\documentclass[a4paper,11pt]{article}

\usepackage[T1]{fontenc}
\usepackage{lmodern}
\usepackage{microtype}
\usepackage[margin=2.5cm]{geometry}
\usepackage{amsmath,amssymb,amsthm,mathtools,bm,mathrsfs}
\usepackage{enumitem}
\usepackage{booktabs}
\usepackage{graphicx}
\usepackage{float}
\usepackage{needspace}
\usepackage{xcolor}
\usepackage{tikz}
\usepackage{hyperref}
\usepackage{aliascnt}
\usepackage[nameinlink,capitalise,noabbrev]{cleveref}

\hypersetup{
  colorlinks=true,
  bookmarksdepth=2,
  linktoc=all,
  linkcolor=blue!55!black,
  citecolor=blue!55!black,
  urlcolor=blue!55!black,
  pdftitle={Pinchoff by surface diffusion},
  pdfauthor={Glen Wheeler}
}

\numberwithin{equation}{section}

\makeatletter
\renewcommand*{\l@section}{\addvspace{.4em}\@dottedtocline{1}{0pt}{1.5em}}
\makeatother

\newtheorem{theorem}{Theorem}[section]
\newtheorem*{conjecture}{Conjecture}
\newaliascnt{proposition}{theorem}
\newtheorem{proposition}[proposition]{Proposition}
\aliascntresetthe{proposition}
\newaliascnt{lemma}{theorem}
\newtheorem{lemma}[lemma]{Lemma}
\aliascntresetthe{lemma}
\newaliascnt{corollary}{theorem}
\newtheorem{corollary}[corollary]{Corollary}
\aliascntresetthe{corollary}
\newaliascnt{claim}{theorem}

\aliascntresetthe{claim}
\newtheoremstyle{spaceddefinition}{\topsep}{\baselineskip}{\normalfont}{}{\bfseries}{.}{.5em}{}
\theoremstyle{spaceddefinition}
\newaliascnt{definition}{theorem}
\newtheorem{definition}[definition]{Definition}
\aliascntresetthe{definition}
\theoremstyle{remark}
\newtheorem*{proofoutlinebody}{Outline of Proof}
\newenvironment{proofoutline}
  {\begin{proofoutlinebody}\renewcommand{\qedsymbol}{\ensuremath{\triangle}}\pushQED{\qed}}
  {\popQED\end{proofoutlinebody}}
\newtheoremstyle{spacedremark}{\topsep}{\baselineskip}{\normalfont}{}{\itshape}{.}{.5em}{}
\theoremstyle{spacedremark}
\newaliascnt{remark}{theorem}
\newtheorem{remark}[remark]{Remark}
\aliascntresetthe{remark}

\newcommand{\statementendmark}{%
  \ifmmode\mathqed\else
    \leavevmode\unskip\nobreak\hfill\enspace\hbox{\qedsymbol}%
  \fi}
\AddToHook{env/definition/begin}{\renewcommand{\qedsymbol}{\ensuremath{\diamond}}\pushQED{\statementendmark}}
\AddToHook{env/definition/end}{\popQED}
\AddToHook{env/remark/begin}{\renewcommand{\qedsymbol}{\ensuremath{\diamond}}\pushQED{\statementendmark}}
\AddToHook{env/remark/end}{\popQED}

\crefname{theorem}{Theorem}{Theorema}
\Crefname{theorem}{Theorem}{Theorema}
\crefname{proposition}{Proposition}{Propositions}
\Crefname{proposition}{Proposition}{Propositions}
\crefname{lemma}{Lemma}{Lemmata}
\Crefname{lemma}{Lemma}{Lemmata}
\crefname{corollary}{Corollary}{Corollaries}
\Crefname{corollary}{Corollary}{Corollaries}
\crefname{definition}{Definition}{Definitions}
\Crefname{definition}{Definition}{Definitions}
\crefname{remark}{Remark}{Remarks}
\Crefname{remark}{Remark}{Remarks}

\newcommand{\R}{\mathbb{R}}
\newcommand{\dd}{\,\mathrm{d}}
\newcommand{\A}{\mathbf{A}}
\newcommand{\tfA}{\mathring{\A}}
\newcommand{\Area}{\operatorname{Area}}
\newcommand{\Vol}{\operatorname{Vol}}
\newcommand{\tr}{\operatorname{tr}}
\newcommand{\dist}{\operatorname{dist}}

\newcommand{\C}{\mathbb{C}}
\newcommand{\N}{\mathbb{N}}
\newcommand{\Sphere}{\mathbb{S}}

\newcommand{\supp}{\operatorname{supp}}
\newcommand{\Span}{\operatorname{span}}
\newcommand{\diag}{\operatorname{diag}}
\newcommand{\opDom}{\operatorname{Dom}}
\newcommand{\initTr}{\operatorname{Tr}}

\newcommand{\graphOp}{\mathcal{G}}
\newcommand{\graphNbhd}{\mathcal{W}}
\newcommand{\radOp}{\mathcal{F}}
\newcommand{\flowSpeed}{\mathcal{V}}
\newcommand{\flowRad}{\mathcal{F}_{\flowSpeed}}
\newcommand{\renArea}{\mathscr{A}_{\rm ren}}
\newcommand{\areaDiss}{\mathscr{D}}
\newcommand{\simTime}{\widehat s}
\newcommand{\pinMan}{\mathcal{M}_{\rm pin}}

\newcommand{\paramSet}{\mathcal{P}}
\newcommand{\coneLin}{\mathcal{A}_\alpha}
\newcommand{\tailRem}{\mathcal{R}_N}
\newcommand{\tailForce}{\mathcal{H}}
\newcommand{\tailInt}{\mathcal{K}}
\newcommand{\tailMap}{\mathcal{T}_{a_{\rm s}}}
\newcommand{\tailDom}{\mathcal{A}_L}
\newcommand{\tailDiag}{\mathsf{D}_0}
\newcommand{\tailCoupling}{\mathcal{B}_1}
\newcommand{\tailLinRem}{\mathcal{R}}
\newcommand{\tailQuad}{\mathcal{Q}}
\newcommand{\coneRem}{\mathcal{R}_{\rm cone}}
\newcommand{\tailGreen}{\mathcal{G}}

\newcommand{\radLin}{\mathcal{A}_U}
\newcommand{\fullLin}{\mathcal{L}_0}
\newcommand{\waistLin}{\mathcal{L}}
\newcommand{\quotLin}{\widehat{\mathcal{L}}}
\newcommand{\minLin}{\mathcal{L}_{0,\min}}
\newcommand{\maxLin}{\mathcal{L}_{0,\max}}
\newcommand{\Pc}{P_{\rm c}}
\newcommand{\Ps}{P_{\rm s}}
\newcommand{\Ec}{E_{\rm c}}
\newcommand{\Es}{E_{\rm s}}
\newcommand{\waistX}{\mathcal{X}}
\newcommand{\waistDom}{\mathcal{D}}
\newcommand{\evenX}{\widetilde X}
\newcommand{\evenY}{\widetilde Y}
\newcommand{\evenDom}{\widetilde D}
\newcommand{\testFns}{\mathcal{Z}}
\newcommand{\dualSpace}{\mathscr{H}}
\newcommand{\distrib}{\mathcal{D}'}

\newcommand{\matD}{\mathfrak{D}_\mu}
\newcommand{\dilPull}{\mathscr{T}_\theta}
\newcommand{\slabLin}{\mathscr{A}_\theta}
\newcommand{\waistOp}{\mathscr{W}_\theta}
\newcommand{\radRem}{\mathcal{R}_F}
\newcommand{\neckNL}{\mathcal{N}}
\newcommand{\waistTr}{\mathcal{T}_k}
\newcommand{\evenTr}{\widetilde{\mathcal{T}}_k}
\newcommand{\coneSize}{\mathfrak{c}}
\newcommand{\boxSize}{\mathfrak{g}}
\newcommand{\coneDiss}{\mathfrak{D}}
\newcommand{\conePhase}{\mathscr{P}}
\newcommand{\slabGraph}{\mathscr{G}}
\newcommand{\slabTr}{\operatorname{tr}\slabGraph}
\newcommand{\slabSrc}{\mathscr{H}}
\newcommand{\coneGraph}{\mathbb{C}}
\newcommand{\coneSrc}{\mathbb{H}}
\newcommand{\traceLift}{\mathscr{E}}
\newcommand{\neckSize}{\mathfrak{X}_\eta}
\newcommand{\algSize}{\mathfrak{a}}
\newcommand{\wkbSize}{\mathfrak{w}}

\newcommand{\matchX}{\widetilde X}
\newcommand{\matchNu}{\widetilde\nu}
\newcommand{\matchSurf}{\widetilde\Sigma}
\newcommand{\refX}{\overline X}
\newcommand{\refNu}{\overline\nu}
\newcommand{\capVariation}{q_{\rm cap}}
\newcommand{\refNuRad}{(\refNu_A)_r}
\newcommand{\refRad}{\overline r}
\newcommand{\refMer}{\overline\gamma}
\newcommand{\refSurf}{\overline\Sigma}
\newcommand{\toNormal}{\mathfrak{N}_A}
\newcommand{\radSpeed}{\mathfrak{R}_A}
\newcommand{\surfSpeed}{\mathcal{S}}
\newcommand{\scaleVel}{\mathcal{T}_A^\sharp}
\newcommand{\refDef}{\mathcal{Q}}
\newcommand{\relNorm}{C^1_{\rm rel}}

\newcommand{\globalHeight}{\mathfrak{u}}
\newcommand{\mixX}{\mathbb{X}}
\newcommand{\mixY}{\mathbb{Y}}
\newcommand{\speedSpace}{\mathbb{Z}_\beta}
\newcommand{\neckPath}{\mathbb{I}_\eta}
\newcommand{\capPath}{\mathbb{O}}
\newcommand{\neckSrc}{\mathbb{Y}_{i,\eta}}
\newcommand{\capSrc}{\mathbb{Y}_o}
\newcommand{\heightRes}{\mathcal{R}}
\newcommand{\mixMap}{\mathscr{S}}
\newcommand{\mixRem}{\mathscr{N}}
\newcommand{\heightFix}{\mathcal{K}_b}
\newcommand{\neckEq}{\mathscr{P}_i}
\newcommand{\compTr}{\mathcal{T}^\#}
\newcommand{\freeTr}{\mathcal{T}_{\rm f}^\#}
\newcommand{\freeProj}{\mathcal{P}_{\rm f}}
\newcommand{\ctrlLift}{\mathcal{J}_{\rm c}}
\newcommand{\volComp}{\mathcal{H}}
\newcommand{\traceVol}{\mathcal{G}}

\title{Pinchoff by surface diffusion}
\author{Glen Wheeler\\
{\small School of Mathematics and Physics, University of Wollongong}\\
{\small Northfields Avenue, Wollongong, NSW 2522, Australia}\\
{\small\texttt{glenw@uow.edu.au}}}
\date{20 November 2026}

\begin{document}
\maketitle

\begin{abstract}
Since Nichols and Mullins analysed the instability of cylindrical wires in 1965, numerical and asymptotic studies have predicted a self-similar neck at pinchoff under surface diffusion.
Yet even the existence of a finite-time singularity for the flow of a smooth closed surface has remained open.
We construct the first such examples, and establish their stability, modulo a finite-dimensional space.
The same shrinking neck closes the central hole of an embedded torus in $\R^3$ and separates the lobes of a genus-zero dumbbell.
In both cases the surface remains embedded until the singular time $T$ and converges smoothly away from the pinching point.
Writing $A(t)$ for the waist radius and $r(z,t)$ for the radius near the neck, we prove that
\[
 A(t)=\{4\mu(T-t)\}^{1/4}(1+o(1)),
 \qquad
 A(t)^{-1}r(A(t)\zeta,t)
 \longrightarrow U(\zeta)
 \quad\text{in }C^\infty_{\mathrm{loc}}(\R).
\]
The profile $U$ above is also for the first time rigorously constructed here.
It was numerically computed by Wong-Miksis-Voorhees-Davis, and identified by Bernoff-Bertozzi-Witelski as the ground state of a family of conical shrinkers.
Our proof also yields a general compactification and stability criterion for fourth-order geometric flows that admit positive symmetric conical shrinkers.
\end{abstract}

\pdfbookmark[1]{\contentsname}{contents}
\tableofcontents
\clearpage

\section{Introduction}

The surface diffusion flow of a smooth closed immersed surface $f_0:\Sigma\to\R^3$ is a one-parameter family of smooth immersions $f:\Sigma\times[0,T)\to\R^3$ satisfying $f(\cdot,0) = f_0$ and
\begin{equation}\label{eq:sdf-intro-new}
 \partial_t f=-(\Delta H)\nu,
\end{equation}
where $H=(\kappa_1+\kappa_2)/2$ is the mean curvature, $\nu$ the unit normal, and $\Delta$ the Laplace-Beltrami operator corresponding to $f(\cdot,t)$.

Mullins proposed surface diffusion flow to model the grooves that form where grain boundaries meet the surface of a heated polycrystal \cite{Mullins1957}.
Mathematically, the flow \eqref{eq:sdf-intro-new} is a fourth-order degenerate quasilinear parabolic geometric evolution that decreases area and preserves (signed, if not embedded) enclosed volume.
It has no comparison or avoidance principle and does not in general preserve embeddedness, convexity, or graphicality.

On a cylindrical wire, surface diffusion amplifies small axisymmetric perturbations whose wavelength exceeds the circumference \cite{NicholsMullins1965,NicholsMullinsTransport1965}.
Nichols computed the evolution of finite rods with rounded ends and found that a sufficiently long rod could break into several particles \cite{Nichols1976}.
Experiments subsequently recorded the breakup of annealed copper and gold nanowires into chains of particles \cite{Karim2006,ToimilMolares2004}.

At fixed axial period, sufficiently thin cylinders are nonlinearly unstable \cite{LeCroneSimonett2013,LeCroneSimonett2016}.
Numerical studies followed perturbed cylinders to breakup \cite{ColemanFalkMoakher1995,ColemanFalkMoakher1996} and exhibited pinchoff in closed dumbbells and tori \cite{Mayer2001}.
Similarity analysis and computations predicted that the shrinking neck approaches a self-similar profile with conical ends \cite{BernoffBertozziWitelski1998,WongMiksisVoorheesDavis1998}.
Yet more than sixty years after Mullins's original proposal, it remained a major open problem to prove that a smooth closed surface could develop a curvature singularity in finite time, by pinchoff or otherwise \cite{DianaFuscoMantegazza2026}.

We address this in the present paper.
Embedded tori (T-type) and genus-zero dumbbells (D-type) give two compact realisations of the same shrinking neck (see \cref{fig:conical-pinch-snapshots}).
In the T-type case the collapsing circle is nonseparating.
In the D-type case it separates the two lobes.
In both cases the waist shrinks at rate $(4\mu(T-t))^{1/4}$, the rescaled neck converges locally smoothly to the same profile $U$, and the flow converges smoothly away from the pinch.
For each admissible outer geometry the initial data form an infinite-dimensional $C^1$ manifold in the rotationally and reflection-symmetric compatible trace class.
At fixed initial waist and volume its finite codimension is the same for both topologies.
We state the geometric result in \cref{thm:main} and describe the initial data and stability in \cref{rem:main-data,rem:main-stability}.

\begin{figure}[htbp]
 \centering
 \begingroup
 \small
 \setlength{\tabcolsep}{4pt}
 \newcommand{\toruspanel}[1]{\raisebox{-.5\height}{%
   \includegraphics[angle=90,width=.195\linewidth]{figures/conical-pinch-#1.pdf}}%
	}
 \newcommand{\dumbbellpanel}[1]{\raisebox{-.5\height}{%
   \includegraphics[width=.273\linewidth]{figures/dumbbell-pinch-#1.pdf}}%
	}
 \begin{tabular}{@{}l@{\hspace{.4em}}ccc@{}}
   & Initial, $t=0$ & Near pinchoff, $t<T$ & Limit, $t=T$ \\[2pt]
   T-type & \toruspanel{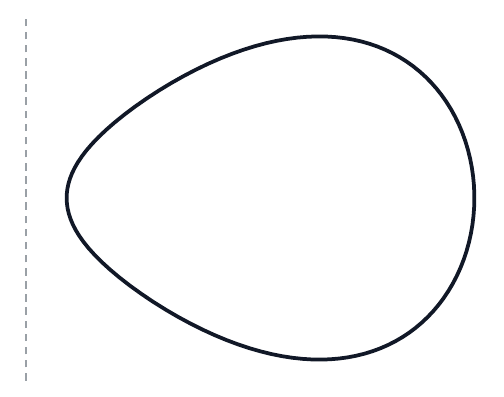}    & \toruspanel{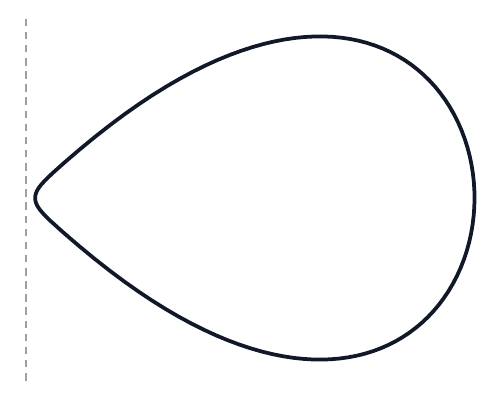}    & \toruspanel{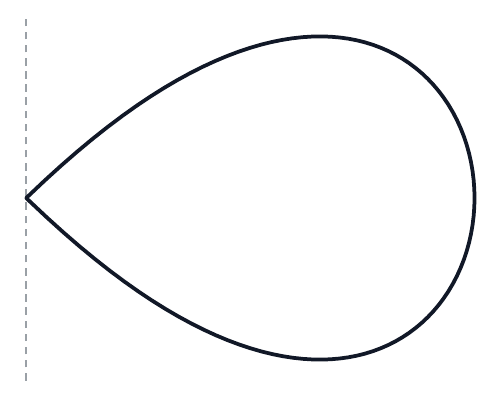}    \\[5pt]
   D-type & \dumbbellpanel{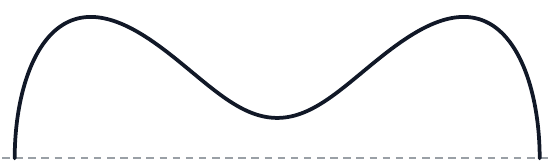} & \dumbbellpanel{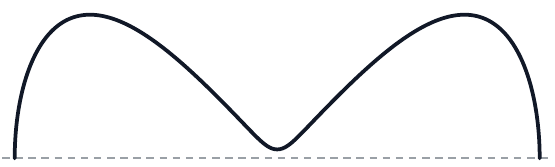} & \dumbbellpanel{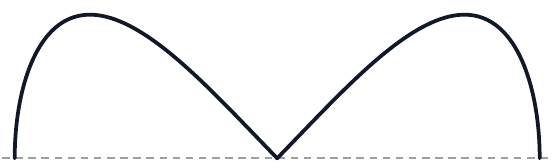}
 \end{tabular}
 \endgroup
 \caption{Schematic meridians of a torus (T-type) and a dumbbell (D-type) with horizontal rotation axis $r=0$.
 The waist radii agree in each column.
 In the D-type row the rescaled profile $AU(z/A)$ rounds the waist for $t<T$.
 At $t=T$ the two lobes meet in a conical point, while the outer poles remain smooth.}
 \label{fig:conical-pinch-snapshots}
\end{figure}

\subsection{History and Related Literature}

Herring described the role of surface tension in sintering \cite{Herring1951}.
Mullins's theory of thermal grooving predicts that groove width and depth grow as $t^{1/4}$ \cite{Mullins1957}, as Mullins and Shewmon confirmed in annealed copper \cite{MullinsShewmon1959}.
Moon and Koo observed elongated cavities in doped tungsten breaking into rows of bubbles during annealing \cite{MoonKoo1971}.
Later experiments used electron microscopy to follow individual gold nanowires through breakup \cite{Schnedlitz2017} and found that crossed wires fragment at their junctions at lower temperatures than isolated wires \cite{Vigonski2018}.

Nichols and Mullins derived the evolution equation for the meridian of a surface of revolution \cite{NicholsMullins1965}.
Using this equation, Coleman, Falk and Moakher showed how short-wave perturbations of a cylinder can excite an unstable long wave \cite{ColemanFalkMoakher1995} and developed a space--time finite element scheme that preserves volume and dissipates area \cite{ColemanFalkMoakher1996}.
In Mayer's computations on triangulated surfaces, a thicker dumbbell neck recovered from its initial contraction and the surface approached a sphere \cite[Figures~1 and~2]{Mayer2001}.

For general surfaces, Cahn and Taylor proposed a minimising-movements scheme \cite{CahnTaylor1994}.
For sufficiently regular initial sets in three dimensions, a localised version of the scheme converges to the classical flow on compact time intervals of smooth existence \cite{CicaleseFuscoJulinKubin2025}.
LeCrone, Shao and Simonett proved local well-posedness for rough normal graphs over uniformly regular hypersurfaces with a uniform tubular neighbourhood \cite{LeCroneShaoSimonett2020}.
A smooth flow can lose embeddedness.
Giga and Ito constructed self-intersections in planar curve diffusion \cite{GigaIto1998}, Mayer and Simonett gave hypersurface examples \cite{MayerSimonett2000}, and Blatt showed loss of convexity and embeddedness for a class of higher-order flows \cite{Blatt2010}.
For complete noncompact planar curves with finite direction energy and nonzero rotation number, Miura and Rupp proved finite-time blow-up under curve diffusion \cite[Corollary~7.9]{MiuraRupp2026}.

McCoy, the author and Williams bounded the lifespan of constrained surface diffusion from below in terms of the initial concentration of curvature \cite{McCoyWheelerWilliams2011}.
The author sharpened the bound for simple constraints \cite{WheelerLifespan2011}.
For surface diffusion itself, the author proved a curvature-concentration criterion for singularities and exponential convergence to a sphere for surfaces of small total trace-free curvature \cite{WheelerNearSpheres2012}.
These arguments draw on the small-energy convergence theorem \cite{KuwertSchaetzle2001} and local lifespan estimates \cite{KuwertSchaetzle2002} of Kuwert and Sch\"atzle for Willmore flow.
Rupp and Schlierf proved finite-time blow-up under Willmore flow for symmetric figure-eight tori of revolution with initial Willmore energy at most $12\pi$ \cite[Theorem~1.1]{RuppSchlierf2026}.
For the flows constructed here, \cref{cor:pinch-curvature-concentration} identifies the origin as a point of curvature concentration and shows that the trace-free curvature energy diverges in every ball about it.
On the rescaled neck, each sufficiently large dyadic annulus contributes an amount of energy bounded below by the same positive constant.
Inspired by this, we present the following conjecture.

\begin{conjecture}
Let $f:\Sigma\times[0,T)\to\R^3$ be a smooth surface diffusion flow of a closed surface with finite maximal time $T$, and suppose $f(\cdot,t)$ is an embedding for every $0\le t<T$.
Suppose $x_*\in\R^3$ is a curvature singularity, in the sense that $f(p_j,t_j)\to x_*$ and $|\A(p_j,t_j)|\to\infty$ for some $p_j\in\Sigma$ and $t_j\uparrow T$.
Then the trace-free second fundamental form $\tfA=\A-Hg$ satisfies
\[
 \lim_{t\uparrow T}\int_{\{p\in\Sigma:\,f(p,t)\in B_\rho(x_*)\}}|\tfA|^2\,\dd\mu_t
 =+\infty\qquad\text{for every }\rho>0.
\]
\end{conjecture}

The existence of a homothetically shrinking immersed torus shows why the embeddedness hypothesis is necessary \cite{McRaeWheelerInPreparation}.
Its curvature blows up at extinction, and its total trace-free curvature energy stays constant because this integral is invariant under dilation.

Elliott and Garcke proved stability near circles \cite{ElliottGarcke1997}.
Escher, Mayer and Simonett proved exponential convergence near spheres \cite{EscherMayerSimonett1998}, and Escher and Mucha extended the result to rougher initial data \cite{EscherMucha2010}.
For axially periodic cylinders, LeCrone and Simonett proved that the bifurcating unduloids are unstable \cite{LeCroneSimonett2013}.
They also established stability and instability of cylinders under perturbations with angular dependence \cite{LeCroneSimonett2016}.
For strictly stable critical sets of area at fixed volume in the flat three-torus, Acerbi, Fusco, Julin and Morini proved nonlinear stability \cite{AcerbiFuscoJulinMorini2019}.
In arbitrary dimension, subsequent proofs use $C^{1,1}$ closeness \cite{DeGennaroDianaKubinKubin2024} and energy estimates \cite{DianaFuscoMantegazza2026}.

For entire graphs, the rough-data theory of Koch and Lamm gives forward self-similar solutions from one-homogeneous data of small slope \cite{KochLamm2012}.
Du and Yip proved their asymptotic stability \cite{DuYip2023}.
Rybka and the author obtained rigidity results for entire planar graphical solitons, with additional hypotheses for expanders \cite{RybkaWheeler2025}.
Giga and Katayama explained why nonlinear small-slope expanders are compatible with this conditional rigidity \cite{GigaKatayama2025}.
For thermal grooves, Asai and Giga constructed self-similar graphs with a prescribed contact angle \cite{AsaiGiga2014}, and Asai and Kohsaka treated the nonlinear no-flux condition \cite{AsaiKohsaka2025}.
For asymptotically one-homogeneous data on higher-dimensional half-spaces, the rescaled flow converges to a self-similar profile when the initial slope is close to that of the plane meeting the boundary at the prescribed contact angle \cite{GigaKatayama2026}.

For a shrinking axisymmetric neck the profile equation also contains the azimuthal curvature.
Wong, Miksis, Voorhees and Davis computed the fundamental profile by shooting and found a half-cone angle of $46.04^\circ$ \cite{WongMiksisVoorheesDavis1998}.
Bernoff, Bertozzi and Witelski identified this solution as the ground state, computed higher profiles, derived their far-field asymptotics and studied their stability formally and numerically \cite[p.~728]{BernoffBertozziWitelski1998}.
They conjectured that every positive solution of the axisymmetric shrinker equation with two conical ends belongs to their countably infinite family of reflection-symmetric profiles, ordered by decreasing cone angle \cite[\S4.1.1, p.~749]{BernoffBertozziWitelski1998}.
This classification remains open.
We call the entire self-similarly shrinking surface generated by the fundamental profile the \emph{WMVD--BBW pincher}, or simply the \emph{ground-state pincher}.
\Cref{thm:certified-positive-conical-profile} gives the first rigorous construction of the ground-state profile.

Our proof draws on singularity constructions for other geometric flows.
In mean curvature flow, Vel\'azquez constructed singularities near minimal cones \cite{Velazquez1994}, and Liu obtained compact examples of this type with bounded mean curvature \cite{Liu2024}.
Degenerate neckpinches were constructed by Angenent and Vel\'azquez for mean curvature flow \cite{AngenentVelazquez1997} and by Angenent, Isenberg and Knopf for Ricci flow \cite{AngenentIsenbergKnopf2015}.
Brendle and Kapouleas used a gluing obstruction to construct an ancient Ricci flow \cite{BrendleKapouleas2017}.
Asymptotically conical shrinkers have been realised in closed Ricci flows \cite{Stolarski2026} and closed mean curvature flows \cite{LeeZhao2024}, where one can also prescribe first-order asymptotics \cite{ChenLeeSunZhao2026}.
In Lagrangian mean curvature flow, Su, Tsai and Wood glued in shrinking Lawlor necks to produce an infinite-time singularity \cite{SuTsaiWood2024}.
Closest to our mode-selection argument is the finite-time Type~II construction of Stolarski and Su \cite{StolarskiSuDynamics2026}, which uses their scale-dependent spectral estimates \cite{StolarskiSuSpectral2026}.

\subsection{Main results}

We prescribe the same conical neck in both constructions.
The \emph{cap} is the part of the terminal surface outside this neck.
A D-type cap consists of two discs, each closing one side of the neck at a smooth pole.
A T-type cap is an annulus joining the two sides.
With the rounded neck in place, the first surface is a genus-zero dumbbell and the second is a torus.

\begin{definition}[Admissible outer cap]\label{def:admissible-outer-cap}
Fix $\alpha,\delta>0$.
An \emph{$(\alpha,\delta)$-admissible terminal meridian} $\Gamma_0$ is a compact simple curve in the half-plane $\{(z,r):r\geq0\}$, invariant under $z\mapsto-z$.
It contains the neck subarc
\[
   (z,\alpha|z|),\qquad -4\delta\leq z\leq4\delta,
\]
and its revolution is smoothly embedded away from the origin.
We allow either of the following geometries.
\begin{description}[leftmargin=\parindent,labelindent=0pt]
	\item[{\it D-type.}] $\Gamma_0$ is an arc with endpoints $(\pm L,0)$ for some $L>0$.
	The open arc from the origin to $(L,0)$ lies in $\{z>0,r>0\}$, and the other half is its reflection.
	Near the positive pole the surface has the form
\[
 (x_1,x_2)\longmapsto(x_1,x_2,L+q_{\rm p}(x_1^2+x_2^2)),
 \qquad q_{\rm p}(0)=0,
\]
with $q_{\rm p}$ smooth.
Thus the two ends are smooth polar discs.
\item[{\it T-type.}] $\Gamma_0$ is a closed loop whose only point on the rotation axis is the origin.
The outer return arc has positive distance from the $\{r=0\}$ axis.
\end{description}
The associated \emph{outer cap} $C_0$ is obtained by deleting the open neck subarc $\{(z,\alpha|z|):|z|<\delta\}$ from $\Gamma_0$ and revolving the remaining meridian.
\end{definition}

\noindent Our main result is the following.

\begin{theorem}[Pinchers exist]\label{thm:main}
There are smooth embedded dumbbells and tori of revolution in $\R^3$ whose surface diffusion flows pinch in finite time $T$.
The surfaces are reflection-symmetric across the waist and remain embedded for every $t<T$.
Curvature blows up at the origin, and the surfaces converge smoothly away from it.
In the D-type case the two limiting lobes meet only at the pinching point, and both outer poles remain smooth.
In the T-type case the collapsing waist closes the central hole.

There is a constant $\mu>0$, common to both topologies, for which the waist radius $A(t)$ satisfies
\begin{equation}\label{eq:main-radius}
 A(t)\sim\{4\mu(T-t)\}^{1/4}\qquad(t\uparrow T).
\end{equation}
After magnification by $1/A(t)$, both necks converge smoothly on compact sets to the same surface of revolution, asymptotic to a double cone.
The terminal neck is asymptotic to this same double cone.
In each topology this profile is stable within a finite-codimensional family of symmetric initial data at fixed waist and volume.
\end{theorem}

The proof also shows how freely the initial surface can be varied while retaining this pinchoff.

\begin{remark}[Choice of cap and initial data]\label{rem:main-data}
Let $\alpha>0$ be the slope of the limiting cone.
Fix $\delta>0$ and any $(\alpha,\delta)$-admissible reference meridian $\Gamma_0$ of either type.
Write $V_0$ for the volume it encloses upon revolution, counting both lobe volumes in the D-type case.
Fix $k\geq10$ and a sufficiently small initial waist radius $A_0>0$.
At fixed volume $V_0$, only $N=\dim \Ec$ neck coordinates need to be selected, where $\Ec$ is the controlled space of \cref{prop:finite-core-splitting}.
The free parameter $h$ ranges over an open ball in the infinite-dimensional complement of the controlled neck directions and the volume direction.
For each $h$, \cref{lem:fully-coupled-selection} determines the $N$ controlled coordinates and the scalar cap adjustment that fixes the volume.
These choices are locally unique and depend $C^1$ on $h$.
Their graph $\pinMan^k$ has codimension $N$ in the fixed-volume class of symmetric compatible traces and $N+1$ before fixing the volume.
Every smooth datum on this graph has the evolution of \cref{thm:main}, with maximal smooth time $T$.

A compatible initial trace records a single normal displacement of the reference surface at $t=0$.
Formula~\eqref{eq:global-height-representatives} expresses it in rescaled coordinates on the neck and physical coordinates on the cap, with prescribed cutoffs on the overlap.
The two representatives recover the original displacement by \eqref{eq:correction-reconstruction}.

The limiting outer cap is a small smooth perturbation of the reference cap.
We vary $\Gamma_0$ by smooth rotationally and reflection-symmetric normal graphs supported away from the neck.
The geometric and linear cap estimates are uniform for variations small in $C^{k+8}$.
Measuring the variations in $C^{k+9}$ gives joint $C^1$ dependence on the cap and the free data for both the parametrisation of $\pinMan^k$ and the limiting outer height in $H^{k+2}$.
The T-type return arc may pass through the same axial values as the neck, so we apply the neck cutoffs only on the neck subarc of \cref{def:admissible-outer-cap}.
At either pole of a D-type cap the height is a smooth invariant function on the whole polar disc.
A common normal tube exists for all sufficiently small cap perturbations.
\end{remark}

To compare neck shapes independently of their pinching times, we compare the flows at equal waist radii (see the stability estimate \eqref{eq:main-relative-stability}).

\begin{remark}[Convergence and stability of the neck]\label{rem:main-stability}
In meridian coordinates write $r(z,t)$ for the radius near the waist, so $A(t)=r(0,t)$.
The smooth even profile $U$ of the limiting neck is normalised by $U(0)=1$.
The geometric convergence in \cref{thm:main} means that
\begin{equation}\label{eq:main-profile}
 \frac{r(A(t)\zeta,t)}{A(t)}\longrightarrow U(\zeta)
 \quad\text{in }C^\infty_{\mathrm{loc}}(\R)\qquad(t\uparrow T).
\end{equation}
The terminal radius $r(z,T):=\lim_{t\uparrow T}r(z,t)$, defined for $z\ne0$ near the origin, satisfies
\begin{equation}\label{eq:main-terminal-slope}
 \lim_{z\to0}\frac{r(z,T)}{|z|}=\alpha,
\end{equation}
where $\alpha$ is the cone slope of $U$.
At fixed cap and initial waist $A_0$, let $h,\tilde h$ be nearby points of the open ball of free coordinates in \cref{rem:main-data}, and let $v_h,v_{\tilde h}$ be the corresponding inner perturbations from \eqref{eq:global-height-representatives}.
Write $s=\mu^{-1}\log(A_0/A)$ for similarity time and take $\eta$ as in \eqref{eq:mixed-weight-choice}.
Then their difference satisfies
\begin{equation}\label{eq:main-relative-stability}
 \|v_h(s)-v_{\tilde h}(s)\|_{H^k_2}
 \leq C e^{-\eta s}\|h-\tilde h\|_{\compTr},\qquad s\geq0 .
\end{equation}
Here $H^k_2$ is the weighted Sobolev space \eqref{eq:pincher-polynomial-space}, and $\|\cdot\|_{\compTr}$ is the quotient norm on compatible initial traces defined in \eqref{eq:compatible-trace-norm}.
The controlled space can include weakly decaying modes to the right of the chosen spectral cutoff.
We leave stability on an open set of symmetric data and stability under perturbations breaking the symmetries as open problems.
\end{remark}

\Needspace{17\baselineskip}
The following theorem establishes the existence of the WMVD--BBW ground-state pincher.
Its profile supplies the shrinking neck in both cases of \cref{thm:main}.

\begin{theorem}[The ground-state pincher]\label{thm:certified-positive-conical-profile}
There exist $\mu>0$ and a smooth even profile $U:\R\to(0,\infty)$ with the following properties.
\begin{enumerate}[label=(\roman*),ref=(\roman*),leftmargin=*,widest=iii]
\item\label{it:ground-state-shrinker} The profile is bounded away from zero and satisfies
\begin{equation}\label{eq:main-profile-equation}
 \radOp[U]+\mu(U-\zeta U')=0,
 \qquad U(0)=1,
\end{equation}
where $\radOp$ is the rotational surface-diffusion operator defined in \eqref{eq:terminal-operator}.
For any $T\in\R$, set $A(t)=\{4\mu(T-t)\}^{1/4}$ for $t<T$.
The meridian $r(z,t)=A(t)U(z/A(t))$ generates a self-similarly shrinking solution of surface diffusion flow.

\item\label{it:ground-state-ends} At both ends the profile is asymptotic to a cone of slope $\alpha>0$.
Its full expansion in the powers $|\zeta|^{1-4n}$, $n=0,1,\ldots$, may be differentiated term by term and begins with
\begin{equation}\label{eq:main-tail}
 U(\zeta)=\alpha|\zeta|
 -\frac{|\zeta|^{-3}}{8\mu\alpha(1+\alpha^2)}
 -\frac{21(11\alpha^2+1)}
 {128\mu^2\alpha^3(1+\alpha^2)^3}|\zeta|^{-7}
 +O(|\zeta|^{-11}).
\end{equation}
\item\label{it:ground-state-bounds} With $\gamma=U''(0)$, the three parameters satisfy the bounds
\begin{align*}
 |\gamma-0.63169185949612983|
 &\le9.248983772691192\times10^{-15},\\
 |\mu-0.030292801498271463|
 &\le1.361237725845061\times10^{-15},\\
 |\alpha-1.037079401503446|
 &\le1.557172962297049\times10^{-12}.
\end{align*}
\end{enumerate}
\end{theorem}

The compactification argument gives a general criterion for realising a conical shrinker as the terminal profile of a compact flow.
We state its geometric content in \cref{thm:general-compactification-intro} below, with the technical hypotheses collected in \cref{rem:general-compactification-hypotheses} following the theorem statement.
We anticipate that it will be of independent interest and enjoy application to other flows.

\begin{theorem}[General compactification criterion]\label{thm:general-compactification-intro}
Consider a homogeneous fourth-order geometric flow in $\R^3$ with a smooth positive even conical shrinking profile $U:\R\to(0,\infty)$, normalised by $U(0)=1$.
Let $C_0$ be a smooth symmetric outer cap matching the cone of $U$, and suppose the flow, profile and cap satisfy the hypotheses in \cref{rem:general-compactification-hypotheses}.

For every sufficiently small initial central radius $A_0>0$, measured at the neck in the plane of reflection, there is a $C^1$ submanifold of finite codimension in the space of rotationally and reflection-symmetric initial data whose smooth members give closed embedded surfaces of revolution pinching in finite time.
The flows remain embedded until pinchoff and converge smoothly away from the origin.
After magnification by the reciprocal central radius, their necks converge smoothly on compact sets to the surface generated by $U$.
The central radius has the same fourth-root collapse law as the prescribed shrinker.
The terminal profile is stable within this family when necks are compared at equal central radii.
Under the additional volume hypotheses in \cref{rem:general-compactification-hypotheses}, the family can also be chosen at fixed enclosed volume.
\end{theorem}

\begin{remark}[Hypotheses and precise conclusions]\label{rem:general-compactification-hypotheses}
Fix an integer $k\ge10$.
The hypotheses of \cref{thm:general-compactification-intro} at this regularity are as follows.
\begin{enumerate}[label=(\roman*),leftmargin=*]
\item The flow belongs to the geometric class of \cref{def:geometric-flow-class}.
Its normal speed is local and quasilinear, respects the geometric symmetries, and scales with degree $-3$ under spatial dilation.
The normal graph equation is parabolic.
\item The profile $U$ is an admissible neck in the sense of \cref{def:admissible-neck}.
This requires the shrinking-profile equation and differentiated conical expansion of \cref{def:normalised-conical-profile}, the linear coefficient bounds of \cref{def:conical-linearisation}, and the local nonlinear estimates in radial and normal coordinates of \cref{def:neck-nonlinear-estimates}.
\item The cap satisfies the geometric cap condition of \cref{def:compactification-hypotheses}(i) and the parabolic cap condition of \cref{def:outer-cap-parabolic}.
These require an injective normal tube with curvature bounds at the local geometric scale, and a uniformly parabolic cap equation with uniform linear and nonlinear estimates.
\end{enumerate}
The family has codimension $N=\dim\Ec$ in the compatible trace space \eqref{eq:compatible-trace-norm} at fixed $A_0$.
Here a trace records one initial normal height, as in \cref{rem:main-data}; the selected family is a $C^1$ graph over an open ball in the free trace space.
The decay estimate \eqref{eq:general-relative-stability} specifies its stability, and \eqref{eq:compact-pincher-radius}--\eqref{eq:compact-pincher-profile} give the collapse rate and rescaled convergence, with $\mu>0$ belonging to the chosen profile.
The parametrisation of the family and the limiting outer height are jointly $C^1$ in the free data and the cap.
Here smooth cap variations are measured in $C^{k+9}$ and the limiting outer height in $H^{k+2}$.

For the fixed-volume conclusion, assume also that the flow preserves volume in the sense of \cref{def:volume-preserving-flow} and the cap satisfies the volume transversality condition of \cref{def:compactification-hypotheses}(ii).
Then the family has codimension $N$ within the hypersurface of data with the volume enclosed by the prescribed terminal surface.
The detailed formulation and proof are given in \cref{thm:compact-pincher-realisation}.
\end{remark}


\subsection{On the proof}

In \cref{sec:geometry-new} we write the axisymmetric equation in terminal similarity variables.
The proof then proceeds in three stages.

First, we construct the entire conical shrinker by shooting from the waist and matching to a solution with a prescribed conical end.
Two modes grow when integrated outwards along the cone.
We impose their far-field conditions by integrating backwards from infinity.
A Volterra fixed-point argument constructs the tail, and an interval-arithmetic degree argument matches it to the solution from the waist.
This gives the \hyperref[proof:ground-state]{proof of \cref*{thm:certified-positive-conical-profile}} at the end of \cref{sec:profile-new}.

Second, we analyse perturbations of a normalised positive conical profile in terminal similarity variables, under hypotheses on the geometric evolution law and its linearisation.
At the conical end the fourth-order part degenerates in logarithmic coordinates and the limiting operator is a first-order outward transport.
This transport does not itself damp perturbations, a major difficulty.
To overcome this we use \emph{weighted spaces} in which the outward shift becomes stable.
In \cref{sec:linear-new} we separate a finite-dimensional controlled space $\Ec$ from an infinite-dimensional stable space.
The convective maximal-regularity and localisation estimates we prove there allow us to choose the initial component in $\Ec$ by prescribing decay as similarity time tends to infinity.

Third, we couple the shrinking core to the outer cap across a fixed interface.
We solve one equation for a normal height on the whole surface, using weighted maximal regularity on the cap and the convective estimates on the neck.
The linear inverse is uniform as the central radius tends to zero.
For a volume-preserving flow, one further scalar equation fixes the enclosed volume.
Together, \cref{sec:linear-new,sec:compact-new} prove the general compactification criterion of \cref{thm:general-compactification-intro}, in the detailed form of \cref{thm:compact-pincher-realisation}.
The local neck estimates are collected in \cref{def:admissible-neck}; the global inverse and the compact pinching flows are conclusions of the criterion.
In \cref{sec:dumbbell-pinchoff}, \cref{prop:sdf-general-hypotheses} verifies the flow and neck hypotheses for surface diffusion and its ground state.
We then verify the cap hypotheses for the two geometries in \cref{prop:db-selection,prop:t-cap-verification}.
For a D-type cap we use Cartesian charts on the two polar discs.
For a T-type cap the annular return stays a positive distance from the axis.
This completes the \hyperref[proof:joint-pinchoff]{proof of \cref*{thm:main}}.

Requiring decay at the chosen exponential rate in similarity time fixes one initial coordinate for each mode in the controlled space $\Ec$.
Thus at fixed waist and volume the selected family has codimension $N=\dim \Ec$.
Our estimates leave $N$ undetermined.
The universality proposed by Wong, Miksis, Voorhees and Davis \cite{WongMiksisVoorheesDavis1998} and the numerical spectra computed by Bernoff, Bertozzi and Witelski \cite[\S5]{BernoffBertozziWitelski1998} both suggest that the ground-state pincher attracts an open set of initial data.
Our argument gives stability within the selected family.
Establishing attraction on an open set would require further analysis.

As this outline suggests, the proof draws on a broad range of classical tools.
We use Fourier multiplier and weighted Sobolev estimates \cite{Grafakos2014}, Volterra integral equations and asymptotic integration \cite{Levinson1948}, maximal regularity \cite{DaPratoGrisvard1975,deSimon1964}, weighted Fredholm and exponential-dichotomy theory \cite{LockhartMcOwen1985,Palmer1988}, and semigroup, spectral and perturbation theory \cite{EngelNagel2000,Gearhart1978,Kato1966,Pruss1984}.
The nonlinear selection step combines a mixed initial--future inverse with finite-dimensional mode selection and one scalar volume equation, while the profile construction is completed using outward-rounded interval arithmetic, validated Taylor integration, a Volterra tail contraction and Poincar\'e--Miranda degree \cite{Miranda1940,Moore1966}.


\subsection{On the appendices}

For results in Sections 2--6 whose detailed arguments are in an appendix, we give a short \emph{Proof} showing how the cited appendix results imply the statement.
An \emph{Outline of Proof} follows for readers who wish to see how the estimates are obtained without working through the appendix.
Where an appendix supplies only part of a proof, we identify that part and carry out the remaining argument in the text.

We briefly give an overview of the appendices here:
\Cref{app:terminal-fredholm} develops the Fredholm theory of the linearised operator and proves decay of the controlled adjoint modes which is critical for the proof of \cref{lem:weighted-fredholm-boundary}.
\Cref{app:convective-terminal-estimates} proves the weighted and conical evolution estimates outlined in \cref{prop:convective-maximal-regularity,prop:terminal-core-semiflow} and the localisation estimate for the nonlinear forcing in \cref{lem:controlled-mode-localisation}.
The radial coefficient calculation in \cref{app:terminal-fredholm} and the product estimates in \cref{app:convective-terminal-estimates} are used in \cref{prop:sdf-general-hypotheses} to check the neck conditions for surface diffusion.
\Cref{app:fixed-interface-details} gives the detailed matching, volume and coordinate estimates for \cref{sec:compact-new}.
It also constructs the inverse for the global height equation used in \cref{lem:fixed-interface-block-bounds}, proves the nonlinear and residual estimates outlined in \cref{lem:section-quadratic-remainder}, and supplies the uniform selection and regularity arguments needed to complete the compactification.
\Cref{app:profile-certificate} supplies the normal-form estimates, parameter dependence and differentiated asymptotics used in \cref{prop:finite-base-tail-chart}.
It also proves the Green-operator estimate and records the interval bounds used in \cref{lem:validated-compactified-tail-bounds,lem:validated-centre-miranda-cube} for the tail at $L=20$ and the signs of the centre--tail mismatch.
It specifies the quantities computed by code and gives instructions for independent reproduction of the calculation.


\section{Geometric setting and terminal scaling}\label{sec:geometry-new}

Fix an orientation of the closed surface $\Sigma$ in \eqref{eq:sdf-intro-new}.
At each time $t$, the pullback metric $g=g(t)=f(\cdot,t)^*g_{\rm Eucl}$ makes $f(\cdot,t)$ an isometric immersion.
We write $\dd\mu$ for the area measure of $g$, $\nu$ for the unit normal compatible with the orientation, $\A$ for the second fundamental form, and
\[
 H=\frac12\tr_g\A
\]
for the averaged scalar mean curvature.
Our sign convention for the Laplace--Beltrami operator of $g$ is
\begin{equation}\label{eq:laplace-sign}
 \int_\Sigma \varphi\Delta\varphi\,\dd\mu
 =-\int_\Sigma|\nabla\varphi|^2\,\dd\mu.
\end{equation}
Reversing the orientation changes the signs of both $H$ and $\nu$ and preserves the velocity in the surface diffusion flow \eqref{eq:sdf-intro-new}.

\begin{lemma}[Area and volume]\label{lem:area-volume}
If $f$ evolves by \eqref{eq:sdf-intro-new}, its area $\Area(f)=\int_\Sigma\dd\mu$ satisfies
\begin{equation}\label{eq:area-dissipation}
 \frac{\dd}{\dd t}\Area(f)
 =-2\int_\Sigma|\nabla H|^2\,\dd\mu.
\end{equation}
Its signed enclosed volume \( \Vol(f)=\tfrac13\int_\Sigma\langle f,\nu\rangle\,\dd\mu \) is constant along the flow,
\begin{equation}\label{eq:volume-preservation}
 \frac{\dd}{\dd t}\Vol(f)=0.
\end{equation}
\end{lemma}

\begin{proof}
Write $V=\langle\partial_t f,\nu\rangle=-\Delta H$ for the normal velocity.
The first variation of area gives $\frac{\dd}{\dd t}\Area(f)=-2\int_\Sigma HV\,\dd\mu$, and integration by parts yields \eqref{eq:area-dissipation}.
The first variation of signed volume gives $\frac{\dd}{\dd t}\Vol(f)=\int_\Sigma V\,\dd\mu=-\int_\Sigma\Delta H\,\dd\mu=0$ since $\Sigma$ is closed.
\end{proof}

In normal graph coordinates we hold the base immersion $f:\Sigma\to\R^3$ fixed and denote its unit normal by $\nu_f$.
For a height $u:\Sigma\to\R$ sufficiently small in $C^1$, the normal graph is the immersion $f^u:\Sigma\to\R^3$ given by
\begin{equation}\label{eq:normal-graph}
 f^u(p)=f(p)+u(p)\nu_f(p).
\end{equation}
For such a height $u\in C^4(\Sigma)$, let $\nu_{f^u}$, $H_{f^u}$ and $\Delta_{f^u}$ denote the unit normal, averaged mean curvature and Laplace--Beltrami operator of $f^u$, expressed on $\Sigma$.
Choose $\nu_{f^u}$ continuously from $\nu_f$, so $\langle\nu_f,\nu_{f^u}\rangle>0$ for small $u$.
If the height varies with time, differentiating \eqref{eq:normal-graph} at fixed $p\in\Sigma$ gives
\[
 \langle\partial_t f^{u(\cdot,t)},\nu_{f^u}\rangle
 =\partial_tu\,\langle\nu_f,\nu_{f^u}\rangle.
\]
Thus the normal velocity prescribed by \eqref{eq:sdf-intro-new} gives the height equation
\begin{equation}\label{eq:surface-diffusion-graph-operator}
 \partial_tu=\graphOp(u),\qquad
 \graphOp(u):=-\frac{\Delta_{f^u}H_{f^u}}
                       {\langle\nu_f,\nu_{f^u}\rangle}.
\end{equation}
We write $h^{s}(\Sigma)$ for the little H\"older space of order $s$.
For each $0<\beta<1$, the operator in \eqref{eq:surface-diffusion-graph-operator} is a smooth map
\[
 \graphOp:\{u\in h^{4+\beta}(\Sigma):\|u\|_{C^1}<\varepsilon_f\}
 \longrightarrow h^\beta(\Sigma),
\]
where $\varepsilon_f>0$ is chosen so that $f^u$ is an immersion and $\langle\nu_f,\nu_{f^u}\rangle>0$ throughout this neighbourhood.
The $C^1$ norm is taken with respect to the fixed metric induced by $f$.

Escher, Mayer and Simonett established local well-posedness for surface diffusion on immersed hypersurfaces \cite[Theorem~1.1]{EscherMayerSimonett1998} using Amann's theory of quasilinear parabolic equations \cite[\S12]{Amann1993}.
Pr\"uss and Simonett give a general treatment of moving interfaces \cite{PrussSimonett2016}, and LeCrone and Simonett treat axisymmetric surface diffusion \cite{LeCroneSimonett2013}.
We use the following normal-graph formulation of Mayer and Simonett, including their regularity result at the initial time \cite[Proposition~2.1]{MayerSimonett2000}.

\begin{theorem}[Local normal-graph semiflow]\label{thm:local-semiflow}
Let $0<\vartheta<1$.
Given a sufficiently small $u_*\in h^{2+\vartheta}(\Sigma)$, there exist $T_0>0$ and an $h^{2+\vartheta}$ neighbourhood $\graphNbhd\subset\{u:\|u\|_{C^1}<\varepsilon_f\}$ of $u_*$ such that, for every $u_0\in\graphNbhd$, surface diffusion has a unique classical solution represented by the normal graphs $f^{u(t;u_0)}$, where
\[
 u(\cdot;u_0)\in
 C([0,T_0],h^{2+\vartheta}(\Sigma))
 \cap C^\infty((0,T_0)\times\Sigma).
\]
The height equation \eqref{eq:surface-diffusion-graph-operator} holds for $t>0$.
The map $(t,u_0)\mapsto u(t;u_0)$ defines a local semiflow, continuous up to $t=0$ and smooth for $t>0$.
If $u_0\in h^{4+\vartheta_0}(\Sigma)$ for some $0<\vartheta_0<\vartheta$, then $u\in C([0,T_0],h^{4+\vartheta_0}(\Sigma))\cap C^1([0,T_0],h^{\vartheta_0}(\Sigma))$ and its initial velocity is
\[
 \partial_t u(0;u_0)=\graphOp(u_0),
\]
with equality in $h^{\vartheta_0}(\Sigma)$ and $\graphOp$ defined in \eqref{eq:surface-diffusion-graph-operator}.
\end{theorem}

To apply the normal-graph theory to surfaces of revolution, we first observe that the symmetries of the initial surface persist under the flow.

\begin{lemma}[Preservation of symmetries]\label{lem:symmetry}
Suppose the initial immersion and reference parametrisation are invariant under a compact group of Euclidean isometries and domain diffeomorphisms.
Then the unique local solution from \cref{thm:local-semiflow} has the same symmetry for as long as the normal-graph representation exists.
\end{lemma}

\begin{proof}
Applying any element of the group to the solution gives another solution of \eqref{eq:sdf-intro-new} by Euclidean invariance.
Both have the same initial immersion and hence agree by uniqueness in \cref{thm:local-semiflow}.
\end{proof}

We now write the flow in coordinates on the neck.
Suppose that on an open axial interval $I$ and a time interval $J$ a meridional arc has smooth positive radius $r:I\times J\to(0,\infty)$.
At time $t$ the neck is parametrised by $f_{r(\cdot,t)}$, where $f_Q:I\times\Sphere^1\to\R^3$ is the immersion obtained by rotating a smooth profile $Q:I\to(0,\infty)$ about the $z$-axis,
\begin{equation}\label{eq:axisymmetric-graph}
 f_Q(z,\theta)=\bigl(Q(z)\cos\theta,Q(z)\sin\theta,z\bigr).
\end{equation}
We introduce three geometric quantities used both in the surface diffusion equation and in the shooting formulation of \cref{prop:pincher-flux-volume}.
For a smooth profile $Q$ as in \eqref{eq:axisymmetric-graph}, write
\begin{equation}\label{eq:profile-geometry}
 \begin{aligned}
 S[Q]&=\sqrt{1+(Q')^2},
 &\qquad&\text{arclength factor},\\
 K[Q]&=\frac{1}{Q S[Q]}-\frac{Q''}{S[Q]^3},
 &&\text{twice the mean curvature},\\
 M[Q]&=\frac{Q}{S[Q]},
 &&\text{flux coefficient},
 \end{aligned}
\end{equation}
where primes denote derivatives in the axial variable.
At each time we apply this notation to $r(\cdot,t)$.
The metric is $S[r]^2\,\dd z^2+r^2\,\dd\theta^2$.
With the unit normal
\begin{equation}\label{eq:axisymmetric-normal}
 \nu=\frac1{S[r]}(-\cos\theta,-\sin\theta,r_z),
\end{equation}
the averaged mean curvature is
\begin{equation}\label{eq:axisymmetric-H}
 H=\tfrac12 K[r].
\end{equation}
For an axisymmetric scalar function $\varphi(z)$,
\begin{equation}\label{eq:axisymmetric-LB}
 \Delta\varphi
 =\frac1{r S[r]}\partial_z\bigl(M[r]\varphi_z\bigr).
\end{equation}
At fixed $(z,\theta)$ the normal component of the graph velocity is $\langle\partial_t f_{r(\cdot,t)},\nu\rangle=-r_t/S[r]$.
Thus \eqref{eq:sdf-intro-new} gives
\begin{equation}\label{eq:profile-equation}
 r_t=\frac1r\partial_z\bigl(M[r]\partial_z H\bigr),
\end{equation}
with $H$ given by \eqref{eq:axisymmetric-H}.

This is the classical surface-of-revolution equation of Nichols and Mullins \cite{NicholsMullins1965}, written in the present conventions.
It is also the equation used in the early numerical and similarity studies \cite{BernoffBertozziWitelski1998,ColemanFalkMoakher1995,ColemanFalkMoakher1996,WongMiksisVoorheesDavis1998}.

To rescale the neck by its central radius, write the radial speed in \eqref{eq:profile-equation} as an operator on profiles.
For a smooth $Q:I\to(0,\infty)$ on an open interval $I\subset\R$, define $\radOp[Q]:I\to\R$ by
\begin{equation}\label{eq:terminal-operator}
 \radOp[Q]=\frac{1}{2Q}\bigl(M[Q]K[Q]'\bigr)'.
\end{equation}

For $a>0$ the dilated profile on $aI$ and its derivative with respect to scale at $a=1$ are
\begin{equation}\label{eq:terminal-dilation-profile}
 Q_a(z)=aQ(z/a),\qquad
 DQ(\zeta):=\left.\partial_a Q_a(\zeta)\right|_{a=1}
 =Q(\zeta)-\zeta Q'(\zeta).
\end{equation}
The same formula defines $Dv$ for any smooth perturbation $v$.

\begin{lemma}[Scaling by the central radius]\label{lem:terminal-scaling}
For every $a>0$ the radial speed satisfies
\[
 \radOp[Q_a](z)=a^{-3}\radOp[Q](z/a),\qquad z\in aI.
\]
Suppose a surface diffusion flow has a positive neck graph $r$ on $I\times J$ as in \eqref{eq:axisymmetric-graph}, where $I$ is symmetric about $0$ and $r$ is even in $z$.
Write $A(t)=r(0,t)$ for the central radius and introduce the rescaled axial coordinate $\zeta$, the profile $P$ and the parabolically rescaled time $\simTime$ by setting
\begin{equation}\label{eq:dynamic-terminal-scaling}
 \zeta=\frac z{A(t)},
 \qquad
 P(\zeta,\simTime)=\frac{r(z,t)}{A(t)},
 \qquad
 \frac{\dd\simTime}{\dd t}=A(t)^{-4}.
\end{equation}
Then $P(0,\simTime)=1$, and on the rescaled interval $\zeta\in I/A(t)$,
\begin{align}
 P_{\simTime}&=\radOp[P]-\radOp[P](0)DP,
 \label{eq:terminal-renormalised-pde}\\
 (\log A)_{\simTime}&=\radOp[P](0).
 \label{eq:terminal-scale-ode}
\end{align}
If $P(\zeta,\simTime)=U(\zeta)$ is stationary and $\radOp[U](0)=-\mu<0$, then $U$ satisfies the normalised profile equation \eqref{eq:main-profile-equation}.
There is a time $T$ for which the central radius satisfies
\begin{equation}\label{eq:exact-similarity-scale}
 A_t=-\mu A^{-3},
 \qquad
 A(t)=\{4\mu(T-t)\}^{1/4},\qquad t\in J.
\end{equation}
Conversely, every smooth positive even solution $U$ of \eqref{eq:main-profile-equation} on $\R$ gives a self-similarly shrinking solution $r(z,t)=A(t)U(z/A(t))$ for $z\in\R$ and $t<T$, with $A$ given by \eqref{eq:exact-similarity-scale}.
\end{lemma}

\begin{proof}
At $\zeta=z/a$ we have $Q_a'(z)=Q'(\zeta)$ and $Q_a''(z)=a^{-1}Q''(\zeta)$.
Thus $K[Q_a](z)=a^{-1}K[Q](\zeta)$ and $M[Q_a](z)=aM[Q](\zeta)$ by \eqref{eq:profile-geometry}.
Substitution in \eqref{eq:terminal-operator}, with $\partial_z=a^{-1}\partial_\zeta$, gives the scaling identity.

Differentiating $r(z,t)=A(t)P(\zeta,\simTime)$ at fixed $z$ and applying the scaling identity to $r_t=\radOp[r]$, we obtain
\[
 A_t DP+A^{-3}P_{\simTime}
 =r_t=A^{-3}\radOp[P].
\]
At $\zeta=0$ we have $P(0,\simTime)=1$ and $P_{\simTime}(0,\simTime)=0$, hence $A^3A_t=\radOp[P](0)$.
Substitution gives \eqref{eq:terminal-renormalised-pde}, and $\dd t/\dd\simTime=A^4$ gives \eqref{eq:terminal-scale-ode}.
For a stationary profile these identities reduce to \eqref{eq:main-profile-equation} and $(A^4)_t=-4\mu$.
Integrating from $t_0\in J$ gives \eqref{eq:exact-similarity-scale} with $T=t_0+A(t_0)^4/(4\mu)$.
Conversely, substitution of $P=U$ and this choice of $A$ gives $r_t=\radOp[r]$.
\end{proof}

We now compare the waist radius and cone angle of the ground state in \cref{thm:certified-positive-conical-profile} with the numerical values reported in \cite{BernoffBertozziWitelski1998,WongMiksisVoorheesDavis1998}.

\begin{remark}[Comparison with the numerical profiles]\label{rem:profile-normalisations}
Bernoff, Bertozzi and Witelski use the sum of the principal curvatures as the mean curvature.
Their radial operator is therefore $2\radOp$, and their profile equation in our notation is $\radOp[R]+\tfrac18(R-\eta R')=0$ \cite[eqs.~(3.5) and~(4.5)]{BernoffBertozziWitelski1998}.
By \cref{lem:terminal-scaling}, dilating $U$ by a factor $c$ changes the coefficient $\mu$ in \eqref{eq:main-profile-equation} to $\mu/c^4$.
Setting $\mu/c^4=1/8$ gives the ground state in their variables,
\begin{equation}\label{eq:classical-profile-normalisation}
 R(\eta)=c\,U(\eta/c),
 \qquad c=(8\mu)^{1/4}.
\end{equation}
Thus the rescaled waist radius is $R(0)=c\approx0.701628717842$.
The dilation preserves the cone slope $\alpha\approx1.0370794015$ and hence the half-cone angle $\arctan\alpha\approx46.042796186^\circ$.
For their numerical ground state, Bernoff, Bertozzi and Witelski report a waist radius of $0.701595$, a cone slope of $1.03714$ and a half-cone angle of $46.0444^\circ$ \cite[Table~I]{BernoffBertozziWitelski1998}.
Wong, Miksis, Voorhees and Davis found a half-cone angle of $46.04^\circ$ for their pre-pinchoff profile \cite{WongMiksisVoorheesDavis1998}.
\end{remark}

\section{Construction of the ground-state pincher}\label{sec:profile-new}

We construct the ground-state profile of \cref{thm:certified-positive-conical-profile} by matching a solution from the waist, where $U(0)=1$, to a conical end.
\Cref{prop:pincher-flux-volume} gives the first-order system used to shoot from the waist and the volume identity satisfied by the profile.
\Cref{prop:pincher-relative-area} relates its area relative to the cone to its curvature dissipation.
\Cref{prop:finite-base-tail-chart} constructs the conical end, and \cref{prop:pincher-profile-certificate} reduces matching to equality of four flux variables.
We use the tail bounds of \cref{lem:validated-compactified-tail-bounds} and the signs on the faces of the parameter cube in \cref{lem:validated-centre-miranda-cube} to find parameters for which the two solutions agree at $\zeta=20$.
These parameters satisfy the bounds in \cref{thm:certified-positive-conical-profile}.
\Cref{fig:ground-state-meridians} illustrates the shrinking meridians and their limiting cone.

\begin{figure}[htbp]
 \centering
 \includegraphics[width=.8\linewidth]{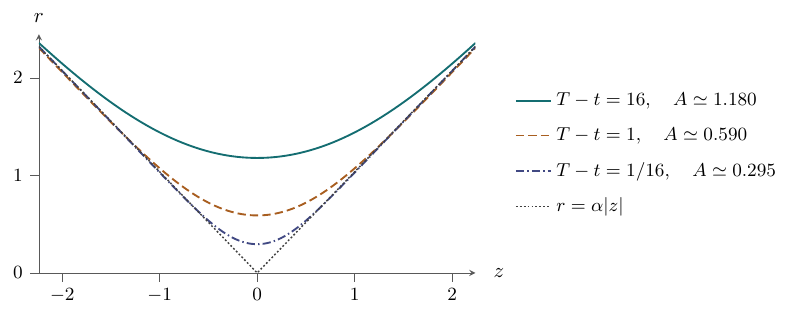}
 \caption{Numerical approximations to the meridians $r(z,t)=A(t)U(z/A(t))$ of the ground-state pincher in \cref{thm:certified-positive-conical-profile}, with $A(t)=\{4\mu(T-t)\}^{1/4}$ as in \eqref{eq:exact-similarity-scale}.
 The dotted line is the limiting cone $r=\alpha|z|$; both coordinate axes have the same scale.}
 \label{fig:ground-state-meridians}
\end{figure}

We begin by determining the first two corrections to the cone from the profile equation \eqref{eq:main-profile-equation}.

\begin{proposition}[Forced conical tail]\label{prop:forced-conical-tail}
Fix $\mu>0$ and $L\ge0$.
Let $P:(L,\infty)\to(0,\infty)$ be smooth and solve the differential equation in \eqref{eq:main-profile-equation}.
Suppose that, as $\zeta\to\infty$,
\begin{equation}\label{eq:conical-tail-ansatz}
 P(\zeta)=\alpha\zeta+c_1\zeta^{-3}+O(\zeta^{-7}),
 \qquad \alpha>0,
\end{equation}
with the expansion differentiated through order four.
This means that the $j$th derivative of the remainder is $O(\zeta^{-7-j})$ for $0\le j\le4$.
Then
\begin{equation}\label{eq:forced-conical-coefficient}
 c_1=-\frac{1}{8\mu\alpha(1+\alpha^2)}<0.
\end{equation}
If $P$ has the finer expansion
\begin{equation}\label{eq:conical-tail-two-term}
 P(\zeta)=\alpha\zeta+c_1\zeta^{-3}
           +c_2\zeta^{-7}+O(\zeta^{-11})
\end{equation}
with remainder of order $O(\zeta^{-11-j})$ after $j$ derivatives, $0\le j\le4$, then
\begin{equation}\label{eq:forced-conical-second-coefficient}
 c_2=-\frac{21(11\alpha^2+1)}
 {128\mu^2\alpha^3(1+\alpha^2)^3}.
\end{equation}
\end{proposition}

\begin{proof}
Substitution of the differentiated expansion \eqref{eq:conical-tail-ansatz} into the radial operator \eqref{eq:terminal-operator} gives
\begin{equation}\label{eq:conical-tail-speed}
 \radOp[P]
 =\frac{1}{2\alpha(1+\alpha^2)}\zeta^{-3}
 -\frac{21(11\alpha^2+1)c_1}{2\alpha^2(1+\alpha^2)^2}\zeta^{-7}
 +O(\zeta^{-11}).
\end{equation}
The dilation term $P-\zeta P'$ is $4c_1\zeta^{-3}+O(\zeta^{-7})$ by \eqref{eq:conical-tail-ansatz}.
Comparing the terms of order $\zeta^{-3}$ in the profile equation \eqref{eq:main-profile-equation} gives \eqref{eq:forced-conical-coefficient}.
Under the finer expansion \eqref{eq:conical-tail-two-term}, the coefficient of $\zeta^{-7}$ in $\mu(P-\zeta P')$ is $8\mu c_2$.
Balancing it against the second term of the radial speed \eqref{eq:conical-tail-speed} and substituting the value of $c_1$ from \eqref{eq:forced-conical-coefficient} gives \eqref{eq:forced-conical-second-coefficient}.
\end{proof}

To shoot from the waist, we write the fourth-order profile equation \eqref{eq:main-profile-equation} as a first-order system.
Its flux also determines the volume of the profile relative to its asymptotic cone.

\begin{proposition}[Flux system and volume neutrality]
\label{prop:pincher-flux-volume}
Fix $\mu>0$ and let $U:\R\to(0,\infty)$ be smooth and even.
We suppress the argument $U$ in the geometric quantities of \eqref{eq:profile-geometry} and write the slope and flux as $q$ and $J$.
Thus
\begin{equation}\label{eq:pincher-flux-variables}
 S=S[U],\quad K=K[U],\quad M=M[U],\qquad
 q=U',\quad J=MK'.
\end{equation}
Restriction to $[0,\infty)$ and the definitions \eqref{eq:pincher-flux-variables} give a one-to-one correspondence between smooth positive even solutions of the normalised profile equation \eqref{eq:main-profile-equation} on $\R$ and solutions $(U,q,K,J):[0,\infty)\to(0,\infty)\times\R^3$ of the first-order system
\begin{align}
 U'&=q,
 &q'&=\frac{1+q^2}{U}-(1+q^2)^{3/2}K,
 \label{eq:pincher-flux-system-a}\\
 K'&=\frac{\sqrt{1+q^2}}{U}J,
 &J'&=-2\mu U(U-\zeta q),
 \label{eq:pincher-flux-system-b}
\end{align}
with centre data
\begin{equation}\label{eq:pincher-flux-centre-data}
 U(0)=1,
 \qquad q(0)=J(0)=0.
\end{equation}
The remaining centre value is $K(0)=1-U''(0)$.
Along each solution the flux satisfies
\begin{equation}\label{eq:pincher-integrated-flux}
 J(\zeta)
 =\mu\left\{
 \zeta U(\zeta)^2-3\int_0^\zeta U(\eta)^2\,\dd\eta
 \right\}.
\end{equation}
If $U$ also has the differentiated conical expansion \eqref{eq:conical-tail-ansatz} at $+\infty$, then
\begin{equation}\label{eq:pincher-volume-neutrality}
 \int_0^\infty
 \bigl(U(\zeta)^2-\alpha^2\zeta^2\bigr)\,\dd\zeta=0.
\end{equation}
\end{proposition}

\begin{proof}
The identities for $q$, $K$ and $J$ in \eqref{eq:pincher-flux-variables} give the first three equations of the flux system \eqref{eq:pincher-flux-system-a}--\eqref{eq:pincher-flux-system-b}.
By the definition of the radial operator \eqref{eq:terminal-operator},
\begin{equation}\label{eq:pincher-flux-speed}
 \radOp[U]=\frac{J'}{2U}.
\end{equation}
Thus the fourth equation of the flux system is precisely the profile equation \eqref{eq:main-profile-equation} multiplied by $2U$.
At the waist, evenness gives $q(0)=J(0)=0$, and $U(0)=1$ gives $K(0)=1-U''(0)$.

Conversely, for a solution of the flux system \eqref{eq:pincher-flux-system-a}--\eqref{eq:pincher-flux-system-b} with $U>0$, the first three equations recover the definitions \eqref{eq:pincher-flux-variables}.
The flux identity \eqref{eq:pincher-flux-speed} then gives the scalar equation \eqref{eq:main-profile-equation}.
The system is invariant under $(\zeta,U,q,K,J)\mapsto(-\zeta,U,-q,K,-J)$.
Since the centre data \eqref{eq:pincher-flux-centre-data} are fixed by this reflection, uniqueness for the smooth first-order system gives a smooth extension to $\R$ with $U,K$ even and $q,J$ odd.

Integrating $J'=-2\mu U(U-\zeta U')$ from the waist, with $J(0)=0$ and $2U(U-\zeta U')=3U^2-(\zeta U^2)'$, gives \eqref{eq:pincher-integrated-flux}.
To obtain the volume identity, put $d(\zeta)=U(\zeta)^2-\alpha^2\zeta^2$.
Using the differentiated expansion \eqref{eq:conical-tail-ansatz} in $d$ and in the flux definition \eqref{eq:pincher-flux-variables} gives, as $\zeta\to+\infty$,
\begin{equation}\label{eq:pincher-conical-flux-limits}
 d(\zeta)=2\alpha c_1\,\zeta^{-2}+O(\zeta^{-6}),\qquad
 J(\zeta)=-\frac{1}{1+\alpha^2}\zeta^{-1}+O(\zeta^{-5}).
\end{equation}
The terms from the cone cancel in the flux formula \eqref{eq:pincher-integrated-flux}, leaving
\begin{equation}\label{eq:pincher-relative-flux}
 \frac{J(\zeta)}\mu
 =\zeta d(\zeta)-3\int_0^\zeta d(\eta)\,\dd\eta.
\end{equation}
By the bounds \eqref{eq:pincher-conical-flux-limits}, $d$ is integrable and both $J(\zeta)$ and $\zeta d(\zeta)$ tend to zero.
Letting $\zeta\to\infty$ in \eqref{eq:pincher-relative-flux} proves the volume identity \eqref{eq:pincher-volume-neutrality}.
\end{proof}

We can also compare the area of the profile with that of its asymptotic double cone.
We show that the difference is finite and compute its variation under dilation to obtain an identity for the dissipation.

\begin{proposition}[Relative area and dissipation identity]
\label{prop:pincher-relative-area}
Let $U:\R\to(0,\infty)$ be a smooth even solution of the normalised profile equation \eqref{eq:main-profile-equation} with the differentiated conical expansion \eqref{eq:conical-tail-ansatz} at $+\infty$.
Write $H_U=K/2$ for its mean curvature, as in \eqref{eq:axisymmetric-H}.
Its relative area and dissipation are
\begin{align}
 \renArea[U]
 &:=
 2\pi\int_{\R}
 \left\{
 US
 -\alpha(1+\alpha^2)^{1/2}|\zeta|
 \right\}\,\dd\zeta,
 \label{eq:pincher-relative-area}\\
 \areaDiss[U]
 &:=
 2\pi\int_{\R}
 M|H_U'|^2\,\dd\zeta.
 \label{eq:pincher-profile-dissipation}
\end{align}
Then the relative area \eqref{eq:pincher-relative-area} and the dissipation \eqref{eq:pincher-profile-dissipation} are finite and satisfy
\begin{equation}\label{eq:pincher-pohozaev}
 \areaDiss[U]
 =\mu\renArea[U]>0.
\end{equation}
\end{proposition}

\begin{proof}
By evenness and the differentiated expansion \eqref{eq:conical-tail-ansatz}, the area integrand in \eqref{eq:pincher-relative-area} is $O(|\zeta|^{-3})$ at both ends.
The curvature definition \eqref{eq:pincher-flux-variables} gives $H_U=O(|\zeta|^{-1})$ and $H_U'=O(|\zeta|^{-2})$, so the dissipation integrand in \eqref{eq:pincher-profile-dissipation} is also $O(|\zeta|^{-3})$.

The cone is invariant under the dilation \eqref{eq:terminal-dilation-profile}.
Changing variables in the relative area \eqref{eq:pincher-relative-area} gives $\renArea[U_\lambda]=\lambda^2\renArea[U]$ for $\lambda>0$.
Differentiating at $\lambda=1$ and integrating by parts gives
\begin{align}
 2\renArea[U]
 &=2\pi\int_{\R}
 \left\{
 S\,DU+\frac{UU'}S(DU)'
 \right\}\,\dd\zeta\nonumber\\
 &=2\pi\int_{\R}UK\,DU\,\dd\zeta
 =4\pi\int_{\R}UH_U\,DU\,\dd\zeta.
 \label{eq:pincher-area-dilation}
\end{align}
Here the boundary term $(UU'/S)DU$ is $O(|\zeta|^{-2})$ by \eqref{eq:conical-tail-ansatz}.
By the flux identity \eqref{eq:pincher-flux-speed} and the profile equation \eqref{eq:main-profile-equation},
\begin{equation}\label{eq:pincher-curvature-drift}
 \radOp[U]
 =\frac1U\left(M H_U'\right)'
 =-\mu DU.
\end{equation}
Multiply \eqref{eq:pincher-curvature-drift} by $UH_U$ and integrate by parts.
The boundary term $M H_UH_U'$ is again $O(|\zeta|^{-2})$, so
\begin{equation}\label{eq:pincher-dissipation-dilation}
 \int_{\R}M|H_U'|^2\,\dd\zeta
 =\mu\int_{\R}UH_U\,DU\,\dd\zeta.
\end{equation}
Combining the dilation identity \eqref{eq:pincher-area-dilation} with \eqref{eq:pincher-dissipation-dilation} proves \eqref{eq:pincher-pohozaev}.
If $\areaDiss[U]=0$, then $H_U$ is constant and \eqref{eq:pincher-curvature-drift} gives $DU\equiv0$.
But $DU(0)=U(0)=1$, which proves strict positivity.
\end{proof}

We next construct a solution on $[L,\infty)$ with prescribed cone slope $\alpha>0$ and rate $\mu>0$.
For $n\ge1$, the dilation term $\mu(1-\zeta\partial_\zeta)$ in the profile equation \eqref{eq:main-profile-equation} multiplies $\zeta^{1-4n}$ by $4n\mu$.
Matching this power gives the recursion
\begin{equation}\label{eq:conical-tail-recursion}
 c_n(\alpha,\mu)
 =-\frac1{4n\mu}
 [\zeta^{1-4n}]
 \radOp\left[
 \alpha\zeta+\sum_{j=1}^{n-1}
 c_j(\alpha,\mu)\zeta^{1-4j}\right].
\end{equation}
Here $[\zeta^m]$ denotes extraction of the coefficient of $\zeta^m$ in the formal Laurent expansion at infinity.
For $n=1,2$ the recursion gives the coefficients in \eqref{eq:forced-conical-coefficient} and \eqref{eq:forced-conical-second-coefficient}.

Bernoff, Bertozzi and Witelski found the conical mode and three exponential modes in their formal WKB analysis of the linearised equation at the cone.
One exponential mode decays and the other two grow \cite[\S4.1]{BernoffBertozziWitelski1998}.
We prescribe the decaying component at the finite point $L$ and impose the two growing-mode conditions at infinity.

\begin{proposition}[Finite-base conical tail chart]
\label{prop:finite-base-tail-chart}
Fix a compact set $\paramSet\Subset(0,\infty)^2$ of pairs $(\alpha,\mu)$ and an integer $N\ge1$.
There exist $L_0,\delta_{\rm tail}>0$, depending only on $\paramSet$ and $N$, such that the following holds for every $L\ge L_0$, $(\alpha,\mu)\in\paramSet$ and $a_{\rm s}\in\R$ with $|a_{\rm s}|<\delta_{\rm tail}/4$.
There is a smooth profile
\[
 U_\infty(\,\cdot\,;\alpha,\mu,a_{\rm s}):[L,\infty)\longrightarrow(0,\infty)
\]
solving the differential equation in \eqref{eq:main-profile-equation}.
Its conical slope is $\alpha$, its stable coordinate at $L$ is $a_{\rm s}$ as in \eqref{eq:finite-base-stable-volterra}, and it has no growing WKB component.
These data determine the profile uniquely among those whose transformed remainder has norm at most $\delta_{\rm tail}$ in \eqref{eq:finite-base-tail-norm}.
With coefficients given by \eqref{eq:conical-tail-recursion}, its expansion is
\begin{equation}\label{eq:finite-base-tail-expansion}
 U_\infty(\zeta)
 =\alpha\zeta+\sum_{n=1}^Nc_n(\alpha,\mu)\zeta^{1-4n}
 +O(\zeta^{-4N-3})\qquad(\zeta\to\infty).
\end{equation}
For each such profile the expansion continues to every order.
After $j$ derivatives the remainder in \eqref{eq:finite-base-tail-expansion} is $O(\zeta^{-4N-3-j})$ for every integer $j\ge0$.
For each fixed $L$, the endpoint map
\[
 (\alpha,\mu,a_{\rm s})\longmapsto
 (U_\infty,q_\infty,K_\infty,J_\infty)(L)\in(0,\infty)\times\R^3
\]
is $C^1$ in the parameters, where the flux variables are defined by \eqref{eq:pincher-flux-variables}.
\end{proposition}

\begin{proof}
The construction is involved and has several moving parts, so we divide the proof into five steps.
In Step~1 we derive the scalar equation and choose coordinates separating the conical and exponential modes.
In Step~2 we prescribe the data through Volterra equations, and in Step~3 we estimate their integral operators.
Step~4 gives the contraction and constructs the tail.
Step~5 completes the parameter dependence and differentiated asymptotics.
The supporting calculations for Step~1 and the regularity arguments for Step~5 are given in \cref{lem:appD-tail-normal-form,lem:appD-tail-regularity}, respectively.

\emph{Step 1. The scalar equation and conical coordinates.}
We now solve the nonlinear profile equation \eqref{eq:main-profile-equation} for $U^{(4)}$.
Write $u=(u_0,u_1,u_2,u_3)=(U,q,r,s)$ for the jet, where $q=U'$, $r=U''$ and $s=U'''$.
Differentiating $K$ in \eqref{eq:profile-geometry} and using \eqref{eq:pincher-flux-variables} gives
\begin{equation}\label{eq:profile-flux-in-jets}
 J=-\frac{q}{US^2}-\frac{qr+Us}{S^4}
      +\frac{3Uqr^2}{S^6}.
\end{equation}
In $J'$, the term $-Us/S^4$ contributes $-UU^{(4)}/S^4$.
Using $(S^2)'=2qr$ to differentiate the flux \eqref{eq:profile-flux-in-jets} and setting $J'=-2\mu U(U-\zeta q)$ by \eqref{eq:pincher-flux-system-b}, we obtain
\begin{equation}\label{eq:scalar-profile-field}
\begin{aligned}
 U^{(4)}=\Phi_\mu(\zeta,u)
 ={}&2S^4\bigg[
 \left(\frac{5qr}{S^6}-\frac{q}{US^4}\right)s
 +\frac{3(1-5q^2)r^3}{2S^8}\\
 &\qquad+\frac{(6q^2-1)r^2}{2US^6}
 +\frac{(q^2-1)r}{2U^2 S^4}
 +\frac{q^2}{2U^3 S^2}
 +\mu(U-\zeta q)\bigg].
\end{aligned}
\end{equation}
For $\zeta\ge L$ and $u\in(0,\infty)\times\R^3$, the first-order equation for the jet is
\begin{equation}\label{eq:tail-jet-vector-field}
 u'=f_\mu(\zeta,u),\qquad
 f_\mu(\zeta,u)=\bigl(u_1,u_2,u_3,\Phi_\mu(\zeta,u)\bigr).
\end{equation}
On the positive cone $Q(\zeta)=\alpha\zeta$, the metric factor in \eqref{eq:profile-geometry} and the fourth-order diffusion coefficient are
\begin{equation}\label{eq:cone-geometric-coefficients}
 S[Q]=\sqrt{1+\alpha^2},\qquad a_\alpha=\frac{1}{2(1+\alpha^2)^2}.
\end{equation}
For a perturbation $v$ of the cone, set
\begin{equation}\label{eq:finite-base-cone-rescaling}
 \kappa=\frac34(\mu/a_\alpha)^{1/3},
 \qquad t=\kappa\zeta^{4/3},
 \qquad v=\zeta^{-5/3}y(t).
\end{equation}
The linearised equation in these coordinates has limiting characteristic polynomial $\lambda^4+\lambda$, as calculated in \cref{lem:appD-tail-normal-form}.
Its roots are $0,-1,\rho_+,\rho_-$, where $\rho_\pm=(1\pm i\sqrt3)/2$.
We subtract the first $N$ corrections before applying the rescaling \eqref{eq:finite-base-cone-rescaling}.
Write
\begin{equation}\label{eq:finite-base-tail-polynomial}
 P_N(\zeta)=\alpha\zeta+
 \sum_{n=1}^Nc_n(\alpha,\mu)\zeta^{1-4n},
 \qquad
 y(t)=\zeta^{5/3}\{U(\zeta)-P_N(\zeta)\}.
\end{equation}
Let $V$ have columns $(1,\lambda,\lambda^2,\lambda^3)^T$ for $\lambda=0,-1,\rho_+,\rho_-$ in that order.
In the coordinates $z=(z_0,z_s,z_+,z_-)^T=V^{-1}(y,y_t,y_{tt},y_{ttt})^T$, \cref{lem:appD-tail-normal-form} gives the equivalent equation
\begin{equation}
 z'=\diag(0,-1,\rho_+,\rho_-)z
    +\frac1tB_1z+\tailRem(t,z),
 \label{eq:finite-base-tail-normal-form}
\end{equation}
with
\begin{equation}
\label{eq:finite-base-tail-remainder}
 B_1=
 \begin{pmatrix}
 2&0&0&0\\
 -2/3&0&0&0\\
 -2/3&0&0&0\\
 -2/3&0&0&0
 \end{pmatrix},
 \qquad
 |\tailRem(t,z)|
 \le C_{\paramSet,N}\{t^{-2}|z|+t^{-1}|z|^2+t^{-3N-2}\}.
\end{equation}
The term $t^{-3N-2}$ in \eqref{eq:finite-base-tail-remainder} bounds the residual of $P_N$ after the cancellations in \eqref{eq:conical-tail-recursion}.

For the parameter dependence in Step~5, take the estimates of \cref{lem:appD-tail-normal-form} uniformly on a compact neighbourhood of $\paramSet$ in $(0,\infty)^2$.
In particular, that lemma supplies constants $\delta_*>0$ and $K\ge1$ such that
\begin{align}
 |\tailRem(t,0)|&\le Kt^{-3N-2},\nonumber\\
 |\tailRem(t,z)-\tailRem(t,\widetilde z)|
 &\le K\bigl[t^{-2}+t^{-1}(|z|+|\widetilde z|)\bigr]|z-\widetilde z|
 \label{eq:tail-remainder-difference}
\end{align}
for $t\ge\kappa L_0^{4/3}$ and $|z|,|\widetilde z|\le\delta_*$.
Here and throughout this proof $|z|$ is the maximum of the four component magnitudes.
On this pointwise ball the reconstructed profile satisfies $U(\zeta)\ge(\min_{\paramSet}\alpha)\zeta/2>0$.

\emph{Step 2. The Volterra equations and weighted space.}
Let $t_0=\kappa L^{4/3}$.
In \eqref{eq:finite-base-tail-normal-form} keep $2z_0/t$ on the centre diagonal and write $\tailForce(t,z)=t^{-1}(B_1z-2z_0e_0)+\tailRem(t,z)$, where $e_0=(1,0,0,0)^T$.
Variation of constants with slope $\alpha$ at infinity, zero growing-mode amplitudes and $z_s(t_0)=a_{\rm s}$ gives
\begin{align}
 z_0(t)&=-t^2\int_t^\infty s^{-2}\tailForce_0(s,z(s))\,\dd s,
 \label{eq:finite-base-centre-volterra}\\
 z_s(t)&=e^{-(t-t_0)}a_{\rm s}
 +\int_{t_0}^t e^{-(t-s)}\tailForce_s(s,z(s))\,\dd s,
 \label{eq:finite-base-stable-volterra}\\
 z_\pm(t)&=-\int_t^\infty e^{\rho_\pm(t-s)}
 \tailForce_\pm(s,z(s))\,\dd s.
 \label{eq:finite-base-unstable-volterra}
\end{align}
Let $\tailMap$ be the map whose components are the right-hand sides of \eqref{eq:finite-base-centre-volterra}--\eqref{eq:finite-base-unstable-volterra}, including the free term $a_{\rm s}e^{-(t-t_0)}$ in the stable component.
Put $r=3N$.
We use the Banach space of continuous $z:[t_0,\infty)\to\C^4$ with $z_0,z_s$ real, $z_-=\overline{z_+}$ and finite norm
\begin{equation}\label{eq:finite-base-tail-norm}
 \|z\|_{r,t_0}
 =\sup_{t\ge t_0}\left(\frac t{t_0}\right)^r
 \max\left\{\frac t{t_0}|z_0(t)|,
 |z_s(t)|,|z_+(t)|,|z_-(t)|\right\}.
\end{equation}

\emph{Step 3. Bounds for the integral operators.}
Write $m(t)=(t_0/t)^r$ for the decay weight in \eqref{eq:finite-base-tail-norm}.
For scalar forcing $h$, let $\tailInt_0h$, $\tailInt_sh$ and $\tailInt_\pm h$ denote the integral terms in the centre equation \eqref{eq:finite-base-centre-volterra}, the stable equation \eqref{eq:finite-base-stable-volterra} and the unstable equations \eqref{eq:finite-base-unstable-volterra}, respectively.
Suppose $|h(s)|\le Hs^{-2}m(s)$ for $s\ge t_0$.
Direct integration of the centre kernel gives
\begin{equation}\label{eq:tail-centre-kernel-bound}
 |(\tailInt_0h)(t)|\le Ht^2\int_t^\infty s^{-4}m(s)\,\dd s
 =\frac{H\,m(t)}{(r+3)t}.
\end{equation}
For the stable integral, $m(s)/m(t)=(t/s)^r\le e^{r(t-s)/t_0}$ on $t_0\le s\le t$.
For the unstable integrals, $m(s)\le m(t)$ on $s\ge t$ and $\Re\rho_\pm=1/2$.
Taking $t_0\ge 2(r+2)$ gives
\begin{align}
 \frac{|(\tailInt_sh)(t)|}{m(t)}
 &\le \frac{H}{t_0^2}\int_{t_0}^t e^{-(1-r/t_0)(t-s)}\,\dd s
 \le\frac{2H}{t_0^2},\nonumber\\
 \frac{|(\tailInt_\pm h)(t)|}{m(t)}
 &\le \frac{H}{t_0^2}\int_t^\infty e^{-(s-t)/2}\,\dd s
 =\frac{2H}{t_0^2}.
 \label{eq:tail-exponential-kernel-bounds}
\end{align}

\emph{Step 4. The contraction and construction of the tail.}
Let $z,\widetilde z$ belong to the closed ball of radius $\delta\le\delta_*$ in \eqref{eq:finite-base-tail-norm}, and put $d=\|z-\widetilde z\|_{r,t_0}$.
Then $|z|,|\widetilde z|\le\delta m$ and $|z-\widetilde z|\le dm$.
Since $t\,m(t)\le t_0$, the difference bound \eqref{eq:tail-remainder-difference} gives
\begin{equation}\label{eq:tail-forcing-difference}
 |\tailRem(t,z(t))-\tailRem(t,\widetilde z(t))|
 \le K(1+2\delta t_0)d\,t^{-2}m(t).
\end{equation}
In the centre equation \eqref{eq:finite-base-centre-volterra}, $\tailForce_0=(\tailRem)_0$, so \eqref{eq:tail-forcing-difference} bounds the entire forcing difference.
Each exponential component also receives $-2(z_0-\widetilde z_0)/(3t)$ from the matrix $B_1$ in \eqref{eq:finite-base-tail-remainder}.
Since $|z_0(t)-\widetilde z_0(t)|\le d(t_0/t)m(t)$ by \eqref{eq:finite-base-tail-norm}, this term contributes at most $(2t_0/3)d\,t^{-2}m(t)$.
Applying \eqref{eq:tail-centre-kernel-bound} and \eqref{eq:tail-exponential-kernel-bounds} gives
\begin{align}
 \|\tailMap z-\tailMap\widetilde z\|_{r,t_0}
 &\le d\,\max\left\{\frac{K(1+2\delta t_0)}{(r+3)t_0},\
 \frac{2}{t_0^2}\left(K(1+2\delta t_0)+\frac{2t_0}{3}\right)\right\}\nonumber\\
 &\le \left(\frac{K+4/3}{t_0}+2K\delta\right)d.
 \label{eq:tail-volterra-contraction}
\end{align}
At $z=0$, the forcing in \eqref{eq:tail-remainder-difference} is bounded by $Kt^{-r-2}=Kt_0^{-r}t^{-2}m(t)$.
The free stable term has norm $|a_{\rm s}|$ since $e^{-(t-t_0)}\le m(t)$.
The same kernel estimates therefore give
\begin{equation}\label{eq:tail-volterra-centre-of-ball}
 \|\tailMap0\|_{r,t_0}\le |a_{\rm s}|+Kt_0^{-r-1}.
\end{equation}
Choose $0<\delta\le\min\{\delta_*,1/(8K)\}$.
This is the radius $\delta_{\rm tail}$ in the statement.
Increase $L_0$ so that $t_0\ge\max\{2(r+2),4(K+4/3)\}$ and $Kt_0^{-r-1}\le\delta/4$ throughout the compact parameter neighbourhood chosen in Step~1.
For $|a_{\rm s}|\le\delta/4$, \eqref{eq:tail-volterra-contraction} gives Lipschitz constant at most $1/2$ and \eqref{eq:tail-volterra-centre-of-ball} gives $\|\tailMap0\|_{r,t_0}\le\delta/2$.
Hence $\|\tailMap z\|_{r,t_0}\le\delta$ on the ball.
The contraction mapping theorem gives the unique fixed point there.
Differentiating the Volterra equations recovers the jet equation \eqref{eq:tail-jet-vector-field}, so the reconstructed positive profile is smooth and solves \eqref{eq:main-profile-equation}.

At the fixed point $z=\tailMap z$, the contraction estimate \eqref{eq:tail-volterra-contraction} gives $\|z\|_{r,t_0}\le \tfrac12\|z\|_{r,t_0}+\|\tailMap0\|_{r,t_0}$.
Hence \eqref{eq:tail-volterra-centre-of-ball} gives the fixed-point estimate
\begin{equation}\label{eq:tail-fixed-point-bound}
 \|z\|_{r,t_0}\le 2\bigl(|a_{\rm s}|+Kt_0^{-r-1}\bigr)\le\delta.
\end{equation}

By the fixed-point estimate \eqref{eq:tail-fixed-point-bound}, the weights in \eqref{eq:finite-base-tail-norm} give $z_0=O(t^{-r-1})$ and $|z|=O(t^{-r})$.
The forcing $-2z_0/(3t)+(\tailRem)_j$ in each exponential component $j=s,+,-$ is therefore $O(t^{-r-2})$ by \eqref{eq:finite-base-tail-remainder}.
The stable and unstable integrals \eqref{eq:finite-base-stable-volterra}--\eqref{eq:finite-base-unstable-volterra} preserve this decay by the same exponential kernel estimate with weight $t^{r+2}$, since $t_0\ge2(r+2)$.
The free stable term decays exponentially.
Hence $z=O(t^{-3N-1})$, and the change of variables \eqref{eq:finite-base-tail-polynomial} gives $U_\infty-P_N=O(\zeta^{-4N-3})$.

\emph{Step 5. Parameter dependence and differentiated asymptotics.}
The uniform contraction and fixed-point bounds established in Step~4 are the hypotheses of \cref{lem:appD-tail-regularity}.
Part~\ref{it:appD-tail-parameters} of that lemma gives $C^1$ dependence of the endpoint at fixed $L$.
Part~\ref{it:appD-tail-expansion} shows that this same profile has the expansion \eqref{eq:finite-base-tail-expansion} to every order, with the stated remainder after each derivative.
This completes the proof.
\end{proof}

We wish to match a conical tail to the solution starting at the waist.
For this, we calculate four matching conditions at $\zeta=L$ using the four flux variables in \eqref{eq:pincher-flux-variables}.

\begin{proposition}[Finite centre--tail matching]
\label{prop:pincher-profile-certificate}
Fix a matching point $L>0$ and a set of tail parameters $\tailDom\subset(0,\infty)^2\times\R$.
Suppose each $(\alpha,\mu,C_L)\in\tailDom$ gives a smooth tail $U_\infty(\cdot;\alpha,\mu,C_L):[L,\infty)\to(0,\infty)$ solving the differential equation in \eqref{eq:main-profile-equation} with the differentiated conical expansion \eqref{eq:finite-base-tail-expansion}.
Here $C_L$ is the stable datum at $L$ in the coordinates used to construct the family.
For centre parameters $\gamma\in\R$ and $\mu>0$, solve the flux system \eqref{eq:pincher-flux-system-a}--\eqref{eq:pincher-flux-system-b} with initial data
\begin{equation}\label{eq:pincher-centre-shooting-data}
 Y_{\rm c}(0;\gamma,\mu)=(1,0,1-\gamma,0).
\end{equation}
Suppose the centre solution $Y_{\rm c}:[0,L]\to(0,\infty)\times\R^3$ exists on the whole interval.
For $(\alpha,\mu,C_L)\in\tailDom$, write $Y_\infty:[L,\infty)\to(0,\infty)\times\R^3$ for the tail in the flux variables \eqref{eq:pincher-flux-variables}.
If the flux vectors agree at the matching point,
\begin{equation}\label{eq:pincher-matching-map}
 Y_{\rm c}(L;\gamma,\mu)=Y_\infty(L;\alpha,\mu,C_L),
\end{equation}
then the two solutions determine a smooth even profile $U:\R\to(0,\infty)$ satisfying \eqref{eq:main-profile-equation}, with $U''(0)=\gamma$ and the differentiated expansion \eqref{eq:finite-base-tail-expansion} at the positive end.
\end{proposition}
\begin{proof}
The flux system \eqref{eq:pincher-flux-system-a}--\eqref{eq:pincher-flux-system-b} has a smooth right-hand side on $(0,\infty)\times\R^3$, so uniqueness with the common data \eqref{eq:pincher-matching-map} joins the solutions smoothly at $L$.
The centre data \eqref{eq:pincher-centre-shooting-data} and \cref{prop:pincher-flux-volume} give the smooth even extension across $\zeta=0$.
Its positive end is the given tail, so \eqref{eq:finite-base-tail-expansion} applies there.
\end{proof}

To apply \cref{prop:pincher-profile-certificate} at $L=20$, we use the same coordinates for the centre and tail jets.
Put $x=\zeta^{4/3}$, $x_0=L^{4/3}$ and $w=x^{-1/4}$, and take $\kappa$ from \eqref{eq:finite-base-cone-rescaling}.
For either profile $U$ write $V(x)=U(\zeta)-P_4(\zeta)$, where $P_4$ is the truncation in \eqref{eq:finite-base-tail-polynomial}.
Define the real matrix
\begin{equation}
\mathsf A(x)=
\begin{pmatrix}
1 & 1-\dfrac{5w^4}{2\kappa}
  & 2+\dfrac{5w^4}{2\kappa} & \dfrac{5\sqrt3w^4}{2\kappa}\\[1.2ex]
\dfrac34w^4 & -\kappa+\dfrac54w^4
  & \kappa+\dfrac52w^4 & -\sqrt3\kappa\\[1.2ex]
0&\kappa^2&-\kappa^2&-\sqrt3\kappa^2\\
0&\kappa^2(-\kappa-\dfrac54w^4)
 &\dfrac{\kappa^2(-8\kappa+5w^4)}4
 &\dfrac{5\sqrt3\kappa^2w^4}4
\end{pmatrix}.
\label{eq:appD-real-transform}
\end{equation}
Its determinant is $6\sqrt3\,\kappa^6>0$.
The jet of $V$ thus determines the matching coordinates $Z=(Z_0,Z_{\rm s},Z_+,Z_-)$ uniquely by
\begin{equation}\label{eq:validated-tail-coordinates}
 (V,V_x,V_{xx},V_{xxx})^{\mathsf T}
 =\mathsf A(x)(Z_0,Z_{\rm s},\Re Z_+,\Im Z_+)^{\mathsf T},
 \qquad Z_-=\overline{Z_+}.
\end{equation}
Here $C_L=Z_{\rm s}(x_0)$ is the stable datum for matching at $L=20$.
In \cref{prop:finite-base-tail-chart} the corresponding datum $a_{\rm s}$ was taken in the rescaled $t$-jet of \eqref{eq:finite-base-stable-volterra}.

A single run of \path{run_certificate.py} supplied the numerical estimates for the next two lemmata.
After verifying the tail cancellation symbolically, the program used $256$-bit Arb interval arithmetic \cite{Moore1966,Johansson2017} to enclose the tail on $[20,\infty)$ and the centre solution on $[0,20]$.
The tail estimates are recorded in \eqref{eq:appD-run-forcing}--\eqref{eq:appD-run-contraction}, and the centre estimates, including the six face signs, in \eqref{eq:appD-run-slabs}--\eqref{eq:appD-run-faces}.
The file \path{certificate_receipt.json} contains the output of this run.
The \hyperref[app:certificate-replay]{replay instructions in Appendix~\ref*{app:profile-certificate}} give the commands to check this record and reproduce the computation.

\Needspace{10\baselineskip}
\begin{lemma}[Validated conical tails at $L=20$]
\label{lem:validated-compactified-tail-bounds}
Fix $L=20$.
For each $(\alpha,\mu,C_L)$ satisfying
\begin{equation}\label{eq:validated-tail-parameter-box}
 1.0370\le\alpha\le1.0372,
 \qquad0.03028\le\mu\le0.03031,\qquad |C_L|\le10^{-11},
\end{equation}
there is a smooth tail $U_\infty:[L,\infty)\to(0,\infty)$ solving the differential equation in \eqref{eq:main-profile-equation}.
It has the full differentiated conical expansion \eqref{eq:finite-base-tail-expansion} of slope $\alpha$ and satisfies $U_\infty(\zeta)/\zeta>1.0369$ for $\zeta\ge L$.

Its endpoint $Z(x_0)$ in the coordinates \eqref{eq:validated-tail-coordinates} depends $C^1$ on $(\alpha,\mu,C_L)$ and satisfies
\begin{align}
 Z_{\rm s}(x_0)&=C_L,\label{eq:validated-stable-base}\\
 \max\{|Z_0(x_0)|,|Z_+(x_0)|,|Z_-(x_0)|\}
 &\le 1.9073486328125\times10^{-11}.
 \label{eq:validated-tail-base-displacement}
\end{align}
\end{lemma}
\begin{proof}
\Cref{prop:finite-base-tail-chart} supplies a tail family for all sufficiently large $L$.
We must show that we can take $L=20$.
For this we use the recorded bounds \eqref{eq:appD-run-forcing}--\eqref{eq:appD-run-tube} to estimate the transformed equation (Step~1), then prove that the Volterra map is a contraction on $x\ge x_0=20^{4/3}$ (Step~2).
Finally, we verify the endpoint bounds and parameter dependence, and recover the full differentiated expansion by identifying the resulting tail near infinity with a member of that family (Step~3).

\emph{Step 1. Bounds for the transformed equation.}
Put $\beta=57/4$.
On the real subspace $Z_0,Z_{\rm s}\in\R$, $Z_-=\overline{Z_+}$, use the norm $|Z|_{\rm c}^2=|Z_0|^2+|Z_{\rm s}|^2+2|Z_+|^2$ and its induced operator norm.
In the coordinates \eqref{eq:validated-tail-coordinates}, the profile equation becomes
\begin{equation}\label{eq:validated-tail-system}
 Z_x=(\Lambda+\tailDiag/x)Z+\tailLinRem(x)Z+\tailQuad(x,Z)+g(x),
\end{equation}
where $\Lambda=\diag(0,-\kappa,\kappa\rho_+,\kappa\rho_-)$, $\tailDiag=\diag(3/4,-5/4,-5/4,-5/4)$ and $\kappa$ is defined in \eqref{eq:finite-base-cone-rescaling}.
Write $f_Z[\overline U]$ for the vector field in these coordinates after subtracting a background $\overline U$.
For $\overline U=P_4$, its forcing $g$ at $Z=0$, linear remainder $\tailLinRem$ and nonlinear Taylor remainder $\tailQuad$ are given in \eqref{eq:appD-field-remainders}.

For the parameters in \eqref{eq:validated-tail-parameter-box} and $x\ge x_0$, we claim
\begin{equation}\label{eq:validated-tail-field-bounds}
\begin{aligned}
 \|\tailLinRem(x)\|&\le100x^{-2},\\
 |g(x)|_{\rm c}&\le4\times10^{14}x^{-\beta-1},\\
 |\tailQuad(x,Z)-\tailQuad(x,\widetilde Z)|_{\rm c}
 &\le20x^{1/4}(|Z|_{\rm c}+|\widetilde Z|_{\rm c})|Z-\widetilde Z|_{\rm c},
\end{aligned}
\end{equation}
where $Z,\widetilde Z$ lie in the ball of radius $(10^{14}+10^{-11}x_0^\beta)x^{-\beta}$ in the conjugate-pair norm.

The recorded linear estimate \eqref{eq:appD-run-linear} gives $x^2\|\tailLinRem(x)\|<44.003107<100$.
The forcing estimate \eqref{eq:appD-run-forcing} gives $x^{\beta+1}|g(x)|_{\rm c}<3.101931095\times10^{14}<4\times10^{14}$.
These are the first two inequalities in \eqref{eq:validated-tail-field-bounds}.

For the third inequality, the recorded Hessian bound in \eqref{eq:appD-run-tube} applies throughout the ball.
Since $D_Z\tailQuad(x,0)=0$, integration along the segment from $0$ to $Y$ gives $\|D_Z\tailQuad(x,Y)\|\le8.632301x^{1/4}|Y|_{\rm c}$.
Integrating this bound from $\widetilde Z$ to $Z$ gives the last line of \eqref{eq:validated-tail-field-bounds}.

\emph{Step 2. The contraction at $L=20$.}
Write $e_{\rm s}=(0,1,0,0)^T$ for the stable coordinate vector and split off the homogeneous stable solution of the diagonal part of \eqref{eq:validated-tail-system} by putting
\begin{equation}\label{eq:validated-stable-split}
 \Phi_{\rm s}(x)=e^{-\kappa(x-x_0)}(x/x_0)^{-5/4},
 \qquad
 Z=C_L\Phi_{\rm s}e_{\rm s}+Z_{\rm r}.
\end{equation}
We impose $(Z_{\rm r})_{\rm s}(x_0)=0$, so that the stable endpoint has the prescribed value \eqref{eq:validated-stable-base}.
Let $\tailGreen$ be the solution operator for $W_x=(\Lambda+\tailDiag/x)W+H$, integrating the stable component from $x_0$ with zero initial value and the other three components from infinity.
In the decomposition \eqref{eq:validated-stable-split}, the Volterra fixed-point equation for \eqref{eq:validated-tail-system} is
\begin{equation}\label{eq:validated-tail-volterra}
 Z_{\rm r}=\tailGreen\{\tailLinRem Z+\tailQuad(\cdot,Z)+g\}.
\end{equation}
We seek $Z_{\rm r}$ in the closed ball of continuous remainders with $\|Z_{\rm r}\|_\beta:=\sup_{x\ge x_0}x^\beta|Z_{\rm r}(x)|_{\rm c}\le R$, where $R=10^{14}$.
To estimate the map on this ball, put
\[
 \Gamma_L=\left\{\frac1{(\beta+3/4)^2}
                    +\frac4{(\kappa x_0)^2}\right\}^{1/2},
 \qquad
 \rho=|C_L|+Rx_0^{-\beta}.
\]
The bounds in \eqref{eq:appD-run-contraction} give $\Gamma_L<0.102$ and $\kappa x_0>26.01517>2(\beta-5/4)=26$.
In particular, the stable solution in \eqref{eq:validated-stable-split} satisfies $|\Phi_{\rm s}(x)|\le(x/x_0)^{-\beta}$ since $\kappa(x-x_0)\ge(\beta-5/4)\log(x/x_0)$.
Since $\kappa x_0>26$, \cref{lem:appD-tail-green-bound} applies and gives
\begin{equation}\label{eq:validated-volterra-operator-bound}
 \|\tailGreen H\|_\beta
 \le\Gamma_L\sup_{x\ge x_0}x^{\beta+1}|H(x)|_{\rm c}.
\end{equation}
Put $K_2=100$, $K_g=4\times10^{14}$ and $K_N=20$, the constants in \eqref{eq:validated-tail-field-bounds}.
Applying the operator estimate \eqref{eq:validated-volterra-operator-bound} to the Volterra map \eqref{eq:validated-tail-volterra} bounds its Lipschitz constant and the relative radius of its image by
\begin{align}
 L_{\rm tail}
 &\le\Gamma_L\left\{\frac{K_2}{x_0}
       +2K_Nx_0^{5/4}\rho\right\},\nonumber\\
 B_{\rm tail}
 &\le\frac{\Gamma_L}{R}\left\{
  K_g+K_2|C_L|x_0^{\beta-1}+K_Nx_0^{\beta+5/4}\rho^2\right\}
 +\frac{\Gamma_LK_2}{x_0}.
 \label{eq:validated-volterra-ratios}
\end{align}
The recorded evaluation of \eqref{eq:validated-volterra-ratios} for $|C_L|\le10^{-11}$ gives $L_{\rm tail}<0.201$ and $B_{\rm tail}<0.742$ by \eqref{eq:appD-run-contraction}.
Thus the Volterra map takes the ball into itself and is a strict contraction there.
Its unique fixed point gives a smooth tail $U_\infty$.

\emph{Step 3. Endpoint bounds and differentiated asymptotics.}
The tube bound in \eqref{eq:appD-run-tube} gives $U_\infty(\zeta)/\zeta>1.0369$ for $\zeta\ge20$.

The parameter argument in part~\ref{it:appD-tail-parameters} of \cref{lem:appD-tail-regularity} applies here: the coefficients are smooth in the parameters after compactification, and differentiating an exponential kernel introduces at most a factor $O(1+|x-y|)$.
The kernels remain integrable with this factor in the weighted norm.
The Volterra map is therefore $C^1$ on the common ball, and the uniform contraction theorem \cite[Theorem~1.174]{Chicone1999} gives $C^1$ dependence of its fixed point and hence of the endpoint $Z(x_0)$.

At $x=x_0$ the free term in \eqref{eq:validated-stable-split} contributes only to $Z_{\rm s}$.
Each of the other three components is bounded by $|Z_{\rm r}(x_0)|_{\rm c}\le R x_0^{-\beta}=10^{14}/20^{19}$, which gives \eqref{eq:validated-tail-base-displacement}.

To recover the differentiated expansion \eqref{eq:finite-base-tail-expansion}, restrict the fixed point to a larger base.
By \eqref{eq:validated-tail-coordinates} and \eqref{eq:validated-stable-split}, $V$ and its first three $x$-derivatives are $O(x^{-57/4})$.
Thus $y=\zeta^{5/3}V=x^{5/4}V$ and its first three $t$-derivatives are $O(t^{-13})$ for $t=\kappa x$.
For a sufficiently large base the norm \eqref{eq:finite-base-tail-norm} with $N=4$ is therefore small.
The uniqueness in \cref{prop:finite-base-tail-chart} identifies the tail in that chart using the stable coordinate of its jet at the new base.
\end{proof}

\Needspace{16\baselineskip}
We next propagate the centre data \eqref{eq:pincher-centre-shooting-data} to $\zeta=20$ and compare the resulting jet with the tail endpoint in \eqref{eq:validated-tail-base-displacement}.

\begin{lemma}[Signs of the centre--tail mismatch]
\label{lem:validated-centre-miranda-cube}
Fix $L=20$.
There is an affine map $p:[-1,1]^3\to\R^3$, $p(\xi)=(\gamma,\mu,\alpha)$, whose image lies in the three parameter intervals in part~\ref{it:ground-state-bounds} of \cref{thm:certified-positive-conical-profile}.
For each $\xi$ the centre data \eqref{eq:pincher-centre-shooting-data} give a smooth solution $Y_{\rm c}:[0,L]\to(0,\infty)\times\R^3$ of the flux system \eqref{eq:pincher-flux-system-a}--\eqref{eq:pincher-flux-system-b}.
Let $Z^{\rm c}(\xi)$ be its endpoint in the coordinates \eqref{eq:validated-tail-coordinates}.
Taking $C_L=Z^{\rm c}_{\rm s}(\xi)$ places $(\alpha,\mu,C_L)$ in the box \eqref{eq:validated-tail-parameter-box}.
Denote the corresponding tail endpoint from \cref{lem:validated-compactified-tail-bounds} by $Z^\infty(\xi)$.

The mismatch in the remaining three real coordinates,
\begin{equation}\label{eq:validated-exact-mismatch}
 M(\xi)=\pi\bigl(Z^{\rm c}(\xi)-Z^\infty(\xi)\bigr),
 \qquad \pi Z=(Z_0,\Re Z_+,\Im Z_+),
\end{equation}
is continuous on $[-1,1]^3$ and satisfies
\begin{equation}\label{eq:validated-miranda-face-sign}
 \sigma M_i(\xi)>0 \qquad\text{when }\xi_i=\sigma,\quad
 i=1,2,3,\quad \sigma\in\{-1,1\}.
\end{equation}
\end{lemma}
\begin{proof}
\Cref{app:profile-certificate} records the enclosures obtained by integrating from the centre.
We show why the six recorded face signs apply to the mismatch of the exact centre and tail solutions.
We take the parameter cube
\begin{equation}\label{eq:validated-centre-box}
 p(\xi)=p_0+rB\xi,
 \qquad \xi\in[-1,1]^3.
\end{equation}
Here $r=3\times10^{-11}$, $B$ is the rational matrix \eqref{eq:validated-centre-matrix} and
\[
 p_0=(0.63169185949612983,
      0.030292801498271463,
      1.037079401503446).
\]
The $i$th coordinate of $p(\xi)$ ranges over an interval of half-width $r\sum_j|B_{ij}|$.
In the order $(\gamma,\mu,\alpha)$ the outward-rounded upper bounds are
\begin{equation}\label{eq:validated-parameter-radii}
 (9.248983772691192\times10^{-15},
  1.361237725845061\times10^{-15},
  1.557172962297049\times10^{-12}).
\end{equation}
Together with $p_0$ these give the three parameter bounds in the statement.
In particular, $(\alpha,\mu)$ lie in the intervals of \eqref{eq:validated-tail-parameter-box}.

To use the face bounds \eqref{eq:appD-run-faces}, we first show that, for every $\xi$ in the cube, the centre solution reaches $L=20$ and its stable coordinate lies in the range of \cref{lem:validated-compactified-tail-bounds}.
The centre stage of the recorded run, \path{centre_validated.py}, covered $[0,20]$ by $400$ slabs of length $h=1/20$, retaining the three cube coordinates as degree-three Taylor-model variables \cite{BerzMakino1998}.
On each slab it constructed a polynomial $P$ of degree $16$ in the integration variable.
Writing $f$ for the right-hand side of the flux system \eqref{eq:pincher-flux-system-a}--\eqref{eq:pincher-flux-system-b}, the program used the shifted expansion \eqref{eq:appD-shifted-polynomial} to bound the defect $P'-f(\zeta,P)$, including the omitted time tail.
The resulting bounds are recorded in \eqref{eq:appD-run-slab-extrema}.

To see why the integration encloses the true solution, let $E_i^0$ bound the incoming error, $R_i$ the defect in component $i$, and $E_i$ the proposed error radius on the slab.
Let $A_{ij}$ bound $|\partial_j f_i|$ throughout this tube.
As long as the solution stays in the tube, the integral equation gives
\begin{equation}\label{eq:centre-slab-error}
 |(Y_{\rm c})_i(\zeta)-P_i(\zeta)|
 \le E_i^0+h\Bigl(R_i+\sum_j A_{ij}E_j\Bigr).
\end{equation}
On every slab the program verified \eqref{eq:appD-run-slabs}, so the right-hand side is strictly less than $E_i$ for every component.
At a first exit we would have $|(Y_{\rm c})_i-P_i|=E_i$ for some $i$, so the solution stays in the tube.
Writing $U_{\rm c}$ for the first component of $Y_{\rm c}$, the lower bound $U_{\rm c}>0.9996$ on these tubes in \eqref{eq:appD-run-centre} therefore holds for the centre solution throughout $[0,20]$.
Smooth dependence on the centre data and $\mu$ follows from the flux system.

At $L=20$ the program converted the flux variables to the matching coordinates \eqref{eq:validated-tail-coordinates}.
The stable endpoint bound in \eqref{eq:appD-run-centre} gives $|C_L|<1.774939\times10^{-12}<10^{-11}$ on the entire cube.
Thus the tail from \cref{lem:validated-compactified-tail-bounds} exists for every $\xi$.
Its endpoint depends $C^1$ on its parameters, so the mismatch \eqref{eq:validated-exact-mismatch} is continuous.

On each face the program subtracted the tail interval with the radius given in \eqref{eq:validated-tail-base-displacement} from the enclosure of the corresponding centre coordinate.
Each resulting signed lower bound in \eqref{eq:appD-run-faces} is positive, which gives \eqref{eq:validated-miranda-face-sign}.
\end{proof}

We can now choose the centre and tail parameters so that all four flux variables agree at $\zeta=20$.

\phantomsection\label{proof:ground-state}
\begin{proof}[Proof of \cref{thm:certified-positive-conical-profile}]
\emph{Part~\ref{it:ground-state-shrinker}.}
Apply the Poincar\'e--Miranda theorem \cite{Miranda1940} to the continuous mismatch $M$ in \eqref{eq:validated-exact-mismatch}.
By the face signs \eqref{eq:validated-miranda-face-sign}, $M$ has a zero $\xi_*$ in the interior of $[-1,1]^3$.
At $p(\xi_*)$ the three projected coordinates agree.
The stable coordinates agree by \eqref{eq:validated-stable-base}, since $C_L=Z^{\rm c}_{\rm s}(\xi_*)$.

The coordinate relation \eqref{eq:validated-tail-coordinates} now gives the same $x$-jet on both sides at $x_0$.
Since $x=\zeta^{4/3}$ has nonzero derivative at $\zeta=20$, the $\zeta$-jets also agree, and hence so do all four flux variables in \eqref{eq:pincher-flux-variables}.
By \cref{prop:pincher-profile-certificate} the joined profile extends smoothly and evenly across $\zeta=0$ to a solution $U$ of \eqref{eq:main-profile-equation}.
The centre profile is positive on the compact interval $[0,20]$ by \cref{lem:validated-centre-miranda-cube}.
Together with the tail estimate $U(\zeta)/\zeta>1.0369$ for $\zeta\ge20$ from \cref{lem:validated-compactified-tail-bounds}, this gives a uniform positive lower bound on $\R$.
By the converse in \cref{lem:terminal-scaling}, dilating $U$ by $A(t)=\{4\mu(T-t)\}^{1/4}$ gives a solution of surface diffusion flow for $t<T$.

\emph{Part~\ref{it:ground-state-ends}.}
On $[20,\infty)$ the matched profile agrees with the tail from \cref{lem:validated-compactified-tail-bounds}, so it has the differentiated expansion \eqref{eq:finite-base-tail-expansion}.
The coefficients in \eqref{eq:forced-conical-coefficient} and \eqref{eq:forced-conical-second-coefficient} give the first two corrections displayed in \eqref{eq:main-tail}.
Since $U(\zeta)=U(-\zeta)$, reflection gives the expansion and its differentiated remainders at the negative end as well.

\emph{Part~\ref{it:ground-state-bounds}.}
Write $p(\xi_*)=(\gamma,\mu,\alpha)$.
The centre data \eqref{eq:pincher-centre-shooting-data} give $\gamma=U''(0)$, and the matched tail has slope $\alpha$.
By \eqref{eq:validated-centre-box}, the $i$th component of $p(\xi_*)$ differs from $(p_0)_i$ by at most $r\sum_{j=1}^3|B_{ij}|$.
The bounds \eqref{eq:validated-parameter-radii} for these three half-widths give the stated enclosures.
\end{proof}

\section{Weighted terminal dynamics and mode selection}\label{sec:linear-new}

Bernoff-Bertozzi-Witelski identified the translation and dilation modes of the ground-state pincher and computed the formal asymptotics of linear perturbations at its conical ends \cite[\S\S4.1 and~5.1]{BernoffBertozziWitelski1998}.
We develop the corresponding theory for a class of fourth-order geometric flows and their positive conical profiles, beginning with the effect of fixing the rescaled radius at the centre.
A polynomial weight then accounts for transport towards the ends and allows us to separate the controlled and stable modes (see \cref{prop:finite-core-splitting}).
We state the conditions on the coefficients and the nonlinear remainder here and verify them for surface diffusion in \cref{prop:sdf-general-hypotheses}.

\subsection{Linearisation and evolution in the weighted space}\label{sec:linear-dilation-evolution}

\begin{definition}[Geometric flow class]\label{def:geometric-flow-class}
Let $\flowSpeed$ assign a scalar normal speed $\flowSpeed[f_0]:\Sigma\to\R$ to each smooth immersion $f_0:\Sigma\to\R^3$ of an oriented surface, possibly noncompact.
With $f(\cdot,0)=f_0$ as in \eqref{eq:sdf-intro-new}, we say that the smooth family $f$ evolves by $\flowSpeed$ if
\begin{equation}\label{eq:general-geometric-flow}
 \partial_t f(p,t)=\flowSpeed[f(\cdot,t)](p)\,\nu(p,t),
 \qquad (p,t)\in\Sigma\times[0,T).
\end{equation}
The flow \eqref{eq:general-geometric-flow} belongs to the \emph{geometric flow class} if $\flowSpeed$ satisfies the following conditions.
\begin{enumerate}[label=(\roman*),leftmargin=*]
\item In local coordinates, the speed depends smoothly on the immersion and its derivatives through order four, and is affine in the fourth derivatives.
The coefficients and their derivatives are uniformly bounded on bounded sets of immersion charts with a common positive lower bound on the induced metric.
\item For every base immersion $f_0$, the normal graph operator $u\mapsto\flowSpeed[f_0^u]/\langle\nu_{f_0},\nu_{f_0^u}\rangle$ in the chart \eqref{eq:normal-graph} is smooth from $C^{j+4}$ to $C^j$ for every integer $j\ge0$.
The negative of its linearisation with respect to the height $u$ at $u=0$ is strongly elliptic of order four.
\item The scalar speed commutes with reparametrisation and is invariant under Euclidean isometries, including reflections, when the unit normal is transported with the immersion.
Reversing the unit normal changes the sign of the scalar speed.
Under dilations the speed satisfies
\begin{equation}\label{eq:general-flow-scaling}
 \flowSpeed[af_0]=a^{-3}\flowSpeed[f_0],\qquad a>0.\qedhere
\end{equation}
\end{enumerate}
\end{definition}

These conditions give the local normal graph theory for smooth compact data by the quasilinear parabolic theory of \cite{Amann1993,PrussSimonett2016}.

For the immersion $f_Q$ of \eqref{eq:axisymmetric-graph}, rotational invariance makes $\flowSpeed[f_Q](z,\theta)$ independent of $\theta$.
Whenever the evolving neck has a radius graph $r$, projection onto the normal \eqref{eq:axisymmetric-normal} gives the radial equation at fixed $z$,
\begin{equation}\label{eq:general-radial-speed}
 \partial_t r(z,t)=\flowRad[r(\cdot,t)](z),\qquad
 \flowRad[Q](z):=-S[Q](z)\,\flowSpeed[f_Q](z,0).
\end{equation}
Here $S[Q]$ is defined in \eqref{eq:profile-geometry}.
For the dilated profile $Q_a$ of \eqref{eq:terminal-dilation-profile}, the symmetries and scaling in \cref{def:geometric-flow-class} give
\[
 \flowRad[Q_a](z)=a^{-3}\flowRad[Q](z/a),\qquad z\in aI.
\]
For surface diffusion \eqref{eq:sdf-intro-new}, the normal speed is $\flowSpeed[f(\cdot,t)]=-\Delta H$ and its radial form is $\radOp$ of \eqref{eq:terminal-operator}.

\begin{definition}[Normalised positive conical profile]\label{def:normalised-conical-profile}
For a flow in the class of \cref{def:geometric-flow-class}, we say that a smooth even function $U:\R\to(0,\infty)$ is a \emph{normalised positive conical profile} if $U(0)=1$, it solves
\begin{equation}\label{eq:general-profile-equation}
 \flowRad[U]+\mu DU=0
\end{equation}
for some $\mu>0$, and there are $\alpha>0$ and $c_1,c_2\in\R$ such that
\begin{equation}\label{eq:general-profile-tail}
 U(\zeta)=\alpha\zeta+c_1\zeta^{-3}+c_2\zeta^{-7}+O(\zeta^{-11})
 \qquad(\zeta\to+\infty),
\end{equation}
whose remainder after $j$ derivatives is $O(\zeta^{-11-j})$ for every integer $j\ge0$.
Here $D=1-\zeta\partial_\zeta$ as in \eqref{eq:terminal-dilation-profile}; even reflection gives the negative-end expansion.
The coefficients $c_1,c_2$ belong to the chosen flow and profile.
\end{definition}

Positivity, smoothness and \eqref{eq:general-profile-tail} give $\inf_{\R}U>0$ and $\|U'\|_{L^\infty(\R)}<\infty$.
For the spectral and evolution estimates we need the following more precise condition on the linearisation.

\begin{definition}[Conical linearisation condition]\label{def:conical-linearisation}
We say that the pair $(\flowSpeed,U)$ satisfies the \emph{conical linearisation condition} if $\radLin=D\flowRad[U]$ has the form
\begin{equation}\label{eq:fredholm-coefficient-form}
 \radLin v=-a_4v^{(4)}+a_3v'''+a_2v''+a_1v'+a_0v,
\end{equation}
with real smooth coefficients, $a_4$ even, $a_j(-\zeta)=(-1)^ja_j(\zeta)$ for $0\le j\le3$, and constants $a_*,a^*,a_\infty>0$ such that $a_*\le a_4\le a^*$.
For every integer $r\ge0$, we require
\begin{align}
 \partial_\zeta^r(a_4-a_\infty)&=O(\langle\zeta\rangle^{-4-r}),
 \label{eq:fredholm-detail-a4-decay}\\
 \partial_\zeta^r(a_3+2a_\infty/\zeta)&=O(|\zeta|^{-5-r}),
 \label{eq:fredholm-detail-a3-leading}\\
 \partial_\zeta^ra_j&=O(|\zeta|^{j-4-r}),\qquad j=0,1,2,
 \label{eq:fredholm-detail-lower-decay}
\end{align}
on $|\zeta|\ge1$.
Above we used $\langle x\rangle=(1+x^2)^{1/2}$ for $x\in\R$.
\end{definition}

The constants, spectral cut and controlled dimension in the linear estimates may depend on the fixed flow and profile.
In the variables of \eqref{eq:dynamic-terminal-scaling}, the rescaled profile satisfies $P(0,\simTime)=1$.
Applying the chain rule to \eqref{eq:general-radial-speed} and using \eqref{eq:general-flow-scaling} gives
\begin{equation}\label{eq:general-normalised-evolution}
 P_{\simTime}=\flowRad[P]-\flowRad[P](0)DP,
 \qquad (\log A)_{\simTime}=\flowRad[P](0).
\end{equation}

\begin{lemma}[The symmetry modes and the normalising projection]\label{lem:pincher-projection}
Let $U$ be a normalised positive conical profile in the sense of \cref{def:normalised-conical-profile}.
Assume the conical linearisation condition of \cref{def:conical-linearisation}.
The linearisation $\radLin=D\flowRad[U]$ has principal part $-a_4v^{(4)}$.
At fixed $\mu$ the linearisation of the profile equation \eqref{eq:general-profile-equation} is
\begin{equation}\label{eq:lin-prof-eqn}
 \fullLin=\radLin+\mu D,
\end{equation}
 with $D$ as in \eqref{eq:terminal-dilation-profile}.
 Its translation and dilation modes satisfy
\begin{equation}\label{eq:pincher-symmetry-modes}
 \fullLin U'=\mu U',\qquad \fullLin(DU)=4\mu DU.
\end{equation}
For a smooth even perturbation $v:\R\to\R$ with $v(0)=0$, the linearisation of the normalised evolution \eqref{eq:general-normalised-evolution} is
\begin{equation}\label{eq:projected-pincher-operator}
 \begin{aligned}
 \waistLin v&=\fullLin v-\ell(v)DU=\Pi\fullLin v,\\
 \ell(v)&=(\radLin v)(0),\qquad \Pi w=w-w(0)DU.
 \end{aligned}
\end{equation}
If $\phi:\R\to\C$ is a smooth even eigenfunction of $\fullLin$, then $\Pi\phi$ is an eigenfunction of $\waistLin$ with the same eigenvalue unless $\phi$ is a multiple of $DU$, in which case $\Pi\phi=0$.
\end{lemma}

\begin{proof}
The principal part is specified in \eqref{eq:fredholm-coefficient-form}.
Differentiating the profile equation at fixed $\mu$ gives $\fullLin$.
Differentiating it in $\zeta$ and using $D(U')-(DU)'=U'$ gives the translation identity.
For the dilation identity, differentiate at $a=1$ the consequence of \eqref{eq:general-flow-scaling},
\[
 \flowRad[U_a]+\mu DU_a=\mu(a-a^{-3})(DU)(\zeta/a).
\]

In \eqref{eq:general-normalised-evolution} the scale speed $\flowRad[P](0)$ has value $-\mu$ at $P=U$ and first variation $\ell(v)$.
The product rule gives $\waistLin v=\fullLin v-\ell(v)DU$.
Since $Dv(0)=v(0)=0$, we have $(\fullLin v)(0)=\ell(v)$, proving \eqref{eq:projected-pincher-operator}.
Finally, $DU(0)=1$ gives $\ker\Pi=\Span\{DU\}$.
This line is invariant by \eqref{eq:pincher-symmetry-modes}, so $\waistLin(\Pi\phi)=\Pi\fullLin\phi$, also after complexification.
\end{proof}

Reflection symmetry excludes the odd translation mode.
At the eigenvalue $4\mu$, the condition at the centre removes the dilation line.
Any further generalised eigenvectors at $4\mu$ contribute to the controlled space.

To select the initial data, we need a space in which all but finitely many modes decay.
The choice of weight is suggested by the evolution on the conical ends.
By \cref{def:conical-linearisation}, the fourth-order coefficient of $\radLin$ tends to $-a_\infty$ and the lower-order coefficients tend to zero at both ends.
Retaining this limiting diffusion and the dilation term $\mu D$ in $\fullLin$ gives the model operator
\begin{equation}\label{eq:far-model-operator}
 Bv=-a_\infty v^{(4)}+\mu Dv.
\end{equation}
Along the shrinker, $(\log A)_{\simTime}=-\mu$ by \eqref{eq:general-normalised-evolution}.
Choosing $\simTime=0$ at $A=A_0$ thus gives $\simTime=s$ along the shrinker, where $s$ is the similarity time of \cref{rem:main-stability}.

\begin{lemma}[Reduction to the biharmonic heat equation]\label{lem:conical-model-evolution}
Let $f\in C_c^\infty(\R)$ and let $B$ be the operator in \eqref{eq:far-model-operator}.
The Schwartz-class solution of $v_s=Bv$ with $v(\cdot,0)=f$ is
\begin{equation}\label{eq:far-semigroup}
 v(\zeta,s)=(e^{sB}f)(\zeta)
 =e^{\mu s}
 \left(e^{-a_\infty\tau_{\rm h}(s)\partial_y^4}f\right)(e^{-\mu s}\zeta),
 \qquad \tau_{\rm h}(s)=\frac{1-e^{-4\mu s}}{4\mu}.
\end{equation}
\end{lemma}
\begin{proof}
Consider the change of variables $y=e^{-\mu s}\zeta$, $v(\zeta,s)=e^{\mu s}u(y,\tau_{\rm h}(s))$.
The chain rule gives $v_s-\mu Dv=e^{\mu s}\tau_{\rm h}'(s)u_\tau$ and $\partial_\zeta^4v=e^{-3\mu s}\partial_y^4u$.
Since $\tau_{\rm h}'(s)=e^{-4\mu s}$, the equation $v_s-\mu Dv=-a_\infty\partial_\zeta^4v$ becomes the biharmonic heat equation $u_\tau=-a_\infty\partial_y^4u$.
At $s=0$ we have $y=\zeta$ and $\tau_{\rm h}=0$, so $u(\cdot,0)=f$.
The biharmonic heat semigroup gives \eqref{eq:far-semigroup}.
\end{proof}

A perturbation of the model equation can travel outwards while its relative height and slope both remain comparable to their initial values.
To see this, we consider an initial perturbation whose relative height and slope are both $\varepsilon>0$ at a point $L>0$ on the cone.
For instance, choose a smooth even function $\phi$ of compact support away from zero with $\phi(x)=x$ near $x=1$, and set $f(\zeta)=\varepsilon L\phi(\zeta/L)$.
Then $f(L)/L=f'(L)=\varepsilon$.
We ask whether these quantities decay as the perturbation evolves on the outgoing ray starting at $L$.

For any starting point $y\ne0$, evaluate the solution at $\zeta(s)=e^{\mu s}y$ at similarity time $s$.
If $u$ is the biharmonic heat solution with initial value $f$, then \eqref{eq:far-semigroup} gives $e^{-\mu s}v(\zeta(s),s)=u(y,\tau_{\rm h}(s))$.
As $s\to\infty$ the heat time increases to the finite value $\tau_\infty=1/(4\mu)$, so the rescaled height on this ray has limit $g(y):=u(y,\tau_\infty)$.
Dividing the height by $|\zeta(s)|$ and differentiating \eqref{eq:far-semigroup} in $\zeta$ give, respectively,
\[
 \frac{v(\zeta(s),s)}{|\zeta(s)|}
 =\frac{u(y,\tau_{\rm h}(s))}{|y|}\longrightarrow\frac{g(y)}{|y|},
 \qquad
 \partial_\zeta v(\zeta(s),s)
 =\partial_y u(y,\tau_{\rm h}(s))\longrightarrow g'(y).
\]
For the perturbation above, take $y=L$ and let $w(x,\tau)$ solve the same biharmonic heat equation with initial value $\phi$.
By scaling, $u(y,\tau)=\varepsilon Lw(y/L,\tau/L^4)$.
The relative height and slope on this ray are therefore $\varepsilon w(1,\tau_{\rm h}(s)/L^4)$ and $\varepsilon w_x(1,\tau_{\rm h}(s)/L^4)$.
Both $w(1,0)$ and $w_x(1,0)$ equal one.
Since $0\le\tau_{\rm h}(s)/L^4\le\tau_\infty/L^4$, continuity at heat time zero lets us choose $L$ large enough that the relative height and slope both stay between $\varepsilon/2$ and $3\varepsilon/2$ for every $s\ge0$.
In particular, both have nonzero limits along the same outgoing ray.

Thus, to measure decay near the centre, we give less weight to a perturbation the farther out it travels along the cone \cite{Sattinger1977,Wayne1997}.
We use the density $\langle\zeta\rangle^{-2m}$ with $m>0$.
A height bounded by $C|\zeta|$ on the ends contributes at most $2C^2\int_R^\infty x^{2-2m}\,\dd x$ to the squared weighted $L^2$ norm on $|\zeta|\ge R$, for $R$ sufficiently large.
This integral is finite precisely when $m>3/2$ and then tends to zero as $R\to\infty$.
For such $m$ and $k\ge0$, let $H^k_m$ consist of functions $v$ on $\R$ for which $\langle\cdot\rangle^{-m}v\in H^k(\R)$.
At integer orders we use the equivalent norm
\begin{equation}\label{eq:pincher-polynomial-space}
 \|v\|_{H^k_m}^2
 =\sum_{j=0}^k\int_{\R}|\partial_\zeta^jv|^2\langle\zeta\rangle^{-2m}\,\dd\zeta.
\end{equation}
On each fixed interval about the centre the density is bounded below by a positive constant, so \eqref{eq:pincher-polynomial-space} controls the ordinary $H^k$ norm there.
We estimate the relative height and slope along the whole conical end separately in \cref{prop:terminal-core-semiflow}.

For $k>1/2$ point evaluation at the centre is continuous.
Reflection symmetry requires $v$ to be even, and fixing the rescaled central radius requires $v(0)=0$.
To have $\waistLin v\in H^k_m$, we need four further derivatives of $v$ for the diffusion term and $\zeta v'\in H^k_m$ for the dilation $\mu Dv=\mu(v-\zeta v')$.
We therefore use the phase space and domain
\begin{equation}\label{eq:pincher-phase-space}
 \begin{aligned}
 \waistX^k_m&=\{v\in H^k_m:v\ \text{is even},\ v(0)=0\},\\
 \waistDom_{m,k}&=\{v\in\waistX^{k+4}_m:\zeta v'\in H^k_m\}.
 \end{aligned}
\end{equation}
To decide which modes can occur at infinity, fix a spectral parameter $\lambda\in\C$ and consider the equation $\fullLin v=\lambda v$ on a half-line $J$.
A solution is \emph{admissible on $J$} if $v\in H^{k+4}_m(J)$ and $\zeta v'\in H^k_m(J)$, with the weighted norms evaluated on $J$.
For the spectral analysis we take $m=2$ and write $\varpi(\zeta)=\langle\zeta\rangle^{-4}$ for the corresponding density.

\begin{lemma}[Admissible modes on a conical end]\label{lem:weighted-end-modes}
Let $U$ be a normalised positive conical profile satisfying \cref{def:conical-linearisation}.
Let $m=2$, let $k\ge2$ be an integer and let $\lambda\in\C$ satisfy $\Re\lambda>-\mu/2$.
For every $R>0$, the admissible solutions of $\fullLin v=\lambda v$ ($\fullLin$ as in \eqref{eq:lin-prof-eqn}) on $(R,\infty)$ form a two-dimensional complex vector space.
It has a basis $v_{\rm a},v_{\rm d}$ with $v_{\rm a}(\zeta)\sim\zeta^{1-\lambda/\mu}$ and $v_{\rm d}(\zeta)$ decaying exponentially as $\zeta\to+\infty$.
The reflected functions $v_{\rm a}(-\zeta),v_{\rm d}(-\zeta)$ form a basis of the admissible solutions on $(-\infty,-R)$.
\end{lemma}
\begin{proof}
For $R_*$ sufficiently large, \cref{lem:fredholm-end-dichotomy} constructs a basis of four solutions of $\fullLin v=\lambda v$ on $(R_*,\infty)$ and identifies the two admissible solutions there.
We recall how the weight selects these two solutions from the asymptotics \eqref{eq:fredholm-algebraic-solution} and \eqref{eq:fredholm-wkb-solutions}, then extend the basis to every $R>0$ and reflect it to obtain the stated basis on the negative half-line.

By \eqref{eq:fredholm-algebraic-solution}, the algebraic solution satisfies $v_{\rm a}(\zeta)\sim\zeta^{1-\lambda/\mu}$.
Since $\varpi(\zeta)\sim\zeta^{-4}$, the integrand in its weighted $L^2$ norm satisfies
\[
 |v_{\rm a}(\zeta)|^2\varpi(\zeta)
 \sim \bigl|\zeta^{1-\lambda/\mu}\bigr|^2\zeta^{-4}
 =\zeta^{-2-2\Re\lambda/\mu}.
\]
This power is integrable at $+\infty$ precisely when $\Re\lambda>-\mu/2$.
The differentiated bounds \eqref{eq:fredholm-algebraic-solution} give $\partial_\zeta^j v_{\rm a},\,\partial_\zeta^j(\zeta v_{\rm a}')=O(\zeta^{1-\Re\lambda/\mu-j})$, so $v_{\rm a}\in H^{k+4}_2$ and $\zeta v_{\rm a}'\in H^k_2$ on this half-line.
Of the three solutions in \eqref{eq:fredholm-wkb-solutions}, $v_{\rm d}$ and its derivatives decay exponentially, while the other two grow exponentially in modulus and fail even the weighted $L^2$ test.
Hence $v_{\rm a}$ and $v_{\rm d}$ form a basis of the admissible space, which has dimension two by \cref{lem:fredholm-end-dichotomy}.

By \cref{lem:pincher-projection}, the eigenvalue equation has smooth coefficients and a nowhere vanishing leading coefficient.
Its solutions extend uniquely across any finite interval, on which their weighted Sobolev norms are finite.
Thus restriction and extension identify the admissible spaces on $(R_*,\infty)$ and $(R,\infty)$ for every $R>0$.
Finally, $\fullLin$ commutes with reflection because the profile equation is invariant under $\zeta\mapsto-\zeta$ and $U$ is even.
Reflection preserves the weighted norms and carries this basis to $(-\infty,-R)$.
\end{proof}

We now use the count of admissible modes at $+\infty$ together with the evenness conditions at the centre to determine the Fredholm index, first for $\lambda-\fullLin$ and then for the normalised operator $\lambda-\waistLin$.

\begin{lemma}[The Fredholm boundary]\label{lem:weighted-fredholm-boundary}
Let $U$ be a normalised positive conical profile satisfying \cref{def:conical-linearisation} and let $k\ge2$ be an integer.
On $\waistX^k_2$ the operator $\waistLin$ with domain $\waistDom_{2,k}$ generates a $C_0$-semigroup.
On the complexified space, $\lambda-\waistLin$ is Fredholm of index zero for $\Re\lambda>-\mu/2$, and the line $-\mu/2+i\R$ lies in its Fredholm essential spectrum.
For each $\delta>0$ the resolvent is uniformly bounded on $\Re\lambda\ge-\mu/2+\delta$ for sufficiently large $|\Im\lambda|$.
\end{lemma}
\begin{proof}
On the full even space, $\fullLin$ generates a $C_0$-semigroup by \cref{lem:fredholm-generation}.
By \cref{lem:fredholm-global-estimate}, $\lambda-\fullLin$ is Fredholm of index zero for $\Re\lambda>-\mu/2$, and \eqref{eq:fredholm-exact-essential-boundary} puts $-\mu/2+i\R$ in the essential spectrum.
By the quotient decomposition \eqref{eq:fredholm-triangular-quotient}, the induced semigroup has generator $\waistLin$ with domain $\waistDom_{2,k}$.
Removing the invariant line $\Span\{DU\}$ preserves the Fredholm index and essential spectrum.
The uniform resolvent bound on the quotient is \eqref{eq:fredholm-quotient-vertical-resolvent}.
\end{proof}

\begin{proofoutline}
On the full even space $\evenX^k$ of \eqref{eq:fredholm-even-spaces}, it suffices to solve the equation on the positive half-line and impose evenness at zero.
For $\Re\lambda>-\mu/2$, the end inverse of \cref{lem:fredholm-end-dichotomy} gives a particular solution, which extends to zero by solving the ordinary differential equation with its variable principal coefficient $-a_4(\zeta)$.
The homogeneous equation has two admissible modes at $+\infty$ by \cref{lem:weighted-end-modes}.
Imposing the two independent boundary conditions $v'(0)=v'''(0)=0$ in \eqref{eq:fredholm-detail-even-boundary} therefore gives index $2-2=0$ by \eqref{eq:fredholm-detail-index-formula}.
On the boundary line, the singular sequence constructed in the proof of \eqref{eq:fredholm-exact-essential-boundary} shows that $\lambda-\fullLin$ fails to be Fredholm at every $\lambda\in-\mu/2+i\R$.

To impose the condition at the centre, recall that $DU(0)=1$ and $\fullLin(DU)=4\mu DU$ by \cref{lem:pincher-projection}.
Thus each class modulo $\Span\{DU\}$ has a unique representative vanishing at the centre, obtained by applying $\Pi$ in \eqref{eq:projected-pincher-operator}.
In these coordinates \eqref{eq:fredholm-triangular-quotient} differs from $\waistLin\oplus4\mu$ by the rank-one functional $\ell$.

Finally, the quotient resolvent estimate \eqref{eq:fredholm-quotient-vertical-resolvent} follows from \cref{lem:fredholm-vertical-resolvent}.
For $\Re\lambda\ge-\mu/2+\delta$ and sufficiently large $|\Im\lambda|$, the estimate \eqref{eq:fredholm-vertical-resolvent} gives injectivity of $\lambda-\fullLin$.
Its index is zero, so its inverse exists and has norm at most $C_\delta$.
Since $\lambda\ne4\mu$ in this region, \eqref{eq:fredholm-triangular-quotient} gives
\[
 (\lambda-\waistLin)^{-1}f=\Pi(\lambda-\fullLin)^{-1}f,
 \qquad f\in\waistX^k_2 .
\]
The projection $\Pi$ is bounded on $\evenX^k$, so the inverse on the quotient is uniformly bounded there.
\end{proofoutline}

The Fredholm half-plane and the resolvent bounds in \cref{lem:weighted-fredholm-boundary} allow us to separate the finitely many modes that require control from a complementary subspace on which the linear evolution decays exponentially.

\begin{proposition}[The controlled and stable subspaces]\label{prop:finite-core-splitting}
Let $U$ be a normalised positive conical profile satisfying \cref{def:conical-linearisation} and let $k\ge2$ be an integer.
One can choose $\omega$ so that
\begin{equation}\label{eq:pincher-omega-choice}
 0<\omega<\mu/2,\qquad
 \{\lambda\in\C:\Re\lambda=-\omega\}\cap\sigma(\waistLin)=\varnothing.
\end{equation}
For any such $\omega$, the spectrum of $\waistLin$ to the right of this line consists of finitely many eigenvalues of finite algebraic multiplicity.
The sum $\Pc$ of their Riesz projections gives the invariant splitting $\waistX^k_2=\Ec\oplus \Es$, where $\Ec=\operatorname{ran}\Pc$ and $\Es=\ker \Pc$.
For every $0<\omega'<\omega$,
\begin{equation}\label{eq:pincher-stable-semigroup}
 \|e^{s\waistLin}v\|_{\waistX^k_2}
 \le C_{\omega'}e^{-\omega's}\|v\|_{\waistX^k_2},
 \qquad s\ge0,\quad v\in \Es.
\end{equation}
\end{proposition}
\begin{proof}
Work on the complexification of $\waistX^k_2$.
By \cref{lem:weighted-fredholm-boundary}, $\waistLin$ generates a $C_0$-semigroup, so $\|e^{s\waistLin}\|\le Me^{bs}$ for some $M\ge1$ and $b\ge0$.
For $f\in\waistX^k_2$, the Laplace representation of the resolvent \cite[Theorem~II.1.10]{EngelNagel2000} gives
\begin{equation}\label{eq:terminal-laplace-resolvent}
 (\lambda-\waistLin)^{-1}f=\int_0^\infty e^{-\lambda s}e^{s\waistLin}f\,\dd s,
 \qquad
 \|(\lambda-\waistLin)^{-1}\|\le\frac{M}{\Re\lambda-b}
 \quad(\Re\lambda>b).
\end{equation}
On the connected open half-plane $\Re\lambda>-\mu/2$, the operators $\lambda-\waistLin:\waistDom_{2,k}\to\waistX^k_2$ form a holomorphic family, since they are bounded on the fixed domain with its graph norm and depend affinely on $\lambda$.
Each is Fredholm of index zero by \cref{lem:weighted-fredholm-boundary}, and \eqref{eq:terminal-laplace-resolvent} gives an invertible member at $\lambda=b+1$.
The analytic Fredholm theorem \cite[Theorem~C.8]{DyatlovZworski2019} therefore gives a meromorphic inverse with finite-rank principal parts.
Its poles are isolated eigenvalues of finite algebraic multiplicity.

Fix $0<\theta<\mu/2$.
By \eqref{eq:terminal-laplace-resolvent} any eigenvalue with $\Re\lambda\ge-\theta$ has $\Re\lambda\le b$.
The vertical resolvent estimate \eqref{eq:fredholm-quotient-vertical-resolvent}, with $\delta=\mu/2-\theta$, also bounds $|\Im\lambda|$ there.
These eigenvalues lie in a compact rectangle inside the Fredholm half-plane, which contains only finitely many poles.
Thus we can choose $0<\omega<\theta$ with $-\omega$ distinct from their real parts, as required in \eqref{eq:pincher-omega-choice}.
For any $\omega$ satisfying \eqref{eq:pincher-omega-choice}, the same argument with $\theta=\omega$ gives finiteness in $\Re\lambda>-\omega$.
The Riesz projections give the asserted invariant splitting \cite{Kato1966}.
Since $\waistLin$ is real, conjugation preserves their sum and the splitting restricts to real data.

Fix $\omega'<\omega_1<\omega$.
Removing the Riesz subspaces puts $\Re\lambda\ge-\omega_1$ in the resolvent set of the restriction to $\Es$.
The estimates \eqref{eq:terminal-laplace-resolvent} and \eqref{eq:fredholm-quotient-vertical-resolvent} bound the restricted resolvent on this half-plane when $\Re\lambda\ge b+1$ and when $|\Im\lambda|$ is large, respectively.
On the remaining compact rectangle it is bounded by continuity.
The restricted semigroup is strongly continuous on the closed Hilbert subspace $\Es$.
Applying the Gearhart--Pr\"uss theorem \cite{Gearhart1978,Pruss1984} to its generator plus $\omega_1 I$ gives \eqref{eq:pincher-stable-semigroup}.
\end{proof}

The controlled space includes any weakly decaying modes with $-\omega<\Re\lambda<0$.
We use the restrictions
\begin{equation}\label{eq:controlled-stable-restrictions}
 \waistLin_{\rm c}:=\waistLin\big|_{\Ec},
 \qquad
 \waistLin_{\rm s}:=\waistLin\big|_{\waistDom_{2,k}\cap \Es}.
\end{equation}

The decay estimate \eqref{eq:pincher-stable-semigroup} controls the $\waistX^k_2$ norm of a solution of the homogeneous equation on $\Es$.
In the nonlinear construction we must solve the forced neck equation \eqref{eq:convective-linear-equation} with right-hand side in $L^2_s H^{k-2}_2$.
On an interval $I$ of similarity time $s$, we write $L^2_s E=L^2(I;E)$ for the strongly measurable functions $u:I\to E$ into a Hilbert space $E$, identified up to equality almost everywhere, for which
\begin{equation}\label{eq:similarity-time-bochner-norm}
 \|u\|_{L^2_s E}^2=\int_I\|u(s)\|_E^2\,\dd s<\infty .
\end{equation}
Estimating the fourth-order terms of the radial speed in this norm requires $v$ in $L^2_s H^{k+2}_2$.
We will also need the time derivative of the physical height.
In the equation on the compact surface, the correction to the scale speed multiplies this derivative (see \eqref{eq:global-height-residual}).

Differentiating the physical height on the shrinking solution with profile $U$ in similarity time at fixed $z$, and using $A_s=-\mu A$, gives
\begin{equation}\label{eq:material-derivative}
 \partial_s\bigl[A(s)v(z/A(s),s)\bigr]=A(s)(\matD v)(z/A(s),s),
 \qquad \matD v:=v_s-\mu Dv .
\end{equation}
Subtracting $\mu Dv$ removes the contribution of the dilation to $v_s$.
Let us see why we cannot estimate $v_s$ itself by ordinary $L^2$ maximal regularity, even on $\Es$.
Indeed, the essential spectral line in \cref{lem:weighted-fredholm-boundary} persists on $\Es$ after removal of the finite-dimensional controlled space.
For $0<\eta<\omega$, the shifted spectrum contains $\eta-\mu/2+i\tau$ for each $\tau\in\R$.
Hence $i\tau$ is at distance at most $\mu/2-\eta$ from that spectrum.
The resolvent norm is at least the reciprocal of this distance, so
\[
 \bigl\|\tau(i\tau-\waistLin_{\rm s}-\eta)^{-1}\bigr\|
 \ge \frac{|\tau|}{\mu/2-\eta},\qquad \tau\in\R.
\]
Ordinary $L^2$ maximal regularity would require the left-hand side to be uniformly bounded \cite{DaPratoGrisvard1975,deSimon1964}.

For an integer $k\ge8$ the forcing, solution and initial trace spaces are
\begin{equation}\label{eq:convective-trace-space}
 X_k=\waistX^{k-2}_2,\qquad Y_k=\waistX^{k+2}_2,\qquad
 \waistTr=\waistX^k_2.
\end{equation}
They carry the evenness and central normalisation conditions of \eqref{eq:pincher-phase-space}, and $\waistTr=(X_k,Y_k)_{1/2,2}$ by \cref{lem:app-interpolation-trace}.
We solve
\begin{equation}\label{eq:convective-linear-equation}
 v_s-\waistLin v=f
\end{equation}
in the sense of distributions, using the material form of \eqref{eq:projected-pincher-operator},
\begin{equation}\label{eq:convective-linear-straightened}
 \matD v-\radLin v+\ell(v)DU=f.
\end{equation}
Write $\Ps=I-\Pc$ and set
\begin{equation}\label{eq:controlled-backward-rate}
 \gamma_{\rm c}=\max\bigl(\{0\}\cup
       \{-\Re\lambda:\lambda\in\sigma(\waistLin_{\rm c})\}\bigr)<\omega.
\end{equation}
The strict inequality follows from the choice of $\Ec$ in \cref{prop:finite-core-splitting}.

We will first prove a convective maximal regularity estimate on each unit time interval $I_n=[n,n+1]$, bounding $v$ in $L^2(I_n;Y_k)$ and $\matD v$ in $L^2(I_n;X_k)$ in terms of the initial trace $v(\cdot,n)$ and the forcing $f$.
We then solve the equation on the half-line with the stable part of the initial trace prescribed at $s=0$, and show that requiring decay at a rate $\eta\in(\gamma_{\rm c},\omega)$ determines the controlled part of that trace from $f$.

\begin{proposition}[Convective maximal regularity]
\label{prop:convective-maximal-regularity}
Let $U$ be a normalised positive conical profile satisfying \cref{def:conical-linearisation} and fix an integer $k\ge8$.
The solutions of \eqref{eq:convective-linear-equation} satisfy the following estimates.
\begin{enumerate}[label=(\roman*),ref=(\roman*),leftmargin=*,widest=ii]
\item\label{it:convective-unit} On each $I_n=[n,n+1]$, $n=0,1,\ldots$, every $v_n\in\waistTr$ and $f\in L^2(I_n;X_k)$ give a unique solution with $v(\cdot,n)=v_n$ in $C(I_n;\waistTr)\cap L^2(I_n;Y_k)$ with $\matD v\in L^2(I_n;X_k)$.
It satisfies
\begin{align}
 \|v\|_{C(I_n;\waistTr)}
 +\|v\|_{L^2(I_n;Y_k)}
 +\|\matD v\|_{L^2(I_n;X_k)}
 \le C\bigl(\|v_n\|_{\waistTr}+\|f\|_{L^2(I_n;X_k)}\bigr).
 \label{eq:convective-maximal-estimate}
\end{align}
The constant is independent of $n$.

\item\label{it:convective-halfline} On $\R_+=[0,\infty)$, fix $\gamma_{\rm c}<\eta<\omega$ and suppose $e^{\eta s}f\in L^2(\R_+;X_k)$.
For every $\xi\in \Ps\waistTr$ there is a unique solution with
\begin{equation}\label{eq:linear-stable-future-conditions}
 \Ps v(\cdot,0)=\xi,\qquad
 \Pc v(\cdot,s)=-\int_s^\infty
 e^{(s-\tau)\waistLin_{\rm c}}\Pc f(\cdot,\tau)\,\dd\tau.
\end{equation}
It satisfies
\begin{align}
 \sup_{s\ge0} e^{\eta s}\|v(\cdot,s)\|_{\waistTr}
 +\|e^{\eta s}v\|_{L^2(\R_+;Y_k)}
 +\|e^{\eta s}\matD v\|_{L^2(\R_+;X_k)}
 \le C_\eta\bigl(\|\xi\|_{\waistTr}
 +\|e^{\eta s}f\|_{L^2(\R_+;X_k)}\bigr).
 \label{eq:linear-convective-halfline-estimate}
\end{align}
The Riesz projections on $X_k$ and $\waistTr$ agree on $\waistTr$.
\end{enumerate}
\end{proposition}
\begin{proof}
For part~\ref{it:convective-unit}, \cref{lem:app-waist-rank-one} gives a unique solution of the straightened problem \eqref{eq:app-normalised-straightened-problem} with the estimate \eqref{eq:app-straightened-fourth-order-estimate}.
The norm equivalence \eqref{eq:app-dilation-equivalence} and the identity for the material derivative \eqref{eq:app-material-pullback} transfer this estimate to \eqref{eq:convective-maximal-estimate}, uniformly in $n$.
For part~\ref{it:convective-halfline}, the assumed gap $\gamma_{\rm c}<\eta<\omega$ allows us to apply \cref{prop:app-convective-halfline-inverse}.
Its Green formula \eqref{eq:app-continuous-green-representation} gives \eqref{eq:linear-stable-future-conditions} for the unique solution, and \eqref{eq:app-continuous-halfline-estimate} gives \eqref{eq:linear-convective-halfline-estimate}.
By \cref{lem:app-riesz-scale-compatibility}, the Riesz projections on $X_k$ and $\waistTr$ agree on $\waistTr$.
\end{proof}

\begin{proofoutline}
\emph{Part~\ref{it:convective-unit}.}
On $I_n$ write $s=n+\theta$ and introduce the dilation
\begin{equation}\label{eq:dilation-straightening}
 y=e^{-\mu\theta}\zeta,\qquad
 (\dilPull u)(y)=e^{-\mu\theta}u(e^{\mu\theta}y),
 \qquad 0\le\theta\le1.
\end{equation}
For $w(\cdot,\theta)=\dilPull v(\cdot,n+\theta)$, \cref{lem:app-dilation-equivalence} identifies $w_\theta$ with the pullback of $\matD v$ by \eqref{eq:app-material-pullback}.
It also gives the uniform equivalence of the forcing, solution and trace norms before and after the change of variables.

Two estimates in these coordinates give the unit-interval result.
First, \cref{lem:app-straightened-fourth-order} solves \eqref{eq:app-unnormalised-straightened-problem} on the full even space.
Multiplication by $\langle y\rangle^{-2}$ converts the weighted Sobolev norms into equivalent unweighted norms, so the whole-space estimate of Dong and Kim applies \cite[Theorem~2.3]{DongKim2011}.
Transferring the estimate back gives \eqref{eq:app-straightened-fourth-order-estimate}, which controls $w$ in $L^2H^{k+2}_2$ and $w_\theta$ in $L^2H^{k-2}_2$.
Second, \cref{lem:app-waist-rank-one} restores the term $\ell(v)DU$ in \eqref{eq:convective-linear-straightened}.
The local $H^5$ bound \eqref{eq:app-waist-trace-bound} controls this term, and the lemma gives the same estimate and uniqueness on the subspace with zero value at the centre.
Pulling its solution back to $I_n$ proves \eqref{eq:convective-maximal-estimate} with a constant independent of $n$.

\emph{Part~\ref{it:convective-halfline}.}
To join these solutions, let $h_n$ be the trace at $n+1$ produced by the forcing $f$ on $I_n$ with zero initial data.
By part~\ref{it:convective-unit}, $\|h_n\|_{\waistTr}\le C\|f\|_{L^2(I_n;X_k)}$.
Thus $(e^{\eta n}h_n)_{n\ge0}$ is square summable, and the incoming traces $x_n=v(\cdot,n)$ must satisfy the recurrence \eqref{eq:app-trace-recurrence}.
With $\gamma_{\rm c}<\eta<\omega$, \cref{lem:app-discrete-green} solves the stable recurrence forwards from $\Ps x_0=\xi$ and the controlled recurrence backwards from infinity.
The backward sum produces the future integral in \eqref{eq:app-continuous-green-representation}.
Summing the squared unit-interval bounds with the weights $e^{2\eta n}$ gives the half-line estimate.
Finally, integrating the resolvent identity \eqref{eq:app-resolvent-scale-identity} on the common spectral contours identifies the Riesz projections on $X_k$ and $\waistTr$.
\end{proofoutline}

\subsection{Conical estimates and selection of the controlled modes}

For an integer $k\ge8$ and a spatial perturbation $v\in H^k_{\rm loc}(\R)$, put
\begin{equation}\label{eq:pincher-global-smallness}
 \coneSize_k(v)=\sum_{j=0}^k\operatorname*{ess\,sup}_{\zeta\in\R}
  \langle\zeta\rangle^{j-1}|\partial_\zeta^jv(\zeta)|,\qquad
 \conePhase^k=\{v\in\waistX^k_2:\coneSize_k(v)<\infty\}.
\end{equation}
We give $\conePhase^k$ the norm $\|v\|_{\waistX^k_2}+\coneSize_k(v)$.
The first two terms of $\coneSize_k$ control the relative height and the slope.
For a normalised positive conical profile $U$, positivity and $U(\zeta)/\langle\zeta\rangle\to\alpha>0$ at both ends give $c_U:=\inf_{\zeta\in\R}U(\zeta)/\langle\zeta\rangle>0$.
If $\coneSize_k(v)\le c_U/2$, the $j=0$ term of \eqref{eq:pincher-global-smallness} gives $|v(\zeta)|\le(c_U/2)\langle\zeta\rangle\le U(\zeta)/2$.
Thus the perturbed rescaled radius $U(\zeta)+v(\zeta)$ is at least $U(\zeta)/2$ for every $\zeta\in\R$.
In particular, it is positive.

For the evolution we use the same weights on boxes transported by the dilation.
Write $\widetilde v(\cdot,\theta)=\dilPull v(\cdot,n+\theta)$ on $I_n$, using \eqref{eq:dilation-straightening}, and likewise for a source $f$.
For $R_0\ge1$ the radius
\begin{equation}\label{eq:transported-conical-radius}
 R(s)=R_0e^{\mu s},\qquad R_n=R(n),
\end{equation}
takes the moving interval $|\zeta|\le2R(s)$ to $|y|\le2R_n$.
Use $J=[-2R_n,2R_n]$ for an estimate restricted to that interval and $J=\R$ for a whole-line estimate.
Put $B_\rho=(\rho-2,\rho+2)$ and $Q_\rho=B_\rho\times(0,1)$.

For $\rho\in J$ and $0\le j\le k$, measure the parabolic regularity of $\partial_y^j\widetilde v$ on $(B_\rho\cap J)\times(0,1)$ by
\begin{align*}
 \boxSize_{j,\rho}(v)={}
  \|\partial_y^j\widetilde v\|_{L^\infty(0,1;H^2(B_\rho\cap J))}
 +\|\partial_y^j\widetilde v\|_{L^2(0,1;H^4(B_\rho\cap J))}
 +\|\partial_y^j\partial_\theta\widetilde v\|_{L^2(0,1;L^2(B_\rho\cap J))}.
\end{align*}
To obtain the graph norm on $I_n$, we take the weighted supremum of these local quantities over box centres $\rho\in J$ for each $j$, then sum over $j$:

\Needspace{12\baselineskip}
\begin{definition}[Conical graph, source and trace norms]\label{def:conical-parabolic-spaces}
For a path $v$ and a source $f$ on $I_n$, and an initial trace $\xi$ on $J$ at $s=n$, define
\begin{subequations}\label{eq:conical-parabolic-norms}
\begin{alignat}{2}
 &\text{Graph norm}\qquad &\|v\|_{\slabGraph_R^k(I_n)}
 &:=\sum_{j=0}^k\sup_{\rho\in J}\langle\rho\rangle^{j-1}\boxSize_{j,\rho}(v),
 \label{eq:conical-parabolic-graph-norm}\\
 &\text{Source norm} &\|f\|_{\slabSrc_R^k(I_n)}
 &:=\sum_{j=0}^k\sup_{\rho\in J}\langle\rho\rangle^{j-1}
 \|\partial_y^j\widetilde f\|_{L^2(0,1;L^2(B_\rho\cap J))},
 \label{eq:conical-parabolic-source-norm}\\
 &\text{Trace norm} &\|\xi\|_{\slabTr_{R_n}^k}
 &:=\inf_{\substack{V\in\slabGraph_R^k(I_n)\\V(\cdot,n)=\xi\ \text{on }J}}\|V\|_{\slabGraph_R^k(I_n)}.
 \label{eq:conical-parabolic-trace-norm}
\end{alignat}
\end{subequations}
Here $\slabGraph_R^k(I_n)$ is the completion, in the graph norm, of smooth even paths of finite norm that vanish at the centre.
The source space $\slabSrc_R^k(I_n)$ is the completion of smooth even sources of finite source norm; no condition at the centre is imposed on this space.
The trace space $\slabTr_{R_n}^k$ consists of the initial values of paths in $\slabGraph_R^k(I_n)$.
Its norm \eqref{eq:conical-parabolic-trace-norm} is the infimum of the graph norms of all such extensions of $\xi$ to $I_n$.
We use the subscript $\infty$ when $J=\R$.
In the restricted norms we identify functions that agree on the moving interval.
\end{definition}

The spaces are complete, and the trace has a bounded linear lift uniformly in $n$ and $R_0$ by \cref{lem:app-conical-banach-trace,lem:app-model-conical-estimate}.
On the half-line we seek a uniform bound for the graph norm on every $I_n$.
Forcing on $I_m$ can contribute to the solution on every later $I_n$, with no decay factor in $n-m$ in the conical trace estimate \eqref{eq:app-model-uniform-trace}.
We therefore sum the source norms to control the accumulated forcing and set
\begin{equation}\label{eq:conical-halfline-norms}
 \|v\|_{\coneGraph_R^k}=\sup_{n\ge0}\|v\|_{\slabGraph_R^k(I_n)},
 \qquad
 \|f\|_{\coneSrc_{R,1}^k}=\sum_{n\ge0}\|f\|_{\slabSrc_R^k(I_n)}.
\end{equation}
For the nonlinear fourth-order term, the difference of principal coefficients multiplies the fourth derivative of the height.
On the ends we will need to sum this derivative over the time intervals $I_n$ (see \eqref{eq:neck-end-input-norm}).
For this part of the graph norm put
\begin{equation}\label{eq:principal-spatial-dissipation}
 \coneDiss_{4,n}(v)=\sum_{j=0}^k\sup_{\rho\in\R}\langle\rho\rangle^{j-1}
 \|\partial_y^{j+4}\widetilde v\|_{L^2(Q_\rho)},\qquad
 \coneDiss_4(v)=\sum_{n\ge0}\coneDiss_{4,n}(v).
\end{equation}
To bound the conical graph norm in \eqref{eq:conical-halfline-norms} and the sum of fourth-derivative norms in \eqref{eq:principal-spatial-dissipation}, we will use exponential decay in the weighted Sobolev spaces.
For $\eta>0$, write
\begin{equation}\label{eq:conical-hilbert-decay}
 Z_\eta(v)=\sup_{s\ge0}e^{\eta s}\|v(\cdot,s)\|_{H^k_2}
       +\|e^{\eta s}v\|_{L^2(0,\infty;H^{k+2}_2)}.
\end{equation}
Under the hypotheses of \cref{prop:convective-maximal-regularity}\ref{it:convective-halfline}, the first two terms on the left of \eqref{eq:linear-convective-halfline-estimate} give
\[
 Z_\eta(v)\le C_\eta\bigl(\|\Ps v(\cdot,0)\|_{\waistTr}
                   +\|e^{\eta s}f\|_{L^2(\R_+;X_k)}\bigr).
\]
We now combine this decay near the centre with the model evolution on the ends to estimate the conical graph norm and the sum of fourth derivatives.

\begin{proposition}[Uniform conical bounds and summable fourth derivatives]
\label{prop:terminal-core-semiflow}
Let $U$ be a normalised positive conical profile satisfying \cref{def:conical-linearisation}, let $k\ge8$ be an integer and let $\eta>0$.
Suppose $v:\R\times[0,\infty)\to\R$ solves \eqref{eq:convective-linear-equation} in the sense of \eqref{eq:convective-linear-straightened}, is even in $\zeta$ and satisfies $v(0,s)=0$.
If $v(\cdot,0)\in\slabTr_\infty^k$, $f\in\coneSrc_{\infty,1}^k$ and $Z_\eta(v)<\infty$, then
\begin{equation}\label{eq:conical-halfline-summable-estimate}
 \|v\|_{\coneGraph_\infty^k}+\coneDiss_4(v)
 \le C\left\{
  \|v(\cdot,0)\|_{\slabTr_\infty^k}
  + Z_\eta(v)+\|f\|_{\coneSrc_{\infty,1}^k}
 \right\}.
\end{equation}
For every $R_0\ge1$ the same right-hand side bounds $\|v\|_{\coneGraph_R^k}$ with the same constant.
\end{proposition}
\begin{proof}
In the norms of \eqref{eq:conical-halfline-norms}, the whole-line estimate \eqref{eq:app-whole-line-conical-green} bounds $\|v\|_{\coneGraph_\infty^k}$ by the right-hand side of \eqref{eq:conical-halfline-summable-estimate}.
Substituting this bound in \eqref{eq:app-summable-principal-dissipation} gives the same bound for $\coneDiss_4(v)$.
Adding the two estimates proves \eqref{eq:conical-halfline-summable-estimate}.
Restricting the boxes in \cref{def:conical-parabolic-spaces} gives $\|v\|_{\coneGraph_R^k}\le\|v\|_{\coneGraph_\infty^k}$ for every $R_0\ge1$.
\end{proof}

\begin{proofoutline}
On each unit interval $I_n$ write
\begin{equation}\label{eq:conical-slab-norms}
 G_n(v)=\|v\|_{\slabGraph_\infty^k(I_n)},\qquad
 H_n(f)=\|f\|_{\slabSrc_\infty^k(I_n)}.
\end{equation}
For the limiting operator $B$ of \eqref{eq:far-model-operator}, the semigroup $S_0(s)=e^{sB}$ maps constant functions to constant functions.
On representatives with zero value at the centre the quotient semigroup is therefore
\begin{equation}\label{eq:model-waist-semigroup}
 S_{\rm q}(s)u=S_0(s)u-(S_0(s)u)(0).
\end{equation}
By \cref{lem:app-model-conical-estimate}, the trace produced by a source on $I_n$ propagates to every later interval with a uniform bound.
This gives the sum of the source norms in \eqref{eq:app-model-whole-line-graph}.
To apply that estimate to $\waistLin$, the proof of \cref{prop:app-whole-line-conical-green} cuts off the solution on the ends at a radius
\begin{equation}\label{eq:slow-conical-radius}
 \widehat R(s)=R_*e^{as},\qquad \widehat R_n=\widehat R(n),\qquad
 0<a<\min\{\mu,\eta/[2(k+1)]\}.
\end{equation}
In the source estimate \eqref{eq:app-slow-cutoff-source}, the error from the end coefficients and spatial cutoff derivatives is bounded by $C\widehat R_n^{-1}G_n(v)$.
Its sum can be absorbed by taking $R_*$ large.
On the cutoff transition and in the interior estimate \eqref{eq:app-expanding-core-graph}, passing from the weighted Sobolev norms to the conical norms costs at most $C\widehat R_n^{k+1}$.
The decay $e^{-\eta n}$ in $Z_\eta(v)$ makes this contribution summable by \eqref{eq:slow-conical-radius}.
The normalisation term is controlled by the local $H^5$ estimate \eqref{eq:app-waist-trace-bound}.
Together with the initial-slab estimate \eqref{eq:app-initial-core-graph}, these bounds give \eqref{eq:app-whole-line-conical-green}, first for smooth data and then by approximation on finite time intervals.

For the fourth derivatives, the model evolution gives an additional decay factor $e^{-3\mu n}$ on $I_n$, $n\ge1$, for a propagated initial trace by \eqref{eq:app-principal-slab-decay}.
\Cref{prop:app-summable-principal-dissipation} uses this decay to sum the contributions on the ends.
On the remaining boxes the interior bound \eqref{eq:app-expanding-core-dissipation} is already summable.
\end{proofoutline}

The future condition \eqref{eq:linear-stable-future-conditions} depends only on the projection of the forcing onto $\Ec$.
We estimate this projection for the nonlinear radial equation on the neck, and then choose initial variations supported near the centre which span $\Ec$ after projection.

For a smooth even perturbation $v$ with $v(0)=0$ and $U+v>0$ on a symmetric interval about the centre, write
\begin{equation}\label{eq:terminal-radial-remainder}
 \radRem(v)=\flowRad[U+v]-\flowRad[U]-\radLin v,\qquad
 \vartheta_{\rm rad}(v)=\ell(v)+\radRem(v)(0).
\end{equation}
At the centre $\flowRad[U+v](0)=-\mu+\vartheta_{\rm rad}(v)$.
In the parabolically rescaled time $\simTime$ of \eqref{eq:general-normalised-evolution}, substitution of $P=U+v$ gives $v_{\simTime}=\waistLin v+\neckNL(v)$, where
\begin{equation}\label{eq:defined-core-nonlinearity}
 \neckNL(v)=\radRem(v)-\radRem(v)(0)DU-\vartheta_{\rm rad}(v)Dv .
\end{equation}
This remainder vanishes to second order at $v=0$.
We also need local estimates in a fixed normal chart, since the neck and cap are joined by a single normal height.
For the immersion $f_U$ of \eqref{eq:axisymmetric-graph}, let $\nu_U:\R\times\Sphere^1\to\Sphere^2$ be its inward unit normal and put $\nu_r=\langle\nu_U,e_r\rangle=-1/S$.
Here $e_r=(\cos\theta,\sin\theta,0)$ is the radial unit vector and $S=S[U]$ is defined in \eqref{eq:profile-geometry}.
For a smooth even $v:\R\to\R$ with $v(0)=0$ and sufficiently small $\coneSize_k(v)$, define
\begin{equation}\label{eq:normal-neck-chart}
 \begin{gathered}
 Y_v=f_U+\nu_r v\nu_U,\qquad b(v)=\langle\nu_U,\nu_{Y_v}\rangle,\\
 C_U^\nu(v)=\frac{\langle Y_v,\nu_{Y_v}\rangle}{\nu_r b(v)}-Dv.
 \end{gathered}
\end{equation}
Here $Y_v:\R\times\Sphere^1\to\R^3$ is an immersion, its unit normal $\nu_{Y_v}$ has the same orientation as $\nu_U$, and $b(v):\R\to\R$ is bounded away from zero.
The spatial operator in material coordinates and its quadratic remainder are
\begin{equation}\label{eq:normal-neck-remainder}
 \flowRad^\nu(v)=\frac{\flowSpeed[Y_v]}{\nu_r b(v)}+\mu C_U^\nu(v),
 \qquad \radRem^\nu(v)=\flowRad^\nu(v)-\radLin v.
\end{equation}
\Cref{lem:app-normal-neck-chart} proves that $\flowRad^\nu(0)=0$ and $D\flowRad^\nu(0)=\radLin$.
It also gives the exact material equation and estimates the geometric coefficient $C_U^\nu(v)-DU$.

The local estimates near the centre will control the forcing of the finite-dimensional modes, while on the ends we need enough decay to sum the forcing over successive time intervals.
Both the radial remainder \eqref{eq:terminal-radial-remainder} and the normal remainder \eqref{eq:normal-neck-remainder} vanish to second order at $v=0$.
It suffices to estimate their first variations to obtain both a quadratic bound and a difference bound.
Indeed, for either remainder $F$,
\[
 F(v)=\int_0^1 DF(tv)v\,\dd t,\qquad
 F(v)-F(w)=\int_0^1 DF\bigl(w+t(v-w)\bigr)(v-w)\,\dd t.
\]
The same identities hold after applying the projection $\Pi$ of \eqref{eq:projected-pincher-operator}.
In each product we use one norm, $K$, for the factor carrying the highest derivative and another, $L$, for the remaining factors.

On $I_n=[n,n+1]$ the conical size of a path $h$ is
\begin{equation}\label{eq:neck-slab-conical-size}
 r_n(h)=\operatorname*{ess\,sup}_{s\in I_n}\coneSize_k(h(\cdot,s)),
\end{equation}
where the spatial suprema in \eqref{eq:pincher-global-smallness} are over $\R$ or, on the truncated neck, over $|\zeta|\le2R(s)$.
For estimates on $Q_\rho$ we measure the input on $Q_\rho^*=B_\rho^*\times(0,1)$, where $B_\rho^*=(\rho-3,\rho+3)$.
On the truncated neck both boxes are intersected with $|y|\le2R_n$.
Write $\slabSrc^k(Q_\rho)$ and $\slabGraph^k(Q_\rho^*)$ for the source and graph spaces of \cref{def:conical-parabolic-spaces} with the norms taken on the indicated box at the single centre $\rho$.
On the ends we need the fourth derivatives of \eqref{eq:principal-spatial-dissipation} with their full weight, while the remaining terms gain a factor $\langle\rho\rangle^{-1}$.
We therefore set
\begin{equation}\label{eq:neck-end-input-norm}
 E_\rho(h)=\sum_{j=0}^k\langle\rho\rangle^{j-1}
       \|\partial_y^{j+4}\widetilde h\|_{L^2(Q_\rho^*)}
       +\langle\rho\rangle^{-1}\|h\|_{\slabGraph^k(Q_\rho^*)}.
\end{equation}

\Needspace{23\baselineskip}
For each operator $T$ the table gives its target space $Y$ and its two input norms.
{\small
\begin{subequations}\label{eq:neck-remainder-norms}
\begin{alignat}{5}
 &\text{Region}&&\quad T&&\quad Y&&\quad L(h)&&\quad K(h)\notag\\
 &\text{Whole line}&&\quad\Pi F&&\quad L^2H^{k-2}_2&&\quad r_n(h)&&\quad
      \|h\|_{L^\infty H^k_2}+\|h\|_{L^2 H^{k+2}_2}\label{eq:neck-weighted-line-norms}\\
 &\text{Conical slab}&&\quad\Pi F&&\quad\slabSrc_R^k(I_n)&&\quad
      \|h\|_{\slabGraph_R^k(I_n)}&&\quad\|h\|_{\slabGraph_R^k(I_n)}\label{eq:neck-conical-slab-norms}\\
 &\text{Interval }J,\ |J|\ge1&&\quad F&&\quad L^2H^{k-2}(J)&&\quad
      \|h\|_{L^\infty H^k(J)}&&\quad\|h\|_{L^2 H^{k+2}(J)}\label{eq:neck-interval-norms}\\
 &Q_\rho&&\quad F&&\quad L^2H^k(B_\rho)&&\quad
      \|\widetilde h\|_{L^\infty H^k(B_\rho^*)}&&\quad
      \|\widetilde h\|_{L^2 H^{k+4}(B_\rho^*)}\label{eq:neck-ordinary-box-norms}\\
 &Q_\rho,\ |\rho|\ge4&&\quad F&&\quad\slabSrc^k(Q_\rho)&&\quad r_n(h)&&\quad E_\rho(h)
      \label{eq:neck-end-box-norms}
\end{alignat}
\end{subequations}
}
The time interval in \eqref{eq:neck-weighted-line-norms}--\eqref{eq:neck-interval-norms} is $I_n$.
On every box $Q_\rho$ we use the unweighted norms \eqref{eq:neck-ordinary-box-norms}, and for $|\rho|\ge4$ we also use the weighted norms \eqref{eq:neck-end-box-norms}.
Both rows use the straightened coordinates $(y,\theta)$ of \eqref{eq:dilation-straightening}, with $0<\theta<1$.

The following condition gives the estimates we shall use to bound the nonlinear forcing near the centre and to sum its contributions along the conical ends.
\begin{definition}[Nonlinear neck condition]\label{def:neck-nonlinear-estimates}
Let $U$ be a normalised positive conical profile for $\flowSpeed$ and fix an integer $k\ge10$.
We say that $(\flowSpeed,U)$ satisfies the \emph{nonlinear neck condition at order $k$} if there are $\varepsilon,C>0$ such that the following holds for each operator $T$ in \eqref{eq:neck-remainder-norms}.
Consider even paths $v$ with $v(0,s)=0$ and $r_n(v)<\varepsilon$, requiring also $L(v)<\varepsilon$ on the interval $J$.
On these paths $T$ extends from smooth inputs to a $C^2$ map into $Y$ for the norm $L+K$, and for each variation $h$,
\begin{equation}\label{eq:neck-first-variation}
 \|DT(v)h\|_Y\le C\{L(v)K(h)+L(h)K(v)\}.
\end{equation}
For the whole-line norms \eqref{eq:neck-weighted-line-norms} and the box norms \eqref{eq:neck-ordinary-box-norms}--\eqref{eq:neck-end-box-norms}, we also require, for all variations $h,g$,
\begin{equation}\label{eq:neck-second-variation}
 \|D^2T(v)[h,g]\|_Y
 \le C\{L(h)K(g)+L(g)K(h)+L(h)L(g)K(v)\}.
\end{equation}
The constants are uniform in $n$, $J$, $\rho$ and the initial conical radius $R_0\ge1$.
\end{definition}

The last term of \eqref{eq:neck-second-variation} allows both variations to fall on a coefficient while the highest derivative remains on $v$.
We verify \eqref{eq:neck-first-variation} and \eqref{eq:neck-second-variation} for surface diffusion in \cref{prop:sdf-general-hypotheses}.

To localise $\neckNL$, take an even $\chi_{\rm in}\in C_c^\infty(\R;[0,1])$ equal to one on $[-1,1]$ and zero outside $[-2,2]$, and set $\chi_R(\zeta)=\chi_{\rm in}(\zeta/R)$.
For the radius $R(s)$ in \eqref{eq:transported-conical-radius}, write $\coneSize_{k,R}(v;s)$ for $\coneSize_k(v(\cdot,s))$ in \eqref{eq:pincher-global-smallness} with the spatial interval restricted to $|\zeta|\le2R(s)$.
Use \eqref{eq:conical-hilbert-decay} to set
\begin{equation}\label{eq:controlled-localisation-norm}
 \|v\|_{\eta,R}=Z_\eta(v)+\sup_{s\ge0}\coneSize_{k,R}(v;s).
\end{equation}
If $\|v\|_{\eta,R}\le c_U/2$, the height bound in \eqref{eq:pincher-global-smallness} gives $U+v\ge U/2$ on the support of $\chi_{R(s)}$, and we extend $\chi_{R(s)}\neckNL(v)$ by zero.

For the nonlinear estimates we take $k\ge10$.
When we differentiate the fourth-order remainder up to $k$ times, at most one factor requires the $L^2$ estimate.
For $k\ge10$, Sobolev embedding bounds all the other factors pointwise, as in the box estimate \eqref{eq:neck-first-variation} with norms \eqref{eq:neck-ordinary-box-norms}.

\begin{lemma}[Controlled-mode localisation]\label{lem:controlled-mode-localisation}
Fix an integer $k\ge10$ and let $U$ be a normalised positive conical profile satisfying \cref{def:conical-linearisation,def:neck-nonlinear-estimates} at order $k$.
Fix $\omega/2<\eta<\omega$ and let $\varepsilon$ be the threshold in \cref{def:neck-nonlinear-estimates}.
There are $0<M_0<\min\{c_U/2,\varepsilon/5\}$, $C>0$ and $\sigma_N>\omega$ such that the following holds for every $R_0\ge1$.
Let $v,w:\R\times[0,\infty)\to\R$ be even paths vanishing at the centre and suppose that, in the norm \eqref{eq:controlled-localisation-norm},
\[
 M:=\max\{\|v\|_{\eta,R},\|w\|_{\eta,R}\}\le M_0.
\]
Then on every $I_n$ we have
\begin{equation}\label{eq:controlled-projected-quadratic}
 \bigl\|\Pc\bigl[\chi_{R(s)}\{\neckNL(v)-\neckNL(w)\}\bigr]\bigr\|_{L^2(I_n;H^k_2)}
 \le C M e^{-\sigma_N n}\|v-w\|_{\eta,R}.
\end{equation}
In particular, taking $w=0$ gives a quadratic bound for the projected forcing.
\end{lemma}
\begin{proof}
\Cref{lem:app-adjoint-localisation} gives the difference estimate corresponding to \eqref{eq:app-projected-localisation}.
By \eqref{eq:defined-core-nonlinearity}, $\neckNL(v)$ and $\neckNL(w)$ both vanish at the centre, so the normalising projection $\Pi$ in \eqref{eq:projected-pincher-operator} fixes their cutoff difference.
The difference estimate therefore gives \eqref{eq:controlled-projected-quadratic}.
\end{proof}

\begin{proofoutline}
We will explain how the quadratic estimate near the centre and the decay of the adjoint modes on the ends give the rate $\sigma_N>\omega$.

By \cref{lem:app-riesz-scale-compatibility}, the $H^k_2$ norm of the projection is controlled by finitely many $H^0_2$ pairings with adjoint Jordan vectors.
For each controlled eigenvalue $\lambda$, we have $\Re\lambda>-\omega>-\mu/2$, so \cref{lem:fredholm-detail-adjoint-chain-decay} applies.
It excludes the algebraic adjoint branch by the weighted $L^2$ test and gives the differentiated exponential decay \eqref{eq:fredholm-detail-adjoint-chain-exponential} for every vector in the adjoint chain.

Since $\eta>\omega/2$, we may choose $0<a<\min\{\mu,\eta/[2(k+1)],(2\eta-\omega)/4\}$.
On $I_n$ split each pairing using the cutoff $\chi_{e^{an}}$.
On the support of $1-\chi_{e^{an}}$ we have $|\zeta|\ge e^{an}$, where the adjoint decay absorbs the polynomial growth allowed by the conical norm.

The inner cutoff is supported on $\overline J_n$, where $J_n=(-2e^{an},2e^{an})$.
Since $a<\mu$ and $R_0\ge1$, we have $e^{an}\le R(s)$ throughout $I_n$ by \eqref{eq:transported-conical-radius}.
On $J_n$ the density in \eqref{eq:pincher-polynomial-space} with $m=2$ satisfies $\varpi^{-1/2}=\langle\zeta\rangle^2=1+\zeta^2\le 1+4e^{2an}\le 5e^{2an}$.
Thus for each spatial derivative,
\[
 \|\partial_\zeta^j v\|_{L^2(J_n)}
 \le 5e^{2an}\|\varpi^{1/2}\partial_\zeta^j v\|_{L^2(J_n)}.
\]
Summing the squares over $j$ and taking the square root gives the same factor at every Sobolev order.
Since $s\ge n$ on $I_n$ and $Z_\eta(v)\le M$ by \eqref{eq:controlled-localisation-norm}, the two terms of \eqref{eq:conical-hilbert-decay} give
\begin{align*}
 \|v\|_{L^\infty(I_n;H^k(J_n))}
 &\le 5e^{2an}\|v\|_{L^\infty(I_n;H^k_2)}
 \le 5M e^{-(\eta-2a)n},\\
 \|v\|_{L^2(I_n;H^{k+2}(J_n))}
 &\le 5e^{2an}\|v\|_{L^2(I_n;H^{k+2}_2)}
 \le 5M e^{-(\eta-2a)n}.
\end{align*}
Since $2a<\eta$, the interval norms \eqref{eq:neck-interval-norms} satisfy $L(v),L(w)\le5M_0<\varepsilon$.
The conical sizes $r_n(v),r_n(w)$ on the truncated neck are at most $M_0$ by \eqref{eq:controlled-localisation-norm}, so the whole segment from $w$ to $v$ satisfies both smallness conditions of \cref{def:neck-nonlinear-estimates}.
Integrating \eqref{eq:neck-first-variation} from $0$ to $v$ now gives $\|\radRem(v)\|_{L^2(I_n;H^{k-2}(J_n))}\le C L(v)K(v)\le CM^2e^{-(2\eta-4a)n}$.
Our choice of $a$ gives $2\eta-4a>\omega$.
Evaluation of $\radRem(v)(0)$ on a fixed interval about the centre incurs no such loss.

For the remaining term $\vartheta_{\rm rad}(v)Dv$ in \eqref{eq:defined-core-nonlinearity}, the pairing itself removes the derivative from $v$.
Indeed, for $R=R(s)$ and an adjoint representative $\phi^*$, integration by parts gives
\begin{equation}\label{eq:controlled-dilation-pairing}
 \int_{\R}\chi_R Dv\,\overline{\phi^*}\,\varpi\,\dd\zeta
 =\int_{\R}v\{\chi_R\overline{\phi^*}\,\varpi
       +\partial_\zeta(\zeta\chi_R\overline{\phi^*}\,\varpi)\}\,\dd\zeta.
\end{equation}
There is no boundary term since $\chi_R$ has compact support.
The bounds $|\zeta\chi_R'(\zeta)|\le2\|\chi_{\rm in}'\|_\infty$ and $|\zeta\varpi'|\le4\varpi$, together with the adjoint decay \eqref{eq:fredholm-detail-adjoint-chain-exponential}, give a bound $C\|v(s)\|_{H^0_2}$ for the pairing by Cauchy--Schwarz, uniformly in $R\ge1$.
Multiplying by the central value $\vartheta_{\rm rad}(v)=O(Me^{-\eta n})$ gives $CM^2e^{-2\eta n}$.
Subtracting the products for $v$ and $w$ puts the difference in one factor at a time, so the same choice $\omega<\sigma_N<2\eta-4a$ works for their difference.
\end{proofoutline}

It remains to realise the controlled initial coordinates by perturbations supported inside the neck.

\begin{lemma}[Compactly supported mode selection]\label{lem:compact-core-selection}
Let $U$ be a normalised positive conical profile satisfying \cref{def:conical-linearisation}.
For each integer $k\ge2$ there are $N=\dim \Ec$ smooth even functions $\psi_j:\R\to\R$ of compact support, with $\psi_j(0)=0$, whose projections $\Pc\psi_j$ form a basis of the real controlled space $\Ec$ in $\waistX^k_2$.
\end{lemma}
\begin{proof}
Let $\testFns$ be the space of all such compactly supported, even functions vanishing at zero.
Fix $v\in\waistX^k_2$.
Even cutoffs and mollification give smooth compactly supported even $v_j\to v$ in $H^k_2$.
The trace at zero is continuous, so $v_j(0)\to0$.
For a fixed smooth even compactly supported bump $b$ with $b(0)=1$, the functions $v_j-v_j(0)b\in\testFns$ still converge to $v$.
Thus $\Pc\testFns$ is a dense linear subspace of the finite-dimensional real space $\Ec$ and equals $\Ec$.
Choosing one preimage $\psi_j\in\testFns$ of each basis vector of $\Ec$ proves the claim.
\end{proof}

If the supports lie in $|\zeta|\le L$, then by \eqref{eq:dynamic-terminal-scaling} the physical variations lie in $|z|\le AL$.
For small central radius $A$ their supports are confined to the neck, and sufficiently small coefficients preserve positivity and embeddedness.
In \cref{sec:compact-new} these coefficients prescribe the controlled initial coordinates.
For a volume-preserving flow, a separate variation of the cap allows us to prescribe the enclosed volume as well.

\begin{definition}[Admissible neck]\label{def:admissible-neck}
We say that $U$ is an \emph{admissible neck} for the speed $\flowSpeed$ if $\flowSpeed$ satisfies the conditions of \cref{def:geometric-flow-class}, $U$ is a normalised positive conical profile in the sense of \cref{def:normalised-conical-profile}, the pair satisfies \cref{def:conical-linearisation}, and the nonlinear neck condition of \cref{def:neck-nonlinear-estimates} holds at every integer order $k\ge10$.
\end{definition}

These local conditions give the finite controlled splitting of \cref{prop:finite-core-splitting}, the evolution estimates of \cref{prop:convective-maximal-regularity,prop:terminal-core-semiflow}, and the summable and projected nonlinear bounds of \cref{cor:app-summable-spatial-remainder,lem:controlled-mode-localisation}.
The initial values lie in the weighted and conical trace spaces of \eqref{eq:convective-trace-space} and \eqref{eq:conical-parabolic-trace-norm}.
In the compact construction these neck traces and the cap trace will be restrictions of one normal height; the compatibility requirement is imposed when the global height space is defined.

\begin{remark}[The dimension of the controlled space]\label{rem:controlled-dimension}
The proof of \cref{prop:finite-core-splitting} gives only finitely many eigenvalues in $\Re\lambda\ge-\mu/4$.
We can therefore take $\omega>0$ so small in \eqref{eq:pincher-omega-choice} that every eigenvalue with negative real part lies to the left of the cut.
For this choice, $\dim \Ec$ is the sum of the algebraic multiplicities of the eigenvalues of $\waistLin$ with $\Re\lambda\ge0$.
This is the smallest number of neck coordinates selected by our exponential decay condition.

In particular, if the even, normalised operator $\waistLin$ has no eigenvalues with $\Re\lambda\ge0$, then $\Ec=\{0\}$ and \cref{thm:compact-pincher-realisation} gives stability on an open set of symmetric compatible data at fixed central radius for caps satisfying its hypotheses.
In the volume-preserving case, the same conclusion holds within the hypersurface of data with prescribed volume.
The translation mode in \eqref{eq:pincher-symmetry-modes} is odd, and the normalising projection removes the dilation line.
Thus, by \eqref{eq:fredholm-triangular-quotient}, it would suffice to show that $4\mu$ is algebraically simple and is the only eigenvalue of $\fullLin$ on the even space with $\Re\lambda\ge0$.

For the ground state of \cref{thm:certified-positive-conical-profile}, we leave this spectral assertion as an open problem.\qedhere

\end{remark}

\section{Compactification of a conical neck}\label{sec:compact-new}

We now join a shrinking conical neck to a smooth outer cap and correct the resulting surface so that it evolves by the geometric flow \eqref{eq:general-geometric-flow}.
Let $\flowSpeed$ satisfy \cref{def:geometric-flow-class} and let $U$ be an admissible neck in the sense of \cref{def:admissible-neck}, with parameters $\mu$ and $\alpha$.
Admissibility prescribes the conical bounds on the linear coefficients and the nonlinear bounds in the radial and normal charts (\cref{def:conical-linearisation,def:neck-nonlinear-estimates}).
For such a neck, \cref{prop:convective-maximal-regularity,prop:terminal-core-semiflow} give the weighted and conical evolution estimates.
For the cap construction we require the following, uniformly as the central radius tends to zero and for nearby symmetric caps:
\begin{enumerate}[label=(\roman*),leftmargin=*]
\item closed embedded matched surfaces, with an injective normal tube and curvature bounds at the local neck and cap scales (\cref{def:compactification-hypotheses}(i));
\item a uniformly parabolic extension of the cap height equation to the doubled cap, with uniform local nonlinear bounds and a short-time maximal-regularity estimate for its linearisation (\cref{def:outer-cap-parabolic}).
\end{enumerate}
Under these hypotheses, \cref{thm:compact-pincher-realisation} gives compact pinching flows and stability of their terminal profile within a finite-codimensional $C^1$ family of initial data.
For a volume-preserving flow we also obtain a family at prescribed volume, using a transverse variation of the outer cap as in \cref{def:compactification-hypotheses}(ii).

That is the goal of this section: to give a compactification criterion separating the local analysis of the shrinking neck from the geometric and parabolic conditions on the cap.
It is our hope that this separation will be useful for other flows and in other settings, and that the criterion will be of independent interest.
For the present work, verifying these conditions for the D-type and T-type geometries of \cref{def:admissible-outer-cap} in \cref{sec:dumbbell-pinchoff} completes the proof of our main theorem.

\subsection{Matching the neck and choosing coordinates}

For this profile $U$, let $\Gamma_0$ be a reflection-symmetric terminal meridian with the conical neck $r=\alpha|z|$ on $|z|\le4\delta$, with the same cone slope $\alpha$ and a fixed collar width $\delta>0$.
Write $\Sigma_0$ for its surface of revolution and $V_0$ for its enclosed volume as in \cref{rem:main-data}.
All restrictions in $z$ below refer to the neck chart.
At central radius $A>0$ we shall take $r=AU(z/A)$ on $|z|\le2\delta$ and join this neck to the displaced cap on $2\delta\le|z|\le4\delta$.

As in \cref{def:admissible-outer-cap}, removal of the open neck subarc $|z|<\delta$ gives the outer cap $C_0\subset\Sigma_0$ upon revolution.
It retains the two conical annuli $\delta\le|z|\le2\delta$, so the circles $|z|=2\delta$ lie in the interior of the cap.
Let $X_0:C_0\hookrightarrow\R^3$ be the inclusion $X_0(p)=p$, and let $\nu_0:C_0\to\Sphere^2$ be its unit normal pointing out of the region enclosed by $\Sigma_0$.
Write $g_0=X_0^*g_{\rm Eucl}$ for the induced metric and $\surfSpeed_0=\flowSpeed[X_0]$ for the scalar normal speed of the terminal cap.

For the self-similar solution generated by $U$ with terminal time $T$, homogeneity \eqref{eq:general-flow-scaling} and the profile equation \eqref{eq:general-profile-equation} give $\dot A=-\mu A^{-3}$.
Thus the time remaining at central radius $A$ is $\tau=T-t=A^4/(4\mu)$.
As physical time $t$ increases towards $T$, $\tau$ decreases to zero.
On each collar the radius differs from the terminal cone by a term of order $A^4$ in \eqref{eq:general-profile-tail}.
To match this term, we construct the outer reference at the earlier time $t=T-\tau$ by taking the first Taylor term about $T$.
Since the forward normal velocity at $T$ is $\surfSpeed_0\nu_0$, this gives the smooth family of cap embeddings
\begin{equation}\label{eq:outer-terminal-normal-variation}
 X_{\rm out}^{(1)}:C_0\times[0,\tau_0)\to\R^3,\qquad
 X_{\rm out}^{(1)}(p,\tau)=X_0(p)-\tau\surfSpeed_0(p)\nu_0(p).
\end{equation}
The terminal cap is a smooth radius graph on a short extension $[\delta,4\delta+\varepsilon]$ of the conical interval, for some $\varepsilon>0$.
To find its displaced radius at fixed $z$, first consider the terminal circle through the meridian point $(x,\alpha x)$, where $x\in[\delta,4\delta]$.
Write $(Z_\tau(x),R_\tau(x))$ for the axial coordinate and radius of the same circle under \eqref{eq:outer-terminal-normal-variation}.
By homogeneity, its forward radial speed on the terminal cone is $\flowRad[\alpha x]=b x^{-3}$, where $b=\flowRad[\alpha\,\cdot](1)$ may have either sign or vanish.
The expansion \eqref{eq:general-profile-tail} and smooth dependence on the radial jet give $\flowRad[U](\zeta)=b\zeta^{-3}+O(\zeta^{-7})$.
Indeed, rescaling a fixed neighbourhood of $\zeta>0$ to a neighbourhood of one gives $\zeta^{-1}U(\zeta y)=\alpha y+O(\zeta^{-4})$ in every fixed $C^j$ norm.
Comparing the terms of order $\zeta^{-3}$ in \eqref{eq:general-profile-equation} therefore gives
\begin{equation}\label{eq:general-collar-coefficient}
 b+4\mu c_1=0,\qquad -b\tau=c_1A^4.
\end{equation}
The normal velocity $\surfSpeed_0\nu_0$ is the projection of $(0,bx^{-3})$ onto the line spanned by $(-\alpha,1)$.
Thus
\[
 (Z_\tau(x),R_\tau(x))=(x,\alpha x)-\frac{b\tau x^{-3}}{1+\alpha^2}(-\alpha,1).
\]
In particular, with $c=\alpha b/(1+\alpha^2)$,
\begin{equation}\label{eq:collar-circle-position}
 Z_\tau(x)=x+c\tau x^{-3},\qquad R_\tau(x)-\alpha Z_\tau(x)=-b\tau x^{-3}.
\end{equation}
At $\tau=0$ we have $Z_0(x)=x$.
On the conical part $Z_\tau'(x)=1-3c\tau x^{-4}$.
On the short outer extension the displacement \eqref{eq:outer-terminal-normal-variation} is also smooth, and $Z_\tau'\to1$ uniformly as $\tau\to0$.
Choose $\tau_0>0$ sufficiently small that, for $0\le\tau<\tau_0$,
\[
 Z_\tau'(x)>0\quad(\delta\le x\le4\delta+\varepsilon),\qquad
 Z_\tau(\delta)<2\delta,\qquad Z_\tau(4\delta+\varepsilon)>4\delta .
\]
Hence each plane $x_3=z$ with $2\delta\le z\le4\delta$ meets exactly one of these circles.
The direction of axial displacement depends on the sign of $b$.
The retained inner annulus and the smooth outer extension ensure that the entire matching collar remains covered in either direction.

The collar radius $r_{\rm out}^{(1)}:\{z\in\R:2\delta\le|z|\le4\delta\}\times[0,\tau_0)\to(0,\infty)$ is thus given by
\[
 r_{\rm out}^{(1)}(z,\tau)=R_\tau(Z_\tau^{-1}(z))\quad(2\delta\le z\le4\delta),
 \qquad r_{\rm out}^{(1)}(-z,\tau)=r_{\rm out}^{(1)}(z,\tau)
\]
by reflection.
It is smooth in $(z,\tau)$, with $r_{\rm out}^{(1)}(z,0)=\alpha|z|$.

We now estimate the difference between the neck radius $r_{\rm in}$ and the cap radius $r_{\rm out}^{(1)}$ on the conical collar.

\begin{lemma}[Collar error]\label{lem:terminal-collar-match}
Let $U$ be a normalised positive conical profile for $\flowSpeed$ (\cref{def:normalised-conical-profile}), and let $r_{\rm out}^{(1)}$ be constructed from the terminal cap by \eqref{eq:outer-terminal-normal-variation}.
Put $r_{\rm in}(z,A)=AU(z/A)$ and $\tau=A^4/(4\mu)$.
For $A/\delta$ sufficiently small and $2\delta\le |z|\le4\delta$,
\begin{equation}\label{eq:fixed-match-geometric}
 \bigl|\partial_z^j(r_{\rm in}-r_{\rm out}^{(1)})\bigr|
 \le C_j A^8\delta^{-7-j},\qquad j\ge0.
\end{equation}
The same bound holds after applying $A\partial_A$.
\end{lemma}
\begin{proof}
\Cref{lem:appC-collar-matching} gives \eqref{eq:fixed-match-geometric} for the two radii constructed here, including the estimate after applying $A\partial_A$.
\end{proof}
\begin{proofoutline}
For convenience, we record the cancellation responsible for the estimate; the differentiated Taylor remainders are proved in \cref{lem:appC-collar-matching}.
On $2\delta\le z\le4\delta$, let $E$ be the remainder in \eqref{eq:general-profile-tail}.
The inner radius is
\begin{equation}\label{eq:collar-inner-expansion}
 r_{\rm in}(z,A)=AU(z/A)=\alpha z+c_1A^4z^{-3}+c_2A^8z^{-7}+AE(z/A).
\end{equation}
The differentiated bounds $|E^{(j)}(\zeta)|\le C_j\zeta^{-11-j}$ make its last term $O(A^{12}z^{-11})$.
Taylor expansion of the displacement \eqref{eq:outer-terminal-normal-variation}, as in \eqref{eq:collar-outer-taylor}, gives
\[
 r_{\rm out}^{(1)}(z,\tau)=\alpha z-b\tau z^{-3}+O(\tau^2z^{-7}).
\]
By \eqref{eq:general-collar-coefficient}, $-b\tau=c_1A^4$ when $\tau=A^4/(4\mu)$.
The terms of order $A^4$ therefore cancel, leaving an error of order $A^8$.
The derivative estimates in the appendix give the stated powers of $\delta$, and reflection gives the same estimate on the other collar.
\end{proofoutline}

Only the first variation \eqref{eq:outer-terminal-normal-variation} is needed for the construction, since it gives the collar error \eqref{eq:fixed-match-geometric}.
When $U$ has a full differentiated expansion, one can continue the matching formally to arbitrary order.
We do not need it in order to prove the main result, and so have recorded it in the appendix.
We refer the interested reader to \cref{rem:appC-higher-order-matching}.

Let us now form a closed reference surface.
Choose a smooth even cutoff $\xi:\R\to[0,1]$ equal to one on a neighbourhood of $[-2,2]$ and zero on a neighbourhood of $\{|x|\ge4\}$.
Keep $r_{\rm in}(z,A)=AU(z/A)$ on $|z|\le2\delta$.
On each matching collar use the radius
\begin{equation}\label{eq:matched-collar-radius}
 r(z,A)=\xi(z/\delta)r_{\rm in}(z,A)
       +(1-\xi(z/\delta))r_{\rm out}^{(1)}(z,\tau),
 \qquad \tau=A^4/(4\mu).
\end{equation}
Beyond the collars keep the displaced cap \eqref{eq:outer-terminal-normal-variation}.
The pieces agree on neighbourhoods of the joining circles.
Denote their surface of revolution by $\matchSurf_A$ and its parametrisation by $\matchX_A$.

For all $A>0$ sufficiently small we use the same parameter surface $M$ for $\matchX_A$.
It is obtained by joining the two boundary circles of $C_0$ with the cylinder $[-\delta,\delta]\times\Sphere^1$, using the azimuthal angle to identify each end.
We use axial and angular coordinates on the neck and terminal cap coordinates on the fixed outer region.
As in \cref{rem:main-data}, we parametrise a nearby terminal cap by $p\mapsto X_0(p)+\capVariation(p)\nu_0(p)$ for $p\in C_0$, where the smooth rotationally and reflection-symmetric height $\capVariation:C_0\to\R$ is supported away from the neck.

We need a normal tube in which the perturbed surfaces remain embedded.
For the volume-preserving version we also need a cap variation transverse to the volume constraint.

\begin{definition}[Geometric cap condition]\label{def:compactification-hypotheses}
Fix an integer $k\ge10$.
We say that the terminal cap $C_0$ satisfies the \emph{geometric cap condition} if part~(i) below holds uniformly for all sufficiently small $A>0$ and all sufficiently small symmetric $C^{k+8}$ perturbations of $C_0$ supported away from the neck.
We say that it also satisfies the \emph{volume transversality condition} if part~(ii) holds.
\begin{enumerate}[label=(\roman*),leftmargin=*]
\item The matched surface $\matchSurf_A$ is closed and embedded.
There exist constants $c,C>0$ such that, with $\nu$ its outward normal, the map $(p,a)\mapsto p+a\nu(p)$ is injective on the tube given by $|a|<c(A+|z|)$ on the neck $|z|\le2\delta$ and $|a|<c\delta$ elsewhere.
The principal curvatures are bounded in absolute value by $C/(A+|z|)$ and $C/\delta$ on these respective regions.
These bounds persist under sufficiently small symmetric $C^2$ variations supported on the fixed outer region.
\item There is a smooth rotationally and reflection-symmetric function $\phi:M\to\R$, fixed in the cap coordinates, with support in the fixed outer region of $C_0$, a positive distance from the matching collars, such that $\int_{C_0}\phi\,\dd\mu_{g_0}\ne0$.
Here $g_0$ is the induced metric on the terminal cap.\qedhere
\end{enumerate}
\end{definition}

To measure normal heights in the tube of \cref{def:compactification-hypotheses}(i), we use a single smooth length scale on the neck and cap,
\begin{equation}\label{eq:local-geometric-scale}
 \ell_A(p)=\frac{\delta\sqrt{A^2+|p|^2}}{\sqrt{\delta^2+A^2+|p|^2}},\qquad p\in\R^3 .
\end{equation}
Here $|p|$ is Euclidean distance from the apex.
The expression in \eqref{eq:local-geometric-scale} is comparable to $\min\{\delta,\sqrt{A^2+|p|^2}\}$.
Positivity and the conical expansion \eqref{eq:general-profile-tail} give $AU(z/A)\asymp A+|z|$, so $\ell_A\asymp A+|z|$ on $|z|\le2\delta$.
On the collars $|z|\asymp\delta$, and on the fixed outer region the surface stays away from the apex.
Thus $\ell_A\asymp\delta$ on both.
The central circle has radius $A$ since $U(0)=1$.
In the coordinate $\zeta=z/A$ a symmetric perturbed radius $U+v$ has derivative zero at the centre.
If $U''(0)>0$ and $|v''(0)|<U''(0)$, the centre remains a strict local waist.

We next compare the normal velocity of the matched surface with the prescribed geometric speed.
We impose the volume constraint when the flow preserves volume.
\begin{definition}[Volume-preserving flow]\label{def:volume-preserving-flow}
We say that the flow \eqref{eq:general-geometric-flow} is \emph{volume preserving} if every smooth solution $f$ on a closed oriented surface $\Sigma$ satisfies
\[
 \Vol(f(\cdot,t))=\Vol(f_0)\qquad\text{for every }t\in[0,T).
\]
Equivalently, for every smooth initial immersion $f_0:\Sigma\to\R^3$ of a closed oriented surface,
\begin{equation}\label{eq:volume-preserving-speed}
 \int_\Sigma\flowSpeed[f_0]\,\dd\mu_{f_0^*g_{\rm Eucl}}=0.\qedhere
\end{equation}
\end{definition}
Indeed, applying \eqref{eq:volume-preserving-speed} to the immersion $f(\cdot,t)$ gives $\frac{\dd}{\dd t}\Vol(f(\cdot,t))=0$ by the first variation of signed volume used in \cref{lem:area-volume}.
Conversely, each $f_0$ gives a short-time solution by the local theory following \cref{def:geometric-flow-class}, so differentiating the volume identity at $t=0$ gives \eqref{eq:volume-preserving-speed}.

Given smooth embeddings $X_A:M\to\R^3$ with outward normals $\nu_A$ and smooth heights $u_A:M\to\R$ in the normal tube, we use the notation \eqref{eq:normal-graph} to write $X_A^u=X_A+u_A\nu_A$.
Its prescribed normal speed and its normal velocity per unit change in $A$, expressed as functions on $M$, are
\begin{equation}\label{eq:reference-normal-speeds}
 \surfSpeed_A^\sharp(u)=\flowSpeed[X_A^u],\qquad
 \scaleVel(u)=\langle\partial_AX_A^u,\nu_{X_A^u}\rangle .
\end{equation}
Here $\nu_{X_A^u}$ is the outward unit normal to the graph, and $\partial_A$ differentiates the entire family, including $u_A$, at a fixed point of $M$.
Changing coordinates on $M$ adds a tangent vector to $\partial_A X_A^u$, which disappears on projection onto $\nu_{X_A^u}$.
If the central radius evolves at speed $\beta_0(A)=-\mu A^{-3}$, define the invariance defect $\refDef_A$ of the reference family and the volume error of the matched surface by
\begin{equation}\label{eq:matching-defect-definitions}
 \refDef_A
 :=\beta_0(A)\scaleVel(0)-\surfSpeed_A^\sharp(0),
 \qquad
 \Delta V_A:=\Vol(\matchSurf_A)-V_0.
\end{equation}
When $\refDef_A=0$, the normal velocity of the family agrees with $\flowSpeed$.
For the matched family $\matchX_A$ this already holds on the neck.
We estimate $\refDef_A$ on the rest of the surface.
For a volume-preserving flow the same estimate permits a correction to the volume.
We write $F=O_\delta(g)$ when $\|F\|\le C_\delta |g|$ in the stated norm, with a constant that may depend on the collar width $\delta$ and the other fixed data but is uniform as the central radius tends to zero.

\begin{proposition}[Reference family and optional volume correction]\label{prop:fixed-interface-matching}
Let $U$ be as in \cref{lem:terminal-collar-match}.
Suppose $C_0$ satisfies the geometric cap condition of \cref{def:compactification-hypotheses}(i).
For $A/\delta$ sufficiently small, the matched family is closed and embedded; take $\refX_A=\matchX_A$ and $d=0$ unless the volume correction below is imposed, and put $\refSurf_A=\refX_A(M)$.
If the flow also preserves volume and $C_0$ satisfies the volume transversality condition of \cref{def:compactification-hypotheses}(ii), fix the variation $\phi$ from that condition.
There is then a unique small smooth amplitude $d(A)$ for which the corrected family encloses volume $V_0$, with
\begin{equation}\label{eq:outer-volume-amplitude}
 \refX_A=\matchX_A+d(A)\phi\nu_0,\qquad
 |(A\partial_A)^i d(A)|\le C_\delta A^8,\qquad 0\le i\le2 .
\end{equation}
In both cases the chosen family $\refSurf_A$ is closed and embedded.
On $|z|\le2\delta$ it agrees with $AU(z/A)$ and its invariance defect $\refDef_A$ of \eqref{eq:matching-defect-definitions} vanishes.
On the rest of the surface $\refDef_A$ is $O_\delta(A^4)$ in each fixed $C^j$ norm for a smooth cap.
The constants for $0\le j\le k$ are uniform on the stated $C^{k+8}$ neighbourhood.
\end{proposition}

\begin{proof}
\Cref{lem:appC-reference-defect-volume} gives the defect estimates for the matched family \eqref{eq:matched-collar-radius}.
Under the additional volume hypotheses, it also gives the unique correction and \eqref{eq:outer-volume-amplitude}, with the same defect bounds.
The correction is supported away from the collars, so the neck is unchanged; embeddedness in both cases follows from \cref{def:compactification-hypotheses}(i).
\end{proof}
\begin{proofoutline}
We explain the volume correction here; \cref{lem:appC-reference-defect-volume} contains the differentiated estimates and the calculation of the corrected defect.
For the matched family, the defect vanishes on the neck by \eqref{eq:general-flow-scaling} and \eqref{eq:general-profile-equation}.
On the collars, differentiating the $O(A^8)$ matching error in the time remaining $\tau=A^4/(4\mu)$ gives an $O_\delta(\tau)$ error in normal velocity.
The first-order displacement \eqref{eq:outer-terminal-normal-variation} gives the same bound on the fixed cap.
Thus $\refDef_A=O_\delta(\tau)$ on its support, which stays away from the shrinking central circle.

Now assume the two additional volume hypotheses.
The first variation of volume and \eqref{eq:volume-preserving-speed} give
\[
 -\frac{\dd}{\dd\tau}\Delta V_A
 =\int_{\matchSurf_A}\refDef_A\,\dd\mu_A
 =O_\delta(\tau).
\]
The matched surfaces converge to the terminal surface of volume $V_0$, so integration from $\tau=0$ gives $\Delta V_A=O_\delta(\tau^2)=O_\delta(A^8)$.
For the volume change
\[
 W(\tau,b)=\Vol\bigl(\matchX_{A(\tau)}+b\phi\nu_0\bigr)
             -\Vol(\matchX_{A(\tau)}),
 \qquad A(\tau)=(4\mu\tau)^{1/4},
\]
the volume transversality condition gives $W_b(0,0)=\int_{C_0}\phi\,\dd\mu_{g_0}\ne0$.
Solving $W(\tau,d(A(\tau)))=-\Delta V_{A(\tau)}$ therefore gives $d=O_\delta(A^8)$.
The differentiated estimates yield $\partial_\tau d=O_\delta(\tau)$; hence the added velocity changes the defect by $O_\delta(\tau)$, while the normal and prescribed geometric speed change by $O_\delta(\tau^2)$.
Since $\phi$ is supported away from the collars, the corrected family retains both the neck and the required defect bound.
\end{proofoutline}

\begin{remark}[Area of the reference surface]\label{rem:reference-area-defect}
The reference surfaces have uniformly bounded area as $A\to0$.
More precisely, for either reference family of \cref{prop:fixed-interface-matching}, with $\renArea[U]$ defined by the integral \eqref{eq:pincher-relative-area},
\[
 \Area(\refSurf_A)=\Area(\Sigma_0)+A^2\renArea[U]+O_\delta(A^4).
\]
Indeed, substituting $z=A\zeta$ in the area difference on $|z|\le2\delta$ gives $A^2$ times that integral truncated at $|\zeta|=2\delta/A$.
The differentiated expansion \eqref{eq:general-profile-tail} gives $US[U]-\alpha\sqrt{1+\alpha^2}|\zeta|=O(|\zeta|^{-3})$, so the omitted tails contribute at most
\[
 CA^2\int_{2\delta/A}^\infty\zeta^{-3}\,\dd\zeta=O_\delta(A^4).
\]
On the fixed collars and cap, \eqref{eq:collar-outer-taylor}, \eqref{eq:fixed-match-geometric} and \eqref{eq:outer-terminal-normal-variation} give a change of $O_\delta(A^4)$ in the embedding and its first spatial derivatives.
When imposed, the volume correction contributes a further $O_\delta(A^8)$ by \eqref{eq:outer-volume-amplitude}.
\end{remark}

\medskip
The chosen reference family of \cref{prop:fixed-interface-matching} has a normal velocity differing from $\flowSpeed$ by $O_\delta(A^4)$ on the cap.
To correct this error we seek a normal height $u_A:M\to\R$ whose graphs $X_A^u=\refX_A+u_A\refNu_A$, with $A=A(t)$, have the normal velocity prescribed by \eqref{eq:general-geometric-flow}.
Here $\refNu_A$ is the outward unit normal to $\refSurf_A$.
We require $u_A$ to vanish on the central circle so that $A$ still measures its radius, and we allow the scale speed $\beta(A)=\dd A/\dd t$ to vary.
At a fixed point of $M$ the chain rule gives $\partial_t X_{A(t)}^u=\beta(A)\partial_A X_A^u$.
By \eqref{eq:reference-normal-speeds} the graphs therefore have the required normal velocity precisely when
\begin{equation}\label{eq:compact-invariance-equation}
 \surfSpeed_A^\sharp(u_A)
 =\beta(A)\scaleVel(u_A).
\end{equation}
We seek a solution of \eqref{eq:compact-invariance-equation} for every central radius $0<A\le A_0$.
The physical time from $A_0$ to pinchoff depends on the unknown speed $\beta$.
To put the neck equation on the time half-line used in \cref{prop:convective-maximal-regularity}, we take the central radius itself as parameter.
The similarity time
\begin{equation}\label{eq:compact-half-line-time}
 s=\frac1\mu\log\frac{A_0}{A},
 \qquad A=A_0e^{-\mu s}
\end{equation}
increases from zero to infinity as $A$ decreases to zero.
Once the scale speed is known, we recover physical time as follows.
Write its deviation from the reference speed $\beta_0=-\mu A^{-3}$ as
\[
 \vartheta(s)=A(s)^3\{\beta(A(s))-\beta_0(A(s))\}.
\]
Then $\beta=A^{-3}(-\mu+\vartheta)$, so that
\begin{equation}\label{eq:actual-clock-ratio}
 \frac{\dd t}{\dd s}
 =\frac{\dd A/\dd s}{\beta(A)}
 =\frac{A(s)^4}{1-\vartheta(s)/\mu},\qquad t(0)=t_0 .
\end{equation}
Thus $t$ increases with $s$ whenever $\vartheta<\mu$.

On the cap the reference embeddings remain regular as $A$ decreases to zero.
We write the cap equation in the forward time $t_*$ of the reference family so that its time interval is known before we solve for $\beta$.
Setting $\vartheta=0$ in \eqref{eq:actual-clock-ratio} and integrating gives
\begin{equation}\label{eq:outer-base-clock}
 t_*(s)=t_0+\int_0^sA(\sigma)^4\,\dd\sigma
 =t_0+\frac{A_0^4}{4\mu}(1-e^{-4\mu s}),
 \qquad
 T_*=t_0+\frac{A_0^4}{4\mu}.
\end{equation}
Write $I_*=[t_0,T_*)$ for the resulting interval.
The parameter $\tau=A^4/(4\mu)=T_*-t_*$ is the remaining time on the reference family.

The speed defect of \cref{prop:fixed-interface-matching} enters the equation for the cap height as a source.
To control the height on $I_*$ we need to solve the linearised equation with a bound in terms of the initial height and the source, uniform as $A_0$ decreases to zero.
Cutting the cap from the neck would introduce boundary circles where the flow prescribes no boundary data.
Write $K_r=M\setminus\bigl((-r,r)\times\Sphere^1\bigr)$ for $0<r\le\delta$, removing the cylinder in the neck chart.
We glue two copies of $K_{\delta/16}$ along their boundary to form the closed surface
\[
 M_o=K_{\delta/16}^{(1)}\cup_{\partial K_{\delta/16}}K_{\delta/16}^{(2)}.
\]
We call the first copy the \emph{physical copy} and identify it with $K_{\delta/16}\subset M$.
We retain the geometric equation on $K_{\delta/8}^{(1)}$ and extend its coefficients across the joining circles using the annuli $\delta/16<|z|<\delta/8$.

Let $X_A$ denote a reference embedding, with outward unit normal $\nu_A$, and let $w:K_{\delta/16}\times I_*\to\R$ be a smooth height.
Write $X_A^w=X_A+w\nu_A$ and $\nu_w=\nu_{X_A^w}$ as in \eqref{eq:reference-normal-speeds}.
In fixed cap coordinates the velocity of the graph is $(\partial_{t_*}w)\nu_A+\beta_0(\partial_A X_A+w\partial_A\nu_A)$, since $\dd A/\dd t_*=\beta_0(A)=-\mu A^{-3}$.
Taking its normal component, the geometric flow \eqref{eq:general-geometric-flow} becomes
\begin{equation}\label{eq:outer-moving-height}
 \partial_{t_*}w=G_A(w)
 :=\frac{\flowSpeed[X_A^w]
       -\beta_0\langle\partial_A X_A+w\partial_A\nu_A,\nu_w\rangle}
      {\langle\nu_A,\nu_w\rangle}
 \qquad\text{on }K_{\delta/16}\times I_* .
\end{equation}
In particular $G_A(0)=-\refDef_A$ by \eqref{eq:matching-defect-definitions}.
This motivates the following requirement.

\begin{definition}[Parabolic cap condition]\label{def:outer-cap-parabolic}
Fix a smooth background metric on $M_o$ and write $H^j=H^j(M_o)$.
The \emph{graph ball} $B^{C^1}_a(0)$ is the open ball of radius $a$ about the zero height in $C^1(M_o;\R)$, measured in this metric.
We say that the terminal cap $C_0$ satisfies the \emph{parabolic cap condition} for the chosen reference family if there exist $A_{\max},\rho,\ell,C>0$ such that the following properties hold for that family and every $0<A_0\le A_{\max}$.
The bounds in parts~(i)--(ii) are uniform on the $C^{k+8}$ cap neighbourhood of \cref{def:compactification-hypotheses}; those in part~(iii) are uniform on a sufficiently small $C^{k+9}$ neighbourhood within it.
When the volume correction is imposed, we require the same properties for both $\matchX_A$ and $\refX_A$.
These constants may depend on $k$, $\delta$, the profile, the chosen cap neighbourhood, the fixed charts and the background metric.
\begin{enumerate}[label=(\roman*),leftmargin=*]
\item For every height in $B^{C^1}_\rho(0)$, the graph on the physical copy lies in the normal tube and $\langle\nu_A,\nu_w\rangle\ge1/2$.
There is a smooth fourth-order quasilinear operator $G_A^{\rm ext}:B^{C^1}_\rho(0)\cap H^{k+4}\to H^k$ on the double which agrees with $G_A$ of \eqref{eq:outer-moving-height} on $K_{\delta/8}^{(1)}\times I_*$.
The equation $\partial_{t_*}w=G_A^{\rm ext}(w)$ on $M_o\times I_*$ is uniformly parabolic on this ball with a common parabolicity bound for all the stated choices.
\item Let $P_o h=\partial_{t_*}h-D_wG_A^{\rm ext}(0)h$ be the linearisation on the double at zero height.
For every $I=[a,b)$ with $t_0\le a<b\le T_*$ and $b-a\le\ell$, and all $f\in L^2(I;H^k)$ and $w_a\in H^{k+2}$, the problem $P_ow=f$ on $M_o\times I$, $w(\cdot,a)=w_a$, has a unique solution in $H^1(I;H^k)\cap L^2(I;H^{k+4})$.
This solution satisfies
\begin{equation}\label{eq:outer-cap-forward-estimate}
 \|w\|_{H^1(I;H^k)}+\|w\|_{L^2(I;H^{k+4})}+\|w\|_{C(\overline I;H^{k+2})}
 \le C\bigl(\|f\|_{L^2(I;H^k)}+\|w_a\|_{H^{k+2}}\bigr).
\end{equation}
\item On every sufficiently small fixed $H^{k+2}$ ball about zero contained in $B^{C^1}_\rho(0)$, the coefficients of $G_A^{\rm ext}$ have uniform smooth local bounds in the fixed charts.
For $1\le j\le3$ their induced height derivatives satisfy
\begin{equation}\label{eq:cap-local-tame-bounds}
 \begin{split}
 \|D_w^jG_A^{\rm ext}(w)[h_1,\ldots,h_j]\|_{H^k}
 \le C\biggl\{&(1+\|w\|_{H^{k+4}})
                  \prod_{i=1}^j\|h_i\|_{H^{k+2}}\\
 &+\sum_{i=1}^j\|h_i\|_{H^{k+4}}
                    \prod_{\ell\ne i}\|h_\ell\|_{H^{k+2}}\biggr\}.
 \end{split}
\end{equation}
The coefficients and these maps depend $C^1$ on smooth cap variations measured in $C^{k+9}$, with the corresponding uniform bounds for one cap derivative.\qedhere
\end{enumerate}
\end{definition}

We shall also use the height equation in physical time $t$ on fixed local charts away from the origin.
Continue the terminal embedding $X_0$ and its normal $\nu_0$ along the prescribed cones to $K_{\delta/16}$.
Let $w:K_{\delta/16}\times J\to\R$ be a smooth height in the normal chart over $X_0$ on a physical time interval $J$.
At each fixed point of $K_{\delta/16}$ the graph $X_0^{w(\cdot,t)}$ has velocity $(\partial_t w)\nu_0$ by \eqref{eq:normal-graph}.
Taking its normal component in \eqref{eq:general-geometric-flow} gives
\begin{equation}\label{eq:fixed-terminal-cap-height}
 \partial_t w=F_0(w):=\frac{\flowSpeed[X_0^w]}{\langle\nu_0,\nu_{X_0^w}\rangle}.
\end{equation}
Escher, Mayer and Simonett use this fixed normal graph formulation in their local existence proof for surface diffusion \cite[equation~(2.2) and Lemma~2.1]{EscherMayerSimonett1998}.
At zero height \eqref{eq:fixed-terminal-cap-height} gives $F_0(0)=\surfSpeed_0$.

To measure perturbations of either reference family, let $h\in C^1(M;\R)$ be a normal height on the common parameter surface.
At $\refX_A(p)$ it prescribes the displacement $h(p)\refNu_A(p)$.
We measure it by
\begin{equation}\label{eq:relative-height-norm}
 \|h\|_{\relNorm(A)}
 :=\sup_{\refSurf_A}\bigl(\ell_A^{-1}|h|+|\nabla h|\bigr),
\end{equation}
where $\nabla$ is the gradient on $\refSurf_A$.
At $p\in M$ we evaluate the length scale \eqref{eq:local-geometric-scale} at $\refX_A(p)$ on the chosen reference surface and at $\matchX_A(p)$ on the matched surface.

When the optional volume correction of \cref{prop:fixed-interface-matching} is imposed, the reference surface encloses volume $V_0$, but a height $h$ will in general change that volume.
We then choose the scalar $\Theta_A(h)$ so that the graph of $h+\Theta_A(h)\phi$ has volume $V_0$, using the variation $\phi$ of \cref{def:compactification-hypotheses}(ii).
Taking $h$ in a closed complement of $\Span\{\phi\}$ in the symmetric heights then gives a chart for the nearby graphs of volume $V_0$, since projection along $\phi$ recovers $h$.

\begin{lemma}[Uniform fixed-volume chart]\label{lem:uniform-volume-chart}
Suppose the flow preserves volume, $C_0$ satisfies both conditions in \cref{def:compactification-hypotheses}, and $\refX_A$ is the volume-corrected family of \cref{prop:fixed-interface-matching}, with variation $\phi$ as in part~(ii).
For $A_0$ sufficiently small there are constants $r_{\rm vol},\theta_*,C>0$, independent of $0<A\le A_0$, with the following property.
For each symmetric height $h\in C^1(M;\R)$ with $\|h\|_{\relNorm(A)}<r_{\rm vol}$ there is a unique scalar $\Theta_A(h)\in(-\theta_*,\theta_*)$, called the \emph{volume correction}, such that
\begin{equation}\label{eq:uniform-volume-chart}
 \Vol\left(\refX_A+\{h+\Theta_A(h)\phi\}\refNu_A\right)
 =V_0 .
\end{equation}
The graph of height $h+\Theta_A(h)\phi$ is embedded and coincides with the graph of $h$ on $M\setminus\supp\phi$, in particular throughout the neck.
The volume correction is jointly $C^2$ in $(-\log A,h)$, with $\Theta_A(0)=0$ and
\begin{equation}\label{eq:uniform-volume-chart-bound}
 |\Theta_A(h)|\le C\|h\|_{\relNorm(A)},\qquad
 \|(A\partial_A)^jD_h^\ell\Theta_A(h)\|\le C,\qquad 1\le j+\ell\le2 .
\end{equation}
In the partial derivative $\partial_A\Theta_A(h)$, the reference point $\refX_A(p)$ and its normal vary with $A$ while $h(p)$ stays fixed for each $p\in M$.
The height derivatives have the operator norms induced by \eqref{eq:relative-height-norm}.
\end{lemma}

\begin{proof}
\Cref{prop:app-uniform-volume-chart} applies to the corrected reference family and the volume variation of \cref{def:compactification-hypotheses}(ii).
It gives the unique correction, embeddedness, joint $C^2$ dependence and the bounds \eqref{eq:uniform-volume-chart-bound}.
\end{proof}

\begin{proofoutline}
For ease of reading, we explain why the implicit-function argument remains uniform as $A\to0$.
Write $W_A(h)=\Vol(\refX_A+h\refNu_A)-V_0$ for the change in volume produced by the height $h$.
Integration in the normal tube gives
\[
 |W_A(h)|\le C\|h\|_{\relNorm(A)}\int_{\refSurf_A}\ell_A\,\dd\mu_A
 \le C_\delta\|h\|_{\relNorm(A)}.
\]
On $\supp\phi$ the reference surfaces converge smoothly to the terminal cap, so
\[
 \partial_\theta W_A(h+\theta\phi)
 =\int_{\refSurf_A}\phi J_A(p,h+\theta\phi)\,\dd\mu_A,
 \qquad
 \left.\partial_\theta W_A(\theta\phi)\right|_{\theta=0}
 \longrightarrow\int_{C_0}\phi\,\dd\mu_{g_0}\ne0,
\]
where $J_A$ is the normal Jacobian in \eqref{eq:volume-normal-jacobian}.
For uniformly small $h$ and $\theta$, this derivative is therefore bounded away from zero.
Monotonicity gives the unique correction, while the derivative bounds \eqref{eq:volume-functional-derivative-bounds} give two uniformly bounded derivatives of $\Theta_A$ by implicit differentiation.
The detailed estimates in \cref{prop:app-uniform-volume-chart} also ensure that the corrected height remains in the normal tube.
\end{proofoutline}

Normal heights describe perturbations of the whole reference surface.
On the neck, the radial equation defined by \eqref{eq:general-radial-speed} and the estimates in \cref{def:admissible-neck} use the change in radius at a fixed axial position.
A displacement along the normal to a sloping meridian changes the axial position as well as the radius.
To recover the radial height we must therefore compare the two radii at the axial position reached by the displaced point.
\Cref{fig:normal-radial-heights} illustrates this distinction.

Write $J=(-4\delta,4\delta)$ for the neck interval and $\refMer_A(z)=(z,\refRad_A(z))$ for the reference meridian.
On a subinterval $I\subseteq J$ we write the relative $C^1$ norm as
\[
 \|f\|_{A,I}=\sup_{z\in I}\left(\frac{|f(z)|}{A+|z|}+|f'(z)|\right),
 \qquad f\in C^1(I;\R).
\]
This is uniformly equivalent to the restriction of \eqref{eq:relative-height-norm} to $I$.
In the following lemma we pass between normal and radial heights and express the normal velocity in terms of the change in radius at fixed axial position.
Although this change of coordinates is standard, we record the identities explicitly, since they also set and explain much of our notation.

\begin{lemma}[Normal and radial heights]\label{lem:radial-normal-gauge}
Fix a compact subinterval $J'\Subset J$.
There is a threshold $\varepsilon_{J'}>0$, depending on $J'$ only through $\dist(J',\partial J)$, such that the following holds for every $0<A\le A_0$.
If a radial height $q\in C^1(J;\R)$ satisfies $\|q\|_{A,J}<\varepsilon_{J'}$, its graph has a unique normal height $u=\toNormal(q)$ over $J'$ with $|u(z)|\le 2\varepsilon_{J'}(A+|z|)$.
This height and the axial shift $\psi_A(q)$ satisfy
\begin{equation}\label{eq:radial-normal-exact-chart}
 \bigl(x_A(z),(\refRad_A+q)(x_A(z))\bigr)=\refMer_A(z)+u(z)\refNu_A(z),
 \qquad x_A(z)=z+\psi_A(q)(z),
\end{equation}
where $\refNu_A$ is the outward unit normal in the meridian plane.
Conversely, any normal height $u\in C^1(J';\R)$ with $\|u\|_{A,J'}<\varepsilon_{J'}$ gives a radial height $q$ on $x_A(J')\subset J$.
Both changes of coordinates preserve smoothness.
Their respective bounds are
\[
 \|u\|_{A,J'}\le C\|q\|_{A,J},\qquad
 \|q\|_{A,x_A(J')}\le C\|u\|_{A,J'},
\]
with $C$ independent of $A$ and $J'$.
Along a smooth radial variation $q$ the first variations of the normal height and the axial shift at the reference meridian are
\begin{equation}\label{eq:radial-normal-linearisation}
 D\toNormal(0)q=q\langle e_r,\refNu_A\rangle,\qquad
 D\psi_A(0)q=-q\frac{\refRad_A'}{1+\refRad_A'^2}.
\end{equation}

For a smooth family of such graphs, let $F:J'\to\R$ give the scalar normal speed in the reference coordinate $z$.
Write $\nu_{X_A^u}(z)$ for the outward unit normal at the displaced point, whose axial coordinate is $x=x_A(z)$.
A change of radius at fixed $x$ gives a velocity parallel to $e_r$.
Its normal component is $F(z)$, so the corresponding radial speed is
\begin{equation}\label{eq:normal-radial-residual-map}
 (\radSpeed[u]F)(x)
 :=\frac{F(z)}{\langle\nu_{X_A^u}(z),e_r\rangle},
 \qquad z=x_A^{-1}(x),\quad x\in x_A(J').
\end{equation}
In particular, the geometric residual in \eqref{eq:compact-invariance-equation} satisfies
\begin{equation}\label{eq:normal-radial-speed-identity}
 \radSpeed[u]\{\beta\scaleVel(u)-\surfSpeed_A^\sharp(u)\}
 =\beta\partial_A(\refRad_A+q)-\flowRad[\refRad_A+q],
\end{equation}
on $x_A(J')$, where $\partial_A$ on the right holds the axial position fixed.
\end{lemma}
\begin{proof}
\Cref{lem:app-normal-radial-coordinates} proves the existence, uniqueness, uniform bounds and first variations of the coordinate changes.
It remains to verify the velocity identity \eqref{eq:normal-radial-speed-identity}.
Put $\rho_A=\refRad_A+q$ and $Y_A(x)=(x,\rho_A(x))$.
Differentiating \eqref{eq:radial-normal-exact-chart} at a fixed reference coordinate $z$ gives
\[
 \partial_A[Y_A(x_A(z))]
 =(\partial_A\rho_A)(x_A(z))e_r+Y_A'(x_A(z))\,\partial_A x_A(z).
\]
The last term is tangent and hence contributes nothing to the normal speed $\scaleVel(u)$.
Projecting the first term onto the normal gives $\scaleVel(u)$, and the definition of the radial speed \eqref{eq:general-radial-speed} gives $\surfSpeed_A^\sharp(u)$.
Thus
\[
 \scaleVel(u)
  =\langle e_r,\nu_{X_A^u}\rangle\,(\partial_A\rho_A)\circ x_A,\qquad
 \surfSpeed_A^\sharp(u)
  =\langle e_r,\nu_{X_A^u}\rangle\,\flowRad[\rho_A]\circ x_A .
\]
The bounded slope gives $|\langle e_r,\nu_{X_A^u}\rangle|=(1+|\rho_A'\circ x_A|^2)^{-1/2}\ge c>0$.
Division by this signed normal component and composition with $x_A^{-1}$ as in \eqref{eq:normal-radial-residual-map} give \eqref{eq:normal-radial-speed-identity} for either orientation of the neck.
\end{proof}

\begin{figure}[htbp]
 \centering
\begin{tikzpicture}[x=1.05cm,y=1.05cm,font=\small,
  line cap=round,line join=round,>=latex]
  \definecolor{heightnormal}{RGB}{147,88,29}
  \definecolor{heightradial}{RGB}{16,105,111}
  \coordinate (base) at (2.5,1.1875);
  \coordinate (displaced) at (1.69,2.2675);
  \coordinate (radialbase) at (1.69,.678415);

  \draw[gray!65,->] (.15,0)--(4.05,0) node[right] {$z$};
  \draw[gray!65,->] (.15,0)--(.15,3.75) node[above] {$r$};
  \draw[gray!60,densely dotted] (radialbase)--(1.69,0);
  \draw[gray!60,densely dotted] (base)--(2.5,0);

  \draw[black!80,line width=.9pt,domain=.3:3.65,samples=80]
    plot (\x,{.25+.15*\x*\x});
  \node[anchor=west] at (3.72,2.2484) {$r=\refRad_A(z)$};
  \draw[black!65,dashed,line width=.9pt,domain=.45:4.15,samples=100]
    plot ({\x-1.35*.3*\x/sqrt(1+.09*\x*\x)},
          {.25+.15*\x*\x+1.35/sqrt(1+.09*\x*\x)});
  \node[anchor=west] at (3.18,3.68) {$r=(\refRad_A+q)(z)$};

  \draw[heightradial,line width=1pt,->] (radialbase)--(displaced);
  \node[anchor=east,text=heightradial] at (1.57,1.2) {$q(x_A(z))$};
  \draw[heightnormal,line width=1pt,->] (base)--(displaced);
  \draw[heightnormal,line width=.35pt] (2.095,1.7275)--(3.15,2.85);
  \node[anchor=west,text=heightnormal] at (3.18,2.85)
    {$u(z)\refNu_A(z)$};
  \fill[black!80] (base) circle (1.4pt);
  \fill[black!80] (radialbase) circle (1.4pt);
  \fill[black!80] (displaced) circle (1.5pt);
  \node[anchor=north west,inner sep=3pt] at (base)
    {$(z,\refRad_A(z))$};

  \draw[gray!70] (1.69,0)--(1.69,-.065);
  \draw[gray!70] (2.5,0)--(2.5,-.065);
  \node[below=2pt] at (1.69,0) {$x_A(z)$};
  \node[below=2pt] at (2.5,0) {$z$};
  \draw[black!70,->] (2.5,-.63)--(1.69,-.63);
  \node[below=2pt] at (2.095,-.63)
    {$x_A(z)-z=\psi_A(q)(z)$};
\end{tikzpicture}
 \caption{Normal and radial heights in \eqref{eq:radial-normal-exact-chart}, illustrated on a sloping meridian with outward normal pointing towards increasing $r$.
 The normal displacement $u(z)\refNu_A(z)$ reaches axial position $x_A(z)=z+\psi_A(q)(z)$; the radial height $q(x_A(z))$ compares the two curves at that new position.}
 \label{fig:normal-radial-heights}
\end{figure}

\subsection{The global height equation and selection of initial data}

We seek a single normal height $\globalHeight:M\times I_*\to\R$ over the reference family $\refSurf_A$ of \cref{prop:fixed-interface-matching}.
There are two descriptions of this height to keep track of.
On the shrinking neck we rescale by the central radius $A$ and estimate the height for all similarity times $s\ge0$.
On the cap the surface remains regular, so we work in fixed coordinates on the finite reference interval $I_*=[t_0,T_*)$ of \eqref{eq:outer-base-clock}.
On the overlap we compare the two descriptions at the same central radius $A=A_0e^{-\mu s}$, taking $t_*=t_*(s)$, since there both must give the same normal displacement.

We choose the controlled neck coordinates so that the rescaled perturbation decays, while allowing the remaining compatible initial data to vary freely.
For a volume-preserving flow we may also prescribe the enclosed volume.
In that case the neck displacement can change the volume, and a cap displacement can change the neck coordinates required for decay, so we must choose the two together.
There are already several moving pieces in this construction, and by the end of this section we will need more. 
To aid the reader, we present an overview in \Cref{tab:global-height-notation}.

\begin{table}[H]
\centering
\caption{Notation for the global height equation and the choice of initial surface.}
\label{tab:global-height-notation}
\small
\begin{tabular*}{\linewidth}{@{\extracolsep{\fill}}lp{.53\linewidth}r@{}}
\toprule
Notation & Meaning & Reference\\
\midrule
$A,\ \zeta=z/A,\ s$ & Central radius, rescaled axial coordinate and similarity time $s=\mu^{-1}\log(A_0/A)$ & \eqref{eq:dynamic-terminal-scaling}, \eqref{eq:compact-half-line-time}\\
$t_*,\ t$ & Reference time on the cap and physical time along the solution & \eqref{eq:outer-base-clock}, \eqref{eq:actual-clock-ratio}\\
$\refX_A,\ \refNu_A,\ \globalHeight$ & Chosen reference embedding, its unit normal and the height of the normal graph $\refX_A+\globalHeight\refNu_A$ & \cref{prop:fixed-interface-matching}, \eqref{eq:global-height-coefficients}\\
$v,\ w$ & Neck and cap representatives of $\globalHeight$ & \eqref{eq:global-height-representatives}\\
$J_i^A,\ J_o^A$ & Maps taking the two representatives back to the physical surface & \eqref{eq:fixed-interface-injections}\\
$\alpha_i,\alpha_o;\ \beta_i,\beta_o$ & Cutoffs for dividing the data and for joining the local solutions, respectively & \eqref{eq:global-parametrix-cutoffs}\\
$\kappa_i,\ \eta_o^4$ & Cutoffs used to measure the height and forcing on the neck and cap & \eqref{eq:global-parametrix-cutoffs}, \eqref{eq:mixed-product-spaces}\\
$\beta,\ \vartheta=\Theta_\beta(\globalHeight)$ & Scale speed $\beta=\dd A/\dd t$ and its correction $\vartheta=A^3(\beta-\beta_0)$ from the reference speed $\beta_0$ & \eqref{eq:actual-clock-ratio}, \cref{lem:physical-waist-reduction}\\
\addlinespace
$\mixX,\ \mixY_{\rm eq}$ & Spaces for the global height and its equation residual & \eqref{eq:mixed-product-spaces}\\
$\tau\in\compTr,\ b\in\freeTr$ & Full initial trace and its free part $b=\freeProj\tau$ & \eqref{eq:compatible-trace-norm}, \eqref{eq:free-trace-projection}\\
$\Lambda\tau=\Pc\tau_i$ & Controlled component of the initial neck perturbation, determined by decay & \eqref{eq:free-trace-projection}, \eqref{eq:future-functional-identity}\\
$\mixMap_b$ & Map imposing the height equation and the prescribed free trace $b$ & \eqref{eq:fully-defined-section}\\
$h\in\volComp,\ (a,\lambda)$ & In the volume-preserving case, we select $a\in\R^N$ and $\lambda\in\R$ so that $h+(\Psi a,\lambda\phi)$ gives decay and volume $V_0$ & \eqref{eq:compact-mode-lift}, \eqref{eq:joint-future-volume-system}\\
\bottomrule
\end{tabular*}
\end{table}

Let $U$ be an admissible neck for $\flowSpeed$ in the sense of \cref{def:admissible-neck}.
\Cref{prop:finite-core-splitting} gives the projection $\Pc$ onto the controlled space $\Ec$.
To prescribe this component of the initial neck trace we use the compactly supported variations from \cref{lem:compact-core-selection}.
The dimension $N$, spectral cutoff, constants and smallness thresholds may depend on the fixed pair $(\flowSpeed,U)$.

\emph{Compatible representatives.}
We use $\alpha_i$ and $\alpha_o=1-\alpha_i$ to divide the forcing and initial data between neck and cap, and $\beta_i,\beta_o$ to join the resulting local solutions.
We localise the height using the cutoff $\kappa_i$ on the neck and the fourth power of the cutoff $\eta_o$ on the cap.
On the neck subarc choose smooth even functions of $z$ with values in $[0,1]$, constant on small neighbourhoods of the closed regions in the following table.
The shaded bands in the sketch mark the intervals where the cutoffs may change as we move away from the centre along the meridian with $z\ge0$.
\begin{equation}\label{eq:global-parametrix-cutoffs}
\begin{array}{c|cc}
 & \text{equal to }1 & \text{equal to }0\\ \hline
 \alpha_i & |z|\le\delta & |z|\ge5\delta/4\\
 \alpha_o & |z|\ge5\delta/4 & |z|\le\delta\\
 \beta_i & |z|\le3\delta/2 & |z|\ge7\delta/4\\
 \beta_o & |z|\ge3\delta/4 & |z|\le\delta/2\\
 \kappa_i & |z|\le15\delta/8 & |z|\ge2\delta\\
 \eta_o & |z|\ge\delta/4 & |z|\le\delta/8
\end{array}
\qquad
\vcenter{\hbox{%
\begin{tikzpicture}[x=3cm,y=1.05cm,font=\footnotesize]
  \foreach \leftend/\rightend/\topline/\bottomline in {
    0.125/0.25/\eta_o/{\scriptstyle 0\to1},
    0.5/0.75/\beta_o/{\scriptstyle 0\to1},
    1/1.25/{\alpha_i\colon\,\scriptstyle 1\to0}/{\alpha_o\colon\,\scriptstyle 0\to1},
    1.5/1.75/\beta_i/{\scriptstyle 1\to0},
    1.875/2/\kappa_i/{\scriptstyle 1\to0}} {
    \pgfmathsetmacro{\midpoint}{(\leftend+\rightend)/2}
    \pgfmathsetmacro{\bandtop}{sqrt(\rightend*\rightend+0.04)+0.15}
    \fill[black!15] (\leftend,0) rectangle (\rightend,\bandtop);
    \node[above,inner sep=1.5pt,align=center] at (\midpoint,\bandtop)
      {$\topline$\\[-1pt]$\bottomline$};
  }
  \draw[->] (0,0) -- (2.27,0) node[right] {$z$};
  \draw[->] (0,0) -- (0,1.55) node[above] {$r$};
  \draw[thick] plot[domain=0:2.15,samples=50] (\x,{sqrt(\x*\x+0.04)});
  \draw (0,0.2) -- ++(-0.03,0) node[left,inner sep=1pt] {$A$};
  \node[below,inner sep=2pt] at (0,0) {$0$};
  \foreach \x/\tick in {1/\delta,2/2\delta}
    \draw (\x,0) -- ++(0,-0.07) node[below,inner sep=1pt] {$\tick$};
\end{tikzpicture}}}
\end{equation}
On the outer cap extend $\alpha_i,\beta_i,\kappa_i$ by zero and $\alpha_o,\beta_o,\eta_o$ by one.
For each such cutoff $\chi$, write $\chi^A(\zeta,s)=\chi(A(s)\zeta)$ on the neck chart.
In particular $\beta_j=1$ on a neighbourhood of $\supp\alpha_j$ and $(1-\beta_i)(1-\beta_o)=0$.
The fixed gap between $\supp\alpha_j$ and $\supp\partial_z\beta_j$ leaves a neighbourhood of the transition of $\beta_j$ on which the forcing and initial data assigned by $\alpha_j$ vanish.
When we join the local solutions, derivatives falling on $\beta_j$ produce the commutators in \eqref{eq:appC-global-error-identity}.
Interior estimates on these neighbourhoods over the short reference interval give the factor $A_0^2$ in the error bound \eqref{eq:appC-global-right-error}.

We use the fourth power $\eta_o^4$ on the cap because up to four derivatives can fall on the cutoff in \eqref{eq:appC-weighted-commutator}.
Since $\eta_o^4$ vanishes near the doubling circles $|z|=\delta/16$ of \cref{def:outer-cap-parabolic}, $\eta_o^4 \globalHeight$ extends by zero to $M_o$.
We use this same weight for heights, sources and initial traces.

On the neck the first variation $D\toNormal(0)$ in \eqref{eq:radial-normal-linearisation} multiplies the radial height by the radial component $\refNuRad=\langle e_r,\refNu_A\rangle$ of the reference normal.
To express an inner height $v$ or an outer height $w$ on the physical reference surface, put
\begin{equation}\label{eq:fixed-interface-injections}
 (J_i^Av)(z,t_*(s))=A \refNuRad(z)v(z/A,s),\qquad
 J_o^Aw=w\big|_{K_{\delta/16}^{(1)}\times I_*}.
\end{equation}
Here $v$ is defined on $\R\times[0,\infty)$ and $w$ on $M_o\times I_*$.
We use $J_i^A$ on the neck and $J_o^A$ on the physical copy of the cap.
Conversely, the canonical representatives of a global height $\globalHeight$ are
\begin{equation}\label{eq:global-height-representatives}
 v(\zeta,s)=\kappa_i(A\zeta)\frac{\globalHeight(A\zeta,t_*(s))}{A \refNuRad(A\zeta)},
 \qquad w=\eta_o^4 \globalHeight,
\end{equation}
where $v$ is extended by zero on the rest of $\R$ and $w$ is extended by zero to the double.
Thus $v$ is a linear coordinate for the normal height.
To recover the radial displacement itself, fix a compact interval $J'\Subset J$.
If $\|\globalHeight(\cdot,t_*)\|_{A,J'}<\varepsilon_{J'}$, the converse in \cref{lem:radial-normal-gauge} gives a unique radial height $q:x_A(J')\to\R$ by \eqref{eq:radial-normal-exact-chart} with $u=\globalHeight(\cdot,t_*)|_{J'}$.
Since $\kappa_i=1$ on $\supp\alpha_i$ and $\eta_o^4=1$ on $\supp\alpha_o$, the partition $\alpha_i+\alpha_o=1$ reconstructs $\globalHeight$ by
\begin{equation}\label{eq:correction-reconstruction}
 h_A[v,w]=\alpha_iJ_i^Av+\alpha_oJ_o^Aw=\globalHeight .
\end{equation}
Thus a neck height $v$ and a cap height $w$ represent a single global normal height precisely when substituting $\globalHeight=h_A[v,w]$ into \eqref{eq:global-height-representatives} recovers both $v$ and $w$.
The normalisation at the centre $\globalHeight(0,t_*(s))=0$ is equivalent to $v(0,s)=0$.

Write $X_A^{\globalHeight}=\refX_A+\globalHeight\refNu_A$.
The formula for $G_A$ in \eqref{eq:outer-moving-height} applies on all of $M$ with this reference embedding.
To account for a change in the scale speed, separate the contribution of the moving reference surface by writing
\begin{equation}\label{eq:global-height-coefficients}
\begin{alignedat}{2}
 b_A(\globalHeight)&=\langle\refNu_A,\nu_{X_A^{\globalHeight}}\rangle,
 &\qquad&\parbox[c]{.36\linewidth}{the factor converting height velocity to normal velocity,}\\[1.5ex]
 V_A(\globalHeight)&=
 \frac{\langle\partial_A\refX_A+
                  \globalHeight\partial_A\refNu_A,\nu_{X_A^{\globalHeight}}\rangle}{b_A(\globalHeight)},
 &\qquad&\parbox[c]{.36\linewidth}{the reference motion per unit change of $A$, expressed in height coordinates with $\globalHeight$ held fixed,}\\[1.5ex]
 G_A(\globalHeight)&=\frac{\surfSpeed_A^\sharp(\globalHeight)}{b_A(\globalHeight)}-\beta_0 V_A(\globalHeight),
 &\qquad&\parbox[c]{.36\linewidth}{the geometric flow height velocity in reference time $t_*$, for which $\dd A/\dd t_*=\beta_0$.}
\end{alignedat}
\end{equation}
The curvature bound of \cref{def:compactification-hypotheses}(i) gives $\|\dd\refNu_A\|\le C/\ell_A$.
In the normal graph tube,
\[
 b_A(\globalHeight)=\left(1+\left|(I+\globalHeight\,\dd\refNu_A)^{-1}\nabla \globalHeight\right|^2\right)^{-1/2}.
\]
Thus $\|\globalHeight\|_{\relNorm(A)}<r$ gives $b_A(\globalHeight)>1/2$ for some $r>0$ independent of $A$, by \eqref{eq:relative-height-norm}.

For a physical height residual $F:M\times I_*\to\R$ write $(Q_i^AF)(\zeta,s)=A^3 F(A\zeta,t_*(s))/\refNuRad(A\zeta)$ on the neck.
At zero height division by $\refNuRad$ converts normal speed to radial speed by \eqref{eq:normal-radial-residual-map}, and $A^3$ converts physical radial speed to similarity units by \eqref{eq:general-flow-scaling}.
We now recover the unknown scale speed from the equation for $\globalHeight$ at $z=0$.

\begin{lemma}[Eliminating the scale speed]\label{lem:physical-waist-reduction}
Let $\refX_A$ be the reference family of \cref{prop:fixed-interface-matching}.
For a smooth symmetric normal height $\globalHeight$ with $\sup_{s\ge0}\|\globalHeight(\cdot,t_*(s))\|_{\relNorm(A)}<r$ and $\globalHeight(0,t_*)=0$ on $I_*$, the invariance equation \eqref{eq:compact-invariance-equation} is equivalent to $\heightRes_A(\globalHeight,\vartheta)=0$, where
\begin{equation}\label{eq:global-height-residual}
 \heightRes_A(\globalHeight,\vartheta)=(1-\vartheta/\mu)\partial_{t_*}\globalHeight-G_A(\globalHeight)
       +A^{-3}\vartheta V_A(\globalHeight).
\end{equation}
At each $s$ the equation at the centre determines the speed correction $\vartheta(s)=\Theta_\beta(\globalHeight)(s)$ uniquely from the height and its first four spatial derivatives there.
It satisfies $\Theta_\beta(0)=0$ and $(D\Theta_\beta(0)\globalHeight)(s)=\ell(v(\cdot,s))$, with $\ell$ as in \eqref{eq:projected-pincher-operator}.
\end{lemma}
\begin{proof}
By \eqref{eq:reference-normal-speeds} and \eqref{eq:global-height-coefficients}, $\scaleVel(\globalHeight)/b_A(\globalHeight)=\partial_A\globalHeight+V_A(\globalHeight)$.
Divide \eqref{eq:compact-invariance-equation} by $b_A(\globalHeight)$ and substitute $\partial_A\globalHeight=-A^3\partial_{t_*}\globalHeight/\mu$ from \eqref{eq:outer-base-clock}.
Writing the scale speed as $\beta=A^{-3}(-\mu+\vartheta)$ gives \eqref{eq:global-height-residual}.

Since $\globalHeight(0,t_*)=0$ throughout $I_*$, its time derivative vanishes at the centre.
Also $\partial_A\refX_A(0)=e_r$ there, and $\refNu_A(0)=\refNuRad(0)e_r$ with $\refNuRad(0)=\pm1$.
At $\globalHeight(0,t_*)=0$ the quotient defining $V_A$ in \eqref{eq:global-height-coefficients} therefore gives $V_A(\globalHeight)(0)=\refNuRad(0)$.
Evaluating \eqref{eq:global-height-residual} at the centre now gives
\[
 \Theta_\beta(\globalHeight)(s)=A^3G_A(\globalHeight)(0)/\refNuRad(0).
\]
The reference defect vanishes at the centre, giving $G_A(0)(0)=0$.
To compute the first variation take the radial height $q$ of the same graph and put $P(\zeta)=U(\zeta)+A^{-1}q(A\zeta)$ near the centre.
Reflection and the vanishing height at the centre give $q(0)=0$ for every $A$.
Evaluating \eqref{eq:normal-radial-speed-identity} at the centre gives $\beta=\flowRad[\refRad_A+q](0)$, since $\partial_A(\refRad_A+q)(0)=1$.
The scaling identity \eqref{eq:general-flow-scaling} now gives $\vartheta=\flowRad[P](0)+\mu$.
At $\globalHeight=0$ the variation of $P$ near the centre is $v$ by \eqref{eq:radial-normal-linearisation} and \eqref{eq:global-height-representatives}.
Differentiating gives $(\radLin v)(0)=\ell(v)$.
\end{proof}
The scale speed is therefore $\beta=\beta_0+A^{-3}\Theta_\beta(\globalHeight)$.
Substituting it in \eqref{eq:global-height-residual} gives the residual $\heightRes(\globalHeight)=\heightRes_A(\globalHeight,\Theta_\beta(\globalHeight))$.

Fix $k\ge10$ and choose the rates
\begin{equation}\label{eq:mixed-weight-choice}
\begin{alignedat}{2}
 &\gamma_{\rm c}<\eta<\omega,
 &\qquad&\parbox[c]{.36\linewidth}{the decay rate in the weighted norm, chosen so that the future condition \eqref{eq:linear-stable-future-conditions} converges,}\\[1.5ex]
 &0<a<\min\left\{\mu,\frac{\eta}{2(k+1)}\right\},
 &\qquad&\parbox[c]{.36\linewidth}{the expansion rate of the inner interval, small enough to keep the nonlinear estimates and their first two variations summable.}
\end{alignedat}
\end{equation}
These choices are possible because $\gamma_{\rm c}<\omega<\mu/2$ by \eqref{eq:controlled-backward-rate} and \cref{prop:finite-core-splitting}.
The restriction $2(k+1)a<\eta$ gives summability on the expanding inner interval in \cref{cor:app-summable-spatial-remainder}, as used in \cref{prop:appC-mixed-nemytskii}.
The future condition uses the weighted source estimate \eqref{eq:linear-convective-halfline-estimate} and the gap $\eta>\gamma_{\rm c}$; it does not require the stronger projected decay rate of \cref{lem:controlled-mode-localisation}.

On the neck we combine the weighted estimate \eqref{eq:linear-convective-halfline-estimate} with the whole-line conical estimate \eqref{eq:conical-halfline-summable-estimate}.
Recall $Z_\eta$ from \eqref{eq:conical-hilbert-decay}, the trace and conical norms from \eqref{eq:conical-parabolic-trace-norm} and \eqref{eq:conical-halfline-norms}, and the summable fourth-derivative norm $\coneDiss_4$ from \eqref{eq:principal-spatial-dissipation}.
Here $L^2_s X_k=L^2((0,\infty);X_k)$ in similarity time, with norm \eqref{eq:similarity-time-bochner-norm}.
Write $\neckEq=\partial_s-\waistLin$ as in \eqref{eq:convective-linear-equation}.
At each $s$ the canonical inner height vanishes outside a bounded interval.
Its source $\neckEq v$ can extend beyond that interval, where it equals $\ell(v)DU$ by \eqref{eq:convective-linear-straightened}.
This is why the source norm is taken on the whole line.

\begin{definition}[Height, source and speed spaces]\label{def:mixed-height-source-spaces}
For the rates \eqref{eq:mixed-weight-choice}, let $\neckPath$ be the space of real-valued paths $v\in C([0,\infty);\waistTr)$ for which \eqref{eq:mixed-inner-space} is finite, and let $\neckSrc$ be the space of strongly measurable maps $f:[0,\infty)\to X_k$ for which \eqref{eq:mixed-inner-forcing-space} is finite.
Here $\waistTr$ and $X_k$ are the real spaces in \eqref{eq:convective-trace-space}, which impose evenness and vanishing at the centre.
All derivatives are understood in the sense of distributions, and sources agreeing almost everywhere are identified.
For $v\in\neckPath$ and $f\in\neckSrc$, the respective norms are
\begin{align}
 \|v\|_{\neckPath}:={}&
 Z_\eta(v)+\|e^{\eta s}\matD v\|_{L^2_s X_k}
 +\|v(\cdot,0)\|_{\slabTr_\infty^k}
 +\|v\|_{\coneGraph_\infty^k}
 +\|\neckEq v\|_{\coneSrc_{\infty,1}^k}
 +\coneDiss_4(v),
 \label{eq:mixed-inner-space}\\
 \|f\|_{\neckSrc}:={}&
 \|e^{\eta s}f\|_{L^2_s X_k}
 +\|f\|_{\coneSrc_{\infty,1}^k}.
 \label{eq:mixed-inner-forcing-space}
\end{align}

On the closed cap double, write $H^j=H^j(M_o)$ and set
\begin{align}
 \capPath&=H^1(I_*;H^k)\cap L^2(I_*;H^{k+4})\cap C(\overline I_*;H^{k+2}),
 \label{eq:mixed-outer-space}\\
 \capSrc&=L^2(I_*;H^k).
 \label{eq:mixed-outer-forcing-space}
\end{align}
Let $\mixX$ and $\mixY_{\rm eq}$ be the completions, in the norms below, of smooth finite-norm heights $\globalHeight:M\times I_*\to\R$ and residuals $F:M\times I_*\to\R$, respectively.
Both fields are rotationally and reflection-symmetric and vanish at the centre.
With $v$ as in \eqref{eq:global-height-representatives}, set
\begin{align}
 \|\globalHeight\|_{\mixX}&=\|v\|_{\neckPath}+\|\eta_o^4 \globalHeight\|_{\capPath},
 \nonumber\\
 \|F\|_{\mixY_{\rm eq}}
 &=\|\kappa_i^AQ_i^AF\|_{\neckSrc}
      +\|\eta_o^4 F\|_{\capSrc},
 \label{eq:mixed-product-spaces}\\
 \|\vartheta\|_{\speedSpace}
 &=\sup_{s\ge0}e^{\eta s}|\vartheta(s)|.
 \label{eq:mixed-speed-space}
\end{align}
The speed space $\speedSpace$ consists of the continuous scalar functions on $[0,\infty)$ with finite norm \eqref{eq:mixed-speed-space}.
\end{definition}

The residual $\heightRes(\globalHeight)$ satisfies the vanishing condition at the centre for membership in $\mixY_{\rm eq}$, since the choice $\vartheta=\Theta_\beta(\globalHeight)$ in \cref{lem:physical-waist-reduction} gives $\heightRes_A(\globalHeight,\Theta_\beta(\globalHeight))(0,t_*)=0$.
The formula for $\Theta_\beta(\globalHeight)$ in \cref{lem:physical-waist-reduction} uses spatial derivatives through order four at the centre, all of which the trace in $Z_\eta(v)$ controls since $k\ge10$.
Thus \cref{lem:appC-waist-speed} extends $\Theta_\beta$ to a $C^2$ map from a ball in $\mixX$ to $\speedSpace$, with $\|\Theta_\beta(\globalHeight)\|_{\speedSpace}\le C\|\globalHeight\|_{\mixX}$ uniformly in $A_0$.

Here, as in \cref{def:conical-parabolic-spaces}, we use \emph{trace} to mean the initial values of a path.
For a global height these are represented by $v(\cdot,0)$ on the neck and $w(\cdot,t_0)$ on the cap.
We split this trace into a \emph{free part}, which we prescribe, and a \emph{controlled part}, which is chosen so that the neck perturbation decays.
We vary the $N$ controlled coordinates by displacing the initial surface only near the centre.

\begin{definition}[Compatible and free initial traces]\label{def:compatible-free-traces}
Use the spaces of \cref{def:mixed-height-source-spaces} and fix the compactly supported functions $\psi_1,\ldots,\psi_N$ of \cref{lem:compact-core-selection}.
Choose $R_\Psi>0$ so that $\supp\psi_j\subset[-R_\Psi,R_\Psi]$ for every $j$, and assume $A_0R_\Psi<\delta/8$.
For $\globalHeight\in\mixX$, $\initTr \globalHeight$ denotes its initial height, expressed by the initial values $(v(\cdot,0),w(\cdot,t_0))$ of its neck and cap representatives in \eqref{eq:global-height-representatives}.
Give these initial traces the quotient norm
\begin{equation}\label{eq:compatible-trace-norm}
 \compTr=\initTr\mixX,\qquad
 \|b\|_{\compTr}=\inf_{\initTr \globalHeight=b}\|\globalHeight\|_{\mixX}.
\end{equation}

For a full trace $\tau\in\compTr$ write $\tau_i$ for its canonical inner representative and define $\Lambda:\compTr\to \Ec$ by $\Lambda\tau=\Pc\tau_i$.
We call $\tau$ a \emph{free trace} if $\Lambda\tau=0$, so its inner representative has zero projection onto $\Ec$.
For $a\in\R^N$, set
\begin{equation}\label{eq:compact-mode-lift}
 \Psi a=\sum_{j=1}^N a_j\psi_j,\qquad B_{\rm c}a=\Pc\Psi a.
\end{equation}
Using the isomorphism $B_{\rm c}:\R^N\to \Ec$ from \cref{lem:compact-core-selection}, define the controlled lift for $c\in\Ec$ and the free projection for $\tau\in\compTr$ by
\begin{align}
 \ctrlLift c&=(\Psi B_{\rm c}^{-1}c,0),
 \label{eq:controlled-compatible-lift}\\
 \freeProj\tau&=\tau-\ctrlLift\Lambda\tau,
 \qquad \freeTr=\ker\Lambda.
 \label{eq:free-trace-projection}\\[-\jot]
 &\qedhere\notag
\end{align}
\end{definition}

We need these maps to be bounded as the initial central radius tends to zero.

\begin{lemma}[Splitting the compatible trace]\label{lem:compatible-trace-splitting}
For the spaces and maps of \cref{def:compatible-free-traces}, the space $\compTr$ is Banach.
The lift $\ctrlLift:\Ec\to\compTr$ and projection $\freeProj:\compTr\to\freeTr$ are bounded independently of $A_0$, with the profile, cap, $k$, rates and functions $\psi_j$ fixed.
Moreover,
\[
 \Lambda\ctrlLift=I,\qquad
 \compTr=\freeTr\oplus\ctrlLift\Ec .
\]
For $\tau=(\tau_i,\tau_o)\in\compTr$, the quotient norm also controls its two component traces:
\[
 \|\tau_i\|_{\waistTr}+\|\tau_i\|_{\slabTr_\infty^k}
       +\|\tau_o\|_{H^{k+2}(M_o)}
 \le C\|\tau\|_{\compTr}.
\]
\end{lemma}
\begin{proof}
The component trace maps are continuous by \eqref{eq:mixed-inner-space} and \eqref{eq:mixed-outer-space}, so their joint kernel in $\mixX$ is closed.
The quotient trace space \eqref{eq:compatible-trace-norm} is therefore Banach.
The same path bounds control the displayed component norms; taking the infimum over lifts gives their bound by $\|\tau\|_{\compTr}$.
In particular the inner Hilbert trace bounds $\Lambda$ uniformly.

For the controlled lift, the support condition in \cref{def:compatible-free-traces} gives $\kappa_i=1$ and $\eta_o^4=0$ on the support of $z\mapsto(\Psi a)(z/A_0)$.
Thus the pair in \eqref{eq:controlled-compatible-lift} is the trace of the single physical height $J_i^{A_0}\Psi B_{\rm c}^{-1}c$, extended by zero off the neck.
Multiply $\Psi B_{\rm c}^{-1}c$ by a fixed smooth function of $s$, equal to one at zero and supported in $[0,1)$, and apply $J_i^A$ to obtain a physical path lift.
Since $A(s)\le A_0$, its support stays where $\kappa_i=1$ and $\eta_o^4=0$, so its outer representative is zero.
Its inner norm involves only these fixed compact core functions and the fixed time cutoff, so it is bounded by $C\|c\|_{\Ec}$ independently of $A_0$.
This proves the uniform bound for $\ctrlLift$, and $\freeProj=I-\ctrlLift\Lambda$ is consequently bounded as well.
By \eqref{eq:compact-mode-lift},
\[
 \Lambda\ctrlLift c=\Pc\Psi B_{\rm c}^{-1}c=c .
\]
Consequently $\tau=\freeProj\tau+\ctrlLift\Lambda\tau$, with $\Lambda\freeProj\tau=0$, and $\Lambda\ctrlLift=I$ makes the decomposition unique.
\end{proof}

Given a free trace $b\in\freeTr$, we now have two equations for $\globalHeight$.
Its residual for the geometric flow must vanish and its free initial trace must equal $b$.
With $\mixY=\mixY_{\rm eq}\times\freeTr$ carrying the sum norm, collect them in the map
\begin{equation}\label{eq:fully-defined-section}
 \mixMap_b(\globalHeight)=
 \bigl(\heightRes(\globalHeight),\freeProj\initTr \globalHeight-b\bigr),
 \qquad \mixMap=\mixMap_0.
\end{equation}
To write the equation for the inner height, put
\begin{equation}\label{eq:localised-inner-forcing}
 F_i(\globalHeight)=\neckEq v-\kappa_i^AQ_i^A\heightRes(\globalHeight).
\end{equation}
The forcing $F_i(\globalHeight)$ includes the terms produced by localising the physical equation to the neck.
Through these terms a perturbation on the cap can change the controlled initial trace on the neck.

\begin{lemma}[Controlled trace of a solution]\label{lem:physical-future-condition}
Let $b\in\freeTr$ and let $\globalHeight\in\mixX$ solve $\mixMap_b(\globalHeight)=0$ on the ball where the scale speed is defined.
The canonical inner height $v$ of \eqref{eq:global-height-representatives} satisfies, with $F_i$ as in \eqref{eq:localised-inner-forcing},
\begin{equation}\label{eq:future-functional-identity}
 \Pc v(\cdot,0)
 =-\int_0^\infty e^{-s\waistLin_{\rm c}}\Pc F_i(\globalHeight)(s)\,\dd s .
\end{equation}
\end{lemma}
\begin{proof}
Since $\mixMap_b(\globalHeight)=0$, its first component gives $\heightRes(\globalHeight)=0$ and hence $\kappa_i^AQ_i^A\heightRes(\globalHeight)=0$.
By \eqref{eq:localised-inner-forcing}, the inner height therefore solves $\neckEq v=F_i(\globalHeight)$ on the half-line.
To apply \cref{prop:convective-maximal-regularity}\ref{it:convective-halfline}, write $F_i(\globalHeight)=\matD v-\radLin v+\ell(v)DU$ by \eqref{eq:convective-linear-straightened}.
The coefficients of $\radLin$ and their derivatives are bounded by smoothness on the core and \eqref{eq:fredholm-detail-a4-decay}--\eqref{eq:fredholm-detail-lower-decay} on the ends.
Together with the bound \eqref{eq:app-waist-relative-bound} for the normalisation term at $\theta=0$, this gives
\[
 \|e^{\eta s}F_i(\globalHeight)\|_{L^2_s X_k}
 \le \|e^{\eta s}\matD v\|_{L^2_s X_k}
      +C\|e^{\eta s}v\|_{L^2_s Y_k}
 \le C\|v\|_{\neckPath}.
\]
By \eqref{eq:conical-hilbert-decay} and \eqref{eq:mixed-inner-space}, $\sup_{s\ge0}e^{\eta s}\|v(\cdot,s)\|_{\waistTr}<\infty$.
Since $\gamma_{\rm c}<\eta<\omega$ by \eqref{eq:mixed-weight-choice}, the uniqueness in \cref{lem:app-discrete-green} identifies the initial trace of $v$ with that of the half-line solution for forcing $F_i(\globalHeight)$ and stable trace $\Ps v(\cdot,0)$.
Evaluating its controlled component \eqref{eq:linear-stable-future-conditions} at $s=0$ gives \eqref{eq:future-functional-identity}.
\end{proof}

To solve $\mixMap_b(\globalHeight)=0$ for each small $b$, we need an inverse for the linearised equation with its free trace prescribed.
Its norm must stay bounded as the initial central radius tends to zero.
\begin{lemma}[Uniform global-height inverse]
\label{lem:fixed-interface-block-bounds}
Let $U$ be an admissible neck for $\flowSpeed$, let the reference family be that of \cref{prop:fixed-interface-matching}, and suppose $C_0$ satisfies the parabolic cap condition of \cref{def:outer-cap-parabolic}.
For $A_0/\delta$ sufficiently small, $L=D\mixMap(0):\mixX\to\mixY$ is an isomorphism and
\begin{equation}\label{eq:fibre-product-linear-estimate}
 \|\globalHeight\|_{\mixX}
 +\|D\Theta_\beta(0)\globalHeight\|_{\speedSpace}
 \le C\|L\globalHeight\|_{\mixY},
\end{equation}
with $C$ independent of $A_0$.
\end{lemma}
\begin{proof}
\Cref{cor:appC-uniform-fibre-inverse} gives \eqref{eq:fibre-product-linear-estimate} once we identify the operator and trace to which it applies.
Differentiating \eqref{eq:global-height-residual} at zero and substituting $D\Theta_\beta(0)\globalHeight=\ell(v)$ from \cref{lem:physical-waist-reduction} gives
\begin{equation}\label{eq:linear-global-operator}
 L_{\rm eq}\globalHeight=P_A^0\globalHeight+K_A\ell(v),\qquad
 P_A^0=\partial_{t_*}-DG_A(0),\quad K_A=A^{-3}V_A(0).
\end{equation}
Here $P_A^0$ acts on the physical surface $M$.
It agrees on $K_{\delta/8}^{(1)}$ with the operator $P_o$ on the double, since their spatial operators agree there by \cref{def:outer-cap-parabolic}(i).
On $|z|\le2\delta$ the reference radius is $AU(z/A)$ and its residual vanishes by \cref{prop:fixed-interface-matching}.
When we differentiate the product in \eqref{eq:normal-radial-speed-identity} on this interval, the term in which the derivative falls on the multiplier vanishes with the reference residual.
In the units of \eqref{eq:convective-linear-equation}, the linearised identity reads
\begin{equation}\label{eq:global-linear-conjugacy}
 \neckEq v=Q_i^A\{P_A^0J_i^Av+K_A\ell(v)\},\qquad Q_i^AK_A=DU.
\end{equation}
Indeed $Q_i^AK_A=V_A(0)/\refNuRad=\partial_A[AU(z/A)]=U(\zeta)-\zeta U'(\zeta)=DU(\zeta)$ at fixed $z$.
The trace component is $\freeProj\initTr \globalHeight$ by \eqref{eq:fully-defined-section}.
Finally $|I_*|=A_0^4/(4\mu)$ by \eqref{eq:outer-base-clock}, so for $A_0$ sufficiently small the parabolic cap condition gives \eqref{eq:outer-cap-forward-estimate} on the whole interval $I_*$.
Thus \cref{cor:appC-uniform-fibre-inverse} applies to $L$ with this free trace and gives the asserted uniform inverse.
\end{proof}

\begin{proofoutline}
For a prescribed residual $F$ and free trace $b$, \cref{prop:appC-global-parametrix} first constructs auxiliary local solutions $v$ on the neck and $w$ on the cap.
On the neck the compact correction \eqref{eq:appC-global-controlled-choice} makes the initial trace satisfy the controlled future condition.
We then solve \eqref{eq:appC-global-outer-reference} on the cap with the remaining forcing, which includes the scale speed just determined on the neck.
The cap solution satisfies \eqref{eq:outer-cap-forward-estimate}.
The cutoffs $\beta_i,\beta_o$ join the two solutions to give $\mathbf P(F,b)$ in \eqref{eq:appC-global-parametrix} with free trace exactly $b$.
Applying \eqref{eq:global-height-representatives} to this joined height gives its canonical neck and cap representatives, which may differ on the overlap from the auxiliary solutions $v$ and $w$.
Near the centre $\mathbf P(F,b)=J_i^Av$, so its canonical inner representative agrees with $v$ there and its scale functional remains $\ell(v)$.
In the residual $L\mathbf P(F,b)-(F,b)$ the scale speed cancels because its coefficient is $1-\beta_i-\beta_o(1-\beta_i)=(1-\beta_i)(1-\beta_o)=0$.
Only the two spatial commutators in \eqref{eq:appC-global-error-identity} remain.
Since $\|L\mathbf P-I\|\le C_\delta A_0^2$ by \eqref{eq:appC-global-right-error}, the Neumann series gives the bounded right inverse $\mathbf P(L\mathbf P)^{-1}$ for small $A_0$.

For injectivity let $L\globalHeight=0$.
Its trace component gives $\freeProj\initTr\globalHeight=0$, and hence
\[
 \initTr\globalHeight=\ctrlLift\Lambda\initTr\globalHeight
\]
by \eqref{eq:free-trace-projection}.
Put $I=\|v\|_{\neckPath}$ and $O=\|\eta_o^4 \globalHeight\|_{\capPath}$, where $v$ is now the canonical inner representative of $\globalHeight$.
In the neck equation \eqref{eq:appC-global-core-localisation} the only forcing comes from the collar commutator and the tail of the normalising projection.
On the cap the uncut equation is $P_A^0\globalHeight=-K_A\ell(v)$.
Localising it gives the extra term $[P_A^0,\eta_o^4]\globalHeight$, controlled by \eqref{eq:appC-weighted-physical-estimate}.
The controlled lift \eqref{eq:controlled-compatible-lift} is supported where $\eta_o^4=0$, so the initial trace of $\eta_o^4 \globalHeight$ is zero.
Thus \eqref{eq:appC-global-inner-kernel-estimate} and \eqref{eq:appC-speed-cross-bound} give
\[
 I\le C_\delta(A_0^{3/2}O+A_0^4 I),\qquad O\le C_\delta A_0^2 I .
\]
Substituting the cap bound into the neck bound gives $I\le C_\delta(A_0^{7/2}+A_0^4)I$, so $I=0$ and then $O=0$ for small $A_0$.
The reconstruction \eqref{eq:correction-reconstruction} now gives $\globalHeight=0$.
\end{proofoutline}

It remains to estimate the nonlinear remainder and the residual of the reference family.
Even when the free trace is zero, a nonzero reference residual requires a correction to the height.

\begin{lemma}[Smoothness and quadratic remainder]
\label{lem:section-quadratic-remainder}
Let $U$ be an admissible neck for $\flowSpeed$, use the reference family of \cref{prop:fixed-interface-matching}, and suppose the cap satisfies \cref{def:outer-cap-parabolic}(iii).
Then the map $\mixMap$ of \eqref{eq:fully-defined-section} is $C^2$ from a ball about zero in $\mixX$ into $\mixY$.
The ball and the constants below are independent of sufficiently small $A_0$.
If $\mixRem(\globalHeight)=\mixMap(\globalHeight)-\mixMap(0)-L\globalHeight$, then
\begin{equation}\label{eq:section-quadratic-estimate}
 \|\mixRem(\globalHeight)-\mixRem(\widetilde{\globalHeight})\|_{\mixY}
 \le C(\|\globalHeight\|_{\mixX}+\|\widetilde{\globalHeight}\|_{\mixX})
                         \|\globalHeight-\widetilde{\globalHeight}\|_{\mixX}.
\end{equation}
Furthermore,
\begin{equation}\label{eq:section-residual-estimate}
 \kappa_i^AQ_i^A\heightRes(0)=0,\qquad
 \|\mixMap(0)\|_{\mixY}\le C_\delta A_0^6.
\end{equation}
\end{lemma}
\begin{proof}
With $\vartheta=\Theta_\beta(\globalHeight)$ in \eqref{eq:global-height-residual}, \cref{prop:appC-mixed-nemytskii} gives the $C^2$ regularity of $\mixMap$ and the difference estimate \eqref{eq:section-quadratic-estimate}.
At $\globalHeight=0$ we have $b_A(0)=1$, $\Theta_\beta(0)=0$, and $\heightRes(0)=-G_A(0)=\refDef_A$ by \eqref{eq:matching-defect-definitions} and \eqref{eq:global-height-coefficients}.
Thus $\mixMap(0)=(\refDef_A,0)$.
The reference family is exactly the scaled profile throughout $\supp\kappa_i$, so $\kappa_i^AQ_i^A\heightRes(0)=0$.
On the cap, \cref{prop:fixed-interface-matching} gives $\|\eta_o^4\refDef_A\|_{H^k}\le C_\delta A^4$.
Since $\dd t_*=A^4\dd s$ by \eqref{eq:outer-base-clock},
\[
 \|\mixMap(0)\|_{\mixY}^2
 =\|\eta_o^4\refDef_A\|_{\capSrc}^2
 \le C_\delta\int_0^\infty A(s)^{12}\,\dd s
 =\frac{C_\delta A_0^{12}}{12\mu}.
\]
Taking square roots proves \eqref{eq:section-residual-estimate}.
\end{proof}

\begin{proofoutline}
On the neck, the nonlinear condition includes the local estimates in the normal coordinates used here.
The highest-order nonlinear term is a small change of coefficient times $\partial_\zeta^4v$.
Its conical source norm is summable because the path norm \eqref{eq:mixed-inner-space} includes $\coneDiss_4(v)$.
The lower-order spatial terms are summable by \eqref{eq:app-summable-spatial-remainder}, and the scale correction multiplies the first-order geometric term in \eqref{eq:app-normal-neck-coefficient-bound} by the decaying scalar $\Theta_\beta(\globalHeight)$.
On the cap, \eqref{eq:appC-weighted-high-derivative} puts the physical weight $\eta_o^4$ on the highest derivative, and Sobolev embedding bounds the other factors.
The only coefficient multiplying the time derivative in \eqref{eq:global-height-residual} is the scalar $1-\Theta_\beta(\globalHeight)/\mu$, whose dependence on $\globalHeight$ is $C^2$ by \cref{lem:appC-waist-speed}.
Together these estimates give $\|D^2\mixMap(\globalHeight)\|\le C$ on the ball, uniformly in $A_0$.
In particular
\[
 \|D\mixRem(\globalHeight)\|=\|D\mixMap(\globalHeight)-D\mixMap(0)\|
 \le\int_0^1\|D^2\mixMap(t\globalHeight)\|\,\|\globalHeight\|_{\mixX}\,\dd t
 \le C\|\globalHeight\|_{\mixX}.
\]
Integration along the segment from $\widetilde{\globalHeight}$ to $\globalHeight$ gives \eqref{eq:section-quadratic-estimate}.

At zero height the inner residual vanishes because the reference family is exactly the scaled profile throughout $\supp\kappa_i$.
On the cap the defect is $O_\delta(A^4)$, while the physical time element is $A^4\dd s$.
The calculation in the proof combines these powers in the squared source norm and gives the $O_\delta(A_0^6)$ bound.
\end{proofoutline}

We can now solve for the global height.
Its controlled initial trace will be the one determined by \eqref{eq:future-functional-identity}.

\begin{theorem}[Coupled half-line height equation]
\label{thm:terminal-mixed-inverse}
Let $U$ be an admissible neck for $\flowSpeed$, and suppose $C_0$ satisfies the geometric cap condition \cref{def:compactification-hypotheses}(i) and the parabolic cap condition of \cref{def:outer-cap-parabolic}.
For the weights \eqref{eq:mixed-weight-choice} there are radii $\rho,r>0$ such that, for all sufficiently small $A_0/\delta$ and $\|b\|_{\compTr}<\rho$, there is exactly one $\globalHeight$ in the closed ball $\|\globalHeight\|_{\mixX}\le r$ with $\mixMap_b(\globalHeight)=0$.
Here $b\in\freeTr$ is the prescribed free trace.
For smooth admissible cap graphs $X_0+\capVariation\nu_0$, the solution map is jointly $C^1$ in $b$ and the cap height $\capVariation\in C^{k+9}(C_0)$ in the fixed charts.
It satisfies
\begin{equation}\label{eq:coupled-height-speed-bound}
 \|\globalHeight\|_{\mixX}+\|\Theta_\beta(\globalHeight)\|_{\speedSpace}
 \le C_\delta A_0^6+C\|b\|_{\compTr}.
\end{equation}
In particular $\|\Theta_\beta(\globalHeight)\|_{L^\infty}<\mu/2$ after reducing $A_0$ and the trace ball.
\end{theorem}
\begin{proof}
Let $C_L$ be the uniform bound for $\|L^{-1}\|$ in \eqref{eq:fibre-product-linear-estimate}.
By \eqref{eq:section-residual-estimate}, $\|\mixMap_b(0)\|_{\mixY}\le d:=C_\delta A_0^6+\|b\|_{\compTr}$.
The equation $\mixMap_b(\globalHeight)=0$ is equivalent to
\[
 \globalHeight=\heightFix(\globalHeight):=-L^{-1}\bigl(\mixMap_b(0)+\mixRem(\globalHeight)\bigr).
\]
Fix a ball of radius $r>0$ on which \eqref{eq:section-quadratic-estimate} gives $\operatorname{Lip}(\heightFix)\le2C_LCr\le1/2$.
Take $A_0$ and the trace ball small enough that $\|\heightFix(0)\|\le C_Ld\le r/2$.
The map preserves this fixed ball and has a unique fixed point there.
At the fixed point $\|\globalHeight\|_{\mixX}\le C_Ld+\tfrac12\|\globalHeight\|_{\mixX}$, so $\|\globalHeight\|_{\mixX}\le2C_Ld$.
The speed estimate in \cref{lem:appC-waist-speed} gives \eqref{eq:coupled-height-speed-bound}.

The joint $C^1$ dependence follows from \cref{lem:appC-cap-parameter-dependence}, applied to this contraction on its common ball.
The $C^{k+9}$ topology allows four derivatives for the displaced reference embedding, one for its normal and four for the operator $\flowSpeed$.
\end{proof}

For $\varepsilon=A_0/\delta$ let $v_\varepsilon[b]$ denote the inner representative of the solution with free trace $b$, and put $\Phi_\varepsilon(b)=\Pc v_\varepsilon[b](\cdot,0)$.
We now describe its full initial trace and, when the flow preserves volume, impose the additional volume constraint.

\begin{lemma}[The selected initial traces]\label{lem:general-trace-selection}
Under the hypotheses of \cref{thm:terminal-mixed-inverse}, the full initial traces of its solutions form a $C^1$ graph of codimension $N=\dim\Ec$ in $\compTr$, parametrised by
\begin{equation}\label{eq:general-selected-trace-graph}
 \tau(b)=b+\ctrlLift\Phi_\varepsilon(b),\qquad b\in\freeTr,\quad \|b\|_{\compTr}<\rho .
\end{equation}
This graph depends $C^1$ on the cap in the topology of \cref{thm:terminal-mixed-inverse}.
\end{lemma}
\begin{proof}
The trace splitting \eqref{eq:free-trace-projection} gives $\tau=\freeProj\tau+\ctrlLift\Lambda\tau$.
For the solution with free trace $b$ we have $\freeProj\tau=b$ and $\Lambda\tau=\Phi_\varepsilon(b)$, proving \eqref{eq:general-selected-trace-graph}.
The solution map and the trace maps are $C^1$ by \cref{thm:terminal-mixed-inverse,lem:compatible-trace-splitting}.
The projection $\freeProj$ inverts this parametrisation, and the complementary space $\ctrlLift\Ec$ has dimension $N$.
\end{proof}

If the flow preserves volume and the cap has the volume variation of \cref{def:compactification-hypotheses}(ii), we can restrict this family to a prescribed volume $V_0$.
For this additional construction, use the volume-corrected reference family in \cref{prop:fixed-interface-matching} and its variation $\phi$.
Extend $\phi$ from the physical copy to the cap double $M_o$ by zero across the joining circles, which lie outside its support.
Varying the initial height by $\phi$ gives the compatible trace $p_\phi=(0,\phi)\in\freeTr$.
To separate this volume coordinate from the remaining free data, give $M_o$ the fixed background measure of \cref{def:outer-cap-parabolic} and put
\[
 \volComp=\{b=(b_i,b_o)\in\freeTr:
                          \langle b_o,\phi\rangle_{L^2(M_o)}=0\}.
\]
The functional in this definition is continuous by the outer trace estimate of \cref{lem:compatible-trace-splitting} and takes the value $\|\phi\|_{L^2(M_o)}^2>0$ on $p_\phi$.
Hence
\[
 \freeTr=\volComp\oplus\Span\{p_\phi\}.
\]
We may now prescribe the component in $\volComp$ and vary the coefficient of $p_\phi$ together with the $N$ neck coordinates in $\Psi a$ from \eqref{eq:compact-mode-lift}.

\begin{lemma}[Selection of the neck coordinates and volume]\label{lem:fully-coupled-selection}
Let $U$ be an admissible neck for $\flowSpeed$, suppose the flow preserves volume in the sense of \cref{def:volume-preserving-flow}, and suppose $C_0$ satisfies both parts of \cref{def:compactification-hypotheses} and the parabolic cap condition of \cref{def:outer-cap-parabolic}.
For $\varepsilon=A_0/\delta$ sufficiently small there are radii $\rho,r>0$, independent of $A_0$, such that every $h\in\volComp$ with $\|h\|_{\compTr}<\rho$ determines a unique $(a,\lambda)\in\R^N\times\R$ with $|(a,\lambda)|<r$ for which $h+(\Psi a,\lambda\phi)$ is the initial trace of a solution from \cref{thm:terminal-mixed-inverse} of volume $V_0$.
The volume equals $V_0$ for every $0<A\le A_0$.
The selected coordinates depend $C^1$ on $h$, and their graph has codimension $N+1$ in $\compTr$ and $N$ in its fixed-volume hypersurface.
\end{lemma}
\begin{proof}
For $h=(h_i,h_o)$ let $\tau=h+(\Psi a,\lambda\phi)$ be the proposed full trace and $b=\freeProj\tau=h+\lambda p_\phi$ its free part.
The controlled trace $\Phi_\varepsilon(b)$ is supplied by \cref{thm:terminal-mixed-inverse}.
It depends $C^1$ on $b$ and is given by \eqref{eq:future-functional-identity}.
By \eqref{eq:free-trace-projection}, the actual trace equals $\tau$ exactly when $\Phi_\varepsilon(b)=B_{\rm c}a$.
Thus the equations for the unknowns $(a,\lambda)$ are
\begin{equation}\label{eq:joint-future-volume-system}
\begin{aligned}
 \traceVol_{\rm c}^\varepsilon(h,a,\lambda)&=B_{\rm c}a-\Phi_\varepsilon(b)=0,\\
 \traceVol_V^\varepsilon(h,a,\lambda)&=
 \Vol\bigl(\refX_{A_0}+u_0\refNu_{A_0}\bigr)-V_0=0,
\end{aligned}
\end{equation}
Here $u_{h,a}=h_{A_0}[h_i+\Psi a,h_o]$ is the initial normal height before the cap adjustment.
Since \eqref{eq:correction-reconstruction} sends $p_\phi$ to $\phi$, the height reconstructed from $\tau$ is $u_0=u_{h,a}+\lambda\phi$.
For small data, uniqueness in \eqref{eq:uniform-volume-chart} therefore identifies the second equation in \eqref{eq:joint-future-volume-system} with $\lambda=\Theta_{A_0}(u_{h,a})$.
We must therefore choose $a$ and $\lambda$ together, since $a$ enters the volume correction and $\lambda$ enters the free trace $b=h+\lambda p_\phi$ which determines the required neck correction.

We compute the Jacobian of these two equations at $(h,a,\lambda)=(0,0,0)$.
In the first row $D_a\traceVol_{\rm c}^\varepsilon=B_{\rm c}$ since $b=h+\lambda p_\phi$ is independent of $a$.
The upper-right entry is $-D_b\Phi_\varepsilon(0)p_\phi$.
\Cref{lem:appC-off-diagonal-outer} bounds this by $C_\delta(A_0^2+A_0^6)\|\phi\|_{H^{k+2}}$.
Indeed, $\phi$ is supported away from the collar, so the approximate inverse has zero controlled trace on this datum.
The correction to the approximate inverse contributes $O_\delta(A_0^2)$ by \eqref{eq:appC-global-right-error}.
Replacing the linearisation at zero height by the linearisation at the solution with $b=0$ contributes $O_\delta(A_0^6)$ by \eqref{eq:coupled-height-speed-bound} and \eqref{eq:section-quadratic-estimate}.

For the volume row set $A=A_0$; the first variation in \eqref{eq:volume-normal-jacobian} integrates the physical normal height.
On the support of $\psi_j$ that height is $A_0\refNuRad(A_0\zeta)\psi_j(\zeta)$ by \eqref{eq:fixed-interface-injections}; the area element there is $A_0^2U(\zeta)S[U](\zeta)\,\dd\zeta\,\dd\theta$.
Hence at the origin
\[
 \begin{aligned}
 \partial_{a_j}\traceVol_V^\varepsilon
 &=2\pi A_0^3\int_{\R}\refNuRad(A_0\zeta)\psi_j(\zeta)U(\zeta)S[U](\zeta)\,\dd\zeta
   =O(A_0^3),\\
 \partial_\lambda\traceVol_V^\varepsilon
 &=\int_{\refSurf_{A_0}}\phi\,\dd\mu_{A_0}\longrightarrow c_\phi\ne0 .
 \end{aligned}
\]
The first integral is over the fixed compact support of $\psi_j$.
In the second limit, $c_\phi:=\int_{C_0}\phi\,\dd\mu_{g_0}\ne0$ by the volume transversality condition.
The limit follows from smooth convergence on $\supp\phi$, as in \cref{lem:uniform-volume-chart}.
Thus, as $A_0/\delta\downarrow0$,
\begin{equation}\label{eq:joint-selection-jacobian}
 D_{(a,\lambda)}
 \begin{pmatrix}\traceVol_{\rm c}^\varepsilon\\
                 \traceVol_V^\varepsilon\end{pmatrix}(0,0,0)
 =
 \begin{pmatrix}
  B_{\rm c}&o(1)\\
  o(1)&c_\phi+o(1)
 \end{pmatrix}.
\end{equation}
The limiting matrix $\diag(B_{\rm c},c_\phi)$ is invertible.
The origin need not solve the selection equations: zero free trace may still require a controlled neck correction.
Since the reference surface has volume $V_0$, \eqref{eq:coupled-height-speed-bound} gives
\[
 \begin{pmatrix}\traceVol_{\rm c}^\varepsilon\\
                 \traceVol_V^\varepsilon\end{pmatrix}(0,0,0)
 =\begin{pmatrix}-\Phi_\varepsilon(0)\\0\end{pmatrix}
 =O_\delta(A_0^6).
\]
\Cref{lem:appC-uniform-joint-selection} applies to \eqref{eq:joint-future-volume-system} and gives a unique small solution $(a,\lambda)$ on balls independent of $A_0$, with $C^1$ dependence on $h$ and on the cap in the stated topology.
Projection onto the direct summand $\volComp$ recovers $h$ from the resulting full trace, so this parametrisation is a $C^1$ embedding after shrinking the ball.
Since
\begin{equation}\label{eq:selected-trace-decomposition}
 \compTr=\volComp\oplus\Span\{p_\phi\}
                 \oplus\ctrlLift \Ec,
\end{equation}
its image has codimension $N+1$ in $\compTr$.
The volume equation is transverse to $p_\phi$, so the same image has codimension $N$ within the fixed-volume hypersurface.
It remains to show that the volume stays equal to $V_0$ as the central radius shrinks.
Let $\Sigma_A^{\globalHeight}$ be the image of the selected immersion $X_A^{\globalHeight}=\refX_A+\globalHeight\refNu_A$ with its induced orientation.
For signed volume, the geometric first-variation formula and the invariance equation \eqref{eq:compact-invariance-equation}, together with \eqref{eq:volume-preserving-speed}, give
\begin{equation}
 \beta(A)\frac{\dd}{\dd A}\Vol(X_A^{\globalHeight})
 =\int_{\Sigma_A^{\globalHeight}}\surfSpeed_A^\sharp(\globalHeight)\,\dd\mu
 =\int_{\Sigma_A^{\globalHeight}}\flowSpeed[X_A^{\globalHeight}]\,\dd\mu=0.
 \label{eq:volume-propagation-along-invariant-family}
\end{equation}
On each compact interval of positive central radii the path norms \eqref{eq:mixed-inner-space} and \eqref{eq:mixed-outer-space} give an absolutely continuous height in the lower Sobolev norms.
The volume chain rule thus holds almost everywhere there and the volume is continuous at the initial trace.
Since $\beta=A^{-3}(-\mu+\Theta_\beta(\globalHeight))<0$ by \eqref{eq:coupled-height-speed-bound}, \eqref{eq:volume-propagation-along-invariant-family} gives $\Vol(X_A^{\globalHeight})=V_0$ for every $A$.
\end{proof}

With or without the volume constraint, the selected family is parametrised by its central radius.
We now recover physical time and check that its surfaces stay embedded as the central radius shrinks.

\begin{lemma}[Invariant families generate flows]
\label{lem:invariant-family-generates-flow}
Let $\flowSpeed$ satisfy the conditions of \cref{def:geometric-flow-class}, and let $X_A:M\to\R^3$, $0<A\le A_0$, be a $C^1$ family of smooth embeddings of a closed surface, with outward unit normal $\nu_A$, and suppose
\begin{equation}\label{eq:invariant-family-hypothesis}
 \flowSpeed[X_A]=\beta(A)
 \langle\partial_AX_A,\nu_A\rangle,
 \qquad \beta(A)<0.
\end{equation}
If $A_t=\beta(A)$, then, after a time-dependent reparametrisation of the domain, $X_{A(t)}$ solves the geometric flow \eqref{eq:general-geometric-flow}.
\end{lemma}
\begin{proof}
At $A=A(t)$ the velocity of the embedding is $\beta(A)\partial_AX_A$, whose normal component is $\flowSpeed[X_A]$ by \eqref{eq:invariant-family-hypothesis}.
Write its tangential component as $\dd X_A(W(\cdot,t))$ and let $\Phi(\cdot,t)$ solve $\partial_t\Phi(p,t)=-W(\Phi(p,t),t)$ with $\Phi(\cdot,t_0)=\mathrm{id}_M$.
Then $f(p,t)=X_{A(t)}(\Phi(p,t))$ satisfies
\[
 \partial_t f(\cdot,t)
   =\bigl(\beta(A)\partial_A X_A-\dd X_A(W(\cdot,t))\bigr)\circ\Phi(\cdot,t)
   =(\flowSpeed[X_A]\nu_A)\circ\Phi(\cdot,t).
\]
By invariance under reparametrisation, the last expression is the right-hand side of \eqref{eq:general-geometric-flow} for $f$.
\end{proof}

To apply the preceding lemma we need an embedding at every central radius.
The global estimate controls the normal displacement in different norms on the neck and on the cap.
On $|z|\le2\delta$ write $R(s)=\delta/A(s)$, as in \eqref{eq:transported-conical-radius}, and put
\[
 \coneSize_{1,R}(v;s)
 =\sup_{|\zeta|\le2R(s)}\langle\zeta\rangle^{-1}|v(\zeta,s)|
  +\sup_{|\zeta|\le2R(s)}|\partial_\zeta v(\zeta,s)|.
\]
These are the relative height and slope in \eqref{eq:pincher-global-smallness}.
Together with the cap $C^1$ norm they control the normal displacement on the whole surface.

\begin{lemma}[Quantitative embeddedness]
\label{lem:quantitative-embeddedness}
Suppose the cap satisfies the tube and curvature bounds in \cref{def:compactification-hypotheses}(i).
There is $\varepsilon_*>0$, independent of sufficiently small $A_0$, such that every normal graph $\refX_A+\globalHeight\refNu_A$ with representatives $(v,w)$ from \eqref{eq:global-height-representatives} satisfying
\begin{equation}\label{eq:embeddedness-smallness}
 \sup_s\coneSize_{1,R}(v;s)
 +\sup_{t_*\in I_*}\|w(t_*)\|_{C^1(M_o)}
 <\varepsilon_*
\end{equation}
is embedded and diffeomorphic to the reference surface.
Its neck is a positive radius graph in the axial coordinate.
\end{lemma}
\begin{proof}
\Cref{cor:appC-analytic-embeddedness} proves this under the tube and curvature bounds in \cref{def:compactification-hypotheses}(i).
The representatives here are those in its hypothesis.
\end{proof}

\begin{proofoutline}
The estimate connecting the neck and cap norms to the normal tube is
\[
 \|\globalHeight\|_{\relNorm(A)}
 \le C_\delta\bigl(\coneSize_{1,R}(v;s)+\|w(t_*(s))\|_{C^1(M_o)}\bigr).
\]
Indeed, where $\kappa_i=1$ the height is $A\refNuRad v$, and where $\eta_o^4=1$ it is $w$; these regions cover the surface.
For small data the graph remains inside the injective tube.
The same relative bound keeps its neck radius positive and the derivative of its axial coordinate positive.
\end{proofoutline}

It remains to identify the terminal time and the limit of the flow in the neck and outer coordinates.

\begin{lemma}[Terminal time and smooth limits]\label{lem:compact-terminal-limits}
Let $U$ be an admissible neck for $\flowSpeed$ and let $\globalHeight\in\mixX$ be a small solution of \eqref{eq:compact-invariance-equation} over the reference family of \cref{prop:fixed-interface-matching}, with smooth initial trace.
Suppose the resulting surfaces are embedded, their necks are radius graphs, and the physical outer height equation is uniformly parabolic in fixed local graph charts.
If $\vartheta=\Theta_\beta(\globalHeight)\in\speedSpace$ with $\|\vartheta\|_{L^\infty}<\mu/2$, then the flow reaches $A=0$ at a finite time $T$ with $A(t)\sim\{4\mu(T-t)\}^{1/4}$.
The rescaled neck converges to $U$ in $C^\infty_{\rm loc}(\R)$, and the surfaces converge smoothly on compact sets away from the origin.
\end{lemma}
\begin{proof}
We first show that the central radius reaches zero in finite physical time.
Since $|\vartheta|<\mu/2$, equations \eqref{eq:compact-half-line-time} and \eqref{eq:actual-clock-ratio} give $0<\dd t/\dd s\le 2A_0^4e^{-4\mu s}$.
Hence $T=\lim_{s\to\infty}t(s)$ is finite.
To find the rate at which the central radius shrinks, we use $|\vartheta(s)|\le Ce^{-\eta s}$ from \eqref{eq:mixed-speed-space} and integrate from $s$ to infinity.
This gives
\begin{equation}\label{eq:compact-terminal-time}
 T-t(s)=\int_s^\infty\frac{A_0^4e^{-4\mu\sigma}}{1-\vartheta(\sigma)/\mu}\,\dd\sigma
 =\frac{A(s)^4}{4\mu}\{1+O(e^{-\eta s})\}.
\end{equation}
For the neck limit put $p(\zeta,s)=A(s)^{-1}r(A(s)\zeta,t(s))-U(\zeta)$.
Fix compact intervals $K\Subset K'$ in $\R$.
For large $s$, $A(s)K'\subset(-\delta,\delta)$, where the reference has radius $AU(z/A)$ and $\kappa_i=1$.
In units of $A$ the reference meridian is $(\zeta,U(\zeta))$, with outward reference unit normal $(\refNu_z,\refNu_r)$ independent of $A$.
Its normal height is $u=\refNu_r v$ by \eqref{eq:global-height-representatives}.
Thus \eqref{eq:radial-normal-exact-chart} takes the form
\begin{equation}\label{eq:rescaled-radial-normal-chart}
 x(\zeta)=\zeta+u(\zeta)\refNu_z(\zeta),\qquad
 p(x(\zeta),s)=U(\zeta)-U(x(\zeta))+u(\zeta)\refNu_r(\zeta).
\end{equation}
The weight of $H^k_2$ has a positive lower bound on $K'$.
From \eqref{eq:mixed-inner-space} and Sobolev embedding we obtain
\[
 \|u(\cdot,s)\|_{C^{k-1}(K')}\le C_{K'}e^{-\eta s}.
\]
For large $s$ this gives $x'\ge1/2$ and $K\subset x(K')$.
Differentiating the formula for $p\circ x$ and composing with $x^{-1}$ gives $\|p(\cdot,s)\|_{C^{k-2}(K)}\le C_{K,K'}e^{-\eta s}$.

\Cref{prop:appC-classical-terminal-regularity}(ii) upgrades this convergence to $C^\infty_{\rm loc}$ and gives the smooth limit in fixed physical charts away from the origin.
Its hypotheses hold by the graphical-neck and outer-parabolicity assumptions, the mixed norm bound and \eqref{eq:compact-terminal-time}.
\end{proof}

We can now give the detailed formulation of \cref{thm:general-compactification-intro}, specifying which compact initial surfaces pinch and how their rescaled necks approach $U$.
\begin{theorem}[Compactification of a positive conical profile]
\label{thm:compact-pincher-realisation}
Let $U$ be an admissible neck for $\flowSpeed$ in the sense of \cref{def:admissible-neck}, fix an integer $k\ge10$, and suppose $C_0$ satisfies the geometric cap condition \cref{def:compactification-hypotheses}(i) and the parabolic cap condition of \cref{def:outer-cap-parabolic}.
For sufficiently small initial central radius $A_0>0$, the compatible trace space $\compTr$ contains a $C^1$ submanifold $\pinMan^k$ of codimension $N=\dim\Ec$, parametrised by the free trace $b$ in an open ball in $\freeTr$ as in \eqref{eq:general-selected-trace-graph}.
If the flow preserves volume and $C_0$ satisfies the volume transversality condition of \cref{def:compactification-hypotheses}(ii), we may prescribe the volume $V_0$ using the fixed-volume reference family of \cref{prop:fixed-interface-matching}.
In these coordinates the slice $\pinMan^k\cap\{\Vol=V_0\}$ is a $C^1$ graph over an open ball in $\volComp$ with free parameter $h$.
It has codimension $N$ in the hypersurface $\{\Vol=V_0\}$ and $N+1$ in $\compTr$, and its flows preserve $V_0$.

With either choice of initial data, every smooth trace in the family gives a closed embedded surface of revolution whose flow \eqref{eq:general-geometric-flow} remains embedded up to a finite pinching time $T$.
Its central radius $A(t)=r(0,t)$ and rescaled profile satisfy
\begin{align}
 A(t)
 &=\{4\mu(T-t)\}^{1/4}(1+o(1)),
 \label{eq:compact-pincher-radius}\\
 \frac{r(A(t)\zeta,t)}{A(t)}
 &\longrightarrow U(\zeta)
 \quad\hbox{locally smoothly}.
 \label{eq:compact-pincher-profile}
\end{align}
The circle of radius $A(t)$ in the plane $z=0$ shrinks to the origin, and the surfaces converge smoothly on compact sets away from it.
The limiting outer cap is a small normal graph over the reference cap.

Fix the cap and $A_0$, and let $v_\varepsilon[b]$ and $v_\varepsilon[\widetilde b]$ be the inner heights in \eqref{eq:global-height-representatives} corresponding to two free traces $b,\widetilde b$ in the parameter ball, where $\varepsilon=A_0/\delta$.
Compare the flows at the same central radius $A$, so that both have similarity time $s=\mu^{-1}\log(A_0/A)$.
For the rate $\eta$ of \eqref{eq:mixed-weight-choice} we have
\begin{equation}\label{eq:general-relative-stability}
 \sup_{s\ge0}e^{\eta s}\|v_\varepsilon[b](s)-v_\varepsilon[\widetilde b](s)\|_{H^k_2}
 \le C\|b-\widetilde b\|_{\compTr}.
\end{equation}
On the fixed-volume slice the same estimate holds for the inner heights corresponding to $h,\widetilde h\in\volComp$ in the parameter ball, with $C\|h-\widetilde h\|_{\compTr}$ on the right-hand side.

For smooth admissible cap graphs $X_0+\capVariation\nu_0$, the parametrisation of the initial data and the limiting outer height in $H^{k+2}$ are jointly $C^1$ functions of the free trace and the cap height $\capVariation\in C^{k+9}(C_0)$.
\end{theorem}
\begin{proof}
Take $A_0$ sufficiently small in the reference family of \cref{prop:fixed-interface-matching}.
For each small $b\in\freeTr$, \cref{thm:terminal-mixed-inverse} gives a family satisfying the invariance equation \eqref{eq:compact-invariance-equation}.
Its initial traces form the graph of codimension $N$ in \cref{lem:general-trace-selection}.
This family is infinite dimensional: smooth symmetric heights supported on an open part of the outer cap away from the neck give infinitely many free traces with zero inner component.

The solution bound \eqref{eq:coupled-height-speed-bound} gives both smallness conditions in \eqref{eq:embeddedness-smallness}.
On the neck, the $j=0,1$ terms of \eqref{eq:conical-parabolic-graph-norm} control the relative height and slope by Sobolev embedding on each box.
On $M_o$, the trace in \eqref{eq:mixed-outer-space} controls $w$ in $C(\overline I_*;H^{k+2})$ and hence in $C(\overline I_*;C^1)$.
Thus \cref{lem:quantitative-embeddedness} gives an embedded surface with a graphical neck for every $A>0$.

The same bound gives $|\vartheta|<\mu/2$, so $\beta=A^{-3}(-\mu+\vartheta)<0$ and \eqref{eq:actual-clock-ratio} defines physical time.
On each compact interval of positive central radii the clock change is regular, and \cref{prop:appC-classical-terminal-regularity}(i) makes the strong height solution classical.
\Cref{lem:invariant-family-generates-flow} removes the tangential part of the velocity.
For smooth initial traces the resulting flow agrees with the local classical solution by the uniqueness discussed after \cref{def:geometric-flow-class}.

On $K_{\delta/16}$ the reference embeddings converge in $C^2$ to $X_0$ by \eqref{eq:outer-terminal-normal-variation}, \eqref{eq:collar-inner-expansion} and \eqref{eq:matched-collar-radius}.
When the volume correction is used, it also tends to zero by \eqref{eq:outer-volume-amplitude}.
The uniform outer tube on the larger $K_{\delta/32}$ gives fixed local graph charts on $K_{\delta/16}$.
In these charts the conical lower-derivative bounds and the physical trace bound make the height small in $C^3$, uniformly in $A$.
The coefficients of the fourth-order principal part therefore stay in a fixed ellipticity neighbourhood of the reference coefficients.
Together with \cref{def:outer-cap-parabolic}(i) and the clock bound $1/2<1-\vartheta/\mu<3/2$, this makes the physical height equation uniformly parabolic there.

It remains to check the same property on each fixed neck annulus $0<d\le |z|\le\delta/16$.
Here $\kappa_i=1$, and the conical graph bound \eqref{eq:conical-parabolic-graph-norm} gives, at fixed $s$,
\begin{equation}\label{eq:physical-annulus-derivatives}
 \bigl|\partial_z^j[Av(z/A,s)]\bigr|
 =A^{1-j}\bigl|\partial_\zeta^jv(z/A,s)\bigr|
 \le C(A^2+z^2)^{(1-j)/2}\|v\|_{\coneGraph_\infty^k},
 \qquad 0\le j\le k.
\end{equation}
The reference meridian and its normal have bounded derivatives on this annulus by \eqref{eq:general-profile-tail}.
Thus \eqref{eq:global-height-representatives} bounds the derivatives of the perturbed immersion there uniformly in $A$.
The tube estimate in \cref{lem:quantitative-embeddedness} keeps its radius at least $c d$ and the derivative of its axial coordinate at least $1/2$.
Its derivatives through order three therefore range over a compact set of immersion jets with nondegenerate metric.
Strong ellipticity at each such jet and continuity of the principal symbol in \cref{def:geometric-flow-class} give uniform parabolicity on the annulus, with a constant that may depend on $d$.
Applying \cref{lem:compact-terminal-limits} to the evolving surface gives \eqref{eq:compact-pincher-radius}--\eqref{eq:compact-pincher-profile} and smooth convergence away from the origin.
At the centre, reflection gives $r_z=0$, so the azimuthal principal curvature has magnitude $1/A(t)$ and the smooth flow cannot continue through $T$.

For stability, fix the cap and $A_0$.
The solution map $b\mapsto\globalHeight_b$ is $C^1$ by \cref{thm:terminal-mixed-inverse}, so its derivative is bounded on a sufficiently small ball about zero.
The mean value estimate on this ball gives
\[
 \|\globalHeight_b-\globalHeight_{\widetilde b}\|_{\mixX}\le C\|b-\widetilde b\|_{\compTr}.
\]
By \eqref{eq:conical-hilbert-decay}, \eqref{eq:mixed-inner-space} and \eqref{eq:mixed-product-spaces}, the mixed norm on the left bounds the weighted supremum of the inner height difference in \eqref{eq:general-relative-stability}.

When the flow preserves volume and the cap satisfies the volume transversality condition, \cref{lem:fully-coupled-selection} imposes the volume constraint on these same traces.
It gives the asserted codimensions and preserves the volume along every selected flow.
Write $\lambda_\varepsilon(h)$ for the scalar cap correction selected by that lemma.
The map $h\mapsto h+\lambda_\varepsilon(h)p_\phi$ is $C^1$, so composing it with the solution map in \eqref{eq:general-relative-stability} gives \eqref{eq:main-relative-stability}.

For cap variations in the topology of \cref{thm:terminal-mixed-inverse}, the solution map, the trace graph and, when used, the volume selection depend jointly $C^1$ on the cap and the free data.
Endpoint evaluation $w\mapsto w(T_*)$ is bounded from \eqref{eq:mixed-outer-space} into $H^{k+2}(M_o)$.
On the outer cap $\eta_o^4=1$ and $\refX_A\to X_0$, so the limiting embedding is $X_0+w(T_*)\nu_0$.
This gives the asserted dependence in the fixed cap charts.
\end{proof}

\Needspace{10\baselineskip}
\section{Pinchoff for surface diffusion flow of dumbbells and tori}\label{sec:dumbbell-pinchoff}

To apply \cref{thm:compact-pincher-realisation} to surface diffusion, we must check the hypotheses on the flow, the neck and the cap.
The checks on the flow and the neck are common to both geometries.
We then verify the cap conditions on the two discs in the D-type case and on the annular return in the T-type case, using one parabolic estimate for both.

\begin{proposition}[Surface diffusion and the ground-state neck]\label{prop:sdf-general-hypotheses}
Surface diffusion flow \eqref{eq:sdf-intro-new} belongs to the class of \cref{def:geometric-flow-class} and preserves volume in the sense of \cref{def:volume-preserving-flow}.
Every normalised positive conical profile for this flow is an admissible neck in the sense of \cref{def:admissible-neck}.
In particular, this holds for the WMVD--BBW ground state of \cref{thm:certified-positive-conical-profile}.
\end{proposition}
\begin{proof}
\emph{The flow and enclosed volume.}
The speed $-\Delta H$ is a local geometric expression in derivatives of the immersion through order four, with the fourth derivatives entering linearly.
It commutes with Euclidean isometries and changes of parametrisation, and reversing the unit normal leaves its velocity vector unchanged.
Its dilation law \eqref{eq:general-flow-scaling} follows from $H_{af_0}=a^{-1}H_{f_0}$ and $\Delta_{af_0}=a^{-2}\Delta_{f_0}$, as in the radial calculation of \cref{lem:terminal-scaling}.
In the normal chart \eqref{eq:normal-graph} over $f_0$, the variation of the averaged mean curvature at $u=0$ in the direction $w$ has principal part $\tfrac12\Delta_{f_0}w$.
Thus the negative of the linearised normal graph operator has principal symbol $\tfrac12|\xi|_{f_0^*g_{\rm Eucl}}^4$.
The coefficients are smooth functions of the derivatives of the immersion and the inverse metric, with uniform bounds in the charts of \cref{def:geometric-flow-class}(i).
The local existence and uniqueness assertion is \cref{thm:local-semiflow}, and \cref{lem:symmetry} preserves the required symmetries.
On each closed immersed time slice with induced metric $g$, the divergence theorem gives
\[
 \int_\Sigma \flowSpeed[f(\cdot,t)]\,\dd\mu_g
 =-\int_\Sigma \operatorname{div}_g(\nabla_g H)\,\dd\mu_g=0.
\]
This proves \eqref{eq:volume-preserving-speed}, and the first variation in \cref{lem:area-volume} gives conservation of signed volume.

\emph{The conical linearisation.}
Let $U$ be any normalised positive conical profile for surface diffusion.
Its radial operator is $\radOp$ of \eqref{eq:terminal-operator}.
Write $a_j[U]$ for the coefficients of $D\radOp[U]$ in \eqref{eq:fredholm-coefficient-form}.
With $q=U'$ and $S=\sqrt{1+q^2}$, differentiating the radial flux as in \cref{lem:sdf-coefficient-bounds} gives
\begin{equation}\label{eq:sdf-neck-leading-terms}
 a_4[U]=\frac{1}{2S^4},\qquad
 a_3[U]=-\frac{q}{US^4}+\frac{5qq'}{S^6}.
\end{equation}
Since $U'$ is bounded, $a_4[U]$ is bounded above and away from zero.
Evenness of $U$ and reflection invariance give the parity of each coefficient in \cref{def:conical-linearisation}.

To check the differentiated estimates on the positive end, rescale the profile to $U_{1/b}(x)=b^{-1}U(bx)$ on $1/2\le x\le2$.
By \eqref{eq:general-profile-tail}, $U_{1/b}(x)=\alpha x+O(b^{-4})$ in every $C^r$ norm on this interval.
Differentiating the dilation identity of \cref{lem:terminal-scaling} in a profile direction and comparing coefficients gives
\[
 a_j[U_{1/b}](x)=b^{4-j}a_j[U](bx),\qquad 0\le j\le4.
\]
The coefficients depend smoothly on the profile and its derivatives where the radius is positive, so $a_j[U_{1/b}]=a_j[\alpha x]+O(b^{-4})$ in every $C^r$ norm there.
For the cone, \eqref{eq:sdf-neck-leading-terms} gives $a_4[\alpha x]=a_\infty=1/(2(1+\alpha^2)^2)$ and $a_3[\alpha x]=-2a_\infty/x$.
Returning to $\zeta=bx$ gives \eqref{eq:fredholm-detail-a4-decay} and \eqref{eq:fredholm-detail-a3-leading}, with a factor $b^{-1}$ for each $\zeta$ derivative.
For $j\le2$ the same rescaling gives $\partial_\zeta^r a_j[U](b)=O(b^{j-4-r})$, which is \eqref{eq:fredholm-detail-lower-decay}.
Reflection gives all three estimates on the negative end.

\emph{The nonlinear neck condition.}
Fix $k\ge10$ and consider the radial and normal remainders \eqref{eq:terminal-radial-remainder} and \eqref{eq:normal-neck-remainder}.
In each chart let $\sigma_v$ be the length of the meridian tangent in the coordinate $\zeta$.
Thus $\sigma_v=\sqrt{1+(U'+v')^2}$ in the radial chart and $\sigma_v=|\partial_\zeta Y_v|$ for the normal graph \eqref{eq:normal-neck-chart}.
In either remainder the term containing the fourth derivative of $v$ is
\[
 \left(\frac{1}{2\sigma_0^4}-\frac{1}{2\sigma_v^4}\right)\partial_\zeta^4v,
\]
by the radial formula \eqref{eq:terminal-operator} and the normal graph calculation \eqref{eq:app-sdf-normal-chart-principal}.
The coefficient is $O(r_n(v))$ in both charts.
Indeed, the change of meridian tangent is bounded by $|v'|$ in the radial chart and by $C(|v'|+\langle\zeta\rangle^{-1}|v|)$ in the normal chart, since $Y_v=f_U+v\nu_r\nu_U$ and \eqref{eq:general-profile-tail} gives $|(\nu_r\nu_U)'|\le C\langle\zeta\rangle^{-1}$.
Both bounds are controlled by \eqref{eq:neck-slab-conical-size}, and $\sigma_0\ge1$.
This accounts for the full weight on the highest derivative in the end norm \eqref{eq:neck-end-input-norm}.
The lower terms gain the factor $\langle\rho\rangle^{-1}$ in that norm by \eqref{eq:app-nonlinear-conical-decay}.

We now check the two variation bounds of \cref{def:neck-nonlinear-estimates}.
On the whole line, \eqref{eq:app-hilbert-second-variation} gives both bounds for the projected remainder in the norms \eqref{eq:neck-weighted-line-norms}.
On every box, \eqref{eq:app-local-mixed-variations} gives both bounds in the unweighted norms \eqref{eq:neck-ordinary-box-norms}.
For $|\rho|\ge4$, \eqref{eq:app-end-remainder-variations} gives the additional bounds in the weighted end norms \eqref{eq:neck-end-box-norms}.
On the conical slab \eqref{eq:neck-conical-slab-norms} and the interval $J$ of \eqref{eq:neck-interval-norms}, only \eqref{eq:neck-first-variation} is required.
In the difference bounds \eqref{eq:app-box-spacetime-difference} and \eqref{eq:app-local-sobolev-quadratic}, put $w=v+th$, divide by $|t|$ and let $t\to0$ to obtain
\[
 \|DT(v)h\|_Y\le C\{L(v)K(h)+L(h)K(v)\}.
\]
The product estimates of \cref{lem:app-conical-tame-calculus} give the $C^2$ extensions in both charts with constants uniform in the slab, the box centre and the initial conical radius.
For the interval estimate the additional restriction $L(v)<\varepsilon$ bounds the coefficients, and the embedding $H^1(J)\hookrightarrow L^\infty(J)$ has a common constant for $|J|\ge1$.
This proves the nonlinear neck condition for every $k\ge10$.

Finally, parts~\ref{it:ground-state-shrinker} and \ref{it:ground-state-ends} of \cref{thm:certified-positive-conical-profile} give the normalisation, profile equation and differentiated expansion in \cref{def:normalised-conical-profile} for the ground state.
The two neck conditions just proved therefore apply to it.
\end{proof}

Fix this profile $U$, with cone slope $\alpha$.
Fix $k\ge10$, $\delta>0$ and an $(\alpha,\delta)$-admissible terminal meridian as in \cref{def:admissible-outer-cap}.
In both geometries the reference surface $\matchX_A$ has the rounded neck $AU(z/A)$, the outer cap \eqref{eq:outer-terminal-normal-variation} and the interpolated collars \eqref{eq:matched-collar-radius}.
To apply \cref{thm:compact-pincher-realisation} we still need to construct an injective normal tube, find a variation of the cap that changes the enclosed volume and prove the parabolic estimate on the doubled cap (\cref{def:compactification-hypotheses,def:outer-cap-parabolic}).
We prove the estimate first, in a form that applies both to $\matchX_A$ and to the volume-corrected embedding $\refX_A$ of \eqref{eq:outer-volume-amplitude}.

Truncating at the neck circles $|z|=\delta/16$ leaves two discs for a D-type cap and one annulus for a T-type cap.
To obtain $M_o$ we double each disc separately in the D-type case and double the annulus in the T-type case.
Once we have checked regularity of the metric, in particular at the two D-type poles, we can use the same estimate on either double.

We compare the height equation for surface diffusion on $M_o$ with the biharmonic equation for a fixed metric.
Fix a smooth cutoff $0\le\psi\le1$ equal to one on $K_{\delta/8}^{(1)}$ and supported in the interior of $K_{\delta/16}^{(1)}$.
With a second cutoff interpolate the terminal metric $g_0$ to a fixed background metric on $M_o$, obtaining a smooth metric $g_*$ equal to $g_0$ near $\supp\psi$.
Choose both cutoffs and the background metric invariant under rotations and reflection, and keep them fixed as the cap varies.
Write $H^j=H^j(M_o)$ as in \cref{def:outer-cap-parabolic}.

\begin{lemma}[The biharmonic estimate on the doubled cap]\label{lem:fixed-cap-linear}
Let $I=[a,b)$ with $0<b-a\le1$.
For $f\in L^2(I;H^k)$ and $w_a\in H^{k+2}$, the problem $(\partial_t+\tfrac12\Delta_{g_*}^2)w=f$ on $M_o\times I$, $w(\cdot,a)=w_a$, has a unique solution satisfying
\begin{equation}\label{eq:fixed-cap-linear-bound}
 \|w\|_{H^1(I;H^k)}+\|w\|_{L^2(I;H^{k+4})}+\|w\|_{C(\overline I;H^{k+2})}
 \le C\bigl(\|f\|_{L^2(I;H^k)}+\|w_a\|_{H^{k+2}}\bigr).
\end{equation}
The constant is uniform in $I$ and under sufficiently small symmetric $C^{k+8}$ perturbations of the terminal cap supported away from the neck.
For every $w$ in this solution space we also have
\begin{equation}\label{eq:outer-short-interval-interpolation}
 \|w\|_{L^2(I;H^{k+3})}
 \le C|I|^{1/4}\|w\|_{C(\overline I;H^{k+2})}^{1/2}
                 \|w\|_{L^2(I;H^{k+4})}^{1/2}.
\end{equation}
\end{lemma}
\begin{proof}
\Cref{lem:app-fixed-cap-linear} proves both estimates for the extended metric $g_*$, uniformly in the interval and the stated cap neighbourhood.
Its proof applies maximal $L^2$ regularity to the positive self-adjoint biharmonic operator on the unperturbed double and absorbs the change in the metric by a Neumann series.
\end{proof}

We now pass from the fixed metric to the moving normal graphs over the two reference surfaces.
For both cap types, the following lemma gives the uniformly parabolic extension required by \cref{def:outer-cap-parabolic}(i), the short-time estimate for its linearisation in part~(ii), and the nonlinear bounds and $C^1$ dependence on the cap in part~(iii).

\begin{lemma}[The parabolic estimate on the doubled cap]\label{lem:cap-double-estimate}
Let $C_0$ be an admissible D-type or T-type cap.
Let $X_A=\matchX_A+d(A)B$, where $B:M\to\R^3$ is a smooth symmetric field independent of $A$, supported in the outer region away from the collars, and $|d(A)|+|Ad'(A)|\le C_\delta A^8$.
The moving height equation \eqref{eq:outer-moving-height} has a uniformly parabolic extension to $M_o$ on a common $C^1$ graph ball.
For its linearisation $P_o$, the problem $P_ow=f$, $w(a)=w_a$, has a unique solution satisfying \eqref{eq:outer-cap-forward-estimate} for $f\in L^2(I;H^k)$, $w_a\in H^{k+2}$ and $I=[a,b)\subseteq I_*$ of length at most $\ell$.
The constants and $\ell>0$ are uniform for small $A_0$ and small admissible $C^{k+8}$ cap perturbations, with a common $C^{k+4}$ bound on $B$ in the fixed charts.
This includes the whole reference interval $I_*$ for $A_0$ sufficiently small.
The solution has the symmetries of the data.
The extension also satisfies the local tame estimates \eqref{eq:cap-local-tame-bounds}, uniformly for caps in a sufficiently small $C^{k+9}$ neighbourhood and fields $B$ bounded in $C^{k+5}$, and depends $C^1$ on these data in those topologies.
\end{lemma}

\begin{proof}
For a D-type cap the azimuthal coordinate is undefined at each pole.
In the Cartesian chart of \cref{def:admissible-outer-cap} write $x=(x_1,x_2)$ and $\rho=|x|^2$.
The induced metric is $g_{ij}=\delta_{ij}+4q_{\rm p}'(\rho)^2x_i x_j$.
Thus $g_{ij}(0)=\delta_{ij}$ and $\partial_\ell g_{ij}(0)=0$, so $\Delta_g$ agrees with the Euclidean Laplacian at the pole.
For a smooth invariant scalar $F(\rho)$,
\[
 \partial_i\partial_j F(\rho)=2\delta_{ij}F'(\rho)+4x_i x_j F''(\rho),
 \qquad \Delta_g F(\rho)\big|_{x=0}=4F'(0).
\]
Writing $H_0(\rho)$ for the terminal mean curvature in this chart, the terminal cap has surface diffusion speed $-4H_0'(0)$ at the pole by \eqref{eq:sdf-intro-new}.
The pole is an interior point of the height equation.
Reflection gives the same conclusion at the other pole.
For small $A$, $X_A$ is $C^2$-close to the terminal embedding there, so its horizontal projection is a diffeomorphism over a smaller fixed polar disc.

For a T-type cap, arclength $\sigma$ and azimuth $\theta$ give the terminal metric $\dd\sigma^2+r_0(\sigma)^2\dd\theta^2$ throughout the annulus $K_{\delta/16}$, where $r_0(\sigma)$ is the terminal meridian radius.
Its inverse quadratic form $\xi_\sigma^2+r_0^{-2}\xi_\theta^2$ is uniformly comparable to $|\xi|^2$ since $0<\min r_0\le\max r_0<\infty$ there.
The same coordinates cover the truncated D-type discs away from the polar charts.
In these fixed charts $C^2$ closeness to the terminal cap gives smooth, uniformly positive definite moving metrics $g_A=X_A^*g_{\rm Eucl}$ for both types.

For both cap types, take the fixed cutoff $\psi$ and extended terminal metric $g_*$ chosen before \cref{lem:fixed-cap-linear}.
On $M_o$ extend $\psi G_A(w)$ by zero and set
\begin{equation}\label{eq:cap-double-extension}
 G_A^{\rm ext}(w)=\psi G_A(w)-(1-\psi)\tfrac12\Delta_{g_*}^2w .
\end{equation}
For the physical graph $X_A^w=X_A+w\nu_A$, the variation of the averaged mean curvature in a height direction $v$ has principal part $\tfrac12\langle\nu_A,\nu_w\rangle\Delta_{X_A^w}v$.
On dividing the speed by $\langle\nu_A,\nu_w\rangle$ in \eqref{eq:outer-moving-height}, this factor cancels from the fourth-order term.
The motion of the reference surface contributes derivatives of $v$ of order at most one.
Thus the principal symbol of the negative linearised spatial operator in \eqref{eq:cap-double-extension} is
\[
 \tfrac12\psi\,|\xi|_{(X_A^w)^*g_{\rm Eucl}}^4+\tfrac12(1-\psi)|\xi|_{g_*}^4 .
\]
For a common sufficiently small $C^1$ graph ball the two inverse metrics are uniformly comparable on $\supp\psi$ and $\langle\nu_A,\nu_w\rangle\ge1/2$.
This gives a uniform positive lower bound for the symbol on unit covectors.

We compare $P_o=\partial_{t_*}-D_wG_A^{\rm ext}(0)$ with $P_*=\partial_{t_*}+\tfrac12\Delta_{g_*}^2$.
On the fixed region $K_{\delta/16}$ the cap formula \eqref{eq:outer-terminal-normal-variation}, the inner expansion \eqref{eq:main-tail} at $|z|\ge\delta/16$ on the neck, and the interpolation \eqref{eq:matched-collar-radius} give
\[
 \|X_A-X_0\|_{C^{k+4}}\le C_\delta A^4,\qquad
 \|\beta_0\partial_A X_A\|_{C^{k+1}}\le C_\delta .
\]
Here $\beta_0\partial_A=-\partial_\tau$, and $A\partial_A$ preserves the $O_\delta(A^4)$ inner expansion at fixed $z$.
The extra field $d(A)B$ contributes $O_\delta(A^8)$ to the embedding and $O_\delta(A^4)$ to the velocity.
On $\supp\psi$ we have $g_A-g_*=O_\delta(A^4)$ in $C^k$, so the fourth-order coefficients of $P_o-P_*$ are $O_\delta(A_0^4)$ there.

To bound the lower order coefficients of $P_o$, we use that the linearisation is taken at $w=0$.
In local coordinates, only the normal components of the fourth derivatives of $X_A+w\nu_A$ enter the surface diffusion speed.
Differentiating $X_A+w\nu_A$ at zero height could therefore introduce a fifth derivative of $X_A$ only through $\langle\nu_A,\partial^\alpha\nu_A\rangle$ with $|\alpha|=4$.
Differentiating $|\nu_A|^2=1$ gives
\[
 2\langle\nu_A,\partial^\alpha\nu_A\rangle
 =-\sum_{0<\beta<\alpha}{\alpha\choose\beta}
       \langle\partial^\beta\nu_A,\partial^{\alpha-\beta}\nu_A\rangle ,
\]
where the sum runs over proper nonzero sub-multi-indices of $\alpha$.
Thus these coefficients use at most four derivatives of $X_A$ and one of its velocity, and are uniformly bounded in $C^k$ for terminal caps bounded in $C^{k+8}$.
At nonzero height the projection uses $\nu_w$, which accounts for the $C^{k+9}$ cap topology in \cref{thm:terminal-mixed-inverse}.

For the terms through order three on a reference time interval $I\subset I_*$ with $|I|\le1$, the interpolation estimate \eqref{eq:outer-short-interval-interpolation} gives a factor $|I|^{1/4}$.
With $\capPath(I)$ as in \eqref{eq:mixed-outer-space} on $I$, the coefficient bounds now give
\begin{equation}\label{eq:cap-operator-perturbation}
 \|(P_o-P_*)w\|_{L^2(I;H^k)}
 \le C_\delta\bigl(A_0^4+|I|^{1/4}\bigr)\|w\|_{\capPath(I)}.
\end{equation}
By \eqref{eq:fixed-cap-linear-bound}, the solution operator for $P_*$ with prescribed initial trace has norm at most $C_*$, independent of the interval and the nearby cap.
Choose $A_{\max}$ and $\ell\le1$ so that $C_*C_\delta(A_{\max}^4+\ell^{1/4})\le1/2$.
For $A_0\le A_{\max}$ and $|I|\le\ell$, a Neumann series applied to \eqref{eq:cap-operator-perturbation} gives the unique solution and \eqref{eq:outer-cap-forward-estimate} with constant $2C_*$.
Since $|I_*|=A_0^4/(4\mu)$ by \eqref{eq:outer-base-clock}, a further reduction of $A_0$ gives the estimate on the whole interval $I_*$.
Rotations and reflection commute with the extended equation, so uniqueness preserves the symmetries of the data.

It remains to verify the nonlinear estimates in \cref{def:outer-cap-parabolic}(iii).
In each fixed chart, the extended height operator has the form
\[
 G_A^{\rm ext}(w)=\sum_{|\beta|=4}a_\beta(x,w,Dw)\partial^\beta w
                 +b(x,w,Dw,D^2w,D^3w),
\]
with smooth coefficients on the chosen graph ball.
The coefficients also depend on $A$, the reference embedding and its velocity.
For $1\le j\le3$, differentiating this expression $j$ times in height directions gives products containing at most one fourth derivative of $w$ or of a direction $h_i$.
The Sobolev product and composition estimates, with $k\ge10$, therefore give
\[
 \|D_w^jG_A^{\rm ext}(w)[h_1,\ldots,h_j]\|_{H^k}
 \le C\left((1+\|w\|_{H^{k+4}})\prod_i\|h_i\|_{H^{k+2}}
 +\sum_i\|h_i\|_{H^{k+4}}\prod_{\ell\ne i}\|h_\ell\|_{H^{k+2}}\right),
\]
which is \eqref{eq:cap-local-tame-bounds}.
Here the graph ball is bounded in $H^{k+2}$, so its lower derivatives are uniformly bounded; interpolation places any derivatives of order three in the displayed products under the same bound.
The displaced reference uses four terminal-cap derivatives, its normal one more, and the speed four more.
Consequently a bounded $C^{k+9}$ neighbourhood controls these coefficients and their cap derivatives, uniformly as $A\to0$.
The same calculation with a cap variation gives the asserted $C^1$ dependence.
\end{proof}

For D-type caps, the following lemma verifies the geometric cap condition of \cref{def:compactification-hypotheses}(i) by establishing the injective normal tube and curvature bounds on the matched sphere.
It also shows that the perturbed surface is a graph over a horizontal disc at each pole, with radius independent of $A$, so we can use Cartesian coordinates on these discs for the surface diffusion equation \eqref{eq:sdf-intro-new} up to pinchoff.

\begin{lemma}[Embeddedness and the two poles]\label{lem:db-polar-tube}
Let $C_0$ be an admissible D-type cap.
The matched sphere has an injective normal tube of radius $c\ell_A$ and principal curvatures bounded by $C/\ell_A$, with $\ell_A$ as in \eqref{eq:local-geometric-scale}.
There are $\varepsilon,d>0$ such that every invariant normal graph of relative $C^1$ norm less than $\varepsilon$ is an embedded sphere with exactly two points on the axis.
Near each pole it is a graph over a horizontal disc of radius $d/2$, smooth whenever the normal height is smooth.
The constants are common to all small $A$ and all caps in a sufficiently small admissible neighbourhood.
\end{lemma}

\begin{proof}
The neck estimates in the proof of \cref{lem:appC-scaled-matched-geometry} use only the rounded profile and the collar estimate \eqref{eq:fixed-match-geometric}, so they apply to both cap types.
We use them to construct the tube on the neck, then extend it over the two discs and show that the perturbed surface is a graph over a fixed horizontal disc at each pole.

Choose $b>0$ so that $B_{4b}(0)$ meets only the neck for all small $A$.
On this neck, \eqref{eq:appC-radius-lower-bound}--\eqref{eq:appC-scaled-derivative-bound} give a local meridian tube on scale $\ell_A$ and a uniform slope bound.
Suppose two normal fibres with distinct base points $p,q\in B_{3b}$ meet.
Then $|p-q|\le c(\ell_A(p)+\ell_A(q))$.
Since $\ell_A$ is uniformly Lipschitz, the two scales are then comparable for small $c$.
The meridian is a graph of bounded slope, so the arclength between $p$ and $q$ is at most $C|z(p)-z(q)|\le C|p-q|$.
Reducing $c$ puts both points in one local tube, where the fibres are disjoint.
By \eqref{eq:appC-radius-lower-bound}, these short normal offsets retain positive radius and hence preserve azimuth, so the revolved tube is injective as well.

The portion outside $B_b$ is contained in two compact discs obtained by cutting the neck at fixed axial coordinates $\pm z_b$ inside $B_b$.
Indeed, $\refRad_A(z)\le C(A+|z|)$ there, so the whole segment $|z|\le 2z_b$ lies in $B_b$ for sufficiently small $z_b$ and $A$.
The two discs extend a little farther along the neck and converge in $C^2$ to disjoint compact embedded discs.
Each extended terminal disc has a fixed normal tube.
On a smaller closed tube the normal-coordinate maps of the matched discs are $C^1$-close to the terminal maps, so remain embeddings.
Reducing the tube radius below one third of the separation between the two terminal discs keeps their tubes disjoint for all small $A$.
It remains to consider a pair with one base point in $B_b$ and the other outside $B_{3b}$.
These points are at least $2b$ apart.
Since $\ell_A$ is uniformly bounded above, we may choose $c$ to keep every offset shorter than $b/2$.
This proves global injectivity.
On the neck the same profile estimates bound the meridian curvature by $C/\ell_A$, and the azimuthal curvature is at most the reciprocal of the radius, hence also at most $C/\ell_A$.
On each outer disc the curvature bound follows by smooth convergence.
Both bounds and the outer tube persist under small $C^2$ perturbations supported away from the collars.

It remains to describe the graph at the poles.
In the fixed horizontal coordinates $x$ at the positive pole of the matched sphere, let $q_{{\rm p},A}(\rho)$, $\rho=|x|^2$, be the height relative to its value at $x=0$.
The horizontal projection of the normal graph of height $\globalHeight$ is
\[
 x\longmapsto
 \left(1-\frac{2\epsilon q_{{\rm p},A}'(\rho)\globalHeight(x)}
                {\sqrt{1+4\rho q_{{\rm p},A}'(\rho)^2}}\right)x ,
 \qquad \epsilon\in\{-1,1\},
\]
where $\epsilon$ gives the choice of normal.
Write the projection as $x+E(x)$ on a closed disc of fixed radius $d$.
Then $E(0)=0$ and $E=\globalHeight\nu_{A,\rm hor}$, where $\nu_{A,\rm hor}$ is the horizontal part of the matched normal.
On this disc $\ell_A\asymp\delta$ and $\nu_{A,\rm hor}$ is uniformly bounded in $C^1$.
Thus $\|DE\|_\infty\le C_\delta\|\globalHeight\|_{\relNorm(A)}<1/2$ after choosing $\varepsilon$.
It follows that $|(x+E(x))-(y+E(y))|\ge |x-y|/2$ on this disc.
For $|y|<d/2$ the map $x\mapsto y-E(x)$ is a contraction of the closed disc into itself.
Its fixed point projects to $y$.
Since $I+DE$ is invertible, the inverse function theorem gives a $C^1$ graph in $y$, smooth for a smooth height.
Reflection gives the graph at the other pole.
Globally, if the equivariant embedding maps a point to the axis, injectivity forces that point to be fixed by every rotation.
Each of the two original poles has this property and no other point does.
\end{proof}

We can now verify the compactification hypotheses on the whole matched sphere.

\begin{proposition}[D-type caps]\label{prop:db-selection}
A D-type cap satisfies the geometric cap condition, the volume transversality condition and the parabolic cap condition of \cref{def:compactification-hypotheses,def:outer-cap-parabolic}.
The matched reference surfaces are spheres.
Every smooth symmetric normal graph sufficiently small in the relative $C^1$ norm is a smooth graph over a fixed horizontal disc at each pole.
\end{proposition}

\begin{proof}
The neck cylinder joins the two discs along their boundary circles to give a sphere.
\Cref{lem:db-polar-tube} gives \cref{def:compactification-hypotheses}(i) and the two polar graphs.
For the volume variation in part~(ii), choose a nonzero nonnegative smooth rotational bump supported away from the pole and collar of one terminal disc.
Adding its reflection on the other disc gives a smooth symmetric $\phi:M\to\R$ with $\int_{C_0}\phi\,\dd\mu_{g_0}>0$.
Keep $\phi$ fixed in the terminal cap coordinates as the cap varies.
Its integral then stays bounded below by continuity of the area density.

\Cref{prop:fixed-interface-matching} now supplies the second reference $\refX_A=\matchX_A+d(A)\phi\nu_0$.
The tube persists under this $O_\delta(A^8)$ outer variation in $C^2$, and $\ell_A$ is bounded below on $K_{\delta/16}$.
We may therefore choose the same $C^1$ radius for both reference surfaces so that each normal graph lies in the corresponding tube.
Applying \cref{lem:cap-double-estimate} first with $d=0$ and then with $B=\phi\nu_0$ and $d$ from \eqref{eq:outer-volume-amplitude} gives all three parts of \cref{def:outer-cap-parabolic}, with the stated $C^{k+8}$ uniformity for the linear estimates and $C^{k+9}$ uniformity for the nonlinear estimates.
All cap cutoffs vary only on fixed annuli of positive orbit radius.
There the invariant surface Sobolev norms are equivalent to the meridian norms, so no polar coordinates are needed on the overlap with the neck.
\end{proof}

\Needspace{7\baselineskip}
For an explicit D-type cap with round ends, take $R>5(1+\alpha^2)\delta$ and a smooth nondecreasing cutoff $\chi_{\rm end}$ equal to zero for $z\le4\delta$ and one for $z\ge5\delta$.
Set
\[
 r_0(z)^2=(1-\chi_{\rm end}(z))\alpha^2 z^2+\chi_{\rm end}(z)(2Rz-z^2),
 \qquad 0\le z\le2R.
\]
On the transition both squared radii are increasing and $2Rz-z^2>\alpha^2z^2$.
Thus the radius rises to $R$ at $z=R$ and falls to zero at $z=2R$.
Near this pole $z=R+\sqrt{R^2-r^2}$ is smooth in $r^2$, as required by \cref{def:admissible-outer-cap}.
Reflection gives the other lobe.

For a T-type cap we must also separate the return from the rounded neck, since the two arcs may have the same axial coordinate.

\begin{proposition}[T-type caps]\label{prop:t-cap-verification}
A T-type cap satisfies the geometric cap condition, the volume transversality condition and the parabolic cap condition of \cref{def:compactification-hypotheses,def:outer-cap-parabolic}.
The matched reference surfaces and their symmetric normal graphs sufficiently small in the relative $C^1$ norm are embedded tori, with the outer return uniformly separated from the rotation axis.
\end{proposition}

\begin{proof}
We first check the tube and curvature bounds of \cref{def:compactification-hypotheses}(i).
By \cref{lem:appC-scaled-matched-geometry}, the meridian has a local normal tube and curvature bounded by $C/\ell_A$.
If normal segments of lengths less than $c\ell_A(p)$ and $c\ell_A(q)$ meet, their base points $p,q$ on the meridian satisfy
\[
 |p-q|<c\bigl(\ell_A(p)+\ell_A(q)\bigr)\le 3c\min\{\ell_A(p),\ell_A(q)\}.
\]
Indeed, $\ell_A$ in \eqref{eq:local-geometric-scale} has a uniform Lipschitz constant, so the first inequality makes the larger scale at most twice the smaller for sufficiently small $c$.
Taking $3c$ below the separation constant in \eqref{eq:appC-global-separation} puts both points on a single local arc, where the normal tube is injective.
Write $\refRad_A$ for the radial coordinate along the whole reference meridian, agreeing with the neck radius $\refRad_A(z)$.
By \eqref{eq:appC-radius-lower-bound}, $\refRad_A\ge c\ell_A$ on the neck.
On the compact return $\refRad_A$ has a fixed positive lower bound.
Thus the azimuthal curvature is at most $1/\refRad_A\le C/\ell_A$, and short normal offsets retain positive radius and preserve azimuth.
Rotating the meridian tube gives the required surface tube.
A symmetric normal graph sufficiently small in the relative norm \eqref{eq:relative-height-norm} lies in this tube and has positive radius.
Its meridian is a regular simple loop, since normal projection inverts the graph and the tangential component of its arclength derivative is $1-\kappa_A h>0$, where $\kappa_A$ is the signed meridian curvature and $h$ the normal height.
Revolution of this loop gives an embedded torus.
The outer tube is uniform under the prescribed small $C^2$ perturbations by \eqref{eq:appC-uniform-outer-reach}.

To vary the volume as required by \cref{def:compactification-hypotheses}(ii), choose a nonzero nonnegative smooth rotational bump in the interior of the terminal return, away from the collars, and add its reflection.
The resulting $\phi$ has positive integral, uniformly for nearby caps as in the D-type case.
\Cref{prop:fixed-interface-matching} now gives $d(A)$ and $\refX_A$ for this return.

At the waist of the torus the outward normal points towards the rotation axis.
Reversing the normal changes the sign of the mean curvature and preserves the velocity vector $-(\Delta H)\nu$.
The signed multiplier \eqref{eq:radial-normal-linearisation} and the speed identity \eqref{eq:normal-radial-speed-identity} therefore give the same radial equation on both types of neck.

On $K_{\delta/16}$ the common lower bound for $\ell_A$ and the outer tube give a fixed $C^1$ graph ball for both references, since the correction in \eqref{eq:outer-volume-amplitude} is $O_\delta(A^8)$ in $C^2$.
Applying \cref{lem:cap-double-estimate} with $d=0$ and then with the correction $d(A)\phi\nu_0$ gives \cref{def:outer-cap-parabolic} for both references on the doubled annulus.
The graph radius and linear estimate are uniform for small $C^{k+8}$ cap perturbations, and the nonlinear estimates and $C^1$ dependence hold in the $C^{k+9}$ cap topology.
\end{proof}

For example, continue the two rays to $(\pm5\delta,5\alpha\delta)$, follow the lines $z=\pm5\delta$ to radius $8\alpha\delta$, and join the upper endpoints by a segment of constant radius.
Rounding the four outer corners symmetrically gives a smooth T-type return and leaves the neck intact.

\phantomsection\label{proof:joint-pinchoff}
\begin{proof}[Proof of \cref{thm:main}]
By \cref{prop:sdf-general-hypotheses,prop:db-selection,prop:t-cap-verification} we may apply the volume-preserving case of \cref{thm:compact-pincher-realisation} to both cap types with the same profile $U$ and the same $\mu$ from \cref{thm:certified-positive-conical-profile}.
For sufficiently small $A_0$, taking $h=0$ in \cref{lem:fully-coupled-selection} gives a smooth initial surface in each topology, since the selected neck variations and the outer volume variation are smooth.
For every smooth selected trace, \eqref{eq:compact-pincher-radius} and \eqref{eq:compact-pincher-profile} give the finite-time law \eqref{eq:main-radius} and the profile limit \eqref{eq:main-profile}.
The compactification theorem also gives embeddedness, conservation of $V_0$, smooth convergence away from the origin and stability at equal waist radii in the fixed-volume family of codimension $\dim \Ec$.

We now show that the central circle is a strict waist.
At the central circle reflection gives $r_z(0,t)=0$.
To compute $r_{zz}(0,t)$, use the rescaled coordinates \eqref{eq:rescaled-radial-normal-chart}, with radial perturbation $p$ and normal height $u=\refNu_r v$.
At the waist $u(0)=u'(0)=0$ by the waist condition and reflection, so $x(0)=0$ and $x'(0)=1$.
Since $U'(0)=p_\zeta(0,s)=0$, differentiating the radial identity twice at fixed $s$ gives
\[
 p_{\zeta\zeta}(0,s)=(p\circ x)''(0)=(u\refNu_r)''(0)
 =\refNu_r(0)^2 v_{\zeta\zeta}(0,s)=v_{\zeta\zeta}(0,s).
\]
Here $\refNu_r(0)=\pm1$ for the two cap types.
On $[-1,1]$ the Sobolev weight is bounded below.
Local Sobolev embedding and the time supremum in \eqref{eq:conical-hilbert-decay} give, for every $s\ge0$,
\[
 |v_{\zeta\zeta}(0,s)|\le C\|v(\cdot,s)\|_{H^k_2}
 \le C e^{-\eta s}Z_\eta(v)\le C e^{-\eta s}\|\globalHeight\|_{\mixX},
\]
where the last bound follows from \eqref{eq:mixed-inner-space} and \eqref{eq:mixed-product-spaces}.
By \eqref{eq:coupled-height-speed-bound} we may take $A_0$ and the trace ball small enough that this is at most $U''(0)/2$, since $U''(0)>0$ by \cref{thm:certified-positive-conical-profile}\ref{it:ground-state-bounds}.
Hence
\[
 A r_{zz}(0,t)=U''(0)+p_{\zeta\zeta}(0,s)\ge\tfrac12 U''(0).
\]
Thus the central circle is a strict local waist.
Curvature blow-up there and maximality of $T$ follow from the proof of \cref{thm:compact-pincher-realisation}.

It remains to identify the physical terminal cone and describe the two cap geometries at the singular time.
On $0\le\xi\le\delta$ write $Z_A(\xi)$ and $R_A(\xi)$ for the axial coordinate and radius of the perturbed meridian, with $\xi$ the reference axial coordinate.
Reflection gives $Z_A(0)=0$, and \cref{lem:quantitative-embeddedness}, \eqref{eq:appC-radius-lower-bound} and \eqref{eq:relative-height-norm} give
\begin{equation}\label{eq:physical-neck-comparability}
 Z_A'(\xi)=1+O(\varepsilon_*),\qquad
 c(A+\xi)\le R_A(\xi)\le C(A+\xi).
\end{equation}
After reducing the data ball, $1/2\le Z_A'\le3/2$ and the physical radius graph satisfies $r(z,t)\asymp A(t)+z$ on a fixed smaller positive neck interval.
These estimates hold for both cap types.

To identify the terminal slope, fix $R\ge2$ and $0<z<\min\{\delta/4,A_0\}$.
Since $A(t)$ decreases to zero, there is a unique time $t_z$ with $A(t_z)=z/R$.
At $t_z$ we can use the rescaled limit \eqref{eq:main-profile} at the fixed point $R$.
To bound how much the radius at $z$ can change between $t_z$ and pinchoff, we first estimate its spatial derivatives on a neck annulus.
The $L^\infty(0,1;H^2(B_\rho))$ terms in \eqref{eq:conical-parabolic-graph-norm}, one-dimensional Sobolev embedding and the dilation \eqref{eq:dilation-straightening} give
\begin{equation}\label{eq:neck-pointwise-conical-derivatives}
 \sup_{s\ge0,\,\zeta\in\R}
 \langle\zeta\rangle^{j-1}|\partial_\zeta^jv(\zeta,s)|
 \le C\|\globalHeight\|_{\mixX},\qquad 0\le j\le4.
\end{equation}
The constants in undoing the dilation are uniform since its time variable lies in $[0,1]$.
Classical regularity from \cref{prop:appC-classical-terminal-regularity}(i) makes these bounds valid at every time.
For $t_z\le t<T$, put $b=A(t)/z\le1/R$ and rescale the reference annulus $\xi=zy$, $1/2\le y\le5/2$, by $1/z$.
On this annulus \eqref{eq:global-height-representatives} expresses the rescaled normal displacement using $b\,v(y/b,s)$, and
\[
 \left|\partial_y^j\bigl[b\,v(y/b,s)\bigr]\right|
 =b^{1-j}|\partial_\zeta^jv(y/b,s)|
 \le C_jb^{1-j}\langle y/b\rangle^{1-j}\le C_j,
 \qquad 0\le j\le4.
\]
The differentiated expansion \eqref{eq:main-tail} and smoothness of $U$ give the same uniform bounds for the rescaled reference meridian $y\mapsto(y,bU(y/b))$ and its unit normal.
Since $Z_A'\ge1/2$, differentiation of the inverse axial coordinate in \eqref{eq:radial-normal-exact-chart} preserves these bounds when we express the perturbed meridian as a radius graph.
The image of the reference annulus contains $[3z/4,5z/4]$ by the bounds on $Z_A'$.
Consequently
\begin{equation}\label{eq:physical-neck-derivative-bounds}
 r(z,t)\asymp z,\qquad
 |\partial_z^jr(z,t)|\le C_jz^{1-j},\qquad
 1\le j\le4,\quad t_z\le t<T,
\end{equation}
with constants independent of $z$ and $R$.

The radial speed equation \eqref{eq:profile-equation} and the homogeneity in \cref{lem:terminal-scaling} now give $|r_t(z,t)|\le Cz^{-3}$ on this interval.
The clock formula \eqref{eq:compact-terminal-time}, with $|\vartheta|<\mu/2$, gives $T-t_z\le A(t_z)^4/(2\mu)$.
Integrating at fixed physical $z$ therefore yields
\[
 \left|\frac{r(z,T)}z-\frac{r(z,t_z)}z\right|
 \le \frac{C(T-t_z)}{z^4}\le\frac{C}{R^4}.
\]
For fixed $R$, $t_z\uparrow T$ as $z\downarrow0$ and $r(z,t_z)/z=r(A(t_z)R,t_z)/(RA(t_z))\to U(R)/R$ by \eqref{eq:main-profile}.
Thus
\[
 \limsup_{z\downarrow0}\left|\frac{r(z,T)}z-\alpha\right|
 \le\frac{C}{R^4}+\left|\frac{U(R)}R-\alpha\right|\longrightarrow0
 \qquad(R\to\infty).
\]
Reflection gives \eqref{eq:main-terminal-slope} on the negative branch.

For a D-type cap the fixed horizontal discs of \cref{lem:db-polar-tube} are interior charts away from the origin, so the smooth convergence in \cref{lem:compact-terminal-limits} includes both poles.
The positive terminal lobe truncated at $z=\delta$ on the neck is compactly contained in $\{z>0\}$ by \cref{def:admissible-outer-cap}.
A sufficiently small outer graph stays in this half-space, also at time $T$.
On the neck, \eqref{eq:physical-neck-comparability} gives $Z_A(\xi)\ge\xi/2$, while points with $A,\xi\to0$ tend to the origin.
Reflection gives the opposite inequalities on the other branch.
The two punctured limiting lobes consequently lie in opposite open half-spaces, and their closures meet exactly at the origin.

For a T-type cap the return stays a fixed positive distance from the axis by \cref{prop:t-cap-verification}, including at time $T$ by smooth convergence.
In the reflection plane the disc $0\le r<A(t)$ runs from the axis to the inner boundary of the solid torus.
It lies in the central hole and shrinks to the origin, while the outer annulus still joins the two sides of the neck.
\end{proof}

The lifespan theorem of \cite{WheelerNearSpheres2012} forces a definite concentration of the $L^2$ norm of $\A$ at a finite singular time.
For the flows constructed here the conical end gives the following stronger conclusion.
\begin{corollary}[Curvature concentration at the pinch]\label{cor:pinch-curvature-concentration}
Let $\Sigma_t$ be a flow of \cref{thm:main}, with pinching time $T$ and pinching point at the origin.
Write $\tfA=\A-Hg$ for its trace-free second fundamental form, where $H$ is the averaged mean curvature.
For every $\rho>0$,
\[
 \lim_{t\uparrow T}
 \int_{\Sigma_t\cap B_\rho(0)}|\tfA|^2\dd\mu_t=+\infty,
\]
where $B_\rho(0)\subset\R^3$ is the Euclidean ball and $\dd\mu_t$ is the induced area measure.
\end{corollary}
\begin{proof}
Let $\Sigma_U$ be the surface of revolution generated by $U$, and write $S=S[U]$ as in \eqref{eq:profile-geometry}.
With the normal \eqref{eq:axisymmetric-normal}, its azimuthal and meridional principal curvatures are $\kappa_\theta=1/(US)$ and $\kappa_\zeta=-U''/S^3$.
Since $|\tfA|^2=(\kappa_\theta-\kappa_\zeta)^2/2$ and the area element is $\dd\mu_U=US\dd\theta\dd\zeta$, its energy on $|\zeta|\le R$, for $R>0$, is
\[
 E_U(R):=\int_{\Sigma_U\cap\{|\zeta|\le R\}}|\tfA|^2\dd\mu_U
 =\pi\int_{-R}^R US\left(\frac1{US}+\frac{U''}{S^3}\right)^2\dd\zeta.
\]
The differentiated expansion \eqref{eq:main-tail} gives $U=\alpha\zeta+O(\zeta^{-3})$, $U'=\alpha+O(\zeta^{-4})$ and $U''=O(\zeta^{-5})$ on the positive end.
Hence
\[
 \pi US\left(\frac1{US}+\frac{U''}{S^3}\right)^2
 =\frac{\pi}{\alpha\sqrt{1+\alpha^2}}\zeta^{-1}+O(\zeta^{-5}),
\]
and reflection gives
\begin{equation}\label{eq:profile-tracefree-energy}
 E_U(R)=\frac{2\pi}{\alpha\sqrt{1+\alpha^2}}\log R+O(1)
 \qquad(R\to\infty).
\end{equation}

Fix $R$ and $\rho$.
By \eqref{eq:main-profile}, the physical neck portion $|z|\le A(t)R$ lies in $B_\rho(0)$ for all sufficiently late $t$, and its dilation by $1/A(t)$ converges smoothly to the corresponding portion of $\Sigma_U$.
The integral of $|\tfA|^2$ is invariant under dilation in dimension two, so the energy of this neck portion tends to $E_U(R)$.
Thus
\[
 \liminf_{t\uparrow T}\int_{\Sigma_t\cap B_\rho(0)}|\tfA|^2\dd\mu_t
 \ge E_U(R).
\]
Letting $R\to\infty$ and using \eqref{eq:profile-tracefree-energy} proves the claim.
\end{proof}

\section*{Acknowledgements}
\addcontentsline{toc}{section}{Acknowledgements}

The author gratefully acknowledges partial support from the Australian Research Council through Future Fellowship FT250100880 and Discovery Project DP250101080.
This research was first announced at the MATRIX Australia--Taiwan Joint Workshop on Geometric Analysis, held 17--21 August 2026.

\section*{AI and tool use statement}
\addcontentsline{toc}{section}{AI and tool use statement}

OpenAI language models, including ChatGPT and Codex (GPT-5.6 Sol), were used for bibliographic discovery, editorial review, and assistance in expanding the exposition of arguments supplied by the author.
Their suggestions affected the organisation and expository framing of the abstract and Introduction and led to additional references in the historical discussion.
The author independently checked the references.
The models were also used to identify typographical and presentational issues, to assist in tidying and documenting code written by the author, and to help organise the validated numerical material associated with Lemmata~\ref{lem:validated-compactified-tail-bounds} and~\ref{lem:validated-centre-miranda-cube} and the computational documentation in \cref{app:profile-certificate}.
The words in this article are the author's own.
The author takes full responsibility for independently checking the manuscript, citations, code, and validated output, and for the correctness of the submission.


\appendix
\crefalias{section}{appendix}
\crefname{appendix}{Appendix}{Appendices}
\Crefname{appendix}{Appendix}{Appendices}

\section{The terminal Fredholm theorem}
\label{app:terminal-fredholm}

We prove the operator assertions used in \cref{prop:finite-core-splitting}.
Let $U$ be a fixed normalised positive conical profile for $\flowSpeed$, satisfying \cref{def:conical-linearisation}.
Constants and the locations from which the end estimates hold may depend on $U$.
We fix $m=2$ and an integer $k\geq2$, and return to $k=0$ when we construct the adjoint modes.
All spaces and operators are complexified in this appendix.
The real even subspaces are recovered by complex conjugation at the end.

We use the weighted Sobolev norm \eqref{eq:pincher-polynomial-space} with $m=2$ and density $\varpi$.
Before imposing the normalisation at the centre, let
\begin{equation}\label{eq:fredholm-even-spaces}
\begin{aligned}
 \evenX^k&=H^k_2(\R)_{\rm even},\\
 \evenDom^k
 &=\{v\in H^{k+4}_2(\R):
       \zeta v'\in H^k_2(\R),\ v\text{ is even}\},
\end{aligned}
\end{equation}
with norm
\begin{equation}\label{eq:fredholm-detail-domain-norm}
 \|v\|_{\evenDom^k}
 =\|v\|_{H^{k+4}_2}+\|\zeta v'\|_{H^k_2}.
\end{equation}
On $\evenX^k$ we study the unnormalised linearisation $\fullLin$ from \cref{lem:pincher-projection}, initially on smooth even functions for which $\fullLin v\in\evenX^k$.
We then pass to the quotient by $\Span\{DU\}$.

\subsection*{Coefficients and solutions at the conical ends}

We first differentiate the radial surface diffusion operator for use in \cref{sec:dumbbell-pinchoff}.

\begin{lemma}[Coefficients of the radial surface diffusion operator]\label{lem:sdf-coefficient-bounds}
Let $U:\R\to(0,\infty)$ be smooth, $q=U'$ and $S=\sqrt{1+q^2}$.
The linearisation $\radLin=D\radOp[U]$ of \eqref{eq:terminal-operator} has the form \eqref{eq:fredholm-coefficient-form}, with leading coefficients given by \eqref{eq:sdf-neck-leading-terms}.
\end{lemma}
\begin{proof}
We use $K$ and $M$ of \eqref{eq:profile-geometry} evaluated at $U$.
For $v\in C_c^\infty(\R;\R)$, let $k_U[v]$ and $m_U[v]$ denote the first variations of $K$ and $M$ along $U+\varepsilon v$.
Differentiating the radial operator \eqref{eq:terminal-operator} gives
\begin{equation}\label{eq:F-linearisation}
 \radLin v=-\frac{v}{2U^2}(MK')'
 +\frac1{2U}\bigl(m_U[v]K'+M(k_U[v])'\bigr)'.
\end{equation}
This differential expression extends complex-linearly to smooth $v$.
By \eqref{eq:profile-geometry} the two variations are
\[
 k_U[v]=k_0v+k_1v'+k_2v'',\qquad
 m_U[v]=m_0v+m_1v',
\]
where
\begin{align*}
 k_0&=-\frac1{U^2S},
 &k_1&=-\frac q{US^3}+\frac{3qq'}{S^5},
 &k_2&=-\frac1{S^3},\\
 m_0&=\frac1S,
 &m_1&=-\frac{Uq}{S^3}.
\end{align*}
After one differentiation, the quantity inside the outer derivative in \eqref{eq:F-linearisation} is
\[
 m_U[v]K'+M(k_U[v])'
 =h_0v+h_1v'+h_2v''+h_3v''',
\]
where
\begin{align}
 h_0&=m_0K'+Mk_0',
 &h_1&=m_1K'+M(k_0+k_1'),
 \label{eq:fredholm-detail-h01}\\
 h_2&=M(k_1+k_2'),
 &h_3&=Mk_2=-\frac U{S^4}.
 \label{eq:fredholm-detail-h23}
\end{align}
The coefficients in \eqref{eq:fredholm-coefficient-form} are
\begin{align}
 a_4&=-\frac{h_3}{2U},
 \label{eq:fredholm-detail-a4}\\
 a_3&=\frac{h_2+h_3'}{2U},
 &a_2&=\frac{h_1+h_2'}{2U},
 \label{eq:fredholm-detail-a32}\\
 a_1&=\frac{h_0+h_1'}{2U},
 &a_0&=\frac{h_0'}{2U}-\frac{(MK')'}{2U^2}.
 \label{eq:fredholm-detail-a10}
\end{align}
Since $S'=qq'/S$, the two contributions to $2Ua_3$ in \eqref{eq:fredholm-detail-a32} are
\[
 h_2=-\frac{q}{S^4}+\frac{6Uqq'}{S^6},\qquad
 h_3'=-\frac{q}{S^4}+\frac{4Uqq'}{S^6}.
\]
Adding these expressions gives $a_3$ in \eqref{eq:sdf-neck-leading-terms}, and $a_4$ follows from $h_3=-U/S^4$.
\end{proof}

For the rest of this appendix we return to the general flow and profile satisfying \cref{def:conical-linearisation}.
In particular,
\begin{align}
 a_4&=a_\infty+O(|\zeta|^{-4}),
 &a_3&=O(|\zeta|^{-1}),
 &a_2&=O(|\zeta|^{-2}),
 \nonumber\\
 a_1&=O(|\zeta|^{-3}),
 &a_0&=O(|\zeta|^{-4}),
 \label{eq:fredholm-coefficient-decay}
\end{align}
with the corresponding differentiated estimates.

The operator $v\mapsto(a_4-a_\infty)v^{(4)}$ is bounded from $\evenDom^k$ to $\evenX^k$, but is not compact.
We therefore retain the variable coefficient $a_4(\zeta)$ in the end equation.

We next construct the admissible subspace of the homogeneous equation on an end.
The transport equation $\mu(v-\zeta v')=\lambda v$ gives the trial power $|\zeta|^{1-\lambda/\mu}$.
In \cref{lem:fredholm-end-dichotomy} we correct this trial to a solution of $(\fullLin-\lambda)v=0$ with expansion
\begin{equation}\label{eq:fredholm-algebraic-solution}
 v_{\rm a}(\zeta)
 =|\zeta|^{1-\lambda/\mu}
   \{1+O(|\zeta|^{-4})\}.
\end{equation}
For the exponential modes on the positive end, the fourth derivative and the drift give the leading terms $-a_\infty(\phi')^4-\mu\zeta\phi'$ for $v=e^\phi$ after division by $v$, as in \cite[Sections~4.1 and~5.1.2]{BernoffBertozziWitelski1998}.
Their cancellation gives $(\phi')^3=-\mu\zeta/a_\infty$.
Thus $\phi'=c\varrho_j\zeta^{1/3}$, where $c=(\mu/a_\infty)^{1/3}$, $\kappa=3c/4$ and $(\varrho_1,\varrho_2,\varrho_3)=(-1,\rho_+,\rho_-)$ with $\rho_\pm=(1\pm i\sqrt3)/2$.
The ansatz $v=\zeta^\beta\exp(\varrho_j\kappa\zeta^{4/3})$ has logarithmic derivative
\[
 \frac{v'}v=\varrho_jc\zeta^{1/3}+\frac\beta\zeta.
\]
The terms of order $\zeta^{4/3}$ cancel because $a_\infty(\varrho_jc)^3=-\mu$.
At the next, constant, order the fourth derivative contributes $4\mu\beta+2\mu$, the leading term $a_3=-2a_\infty\zeta^{-1}$ contributes $2\mu$, and the transport term contributes $\mu(1-\beta)$.
Equating their sum to $\lambda$ gives $\beta=\lambda/(3\mu)-5/3$.
We therefore seek three exponential solutions with the asymptotics
\begin{equation}\label{eq:fredholm-wkb-solutions}
 v_j(\zeta)
 =|\zeta|^{\lambda/(3\mu)-5/3}
   \exp\{\varrho_j\kappa|\zeta|^{4/3}\}
   \{1+O(|\zeta|^{-4/3})\}.
\end{equation}
All error estimates in this paragraph may be differentiated a fixed number of times, uniformly for $\lambda$ in a compact set.
In the proof of \cref{lem:fredholm-end-dichotomy} we construct solutions with these asymptotics by a Volterra argument.

Among the roots of $\varrho^3=-1$, the root $-1$ gives a decaying branch, whereas $e^{i\pi/3}$ and $e^{-i\pi/3}$ give two growing branches.
For the algebraic branch the weighted square is
\[
 |v_{\rm a}(\zeta)|^2\varpi(\zeta)
 =|\zeta|^{-2-2\Re\lambda/\mu}(1+o(1)).
\]
The same calculation applies to $\zeta v_{\rm a}'$, while ordinary derivatives only improve the power.
Hence the algebraic branch belongs to the graph space on an end precisely when
\begin{equation}\label{eq:fredholm-algebraic-threshold}
 \Re\lambda>-\frac\mu2.
\end{equation}

For later estimates we make the logarithmic conjugation exact.
On the positive end set $\zeta=e^y$ and define
\begin{equation}\label{eq:fredholm-detail-unitary-end-map}
 (Jv)(y)=\{\varpi(e^y)e^y\}^{1/2}v(e^y)
       =\frac{e^{-3y/2}}{1+e^{-2y}}v(e^y).
\end{equation}
Then $J$ is unitary from the weighted $L^2$ space on the positive end to $L^2(\dd y)$, and a direct calculation gives
\begin{equation}\label{eq:fredholm-detail-exact-log-transport}
 J(1-\zeta\partial_\zeta)J^{-1}
 =-\frac12-\partial_y+\frac{2}{1+e^{2y}}.
\end{equation}
The last coefficient is $O(e^{-2y})$ as $y\to\infty$.
The fourth-order part has the form
\begin{equation}\label{eq:fredholm-detail-log-fourth-order}
 J\partial_\zeta^4J^{-1}
 =e^{-4y}\prod_{r=0}^3
   \left(\partial_y+\frac32-r+O(e^{-2y})\right),
\end{equation}
where differentiating an $O(e^{-2y})$ coefficient preserves that order.
Thus the exact limiting operator in the weighted space is $\mu(-1/2-\partial_y)$.

We use the following elementary form of the Volterra dichotomy argument.

\begin{lemma}[Volterra inversion on separated channels]
\label{lem:fredholm-detail-volterra}
Let $\Lambda=\diag(\lambda_1,\ldots,\lambda_N)$, with $|\Re\lambda_j|\geq\gamma>0$, and let $\mathsf D$ range over a bounded set of diagonal matrices.
For $m\in\R$ and $-\gamma<\eta<\gamma$, let $w(x)=e^{\eta(x-X)}(x/X)^m$ and let $E$ be either $L^\infty([X,\infty);\C^N)$ or $L^2([X,\infty);\C^N)$.
Suppose $R$ is a measurable matrix function with $\|R(x)\|\le Cx^{-1-\epsilon}$ for some $\epsilon>0$.
There exist $X_0$ and $C'$ depending only on $\gamma-|\eta|$, $C$, $\epsilon$, $|m|$ and the bound for $\mathsf D$ such that the following holds for every $X\ge X_0$.
Given $wF\in E$ and prescribed values $z_{\rm s}$ at $X$ for the components with $\Re\lambda_j<0$, the system
\begin{equation}\label{eq:separated-channel-system}
 Z_x=\Lambda Z+x^{-1}\mathsf D Z+R(x)Z+F(x)
\end{equation}
has a unique locally absolutely continuous solution with $wZ\in E$ and these boundary values.
It satisfies $\|wZ\|_E\le C'(|z_{\rm s}|+\|wF\|_E)$.
The solution operator depends analytically on parameters on which $\mathsf D$ and $R$ depend analytically under these bounds.
\end{lemma}

\begin{proof}
Integrate the stable components of \eqref{eq:separated-channel-system} from $X$ to $x$ and the unstable components from infinity to $x$, treating $RZ+F$ as the forcing.
If $d_j=\mathsf D_{jj}$, the weighted kernel on the chosen integration interval has modulus
\[
 \left|e^{(\lambda_j+\eta)(x-s)}(x/s)^{d_j+m}\right|
 \le \exp\!\left(-(\gamma-|\eta|)|x-s|
         +\frac{|\Re d_j+m|}{X}|x-s|\right)
 \le e^{-(\gamma-|\eta|)|x-s|/2}
\]
for sufficiently large $X$, since $|\log(x/s)|\le |x-s|/X$.
Taking $s=X$ bounds the homogeneous stable term in $E$ by $C'|z_{\rm s}|$.
Convolution with the displayed exponential is bounded on both spaces, and multiplication by $R$ has norm at most $CX^{-1-\epsilon}$.
Increasing $X$ makes the integral operator on $wZ$ a contraction.
The Neumann series gives the estimate and analytic dependence.
The integral equations give a locally absolutely continuous solution, and the condition $wZ\in E$ excludes a homogeneous growing term in each unstable component.
\end{proof}

We now apply this argument to the conical operator.

\begin{lemma}[Uniform end dichotomy]
\label{lem:fredholm-end-dichotomy}
Let $K$ be a compact subset of $\{\Re\lambda>-\mu/2\}$.
For $R$ sufficiently large, the first-order system associated with $(\lambda-\fullLin)v=f$ on either end has a two-dimensional admissible bundle.
It is spanned asymptotically by \eqref{eq:fredholm-algebraic-solution} and the decaying branch in \eqref{eq:fredholm-wkb-solutions}; the complementary two-dimensional bundle is inadmissible.
The one-sided Green operator is bounded from $H^k_2$ data to the end restriction of $\evenDom^k$, uniformly for $\lambda\in K$, and depends analytically on $\lambda$.
\end{lemma}

\begin{proof}
We work on the positive end; reflection gives the negative end.
With $W=(v,v',v'',v''')^T$, direct solution of $(\lambda-\fullLin)v=f$ for $v^{(4)}$ gives
\begin{equation}\label{eq:fredholm-first-order-system}
 W'=\begin{pmatrix}
 0&1&0&0\\
 0&0&1&0\\
 0&0&0&1\\
 (a_0+\mu-\lambda)/a_4&(a_1-\mu\zeta)/a_4&a_2/a_4&a_3/a_4
 \end{pmatrix}W
 +a_4^{-1}(0,0,0,f)^T.
\end{equation}

We change variables in \eqref{eq:fredholm-first-order-system} to separate the algebraic solution from the three exponential solutions.
Put
\[
 \theta=\frac\lambda\mu,\qquad
 \gamma=1-\theta,\qquad
 x=\zeta^{4/3},
\]
and scale the Cauchy vector by
\begin{equation}\label{eq:fredholm-detail-Y-scaling}
 Y=\diag
   (1,\zeta^{-1/3},\zeta^{-2/3},\zeta^{-1})W.
\end{equation}
Since $\dd x/\dd\zeta=(4/3)\zeta^{1/3}$, substitution of \eqref{eq:fredholm-detail-a4-decay}-- \eqref{eq:fredholm-detail-lower-decay} into \eqref{eq:fredholm-first-order-system} gives
\begin{equation}\label{eq:fredholm-detail-raw-four-system}
 Y_x=\{C_0+x^{-1}C_1+R_2(x,\lambda)\}Y+F_Y,
 \qquad \|\partial_x^rR_2(x,\lambda)\|
       \leq C_{K,r}x^{-2-r},
\end{equation}
where
\begin{equation}\label{eq:fredholm-detail-C0C1}
 C_0=\frac34
 \begin{pmatrix}
 0&1&0&0\\
 0&0&1&0\\
 0&0&0&1\\
 0&-\mu/a_\infty&0&0
 \end{pmatrix},
 \qquad
 C_1=
 \begin{pmatrix}
 0&0&0&0\\
 0&-1/4&0&0\\
 0&0&-1/2&0\\
 3(\mu-\lambda)/(4a_\infty)&0&0&-9/4
 \end{pmatrix},
\end{equation}
and, with $e_4=(0,0,0,1)^T$, the forcing is exactly
\begin{equation}\label{eq:fredholm-detail-raw-forcing}
 F_Y=\frac{3}{4a_4(\zeta)}x^{-1}e_4f.
\end{equation}
For example, the last diagonal entry $-9/4$ in $C_1$ is the sum of the derivative of the row weight $\zeta^{-1}$ and the leading coefficient $a_3=-2a_\infty\zeta^{-1}+O(\zeta^{-5})$.
This is the term which changes the WKB power from the constant-principal-part value.

With the roots $\varrho_j$ in \eqref{eq:fredholm-wkb-solutions}, set
\[
 \nu_0=0,\qquad \nu_j=\frac34c\varrho_j,
 \quad 1\leq j\leq3.
\]
Use the columns
\begin{align*}
 s_{\rm a}(x)
 &=\bigl(1,\gamma x^{-1},
   \gamma(\gamma-1)x^{-2},
   \gamma(\gamma-1)(\gamma-2)x^{-3}\bigr)^T,\\
 s_j&=(1,c\varrho_j,c^2\varrho_j^2,c^3\varrho_j^3)^T,
 \qquad 1\leq j\leq3,
\end{align*}
and let $V(x,\lambda)=[s_{\rm a},s_1,s_2,s_3]$.
As $x\to\infty$, $V(x,\lambda)$ converges uniformly for $\lambda\in K$ to the Vandermonde matrix with distinct nodes $0,c\varrho_1,c\varrho_2,c\varrho_3$.
Thus $V(x,\lambda)^{-1}$ is uniformly bounded for large $x$.
Substituting $Y=VQ$ in \eqref{eq:fredholm-detail-raw-four-system} gives
\begin{equation}\label{eq:fredholm-detail-B1-system}
 Q_x=\{\Lambda+x^{-1}B_1(\theta)+O_K(x^{-2})\}Q
     +V^{-1}F_Y,
 \qquad
 \Lambda=\diag(\nu_0,\nu_1,\nu_2,\nu_3).
\end{equation}
Writing indices in $\{0,1,2,3\}$, direct matrix multiplication gives
\begin{align}
 (B_1)_{00}&=\frac34(1-\theta),
 &(B_1)_{0j}&=\frac34(4-\theta),
 &(B_1)_{j0}&=0,
 \label{eq:fredholm-detail-B1-algebraic}\\
 (B_1)_{ij}
 &=\frac13\left(-3+\frac{3\theta}{4}
   -\frac{\varrho_j}{4\varrho_i}
   -\frac{\varrho_j^2}{2\varrho_i^2}\right),
 &&1\leq i,j\leq3.
 \label{eq:fredholm-detail-B1-WKB}
\end{align}
In particular,
\[
 (B_1)_{jj}=\frac\theta4-\frac54,
 \qquad 1\leq j\leq3.
\]
The entries $(B_1)_{0j}$ in \eqref{eq:fredholm-detail-B1-algebraic} couple the three exponential components into the algebraic equation at order $x^{-1}$.

Define the full $4\times4$ matrix $T_1$ by
\begin{equation}\label{eq:fredholm-detail-T1}
 (T_1)_{ij}=-\frac{(B_1)_{ij}}{\nu_i-\nu_j}\quad(i\neq j),
 \qquad (T_1)_{ii}=0.
\end{equation}
All denominators are nonzero.
With $Q=(I+x^{-1}T_1)Z$, the identity
\[
 (B_1+[\Lambda,T_1])_{ij}=0\qquad(i\neq j)
\]
removes every off-diagonal $x^{-1}$ entry, including the top row.
Writing $Z=(z_0,z_1,z_2,z_3)^T$, we obtain the full normal form
\begin{equation}\label{eq:fredholm-wkb-normal-form}
 Z_x=\{\Lambda+x^{-1}\mathsf D(\theta)+R(x,\lambda)\}Z
 +x^{-1}f(x^{3/4})g+G_2,
\end{equation}
where
\begin{align}
 \mathsf D(\theta)&=\diag
 \left(\frac34(1-\theta),
       \frac\theta4-\frac54,
       \frac\theta4-\frac54,
       \frac\theta4-\frac54\right),
 \label{eq:fredholm-detail-D}\\
 g&=\left(\frac{3}{4\mu},-\frac1{4\mu},
                 -\frac1{4\mu},-\frac1{4\mu}\right)^T,
 \label{eq:fredholm-detail-g}\\
 \|\partial_x^rR(x,\lambda)\|&\leq C_{K,r}x^{-2-r},
 &|G_2|&\leq C_Kx^{-2}|f(x^{3/4})|.
 \label{eq:fredholm-detail-RG2}
\end{align}
The leading forcing vector follows by applying the inverse limiting Vandermonde matrix to $3e_4/(4a_\infty)$; the relation $c^3=\mu/a_\infty$ gives exactly \eqref{eq:fredholm-detail-g}.

At $\theta=0$ the matrix \eqref{eq:fredholm-detail-D} is precisely $\tailDiag$ from \eqref{eq:validated-tail-system}.
The diagonal powers recover the formal solutions.
The algebraic exponent is
\[
 x^{\mathsf D_{00}(\theta)}=x^{3(1-\theta)/4}=\zeta^{1-\theta},
\]
and the WKB exponents are
\[
 e^{\nu_jx}x^{\theta/4-5/4}
 =\exp\{\varrho_j\kappa\zeta^{4/3}\}
  \zeta^{\theta/3-5/3}.
\]
To obtain the relative error in \eqref{eq:fredholm-algebraic-solution}, put $d=3\gamma/4$ and $e_0=(1,0,0,0)^T$ in the $Z$ coordinates.
The trial $v=\zeta^\gamma$ has scaled Cauchy vector $Y=x^d s_{\rm a}(x)$, hence $Z=x^de_0$ since $T_1e_0=0$ by \eqref{eq:fredholm-detail-B1-algebraic} and \eqref{eq:fredholm-detail-T1}.
Its scalar residual divided by $\zeta^\gamma$ is $O(\zeta^{-4})=O(x^{-3})$ by \eqref{eq:fredholm-coefficient-decay}.
The forcing factor $x^{-1}$ in \eqref{eq:fredholm-detail-raw-forcing} therefore gives $R(x,\lambda)e_0=O(x^{-4})$, with each $x$ derivative improving the bound by one power.

For the homogeneous equation write $Z=x^d(e_0+h)$.
By \eqref{eq:fredholm-wkb-normal-form},
\[
 h'=\bigl(\Lambda+x^{-1}(\mathsf D-dI)\bigr)h+R(x,\lambda)(e_0+h).
\]
In the first component the diagonal coefficient is zero.
Integrating the $O(x^{-4})$ forcing from infinity gives $O(x^{-3})$.
In the other three components, use \cref{lem:fredholm-detail-volterra} with $E=L^\infty$, $\eta=0$ and $m=4$.
Its diagonal kernels preserve the $O(x^{-4})$ bound, with zero data at $X$ in the decaying component.
In the norm $\sup_{x\ge X}x^3|h(x)|$, the remaining forcing $Rh$ is bounded by $Cx^{-5}\|h\|$.
Its integral has norm at most $C/X$ times $\|h\|$, so for large $X$ the integral equation has a solution $h=O(x^{-3})$.
Differentiating the equation $j$ times gives the same diagonal system for $h^{(j)}$ with coupling $Rh^{(j)}$ and forcing $O(x^{-4-j})$ by induction, using \eqref{eq:fredholm-detail-RG2}.
The same integral estimates give $h^{(j)}=O(x^{-3-j})$.
Returning to $v=Y_1$ gives \eqref{eq:fredholm-algebraic-solution} with its differentiated error.

For the exponential solution with index $j\in\{1,2,3\}$, write $d_j=\mathsf D_{jj}$ and $Z=e^{\nu_j x}x^{d_j}(e_j+h)$, where $e_j$ is the corresponding coordinate vector.
By \eqref{eq:fredholm-wkb-normal-form}, the forcing in the diagonal equation for $h$ is $R(e_j+h)$, and its $i$th component has propagator
\[
 e^{(\nu_i-\nu_j)(x-s)}(x/s)^{\mathsf D_{ii}-d_j}.
\]
For $\Re\nu_i<\Re\nu_j$ integrate forwards from $X$ with zero initial value, and for $\Re\nu_i\ge\Re\nu_j$ integrate backwards from infinity.
When the real parts differ, the propagator is bounded in the chosen direction by $C_K e^{-c|x-s|}$ for sufficiently large $X$.
When they agree, \eqref{eq:fredholm-detail-D} gives $\mathsf D_{ii}=d_j$, including for the two distinct growing roots, so the propagator has modulus one.
The integral of $Re_j$ is therefore $O_K(x^{-1})$, since the backward integral of $s^{-2}$ is $x^{-1}$.
In the norm $\sup_{x\ge X}x|h(x)|$ the integral of $Rh$ has operator norm at most $C_K/X$ by the same bounds, so for large $X$ we obtain $h=O_K(x^{-1})$.

The differential equation and \eqref{eq:fredholm-detail-RG2} first give $h^{(\ell)}=O_{K,\ell}(x^{-1})$ for every fixed $\ell$, hence zero limiting data in each backward integral.
On differentiating $\ell$ times, the forcing apart from $Rh^{(\ell)}$ is $O_{K,\ell}(x^{-2-\ell})$ if $h^{(m)}=O_{K,m}(x^{-1-m})$ for $m<\ell$.
The preceding integral bounds in the norm $\sup_{x\ge X}x^{1+\ell}|h^{(\ell)}(x)|$ give $h^{(\ell)}=O_{K,\ell}(x^{-1-\ell})$ by induction, retaining the finite data $h_i^{(\ell)}(X)$ in the forward components, where their contribution decays exponentially.
Returning to $Y=V(I+T_1/x)Z$ and $x=|\zeta|^{4/3}$ gives \eqref{eq:fredholm-wkb-solutions} with its differentiated errors.

It remains to bound the solution of $(\lambda-\fullLin)v=f$ and its derivatives in terms of $f$ and the two prescribed boundary values.
In the $x$ coordinate the weighted measure is
\begin{equation}\label{eq:fredholm-detail-x-measure}
 \dd\omega(x)=\frac34x^{-1/4}(1+x^{3/2})^{-2}\,\dd x,
\end{equation}
which is equivalent on $[X,\infty)$ to $x^{-13/4}\dd x$.
Moreover
\[
 \partial_\zeta=\frac43x^{1/4}\partial_x,
 \qquad
 \zeta\partial_\zeta=\frac43x\partial_x.
\]
Thus the squared end norm $\|f\|_{H^k_2}^2$ is the sum of the squared $L^2(\dd\omega)$ norms of $(\partial_\zeta^rf)(x^{3/4})$ for $0\leq r\leq k$.

For the algebraic coordinate put $t=\log x$ and
\[
 \widehat z_0(t)=e^{-9t/8}z_0(e^t).
\]
The factor $9/8$ is one half of the logarithmic density exponent because $x^{-13/4}\dd x=e^{-9t/4}\dd t$.
In the same notation set
\[
 \widehat f(t)=e^{-9t/8}f(e^{3t/4}),
 \qquad
 \widehat Z(t)=e^{-9t/8}Z(e^t).
\]
The exact density in \eqref{eq:fredholm-detail-x-measure} adds a multiplier $1+O(e^{-3t/2})$, which is absorbed into the integrable remainder.
The first equation in \eqref{eq:fredholm-wkb-normal-form} becomes
\[
 \partial_t\widehat z_0
 =-\beta(\lambda)\widehat z_0
  +\frac{3}{4\mu}\widehat f+O_K(e^{-t})\widehat Z
  +O_K(e^{-t})\widehat f,
 \qquad
 \beta(\lambda)
 =\frac34\left(\frac{\lambda}{\mu}+\frac12\right).
\]
Its real part is bounded below by a positive constant on $K$.
Forward convolution in $t$ is therefore bounded on $L^2(\dd t)$, with norm at most $(\inf_K\Re\beta)^{-1}$, once the algebraic boundary value at $x=X$ has been prescribed.

For the WKB coordinates, the root $\varrho_1=-1$ is integrated forward from its prescribed boundary value and the two roots of positive real part are integrated backward from infinity.
After extracting the powers in $\mathsf D(\theta)$ their kernels have the form
\begin{align*}
 K_1(x,s)&={\bf1}_{s\leq x}
 e^{-3c(x-s)/4}(x/s)^{\theta/4-5/4},\\
 K_j(x,s)&=-{\bf1}_{s\geq x}
 e^{\nu_j(x-s)}(x/s)^{\theta/4-5/4},
 \qquad j=2,3.
\end{align*}
The exponential factors dominate the polynomial ratios uniformly on $K$.
To include the coupling to the algebraic coordinate, put
\[
 \algSize=\|z_0\|_{L^2(\dd\omega)},\qquad
 \wkbSize=\sum_{j=1}^3\bigl(
 \|z_j\|_{L^2(\dd\omega)}+\|xz_j\|_{L^2(\dd\omega)}\bigr).
\]
Schur's test for $K_j(x,s)s^{-1}$ and $xK_j(x,s)s^{-1}$ bounds the WKB response to the leading forcing by $C_K\|f\|_{L^2(\dd\omega)}$.
A remainder entry $R_{j\ell}(s)=O_K(s^{-2})$ has an additional factor $s^{-1}$, so its contribution to $\wkbSize$ is at most $C_KX^{-1}(\algSize+\wkbSize)$.
Here the term $C_KX^{-1}\algSize$ accounts for the entries $R_{j0}$ coupling the algebraic solution to the exponential components.
For the first component of \eqref{eq:fredholm-wkb-normal-form}, the kernel $e^{-\beta(\lambda)t}\mathbf 1_{t\ge0}$ has $L^1$ norm $(\Re\beta(\lambda))^{-1}$.
Thus $xR(x)=O_K(x^{-1})$ gives the same factor $C_KX^{-1}$ in the algebraic component.
For any prescribed $\varepsilon>0$, increasing $X$ gives the coupled inequalities
\begin{align*}
 \algSize&\leq C_K\bigl(\|f\|_{L^2(\dd\omega)}+|\xi_{\rm a}|\bigr)
       +\varepsilon(\algSize+\wkbSize),\\
 \wkbSize&\leq C_K\bigl(\|f\|_{L^2(\dd\omega)}+|\xi_-|\bigr)
       +\varepsilon(\algSize+\wkbSize),
\end{align*}
where $\xi_{\rm a}$ and $\xi_-$ are the prescribed algebraic and decaying-WKB boundary coordinates at the now fixed $X$.
Adding these inequalities with $2\varepsilon<1$ allows us to absorb $2\varepsilon(\algSize+\wkbSize)$ on the left.
Equivalently, the coupled Volterra map is inverted by a Neumann series in that norm.
In particular, with $|\xi|=|\xi_{\rm a}|+|\xi_-|$,
\begin{equation}\label{eq:fredholm-detail-WKB-forcing-bound}
 \sum_{j=1}^3\left(
  \|z_j\|_{L^2(\dd\omega)}+
  \|xz_j\|_{L^2(\dd\omega)}\right)
 \leq C_K\bigl(\|f\|_{L^2(\dd\omega)}+|\xi|\bigr).
\end{equation}
This proves existence and uniqueness for each prescribed pair $\xi=(\xi_{\rm a},\xi_-)$ of admissible boundary coordinates.

We now use the bound for $xz_j$ in \eqref{eq:fredholm-detail-WKB-forcing-bound} to estimate $\zeta v'$.
From \eqref{eq:fredholm-detail-Y-scaling},
\[
 v=Y_1,
 \qquad \zeta v'=xY_2.
\]
Write $E_R=\{|\zeta|>R\}$ for the two ends, using reflection for the estimate on $\zeta<-R$.
The algebraic column satisfies $x(s_{\rm a})_2=\gamma$, while the WKB columns have constant second entries.
Therefore
\begin{equation}\label{eq:fredholm-detail-base-graph-bound}
 \|v\|_{H^0_2(E_R)}+\|\zeta v'\|_{H^0_2(E_R)}
 \leq C_K\bigl(\|f\|_{H^0_2(E_R)}+|\xi|\bigr).
\end{equation}

It remains to estimate the ordinary derivatives of $v$.
At the inner endpoint write
\[
 \Gamma_Rv=(v(R),v'(R),v''(R),v'''(R))\in\C^4.
\]
Applying the interior elliptic estimate on dyadic intervals of the end, with $a_4$ bounded below by \cref{def:conical-linearisation}, gives
\begin{equation}\label{eq:fredholm-detail-end-elliptic-upgrade}
 \|v\|_{H^4_2(E_{2R})}
 \leq C_K\left(
  \|f\|_{H^0_2(E_R)}+
  \|v\|_{H^0_2(E_R)}+
  \|\zeta v'\|_{H^0_2(E_R)}+
  \|\Gamma_Rv\|\right).
\end{equation}
To pass from $k-1$ to $k$, commute one ordinary derivative through the equation.
The drift identity
\[
 [\partial_\zeta^r,\mu(1-\zeta\partial_\zeta)]
 =-r\mu\partial_\zeta^r
\]
shifts the spectral parameter to $\lambda+r\mu$, farther into the same Fredholm half-plane, while differentiated coefficient commutators contain at most $r+3$ derivatives and are controlled by the previous elliptic step.
For fixed $k$, the finite union $\bigcup_{r=0}^k(K+r\mu)$ is still a compact subset of that half-plane, so all constants remain uniform.
Induction using \eqref{eq:fredholm-detail-base-graph-bound} and \eqref{eq:fredholm-detail-end-elliptic-upgrade} gives
\[
 \|v\|_{H^{k+4}_2(E_{2R})}
 +\|\zeta v'\|_{H^k_2(E_{2R})}
 \leq C_K\left(\|f\|_{H^k_2(E_R)}+\|\Gamma_Rv\|\right).
\]
This is the asserted Green-map graph estimate.
All matrices and Volterra kernels depend analytically on $\lambda$, and their Neumann series converge uniformly on $K$, proving analytic dependence.
\end{proof}

Cutting off on $A_R=\{R<|\zeta|<2R\}$, we obtain
\begin{equation}\label{eq:fredholm-detail-end-estimate}
 \|v\|_{\evenDom^k(E_{2R})}
 \leq C_K\left(
   \|(\lambda-\fullLin)v\|_{H^k_2(E_R)}
   +\|v\|_{H^{k+3}(A_R)}
 \right).
\end{equation}
To extend an end solution to the compact interval we need its Cauchy data at $\zeta=R$.
We show how the annular term in \eqref{eq:fredholm-detail-end-estimate} controls the two coordinates that are still free.
The dichotomy gives an analytic rank-two projection $P_{\rm ad}(R,\lambda)$ on the space $\C^4$ of Cauchy data at $R$.
For data $f$ on the end, the Green construction produces a particular solution $G_R(\lambda)f$ for which the freely prescribed admissible coordinate is zero.
Every graph-space solution is then uniquely of the form
\begin{equation}\label{eq:fredholm-detail-end-solution-form}
 v=G_R(\lambda)f+\Phi_R(\lambda)c,
 \qquad c\in\operatorname{ran}P_{\rm ad}(R,\lambda),
\end{equation}
Here $\Phi_R(\lambda)c$ is the scalar solution of $(\lambda-\fullLin)v=0$ with Cauchy data $\Gamma_Rv=c$ at $\zeta=R$.
Both maps in \eqref{eq:fredholm-detail-end-solution-form} are bounded in the graph norm.
By the trace theorem on $A_R$, the annular term in \eqref{eq:fredholm-detail-end-estimate} bounds $c$.
Conversely, every $c\in\operatorname{ran}P_{\rm ad}(R,\lambda)$ gives a graph-space solution.

\subsection*{Fredholm index, generation and resolvent bounds}

On every fixed compact interval, uniform ellipticity and the usual one-dimensional interior estimate give, for a cutoff $\chi$ and $\lambda$ in a compact set,
\begin{equation}\label{eq:fredholm-detail-local-elliptic}
 \|\chi v\|_{H^{k+4}}
 \leq C\left(
   \|\chi(\lambda-\fullLin)v\|_{H^k}
   +\|v\|_{H^{k+3}(\supp\chi)}
 \right).
\end{equation}
This follows by freezing $a_4$, applying the Fourier estimate for $1+\partial_\zeta^4$, and absorbing the coefficient oscillation on a finite cover.
The drift is bounded on the support of $\chi$ and is lower order.

\begin{lemma}[Fredholm estimate]
\label{lem:fredholm-global-estimate}
For every compact $K\Subset\{\Re\lambda>-\mu/2\}$ there are $C,R<\infty$ such that
\begin{equation}\label{eq:fredholm-global-estimate}
 \|v\|_{\evenDom^k}
 \leq C\left\{
 \|(\lambda-\fullLin)v\|_{\evenX^k}
 +\|v\|_{H^{k+3}(-R,R)}\right\}
\end{equation}
for $\lambda\in K$ and $v\in\evenDom^k$.
Consequently $\lambda-\fullLin:\evenDom^k\to\evenX^k$ is Fredholm of index zero throughout this half-plane.
\end{lemma}

\begin{proof}
Choose an even, compactly supported cutoff $\chi_0$ and cutoffs $\chi_\pm$ supported on the respective ends, with $\chi_-(\zeta)=\chi_+(-\zeta)$ and $\chi_0+\chi_++\chi_-=1$.
Apply \eqref{eq:fredholm-detail-end-estimate} to $\chi_\pm v$ and \eqref{eq:fredholm-detail-local-elliptic} to $\chi_0v$.
Every commutator $[\fullLin,\chi_j]$ has order at most three and is supported in a fixed compact annulus.
The resulting annular terms are all bounded by one $H^{k+3}(-R,R)$ norm.
This proves \eqref{eq:fredholm-global-estimate}.

To compute the index we work on $[0,\infty)$ and impose evenness at the centre.
Fix $\lambda\in K$, put $P=\lambda-\fullLin$, and let $X_+^k$ be the restrictions of $\evenX^k$ to the half-line.
For the moment allow arbitrary Cauchy data at zero.
The domain of $P$ is then
\[
 D_+^k=\{v\in H^{k+4}_2(\R_+):\zeta v'\in H^k_2(\R_+),\ Pv\in X_+^k\},
\]
with the half-line norm of \eqref{eq:fredholm-detail-domain-norm}.
This is a closed subspace since even extendibility in $H^k_2$ is a closed condition on $Pv$.
For $f\in X_+^k$, continue the particular end solution $G_R(\lambda)f$ of \eqref{eq:fredholm-detail-end-solution-form} from $R$ to zero by solving $Pv=f$ with its four Cauchy data at $R$.
Thus $P:D_+^k\to X_+^k$ is onto.
The same continuation identifies its kernel with the two admissible end modes of \cref{lem:fredholm-end-dichotomy}.

The boundary conditions for evenness are
\begin{equation}\label{eq:fredholm-detail-even-boundary}
 B_{\rm ev}v:=(v'(0),v'''(0))=(0,0).
\end{equation}
To construct a right inverse $S:\C^2\to D_+^k$ for the bounded trace $B_{\rm ev}$, solve $Pv=0$ near zero with the prescribed odd derivatives and the two even Cauchy data zero, then multiply by a smooth cutoff supported there and equal to one near zero.
The resulting $Pv$ vanishes near zero and hence belongs to $X_+^k$.
For $v\in\ker B_{\rm ev}$, solve the Cauchy problem across zero using the even extension of $Pv$.
Reflection preserves the equation by \eqref{eq:fredholm-coefficient-form} and preserves the four initial values by \eqref{eq:fredholm-detail-even-boundary}.
Uniqueness makes the solution even, so its $H^{k+4}$ regularity across zero supplies all the higher odd trace conditions.
This identifies $\ker B_{\rm ev}$ with the restriction of $\evenDom^k$.

The map $(P,0):D_+^k\to X_+^k\times\C^2$ has kernel and cokernel of dimension two.
The finite-rank perturbation $(P,B_{\rm ev})$ has the same index by \cite[Chapter~IV, Theorem~5.26]{Kato1966}.
Writing $v=w+Sb$ with $b=B_{\rm ev}v$ gives $(P,B_{\rm ev})v=(Pw+PSb,b)$, so subtraction of $PSb$ from the first component identifies this map with $P|_{\ker B_{\rm ev}}\oplus I_{\C^2}$.
This gives the half-line index formula of the dichotomy argument \cite{Palmer1988},
\begin{equation}\label{eq:fredholm-detail-index-formula}
 \operatorname{ind}(P|_{\ker B_{\rm ev}})=\dim\ker P-\operatorname{rank}B_{\rm ev}=2-2=0 .
\end{equation}
\end{proof}

We next identify the closed operator to which the Fredholm calculation applies.
Let $\minLin$ be the closure in $\evenX^k$ of $\fullLin$ acting on $C_c^\infty(\R)_{\rm even}$, and let $\maxLin$ be the distributional realisation with
\[
 \opDom(\maxLin)
 =\{v\in\evenX^k:\fullLin v\in\evenX^k
   \text{ in }\distrib(\R)\}.
\]

\begin{lemma}[Maximal realisation]
\label{lem:fredholm-generation}
One has
\[
 \minLin=\maxLin,
 \qquad
 \opDom(\maxLin)=\evenDom^k.
\]
This operator generates a $C_0$-semigroup on $\evenX^k$.
\end{lemma}

\begin{proof}
We divide the proof into the domain, energy, and range assertions.

\smallskip
\noindent\emph{Step 1: compactly supported functions form a core.}
Take the cutoff $\chi_{\rm in}$ from \cref{lem:controlled-mode-localisation}, equal to one on $[-1,1]$, and, for $R>e$, set
\[
 \xi_R(\zeta)=
 \chi_{\rm in}\left(\frac{\log\langle\zeta\rangle}{\log R}\right).
\]
On the transition region, for every $j\geq1$,
\begin{equation}\label{eq:fredholm-detail-log-cutoff}
 |\partial_\zeta^j\xi_R(\zeta)|
 \leq \frac{C_j}{\log R}\langle\zeta\rangle^{-j},
 \qquad
 |\zeta\xi_R'(\zeta)|\leq\frac C{\log R}.
\end{equation}
Dominated convergence and the product rule give $\xi_Rv\to v$ in $H^{k+4}_2$.
Moreover,
\begin{align*}
 \zeta\partial_\zeta\{(1-\xi_R)v\}
 &=(1-\xi_R)\zeta v'-\zeta\xi_R'v,
\end{align*}
and \eqref{eq:fredholm-detail-log-cutoff}, after commuting up to $k$ derivatives, shows convergence to zero in $H^k_2$.
Thus $\xi_Rv\to v$ in $\evenDom^k$.
Even mollification of $\xi_Rv$ proves that $C_c^\infty(\R)_{\rm even}$ is dense in $\evenDom^k$.

The global estimate \eqref{eq:fredholm-global-estimate} at a fixed large positive $\lambda$, together with local interpolation
\[
 \|v\|_{H^{k+3}(-R,R)}
 \leq\varepsilon\|v\|_{H^{k+4}(-R,R)}
       +C_\varepsilon\|v\|_{H^k(-R,R)},
\]
shows that the norm in \eqref{eq:fredholm-detail-domain-norm} is equivalent to the operator graph norm
\[
 \|v\|_{\evenX^k}+\|\fullLin v\|_{\evenX^k}.
\]
Hence the minimal domain contains $\evenDom^k$.

Conversely, if $v\in\opDom(\maxLin)$, local elliptic regularity first gives $v\in H^{k+4}_{\rm loc}$.
Fix a $\lambda$ in the Fredholm half-plane and write $f=(\lambda-\fullLin)v\in\evenX^k$.
On an end, let $u=G_R(\lambda)f$, where $G_R$ is the particular Green map from \eqref{eq:fredholm-detail-end-solution-form}.
Then $u$ belongs to the end graph space and $h=v-u$ is a distributional homogeneous solution.
Local regularity makes $h$ smooth.
The asymptotic basis constructed in \cref{lem:fredholm-end-dichotomy} represents $h$ as a sum of the four channels.
Since both $v$ and $u$ lie in the weighted phase space, the two growing WKB coefficients vanish.
The algebraic branch belongs to the graph space by \eqref{eq:fredholm-algebraic-threshold}, and the decaying branch does so by the differentiated asymptotics \eqref{eq:fredholm-wkb-solutions}.
Consequently $v$ lies in the graph space on both ends.
Compact interior regularity then proves $v\in H^{k+4}_2$ and $\zeta v'\in H^k_2$.
Thus the maximal domain is $\evenDom^k$, and the minimal and maximal realisations agree.

\smallskip
\noindent\emph{Step 2: the weighted G\aa rding estimate.}
Use the Hilbert inner product associated with the derivative-sum norm on $H^k_2$.
Twice integrating the principal term in \eqref{eq:fredholm-coefficient-form} by parts gives, for $j=0$,
\[
 \Re\int -a_4v^{(4)}\overline v\,\varpi\,\dd\zeta
 \leq-c\|v''\|_{H^0_2}^2
     +C\|v\|_{H^1_2}^2.
\]
For $0\leq j\leq k$, commute $\partial_\zeta^j$ through the operator and repeat the calculation.
Derivatives of $a_4$ are bounded, the terms of order at most three are controlled by interpolation, and all logarithmic derivatives of $\varpi$ are bounded.
The drift produces no derivative loss:
\[
 \Re\int -\mu\zeta v'\overline v\varpi\,\dd\zeta
 =\frac\mu2\int(\zeta\varpi)'|v|^2\,\dd\zeta,
 \qquad
 [\partial_\zeta^j,\zeta\partial_\zeta]
 =j\partial_\zeta^j.
\]
Summing and absorbing the interpolation terms yields
\begin{equation}\label{eq:fredholm-garding}
 \Re\langle\fullLin v,v\rangle_{\evenX^k}
 \leq-c\|v\|_{H^{k+2}_2}^2
       +C\|v\|_{\evenX^k}^2.
\end{equation}

\smallskip
\noindent\emph{Step 3: maximal dissipativity after a shift.}
For real $\lambda>C$, estimate \eqref{eq:fredholm-garding} implies that $(\lambda-\fullLin)v=0$ only when $v=0$.
The operator is Fredholm of index zero by \cref{lem:fredholm-global-estimate}, and therefore it is also surjective.
Thus $\fullLin-C$ is dissipative and $\lambda-(\fullLin-C)$ has full range for one, hence every sufficiently large, positive $\lambda$.
The Lumer--Phillips theorem \cite{LumerPhillips1961} shows that $\fullLin-C$ is the generator of a contraction semigroup.
Equivalently, $\fullLin$ generates a $C_0$-semigroup with growth at most $e^{Cs}$.
\end{proof}

To isolate finitely many eigenvalues we also need to bound the resolvent as $|\Im\lambda|\to\infty$ in each half-plane $\Re\lambda\ge-\mu/2+\delta$, $\delta>0$.
We combine coercivity on the ends with compactness on each fixed bounded interval.

For a function $u$ supported in $E_R=\{|\zeta|>R\}$, the exact transport calculation gives
\begin{align*}
 \Re\int \mu(u-\zeta u')\overline u\varpi\,\dd\zeta
 &=\mu\int\left(1+\frac{(\zeta\varpi)'}{2\varpi}\right)|u|^2\varpi\,\dd\zeta\\
 &=\int\left(-\frac\mu2+O(R^{-2})\right)|u|^2\varpi\,\dd\zeta.
\end{align*}
After $j$ derivatives the commutator with the drift changes the limiting coefficient by $-j\mu$, which only improves the estimate.
More explicitly, if $u_j=\partial_\zeta^ju$, then
\[
 \partial_\zeta^j\{\mu(u-\zeta u')\}
 =\mu\{(1-j)u_j-\zeta u_j'\},
\]
so $\Re\lambda\geq-\mu/2+\delta$ leaves a gap at least $\delta+j\mu-O(R^{-2})$.
Twice integrating the principal term by parts retains a positive multiple of $\sum_{j=0}^k\|u_j''\|_{H^0_2}^2$.
Every remaining coefficient in this energy calculation is $O(R^{-1})$ or smaller: this follows from the logarithmic derivatives of $\varpi$ and \eqref{eq:fredholm-detail-a4-decay}--\eqref{eq:fredholm-detail-lower-decay}.
Weighted interpolation and Cauchy's inequality absorb these terms, first into the principal term and then into the transport gap.
Thus, for every $\delta>0$, one may choose $R$ and positive $c_\delta,c'_\delta$ so that
\begin{equation}\label{eq:fredholm-detail-end-coercivity}
 \Re\langle(\lambda-\fullLin)u,u\rangle_{\evenX^k}
 \geq c_\delta\|u\|_{H^k_2}^2
       +c'_\delta\|u\|_{H^{k+2}_2}^2
\end{equation}
for all $u\in\evenDom^k$ supported in $E_R$ and all such $\lambda$.
There are no finite-end boundary terms for these globally defined functions.
The full variable principal coefficient is included throughout; no compactness assertion about $(a_4-a_\infty)\partial_\zeta^4$ is used.

\begin{lemma}[Uniform vertical resolvent]
\label{lem:fredholm-vertical-resolvent}
For every $\delta>0$ there are $C_\delta,R_\delta<\infty$ such that
\begin{equation}\label{eq:fredholm-vertical-resolvent}
 \|v\|_{\evenX^k}
 \leq C_\delta
 \|(\lambda-\fullLin)v\|_{\evenX^k}
\end{equation}
whenever
\[
 \Re\lambda\geq-\frac\mu2+\delta,
 \qquad |\Im\lambda|\geq R_\delta.
\]
\end{lemma}

\begin{proof}
If the estimate failed for arbitrarily large $|\Im\lambda|$, there would be $v_n\in\evenDom^k$ and $\lambda_n$ satisfying
\[
 \|v_n\|_{H^k_2}=1,\qquad
 \Re\lambda_n\geq-\mu/2+\delta,\qquad
 |\Im\lambda_n|\longrightarrow\infty,\qquad
 f_n:=(\lambda_n-\fullLin)v_n\longrightarrow0
 \quad\hbox{in }H^k_2.
\]
The global G\aa rding estimate \eqref{eq:fredholm-garding} gives
\[
 \Re\lambda_n+c\|v_n\|_{H^{k+2}_2}^2
 \leq C+\Re\langle f_n,v_n\rangle_{H^k_2}.
\]
Hence $(v_n)$ is bounded in $H^{k+2}_2$ and $(\Re\lambda_n)$ is bounded above.
On each fixed compact interval the differential expression $\fullLin$ maps $H^{k+2}$ boundedly into $H^{k-2}$.
Dividing $\lambda_nv_n=\fullLin v_n+f_n$ by $\lambda_n$ therefore gives $v_n\to0$ in $H^{k-2}_{\rm loc}$.
Interpolation with the bounded $H^{k+2}_{\rm loc}$ norm yields
\begin{equation}\label{eq:fredholm-detail-vertical-local-decay}
 v_n\longrightarrow0\quad\hbox{in }H^{k+1}_{\rm loc}.
\end{equation}
The same argument applies at $k=0$, with the usual local negative-order Sobolev norm in the first step.

Fix the $R$ supplied by \eqref{eq:fredholm-detail-end-coercivity} and a single smooth even cutoff $\eta$, zero on $[-R,R]$ and one outside $[-2R,2R]$.
For $u_n=\eta v_n$,
\[
 (\lambda_n-\fullLin)u_n
 =\eta f_n-[\fullLin,\eta]v_n.
\]
The commutator has order at most three and is supported in a fixed bounded annulus $A$ containing the transition of $\eta$.
In its $H^k_2$ energy pairing, integrate one derivative from each top-order $\partial_\zeta^{k+3}v_n$ term onto the other factor.
The product rule and boundedness of the coefficients on $A$ give
\begin{equation}\label{eq:fredholm-detail-vertical-commutator}
 \bigl|\langle[\fullLin,\eta]v_n,\eta v_n\rangle_{H^k_2}\bigr|
 \leq C_R\|v_n\|_{H^{k+2}(A)}\|v_n\|_{H^{k+1}(A)}
 \longrightarrow0.
\end{equation}
Estimating the commutator in this pairing allows us to use the local convergence in $H^{k+1}$ from \eqref{eq:fredholm-detail-vertical-local-decay}.
Also $\langle\eta f_n,u_n\rangle_{H^k_2}\to0$.
The end coercivity estimate \eqref{eq:fredholm-detail-end-coercivity} now gives $\|\eta v_n\|_{H^k_2}\to0$, whereas \eqref{eq:fredholm-detail-vertical-local-decay} gives $\|(1-\eta)v_n\|_{H^k_2}\to0$.
This contradicts the normalisation and proves \eqref{eq:fredholm-vertical-resolvent}.
The index-zero conclusion of \cref{lem:fredholm-global-estimate} then makes $\lambda-\fullLin$ invertible throughout this vertical region, with the asserted uniform inverse bound.
\end{proof}

\subsection*{The essential spectrum, normalisation and adjoint}

Write $\sigma_{\rm ess}(\fullLin)$ for the Fredholm essential spectrum, consisting of the $\lambda\in\C$ for which $\lambda-\fullLin:\evenDom^k\to\evenX^k$ is not Fredholm, and set
\[
 s_{\rm ess}(\fullLin)=\sup\{\Re\lambda:\lambda\in\sigma_{\rm ess}(\fullLin)\}.
\]
\Cref{lem:fredholm-global-estimate} gives the upper bound $s_{\rm ess}(\fullLin)\leq-\mu/2$.
We construct a singular sequence at each point of the boundary line.
Fix $\tau\in\R$ and choose a nonzero $\chi\in C_c^\infty((0,1))$.
Let $y_n=n^2$, $L_n=n$, and set
\[
 w_n(y)=L_n^{-1/2}\chi\left(\frac{y-y_n}{L_n}\right)
        e^{-i\tau y/\mu}.
\]
On the positive end put $v_n=J^{-1}w_n$, with $J$ as in \eqref{eq:fredholm-detail-unitary-end-map}.
Extend $v_n$ by zero towards the origin and reflect it evenly to the negative end.
We normalise each $v_n$ to have $\|v_n\|_{\evenX^k}=1$.
The normalising factors are bounded above and away from zero since $J$ is unitary in $L^2$, $\|w_n\|_{L^2}=\|\chi\|_{L^2}$, and each $\zeta$-derivative contributes a factor $e^{-y}$ on the support.

We estimate $(\fullLin+\mu/2-i\tau)v_n$ first for the unnormalised functions.
For the limiting transport operator $\mu(-1/2-\partial_y)$ in \eqref{eq:fredholm-detail-exact-log-transport}, subtracting the spectral value $-\mu/2+i\tau$ gives
\[
 (-\mu\partial_y-i\tau)w_n
 =-\mu L_n^{-3/2}
   \chi'\left(\frac{y-y_n}{L_n}\right)e^{-i\tau y/\mu},
\]
whose $L^2(\dd y)$ norm is $O(L_n^{-1})$.
The exact-conjugation error in \eqref{eq:fredholm-detail-exact-log-transport} is $O(e^{-2y_n})$, and \eqref{eq:fredholm-detail-log-fourth-order} together with the coefficient estimates \eqref{eq:fredholm-detail-a4-decay}--\eqref{eq:fredholm-detail-lower-decay} bounds the fourth-order and all lower geometric errors by
\[
 C e^{-4y_n}
 \sum_{r=0}^4\|\partial_y^rw_n\|_{L^2}.
\]
The sum is uniformly bounded because the oscillatory frequency $\tau/\mu$ is fixed and derivatives of the cutoff gain powers of $L_n^{-1}$.
The same bounds hold after $k$ ordinary $\zeta$-derivatives; each such derivative adds a factor $e^{-y}$ and a bounded polynomial in $\partial_y$.
Therefore
\[
 \|\{\fullLin+\mu/2-i\tau\}v_n\|_{\evenX^k}
 \longrightarrow0.
\]
The sequence is bounded in $\evenDom^k$: in logarithmic variables $\zeta v_n'$ is a fixed first-order expression in $w_n$, while the ordinary higher derivatives gain powers of $e^{-y_n}$.
Its supports escape to infinity, so it converges weakly to zero both in the phase space and in the graph space.

If $-\mu/2+i\tau-\fullLin$ were Fredholm, its graph-domain quotient by the finite-dimensional kernel would satisfy
\[
 \dist_{\evenDom^k}
   (v,\ker(-\mu/2+i\tau-\fullLin))
 \leq C\|(-\mu/2+i\tau-\fullLin)v\|_{\evenX^k}.
\]
The residual tends to zero, so the displayed estimate would give $k_n$ in the kernel with $\|v_n-k_n\|_{\evenDom^k}\to0$.
Since $v_n\rightharpoonup0$, the same holds for $k_n$.
On the finite-dimensional kernel this convergence is strong, contradicting $\|v_n\|_{\evenX^k}=1$.
Thus
\begin{equation}\label{eq:fredholm-exact-essential-boundary}
 -\frac\mu2+i\R\subset\sigma_{\rm ess}(\fullLin),
 \qquad s_{\rm ess}(\fullLin)=-\frac\mu2.
\end{equation}

It remains to impose the dynamic neck normalisation.
Since $k\geq2$, the trace $v\mapsto v(0)$ is continuous on $\evenX^k$.
The differentiated expansion \eqref{eq:general-profile-tail} gives $DU=4c_1|\zeta|^{-3}+O(|\zeta|^{-7})$ and hence $DU\in\evenDom^k$.
By \eqref{eq:pincher-symmetry-modes}, $\fullLin DU=4\mu DU$.
Since $DU(0)=1$, the projection $\Pi$ of \eqref{eq:projected-pincher-operator} gives topological direct sums
\[
 \evenX^k=\waistX^k_2\oplus\Span\{DU\},
 \qquad
 \evenDom^k=\waistDom_{2,k}\oplus\Span\{DU\}.
\]
For $v\in\waistDom_{2,k}$, the functional $\ell$ from \eqref{eq:projected-pincher-operator} satisfies $\ell(v)=(\fullLin v)(0)$, since $v(0)=0$.
It is continuous in the graph norm, and
\[
 \fullLin v=\Pi\fullLin v+\ell(v)DU=\waistLin v+\ell(v)DU.
\]
Relative to the displayed decompositions,
\begin{equation}\label{eq:fredholm-triangular-quotient}
 \fullLin=
 \begin{pmatrix}
  \waistLin&0\\ \ell&4\mu
 \end{pmatrix}.
\end{equation}

The line $\Span\{DU\}$ is invariant under the semigroup of $\fullLin$.
Hence the induced quotient semigroup, identified with $\waistX^k_2$ by $\Pi$, is
\[
 e^{s\waistLin}v=\Pi e^{s\fullLin}v,
\]
and its generator has domain $\waistDom_{2,k}$.
The matrix \eqref{eq:fredholm-triangular-quotient} differs from the block diagonal operator $\waistLin\oplus4\mu$ only by a finite-rank graph-bounded map.
It follows that $\lambda-\waistLin$ has the same Fredholm index and the same Fredholm essential boundary as $\lambda-\fullLin$, including at $\lambda=4\mu$.

The localisation argument in \cref{lem:controlled-mode-localisation} uses dual representatives in the weighted $L^2$ space.
At $k=0$ a point trace is not defined on the phase space, so one must use the quotient itself rather than the cross-section of functions vanishing at zero.
The end estimates of \cref{lem:fredholm-end-dichotomy}, the index calculation in \cref{lem:fredholm-global-estimate} and the generation proof in \cref{lem:fredholm-generation} apply at $k=0$.
The two boundary traces in \eqref{eq:fredholm-detail-even-boundary} are taken on the half-line domain $D_+^0$, whose elements belong locally to $H^4$ up to zero.
We therefore realise $\fullLin$ on $\evenX^0$ with domain $\evenDom^0$ and put
\[
 \widehat X^0
 =\evenX^0/\Span\{DU\},
 \qquad
 \pi_{\rm q}:\evenX^0\longrightarrow\widehat X^0.
\]
Since $\fullLin DU=4\mu DU$, the formula
\[
 \quotLin\,\pi_{\rm q}v=\pi_{\rm q}\fullLin v,
 \qquad
 \opDom(\quotLin)=\pi_{\rm q}\evenDom^0,
\]
defines the quotient generator.
Under the weighted Riesz pairing, its dual space is the annihilator
\begin{equation}\label{eq:fredholm-detail-quotient-annihilator}
 \dualSpace
 =\left\{\phi^*\in\evenX^0:
   \langle DU,\phi^*\rangle_{H^0_2}=0\right\}.
\end{equation}

The weighted formal adjoint is
\begin{align}
 \fullLin^\dagger\phi^*
 =\varpi^{-1}\bigl[&-\partial_\zeta^4(a_4\varpi\phi^*)
 -\partial_\zeta^3(a_3\varpi\phi^*)
 +\partial_\zeta^2(a_2\varpi\phi^*)
 \nonumber\\
 &-\partial_\zeta\{(a_1-\mu\zeta)\varpi\phi^*\}
 +(a_0+\mu)\varpi\phi^*\bigr].
 \label{eq:fredholm-detail-weighted-formal-adjoint}
\end{align}
The domain of $\quotLin^{\,*}$ is defined intrinsically by
\begin{equation}\label{eq:fredholm-detail-quotient-adjoint-domain}
 \opDom(\quotLin^{\,*})
 =\left\{\phi^*\in\dualSpace:
  \text{there is }h\in\dualSpace\text{ with }
  \langle\fullLin v,\phi^*\rangle_{H^0_2}
  =\langle v,h\rangle_{H^0_2}
  \text{ for every }v\in\evenDom^0\right\},
\end{equation}
and then $\quotLin^{\,*}\phi^*=h$.
In particular, we can test this identity against any compactly supported smooth even function.
Since $\fullLin$ preserves parity and $\phi^*,h,\varpi$ are even, replacing an arbitrary test function by its even part leaves both pairings unchanged.
Thus \eqref{eq:fredholm-detail-weighted-formal-adjoint} gives $\fullLin^\dagger\phi^*=h$ in distributions on the whole line, including at the centre.

\begin{lemma}[Decay of controlled adjoint chains]
\label{lem:fredholm-detail-adjoint-chain-decay}
Let $\Re\lambda>-\mu/2$, and let $\phi_0^*,\ldots,\phi_r^*\in\dualSpace$ be a generalised eigenchain of the $k=0$ quotient adjoint:
\[
 (\quotLin^{\,*}-\overline\lambda)\phi_0^*=0,
 \qquad
 (\quotLin^{\,*}-\overline\lambda)\phi_j^*=\phi_{j-1}^*
 \quad(1\leq j\leq r).
\]
Then every member is smooth on $\R$ and, for some $c_*>0$,
\begin{equation}\label{eq:fredholm-detail-adjoint-chain-exponential}
 |\partial_\zeta^\ell\phi_j^*(\zeta)|
 \leq C_{j,\ell}e^{-c_*|\zeta|^{4/3}}
 \qquad(|\zeta|\geq1)
\end{equation}
for all $j$ and $\ell$.
\end{lemma}

\begin{proof}
Put $\phi_{-1}^*=0$.
The distributional chain equation is $(\fullLin^\dagger-\overline\lambda)\phi_j^*=\phi_{j-1}^*$ on $\R$.
This fourth-order operator has smooth coefficients and principal coefficient $-a_4$, bounded away from zero by \cref{def:conical-linearisation}.
Interior elliptic regularity \cite[Corollary~8.3.2]{Hormander2003} gives $\phi_0^*\in C^\infty(\R)$ and then the same conclusion successively for each $\phi_j^*$.

On the positive end write \eqref{eq:fredholm-wkb-normal-form} as $Z_x=CZ+bf$, where $b=x^{-1}g+O(x^{-2})$, and write $\omega_x=\dd\omega/\dd x$ in \eqref{eq:fredholm-detail-x-measure}.
If $\Phi$ is a fundamental matrix of the homogeneous system $Z_x=CZ$, the columns $Y^\vee$ of $(\Phi^{-1})^*$ solve $(Y^\vee)_x=-C^*Y^\vee$.
The identity $((Y^\vee)^*Z)_x=(Y^\vee)^*bf$ identifies the scalar adjoint modes in the measure $\omega_x\,\dd x$ as $\phi^*=b^*Y^\vee/\omega_x$.
Every entry of $g$ is nonzero by \eqref{eq:fredholm-detail-g}, so these four modes retain the distinct leading phases of the four dual columns.
For the algebraic column, $Y^\vee=x^{-3(1-\overline\lambda/\mu)/4}(e_0+o(1))$ by \eqref{eq:fredholm-detail-D}.
Since $\omega_x\sim(3/4)x^{-13/4}$ and $x=\zeta^{4/3}$, its scalar mode is asymptotic to a nonzero multiple of $\zeta^{2+\overline\lambda/\mu}$.
Its weighted square is integrable only for $\Re\lambda<-\mu/2$.
Membership in $H^0_2$ thus excludes both this mode and the growing exponential, leaving only exponentially decaying modes for $\phi_0^*$.

For the induction write $v=\mathsf r(x)Z$, where $\mathsf r$ is the bounded first row of $V(I+T_1/x)$ from \eqref{eq:fredholm-detail-Y-scaling}--\eqref{eq:fredholm-wkb-normal-form}.
The same pairing gives the dual equation
\[
 (Y^\vee)_x=-C^*Y^\vee+\mathsf r^*\omega_x\phi_{j-1}^*,\qquad \phi_j^*=b^*Y^\vee/\omega_x .
\]
If $\phi_{j-1}^*$ decays exponentially in $x$, variation of constants gives an exponentially decaying particular solution by integrating the algebraic and growing components from infinity and the decaying components from a fixed base.
The polynomial factors in the kernels are absorbed by choosing a smaller positive decay rate.
The $H^0_2$ condition again excludes the inadmissible homogeneous modes.
Differentiating the system gives the same estimate up to polynomial factors for all spatial derivatives by \eqref{eq:fredholm-detail-RG2}.
Reflection gives \eqref{eq:fredholm-detail-adjoint-chain-exponential} on the other end.
\end{proof}

Consequently, writing $N=\dim\Ec$, one may choose smooth primal representatives $\phi_j\in\waistX^k_2$ and biorthogonal dual vectors $\phi_j^*\in\dualSpace$ such that
\[
 \Pc f=\sum_{j=1}^{N}\phi_j
   \langle f,\phi_j^*\rangle_{H^0_2},
 \qquad f\in\waistX^k_2.
\]
The dual factors satisfy \eqref{eq:fredholm-detail-adjoint-chain-exponential}.
Since the primal factors are smooth and finite in number, this representation is compatible across the weighted Sobolev scale and is the representation used in the controlled-mode localisation estimate \eqref{eq:controlled-projected-quadratic}.

The resolvent estimate \eqref{eq:fredholm-vertical-resolvent} also descends to the quotient.
For $v\in\waistDom_{2,k}$ and $\lambda\ne4\mu$, put
\[
 w=v+\frac{\ell(v)}{\lambda-4\mu}DU.
\]
Then
\[
 (\lambda-\fullLin)w=(\lambda-\waistLin)v.
\]
Applying \eqref{eq:fredholm-vertical-resolvent} to $w$ and then applying the projection $\Pi$ of \eqref{eq:projected-pincher-operator} gives
\begin{equation}\label{eq:fredholm-quotient-vertical-resolvent}
 \|v\|_{\waistX^k_2}
 \leq C_\delta\|(
 \lambda-\waistLin)v\|_{\waistX^k_2}
\end{equation}
for the same vertical region after increasing $R_\delta$.

This completes the operator estimates for \cref{lem:weighted-fredholm-boundary}.
The proof of \cref{prop:finite-core-splitting} uses them to construct the two spectral subspaces and obtain decay on $\Es$.

\section{Convective terminal estimates}
\label{app:convective-terminal-estimates}

This appendix proves the estimates used in \cref{prop:convective-maximal-regularity,prop:terminal-core-semiflow,lem:controlled-mode-localisation}.
Let $U$ be a fixed normalised positive conical profile for $\flowSpeed$, satisfying \cref{def:conical-linearisation}.
On the conical ends the dilation term dominates at infinity, so the usual sectorial maximal regularity theorem does not apply to $\waistLin$.
We straighten the dilation on each unit time interval, prove a uniformly parabolic estimate there, and use the discrete Green operator to join the solutions.
We also use uniformly local norms to estimate perturbations that change the cone slope.

We use the time norm \eqref{eq:similarity-time-bochner-norm} throughout this appendix.
In $L^2_s H^j_2$ the spatial norm is \eqref{eq:pincher-polynomial-space} with $m=2$ and density $\varpi$.
We take $k\geq8$ for the linear estimates in the Hilbert and conical norms.
Constants may depend on $U,\mu,k$ and on a fixed spectral gap, but not on the unit interval $I_n=[n,n+1]$ or on the moving conical radius.

\subsection{The weighted half-line inverse}

The dilation pullback $\dilPull$ in \eqref{eq:dilation-straightening} has inverse
\begin{equation}\label{eq:app-dilation-pullback}
 (\dilPull^{-1}w)(\zeta)
 :=e^{\mu\theta}w(e^{-\mu\theta}\zeta).
\end{equation}

\begin{lemma}[Uniform dilation equivalence]
\label{lem:app-dilation-equivalence}
For every integer $j\geq0$ there is $C_j\geq1$ such that
\begin{equation}\label{eq:app-dilation-equivalence}
 C_j^{-1}\|u\|_{H^j_2}
 \leq \|\dilPull u\|_{H^j_2}
 \leq C_j\|u\|_{H^j_2},
 \qquad 0\leq\theta\leq1.
\end{equation}
The same statement holds for $\dilPull^{-1}$, for the even subspaces, and for the subspaces with zero value at the centre.
If $v=v(\zeta,n+\theta)$, then
\begin{equation}\label{eq:app-material-pullback}
 \partial_\theta\{\dilPull v(n+\theta)\}
 =\dilPull\matD v(n+\theta).
\end{equation}
Consequently the corresponding $L^2$-in-time norms in $H^{k+2}_2$, $H^{k-2}_2$ and $H^k_2$ are uniformly equivalent before and after straightening.
\end{lemma}

\begin{proof}
The identity $\partial_y^r\dilPull u =e^{(r-1)\mu\theta}u^{(r)}(e^{\mu\theta}y)$ and the uniform equivalence of $\langle e^{-\mu\theta}x\rangle^{-4}$ and $\langle x\rangle^{-4}$ prove \eqref{eq:app-dilation-equivalence} by a change of variables.
The inverse, parity and trace assertions follow immediately.
Differentiating \eqref{eq:dilation-straightening} proves \eqref{eq:app-material-pullback}.
\end{proof}

For the unnormalised problem write $\evenX_k=\evenX^{k-2}$, $\evenTr=\evenX^k$ and $\evenY_k=\evenX^{k+2}$ for the full even spaces of \eqref{eq:fredholm-even-spaces}.
The spaces $X_k,\waistTr,Y_k$ of \eqref{eq:convective-trace-space} are the respective kernels of evaluation at the centre.

\begin{lemma}[Interpolation and time trace]
\label{lem:app-interpolation-trace}
The interpolation identities on the full spaces and on the subspaces with zero value at the centre
\begin{equation}\label{eq:app-interpolation-identity}
 (\evenX_k,\evenY_k)_{1/2,2}
 =\evenTr,
 \qquad
 (X_k,Y_k)_{1/2,2}=\waistTr
\end{equation}
hold with equivalent norms.
Moreover,
\begin{equation}\label{eq:app-time-trace-map}
 W^{1,2}(0,1;X_k)\cap L^2(0,1;Y_k)
 \hookrightarrow C([0,1];\waistTr),
\end{equation}
and the same embedding holds with tildes on all three spaces.
Its norm is unchanged by translation of the unit interval.
\end{lemma}

\begin{proof}
Write $Wu=\varpi^{1/2}u=\langle y\rangle^{-2}u$.
Since all positive-order derivatives of $\log\varpi$ are bounded, Leibniz' rule gives $\|Wu\|_{H^j}\asymp\|u\|_{H^j_2}$.
The full-space interpolation identity follows from the unweighted Sobolev identity.
Evenness and the condition of vanishing at the centre are imposed by bounded projections on all three levels; for the latter use $u\mapsto u-u(0)\vartheta$, where $\vartheta$ is a fixed even smooth core bump with $\vartheta(0)=1$.
Exact interpolation of complemented subspaces gives the second identity.
For the time trace, reflect across $\theta=1$ and cut off smoothly before $\theta=2$.
The unweighted time-trace theorem \cite[Proposition~3.1(b)]{PrussSimonett2004} on $\R_+$ then gives \eqref{eq:app-time-trace-map}.
\end{proof}

We first solve the equation on a single straightened time interval.
For the coefficients $a_j$ of $\radLin$ we use the ellipticity and symbol bounds of \cref{def:conical-linearisation}.
After straightening, set
\begin{equation}\label{eq:app-straightened-operator}
 \slabLin
 :=\dilPull\radLin\dilPull^{-1}.
\end{equation}
Its principal term is
\begin{equation}\label{eq:app-straightened-principal}
 -a_{4,\theta}(y)\partial_y^4,
 \qquad
 a_{4,\theta}(y)
 =e^{-4\mu\theta}a_4(e^{\mu\theta}y).
\end{equation}
In particular,
\begin{equation}\label{eq:app-uniform-parabolicity}
 e^{-4\mu}a_*\leq a_{4,\theta}(y)\leq a^*.
\end{equation}
The remaining coefficients are obtained in the same way and have uniformly bounded derivatives in the weighted multiplier classes required below.

\begin{lemma}[Weighted fourth-order estimate]
\label{lem:app-straightened-fourth-order}
Let $F\in L^2(0,1;\evenX_k)$ and $w_0\in\evenTr$.
The unnormalised problem
\begin{equation}\label{eq:app-unnormalised-straightened-problem}
 w_\theta-\slabLin w=F,
 \qquad w(0)=w_0,
\end{equation}
has a unique solution in
\[
 W^{1,2}(0,1;\evenX_k)\cap L^2(0,1;\evenY_k)
 \cap C([0,1];\evenTr),
\]
and
\begin{align}
 &\|w\|_{C\evenTr}
 +\|w\|_{L^2\evenY_k}
 +\|w_\theta\|_{L^2\evenX_k}
 \nonumber\\
 &\hspace{7em}\leq
 C\bigl(\|w_0\|_{\evenTr}
       +\|F\|_{L^2\evenX_k}\bigr).
 \label{eq:app-straightened-fourth-order-estimate}
\end{align}
The same constant applies on every translated unit interval in the coordinates \eqref{eq:dilation-straightening}.
\end{lemma}

\begin{proof}
Put $h=Ww$ and $L_\theta=-W\slabLin W^{-1}$, with $W$ as in the proof of \cref{lem:app-interpolation-trace}.
Then $h_\theta+L_\theta h=WF$.
Conjugation replaces $\partial_y$ by $\partial_y-\tfrac12\partial_y\log\varpi$, so the lower coefficients and their spatial derivatives remain bounded.
The leading coefficient $a_{4,\theta}$ satisfies \eqref{eq:app-uniform-parabolicity} and is uniformly Lipschitz in $y$.
Its mean oscillation on each interval of radius $r$ is therefore at most $Cr$, uniformly in $\theta$.

We may therefore apply the whole-space estimate of Dong and Kim in dimension one with $p=2$ and order four \cite[Theorem~2.3 and Remark~2.4]{DongKim2011}.
On $0<\theta<1$ it gives unique solvability with zero initial data and controls $h_\theta$ and $\partial_y^j h$, $0\le j\le4$, in $L^2$ by the $L^2$ norm of the source.
Their exponential shift removes the spectral parameter on this fixed interval.
For initial data $h_0=Ww_0$, subtract $g(\theta)=e^{-\theta(1+\partial_y^4)}h_0$.
By Plancherel,
\[
 \|g_\theta\|_{L^2H^{k-2}}+\|g\|_{L^2H^{k+2}}
 \le C\|h_0\|_{H^k},
 \qquad \int_0^\infty e^{-2\theta(1+\xi^4)}\,\dd\theta=\frac{1}{2(1+\xi^4)} .
\]
Thus $z=h-g$ has zero initial trace and source $f=WF-g_\theta-L_\theta g$ bounded in $L^2H^{k-2}$ by the data.
To gain the remaining derivatives, apply the same estimate to $\partial_y^j z$ for $1\le j\le k-2$.
Its source is $\partial_y^j f-[\partial_y^j,L_\theta]z$, with
\[
 \|[\partial_y^j,L_\theta]z\|_{L^2L^2}\le C_j\|z\|_{L^2H^{j+3}} .
\]
The preceding step controls this norm, since at least one derivative in the commutator falls on a coefficient.
Difference quotients give the same induction for Sobolev data.
Returning by $W^{-1}$ and applying \eqref{eq:app-time-trace-map} proves \eqref{eq:app-straightened-fourth-order-estimate}.
Reflection preserves the equation, so uniqueness preserves evenness.
The coefficients as functions of $(y,\theta)$ are independent of $n$, giving the same constant on every slab.
\end{proof}

To impose the condition at the centre in these coordinates, let $q_\theta=\dilPull DU$ and define
\begin{equation}\label{eq:app-waist-functional}
 \ell_\theta(w):=\ell(\dilPull^{-1}w),
 \qquad
 \waistOp w=\ell_\theta(w)q_\theta,
\end{equation}
where $\ell$ is the functional \eqref{eq:projected-pincher-operator}.
The straightened normalised equation is
\begin{equation}\label{eq:app-normalised-straightened-problem}
 w_\theta-\slabLin w+\waistOp w=F.
\end{equation}

\begin{lemma}[Uniform normalising perturbation]
\label{lem:app-waist-rank-one}
For $0\leq\theta\leq1$ the family $\waistOp$ is rank one and bounded from $\evenY_k$ to every $H^r_2$ with $r\leq k$.
More precisely,
\begin{align}
 |\ell_\theta(w)|&\leq C\|w\|_{H^{5}(-2,2)},
 \label{eq:app-waist-trace-bound}\\
 \|\waistOp w\|_{\evenX_k}
 &\leq \varepsilon\|w\|_{\evenY_k}
       +C_\varepsilon\|w\|_{\evenTr}.
 \label{eq:app-waist-relative-bound}
\end{align}
If $w_0\in\waistTr$ and $F\in L^2(0,1;X_k)$, equation \eqref{eq:app-normalised-straightened-problem} has a unique solution in $W^{1,2}X_k\cap L^2Y_k\cap C\waistTr$ satisfying \eqref{eq:app-straightened-fourth-order-estimate}.
In particular, $w(0,\theta)=0$ for every $\theta$.
\end{lemma}

\begin{proof}
Formula \eqref{eq:fredholm-coefficient-form} shows that $(\radLin u)(0)$ is a linear combination of $u(0),\ldots,u^{(4)}(0)$ with smooth coefficients depending on the fixed profile.
One-dimensional trace theory gives
\[
 |(\radLin u)(0)|\leq C\|u\|_{H^5(-1,1)}.
\]
The dilation equivalence \eqref{eq:app-dilation-equivalence} transfers this to \eqref{eq:app-waist-trace-bound}.
The differentiated conical tail \eqref{eq:general-profile-tail} gives $DU=O(\langle\zeta\rangle^{-3})$ with all differentiated analogues, and hence $q_\theta$ is uniformly bounded in every weighted Sobolev space.
Since $k-2\ge6$, the local $H^5$ norm is bounded by $\|w\|_{\evenX_k}$.
Hence $\|\waistOp w\|_{\evenX_k}\le C\|w\|_{\evenX_k}$, which implies \eqref{eq:app-waist-relative-bound}.

On a subinterval $I\subset[0,1]$ let $S_I$ solve \eqref{eq:app-unnormalised-straightened-problem} with zero initial value.
Extending the source by zero on the rest of the slab gives $\|\partial_\theta S_I f\|_{L^2(I;\evenX_k)}\le C\|f\|_{L^2(I;\evenX_k)}$ by \eqref{eq:app-straightened-fourth-order-estimate} and uniqueness.
Integrating from the initial endpoint gives $\|\waistOp S_I f\|_{L^2(I;\evenX_k)}\le C|I|\|f\|_{L^2(I;\evenX_k)}$.
Thus $\mathrm{Id}+\waistOp S_I$ is invertible by a Neumann series on intervals with $C|I|\le1/2$.
Subtracting the unperturbed solution for the prescribed initial trace and applying this inverse solves \eqref{eq:app-normalised-straightened-problem} on each such interval.
The trace lift in \cref{lem:app-straightened-fourth-order} works from any initial endpoint, so a fixed finite number of steps gives the full estimate on $[0,1]$.
Evaluating \eqref{eq:convective-linear-straightened} at the origin gives
\[
 \partial_s v(0,s)-\mu v(0,s)
 -(\radLin v)(0,s)+(\radLin v)(0,s)DU(0)=f(0,s).
\]
Since $DU(0)=1$ and $f(0,s)=0$, the value at the centre solves $\partial_sv(0,s)=\mu v(0,s)$; zero initial trace is therefore invariant.
Restriction of the full estimate to this invariant complemented subspace gives the stated estimate.
This restriction is the evolution form of the triangular quotient identity \eqref{eq:fredholm-triangular-quotient}.
\end{proof}

Combining the preceding three lemmas and returning by \cref{lem:app-dilation-equivalence} proves the unit-interval part of \cref{prop:convective-maximal-regularity}.
We next make its half-line construction explicit.

The time-one map $e^{\waistLin}$ on $\waistTr$ commutes with both spectral projections.
For the restrictions \eqref{eq:controlled-stable-restrictions} write $S_{\rm s}=e^{\waistLin_{\rm s}}$ and $S_{\rm c}=e^{\waistLin_{\rm c}}$.
Fix
\begin{equation}\label{eq:app-halfline-gap-choice}
 \gamma_{\rm c}<\eta<\omega.
\end{equation}
Apply \eqref{eq:pincher-stable-semigroup} to the stable restriction.
On the controlled space, the spectral bound \eqref{eq:controlled-backward-rate} allows the polynomial factors from the Jordan blocks to be absorbed by increasing the backward growth rate.
This gives
\begin{equation}\label{eq:app-discrete-dichotomy}
 \|S_{\rm s}^r\|\leq C e^{-\omega_1r},
 \qquad
 \|S_{\rm c}^{-r}\|\leq C e^{\gamma_1r},
 \qquad r\in\N,
\end{equation}
for some $\eta<\omega_1<\omega$ and $\gamma_{\rm c}<\gamma_1<\eta$.
The strict inequality $\eta<\omega_1$ will make the stable sum converge after multiplication by $e^{\eta n}$.

For $f\in L^2_{\rm loc}(\R_+;X_k)$ let $h_n$ be the time-one trace of the solution of the normalised equation on $I_n$ with zero trace at $n$.
By the unit-interval estimate \eqref{eq:convective-maximal-estimate},
\begin{equation}\label{eq:app-forced-trace-bound}
 \|h_n\|_{\waistTr}
 \leq C\|f\|_{L^2(I_n;X_k)}.
\end{equation}

\begin{lemma}[Discrete Green construction]
\label{lem:app-discrete-green}
Suppose that
\[
 \sum_{n=0}^\infty e^{2\eta n}\|h_n\|_{\waistTr}^2<\infty,
 \qquad \xi\in \Ps\waistTr.
\]
There is a unique sequence $(x_n)_{n\geq0}$ satisfying
\begin{equation}\label{eq:app-trace-recurrence}
 x_{n+1}=e^{\waistLin}x_n+h_n,
 \qquad \Ps x_0=\xi,
 \qquad \sup_ne^{\eta n}\|x_n\|_{\waistTr}<\infty.
\end{equation}
It is given by
\begin{align}
 \Ps x_n
 &=S_{\rm s}^n\xi
   +\sum_{j=0}^{n-1}S_{\rm s}^{n-1-j}\Ps h_j,
 \label{eq:app-discrete-green-stable}\\
 \Pc x_n
 &=-\sum_{j=n}^{\infty}S_{\rm c}^{n-1-j}\Pc h_j,
 \label{eq:app-discrete-green-controlled}
\end{align}
and obeys
\begin{equation}\label{eq:app-discrete-green-estimate}
 \sup_ne^{\eta n}\|x_n\|_{\waistTr}
 \leq C_\eta\left(
 \|\xi\|_{\waistTr}
 +\Big\{\sum_{n=0}^\infty
 e^{2\eta n}\|h_n\|_{\waistTr}^2\Big\}^{1/2}
 \right).
\end{equation}
The same right-hand side also bounds $\bigl(\sum_{n=0}^\infty e^{2\eta n}\|x_n\|_{\waistTr}^2\bigr)^{1/2}$.
\end{lemma}

\begin{proof}
By \eqref{eq:app-discrete-dichotomy}, the stable series in \eqref{eq:app-discrete-green-stable} is bounded by convolution with $e^{-(\omega_1-\eta)r}\mathbf1_{r\geq0}$ after multiplication by $e^{\eta n}$.
The controlled series in \eqref{eq:app-discrete-green-controlled} is bounded by convolution with $e^{-(\eta-\gamma_1)r}\mathbf1_{r\geq1}$.
Both kernels belong to $\ell^1$, so Young's inequality proves \eqref{eq:app-discrete-green-estimate} and its $\ell^2$ version.
Direct substitution proves the recurrence.
If $z_n$ is the difference of two solutions, then $\Ps z_n=S_{\rm s}^n\Ps z_0=0$, while $\Pc z_n=S_{\rm c}^{n-N}\Pc z_N$ for $N>n$.
Letting $N\to\infty$ and using the weighted bound together with the second estimate in \eqref{eq:app-discrete-dichotomy} gives $\Pc z_n=0$.
\end{proof}

\begin{proposition}[Convective half-line inverse]
\label{prop:app-convective-halfline-inverse}
Under \eqref{eq:app-halfline-gap-choice}, if $e^{\eta s}f\in L^2(\R_+;X_k)$ and $\xi\in \Ps\waistTr$, there is a unique solution of \eqref{eq:convective-linear-straightened} satisfying \eqref{eq:linear-stable-future-conditions} and
\begin{align}
 &\sup_{s\geq0}e^{\eta s}\|v(s)\|_{\waistTr}
 +\left\{\int_0^\infty e^{2\eta s}
 \bigl(\|v(s)\|_{H^{k+2}_2}^2
 +\|\matD v(s)\|_{H^{k-2}_2}^2\bigr)\,\dd s\right\}^{1/2}
 \nonumber\\
 &\hspace{6em}\leq C_\eta\left(
 \|\xi\|_{\waistTr}
 +\|e^{\eta s}f\|_{L^2(\R_+;X_k)}\right).
 \label{eq:app-continuous-halfline-estimate}
\end{align}
Writing the semigroup integrals in the extrapolation space of $X_k$, we have
\begin{align}
 v(s)={}&e^{s\waistLin_{\rm s}}\xi
 +\int_0^s e^{(s-\tau)\waistLin_{\rm s}}\Ps f(\tau)\,\dd\tau
 \nonumber\\
 &-\int_s^\infty
 e^{(s-\tau)\waistLin_{\rm c}}\Pc f(\tau)\,\dd\tau.
 \label{eq:app-continuous-green-representation}
\end{align}
The unit-interval estimate \eqref{eq:convective-maximal-estimate} gives the $\waistTr$ traces of the integrals.
\end{proposition}

\begin{proof}
Use \eqref{eq:app-forced-trace-bound} and \cref{lem:app-discrete-green} to prescribe the traces $x_n=v(n)$.
On each $I_n$, solve the initial-value problem furnished by \cref{lem:app-straightened-fourth-order,lem:app-waist-rank-one}.
At the right endpoint the solution has trace $e^{\waistLin}x_n+h_n=x_{n+1}$, precisely the initial trace prescribed on $I_{n+1}$.
On $I_n$ we have $1\le e^{\eta s}/e^{\eta n}\le e^\eta$.
Squaring \eqref{eq:convective-maximal-estimate}, summing in $n$, and applying the weighted square-sum bound of \cref{lem:app-discrete-green} proves \eqref{eq:app-continuous-halfline-estimate}.

For smooth compactly supported $f$, variation of constants gives \eqref{eq:app-continuous-green-representation}; the controlled integral is finite-dimensional and converges by \eqref{eq:app-halfline-gap-choice}.
Approximation in the norm $\|e^{\eta s}f\|_{L^2(0,\infty;X_k)}$ and \eqref{eq:app-continuous-halfline-estimate} give the general formula.
At $s=0$ its controlled component is exactly the future condition in \eqref{eq:linear-stable-future-conditions}.
Uniqueness follows either from the discrete uniqueness argument or by applying the weighted estimate to a homogeneous difference.
\end{proof}

\begin{lemma}[Sobolev compatibility of the spectral projections]
\label{lem:app-riesz-scale-compatibility}
Let $1/2<j_1<j_2$ be two Sobolev indices for which the realisations $\waistLin^{(j_i)}$ are defined.
In the Fredholm half-plane their discrete spectra, algebraic multiplicities and Jordan chains agree.
Off that common discrete spectrum,
\begin{equation}\label{eq:app-resolvent-scale-identity}
 (\lambda-\waistLin^{(j_1)})^{-1}f
 =(\lambda-\waistLin^{(j_2)})^{-1}f,
 \qquad f\in\waistX^{j_2}_2.
\end{equation}
Consequently their Riesz projections agree on every common trace space, and
\begin{equation}\label{eq:app-riesz-smoothing}
 \Pc:\waistX^{j}_2\longrightarrow\waistX^{j+r}_2
\end{equation}
is bounded for every finite $r$ and every input order $j$ used here, including $j=k-2$.
The adjoint Riesz representatives and their Jordan chains enjoy the corresponding bootstrap and can be chosen independently of the Sobolev realisation.
\end{lemma}

\begin{proof}
Suppose first that $\lambda$ belongs to the lower-level resolvent set and $f\in\waistX^{j_2}_2$.
On a compact interval the equation bootstraps by four derivatives.
On each end, differentiate the Volterra equations in \cref{lem:fredholm-end-dichotomy}; the coefficient bounds absorb the commutators and give the same weighted gain.
Iteration places the lower-level solution in the $j_2$ graph domain.
Uniqueness proves \eqref{eq:app-resolvent-scale-identity} whenever both sides are defined.

If $u$ is an eigenvector or generalised eigenvector at the lower level, the same local and end bootstrap applied successively to $(\waistLin-\lambda)u_r=u_{r-1}$ makes the entire Jordan chain smooth in all the weighted levels.
Conversely a higher-level chain is already a lower-level chain.
Fredholm index zero then shows that the kernels and cokernels have the same dimensions; hence the discrete spectrum and its algebraic multiplicities are independent of $j$.
Integrating \eqref{eq:app-resolvent-scale-identity} around the common spectral circles proves compatibility of the Riesz projections.

For the adjoint, write the controlled projection in the weighted Hilbert pairing as
\[
 \Pc f=\sum_{r=1}^{N}\phi_r\langle f,\phi_r^*\rangle_{H^0_2},
 \qquad N=\dim\Ec .
\]
The vectors $\phi_r$ are the smooth primal chains just obtained.
The distributional adjoint equations for $\phi_r^*$ bootstrap on the compact core, and the adjoint end dichotomy bootstraps with the dual weight.
Thus every $\phi_r^*$ and its chain act continuously on all input levels used here.
This proves \eqref{eq:app-riesz-smoothing} and the asserted adjoint compatibility.
\end{proof}

This proves \cref{prop:convective-maximal-regularity}.

\subsection{Conical graph spaces and propagation}

We next justify the functional-analytic facts implicit in the norms \eqref{eq:conical-parabolic-norms}.
The graph and source spaces are the completions specified in \cref{def:conical-parabolic-spaces}.

\begin{lemma}[Banach and trace-lift properties]
\label{lem:app-conical-banach-trace}
For every $R_0\geq1$ the spaces $\slabGraph^k_R(I_n)$ and $\slabSrc^k_R(I_n)$ are Banach spaces.
Evaluation at $s=n$ is a bounded map from $\slabGraph^k_R(I_n)$ onto $\slabTr^k_{R_n}$, and the quotient norm in \eqref{eq:conical-parabolic-trace-norm} makes the latter a Banach space.
There is a bounded linear right inverse
\begin{equation}\label{eq:app-conical-trace-lift}
 \traceLift_n:
 \slabTr^k_{R_n}\longrightarrow
 \slabGraph^k_R(I_n),
 \qquad
 (\traceLift_n\xi)(n)=\xi,
\end{equation}
whose norm is independent of $n$ and $R_0\geq1$.
The intersection space $\conePhase^k$ in \eqref{eq:pincher-global-smallness} is also Banach.
\end{lemma}

\begin{proof}
On a fixed box, the maximal-regularity space
\[
 W^{1,2}(0,1;L^2(B_\rho))
 \cap L^2(0,1;H^4(B_\rho))
\]
embeds continuously in $C([0,1];H^2(B_\rho))$.
Hence evaluation at the initial face is continuous in every summand defining $\boxSize_{j,\rho}$, and its kernel is closed.
A Cauchy sequence in the supremum over boxes converges on each box; uniqueness on overlaps pastes the limits.
The weighted supremum bound passes to the limit, proving completeness of $\slabGraph^k_R$.
The same argument with the local $L^2$ norms gives completeness of the source space.
The trace space is Banach since it is the quotient of the graph space by the closed kernel of evaluation at the initial time.

We construct the right inverse by extending the initial height and evolving it under the model equation on the quotient by constants.
The local trace theorem first gives
\begin{equation}\label{eq:app-static-conical-trace-bound}
 \sum_{j=0}^k\sup_{|\rho|\leq2R_n}
 \langle\rho\rangle^{j-1}
 \|\partial_y^j\xi\|_{H^2(B_\rho\cap[-2R_n,2R_n])}
 \leq C\|\xi\|_{\slabTr^k_{R_n}}.
\end{equation}
Indeed this holds for the trace of every admissible lift, and taking the infimum proves the display.
A standard finite-order Sobolev reflection operator at the two outer endpoints, followed by a cutoff on collars whose widths are comparable to $R_n$, defines a linear spatial extension $E_n^{\rm sp}\xi$ to $\R$.
Use even reflection at the centre and subtract its value there against a fixed even core bump; because $\xi(0)=0$, this last operation vanishes on the original interval.
On the outer collars the cutoff may be chosen so that $|\partial_y^r\chi_n|\leq C_rR_n^{-r}$; hence each cutoff derivative is exactly balanced by the corresponding conical weight.
The reflected collar has fixed relative width, so $\langle y\rangle\asymp R_n$ there.
Finite overlap and \eqref{eq:app-static-conical-trace-bound} show that the extension is bounded uniformly for $R_n\geq1$, both in $H^k_2$ and in the displayed uniformly local trace seminorm.

Evolve $E_n^{\rm sp}\xi$ for one slab by the model equation on the quotient by constants, and let $\traceLift_n\xi$ be the restriction of this solution to the moving cylinder.
\Cref{lem:app-model-conical-estimate} derives the whole-line graph estimate \eqref{eq:app-model-whole-line-graph} directly from the kernel and the static trace norm.
Together with \eqref{eq:app-static-conical-trace-bound} this gives
\begin{equation}\label{eq:app-auxiliary-lift-estimate}
 \|\traceLift_n\xi\|_{\slabGraph_R^k(I_n)}
 \leq C\|\xi\|_{\slabTr^k_{R_n}}.
\end{equation}
The evolution is linear and preserves the zero value at the centre, and its initial value on $[-2R_n,2R_n]$ is $\xi$.
The same construction without spatial extension gives the whole-line trace lift.
All constants are uniform in $n$ and $R_0$.

Finally, if $(v_n)$ is Cauchy in $\conePhase^k$, then $v_n$ converges in $H^k_2$ and $\langle\zeta\rangle^{j-1}\partial_\zeta^j v_n$ converges in $L^\infty(\R)$ for each $0\le j\le k$.
The latter limit is the weighted weak derivative of the $H^k_2$ limit, by uniqueness of distributional limits.
This proves completeness.
\end{proof}

The fixed normal chart in \eqref{eq:normal-neck-chart} has the same linearised stationary equation as the radial chart.
We first record the geometric calculation which is needed to compare their nonlinear terms.
We measure the conical size on each slab by $r_n$ of \eqref{eq:neck-slab-conical-size}.

\begin{lemma}[Material equation in a fixed normal chart]\label{lem:app-normal-neck-chart}
Let $U$ be a normalised positive conical profile for a flow in the class of \cref{def:geometric-flow-class}.
The operators in \eqref{eq:normal-neck-remainder} satisfy
\[
 \flowRad^\nu(0)=0,\qquad D\flowRad^\nu(0)=\radLin,
 \qquad \radRem^\nu(0)=D\radRem^\nu(0)=0.
\]
Let $\vartheta:[0,\infty)\to(-\infty,\mu)$ and $v:\R\times[0,\infty)\to\R$ be smooth, with $v(\cdot,s)$ even, zero at the centre and in the normal chart \eqref{eq:normal-neck-chart} for every $s$.
For $A_0>0$, put $A(s)=A_0e^{-\mu s}$ and $\dd t/\dd s=A^4/(1-\vartheta/\mu)>0$.
In physical time $t$, the family $A(s)Y_{v(s)}$ has the normal speed prescribed by \eqref{eq:general-geometric-flow} if and only if
\begin{equation}\label{eq:normal-neck-material-equation}
 (1-\vartheta/\mu)\matD v
 =\flowRad^\nu(v)-\vartheta C_U^\nu(v),
 \qquad \vartheta=\ell(v)+\radRem^\nu(v)(0).
\end{equation}
On a sufficiently small conical coefficient ball the geometric coefficient satisfies, uniformly in $n$ and $R$,
\begin{equation}\label{eq:app-normal-neck-coefficient-bound}
 \begin{aligned}
 \|C_U^\nu(v)-DU\|_{L^2(I_n;H^{k-2}_2)}
 &\le C\|v\|_{L^2(I_n;H^{k-1}_2)},\\
 \|C_U^\nu(v)-DU\|_{\slabSrc_R^k(I_n)}
 &\le C\|v\|_{\slabGraph_R^k(I_n)}.
 \end{aligned}
\end{equation}
For the Hilbert estimate, $v\mapsto C_U^\nu(v)-DU$ is $C^2$ into $L^2(I_n;H^{k-2}_2)$ on the set $r_n(v)<\varepsilon$ of even paths in $L^2(I_n;H^{k-1}_2)$ vanishing at the centre, with domain norm $\|v\|_{L^2(I_n;H^{k-1}_2)}+r_n(v)$ and $\varepsilon>0$ sufficiently small.
For the conical estimate the map is $C^2$ from the open subset of $\slabGraph_R^k(I_n)$ given by the same smallness condition to $\slabSrc_R^k(I_n)$, taking the spatial suprema in $r_n(v)$ over $|\zeta|\le2R(s)$ for the restricted graph norm.
\end{lemma}

\begin{proof}
Write $Y_v=(R_v\cos\theta,R_v\sin\theta,Z_v)$.
We use $S=S[U]$ from \eqref{eq:profile-geometry}.
The formula for the reference normal gives
\begin{equation}\label{eq:app-normal-neck-meridian}
 R_v=U+\frac{v}{S^2},\qquad
 Z_v=\zeta-\frac{U'v}{S^2}.
\end{equation}
For small $\coneSize_k(v)$, $R_v\asymp\langle\zeta\rangle$ and $Z_v'>0$.
The radial perturbation $q$ of the same surface satisfies $q(Z_v)=R_v-U(Z_v)$, whose first variation at zero is
\[
 \frac{v}{S^2}-U'\left(-\frac{U'v}{S^2}\right)=v.
\]
The stationary shrinker vector fields in the two charts are related by this coordinate change.
The radial vector field vanishes at $U$, so its linearisation is conjugated by the identity just calculated.
Consequently the derivative of $\flowRad^\nu(v)+\mu Dv$ is $\fullLin=\radLin+\mu D$; subtracting $\mu D$ proves the asserted linearisation.
At zero the profile equation gives $\flowSpeed[f_U]/\nu_r=-\mu DU$, while $C_U^\nu(0)=DU$, proving the remaining zero identities.

To calculate $C_U^\nu$, put $\kappa_\nu=U''/S^3$ and $T=(U'e_r+e_z)/S$, where $T$ is the unit meridian tangent and $e_z=(0,0,1)$ is the axial unit vector.
Since $\partial_\zeta\nu_U=S\kappa_\nu T$, the meridional tangent and the perturbed normal obey
\[
 \partial_\zeta Y_v=S(1+\nu_r v\kappa_\nu)T+(\nu_r v)'\nu_U,
 \qquad
 \frac{\nu_{Y_v}}{b(v)}
 =\nu_U-\frac{(\nu_r v)'}{S(1+\nu_r v\kappa_\nu)}T.
\]
Taking the scalar product with $Y_v$ in \eqref{eq:normal-neck-chart} yields
\begin{equation}\label{eq:app-normal-neck-dilation}
 \begin{aligned}
 C_U^\nu(v)-DU
 &=\left(\zeta-\frac{\zeta+UU'}{S^2(1+\nu_r v\kappa_\nu)}\right)v'\\
 &\quad-\frac{\zeta+UU'}{S^2(1+\nu_r v\kappa_\nu)}
           \frac{\nu_r'}{\nu_r}v.
 \end{aligned}
\end{equation}
There is no unbounded transport coefficient in this expression: by \eqref{eq:general-profile-tail},
\[
 \zeta-\frac{\zeta+UU'}{S^2}=-\frac{U'DU}{S^2}=O(|\zeta|^{-3}),
 \qquad \frac{\nu_r'}{\nu_r}=-\frac{U'U''}{S^2}=O(|\zeta|^{-5}),
 \qquad \kappa_\nu=O(|\zeta|^{-5}).
\]
These bounds hold after differentiation, with an additional inverse power for each derivative.
Since $|v|\le C\coneSize_k(v)\langle\zeta\rangle$, the two coefficients in \eqref{eq:app-normal-neck-dilation} are respectively $O(\langle\zeta\rangle^{-3})$ and $O(\langle\zeta\rangle^{-4})$, uniformly on the coefficient ball.
Leibniz' rule now gives \eqref{eq:app-normal-neck-coefficient-bound}.

To see where the uniform bound in time enters the variations, write $F(v)=C_U^\nu(v)-DU$ and put
\[
 P_v=\frac{\zeta+UU'}{S^2(1+\nu_r v\kappa_\nu)},\qquad
 m_v=\frac{\nu_r\kappa_\nu}{1+\nu_r v\kappa_\nu}.
\]
Then $F(v)=\zeta v'-(P_v/\nu_r)(\nu_r v)'$ and $DP_v[h]=-P_vm_vh$.
Differentiating twice gives
\[
 D^2F(v)[h,g]=\frac{P_vm_v}{\nu_r}
 \bigl\{h(\nu_r g)'+g(\nu_r h)'-2m_vhg(\nu_r v)'\bigr\}.
\]
By \eqref{eq:general-profile-tail}, $P_vm_v=O(\langle\zeta\rangle^{-4})$ and $m_v=O(\langle\zeta\rangle^{-5})$, with an additional inverse power for each spatial derivative through order $k$ on the coefficient ball.
In the Hilbert norm, place one variation in $L^2(I_n;H^{k-1}_2)$ and use \eqref{eq:neck-slab-conical-size} for the other.
Even its undifferentiated height grows at most like $\langle\zeta\rangle$, which the factor $P_vm_v$ absorbs.
Thus
\begin{equation}\label{eq:app-normal-coefficient-second-variation}
 \|D^2F(v)[h,g]\|_{L^2(I_n;H^{k-2}_2)}
 \le C\bigl(r_n(h)\|g\|_{L^2(I_n;H^{k-1}_2)}
           +r_n(g)\|h\|_{L^2(I_n;H^{k-1}_2)}\bigr).
\end{equation}
For the box estimate the graph norm \eqref{eq:conical-parabolic-graph-norm} controls the time supremum of the conically weighted derivatives through order $k$ by local Sobolev embedding.
If all $k$ spatial derivatives fall on $h'$, $g'$ or $v'$, the $L^2_\theta H^4$ term at index $k$ controls that factor with the output weight $\langle\rho\rangle^{k-1}$.
The same product estimates applied to differences of the displayed formula give continuity of the second variation in each stated domain norm, since $P_v-P_w=-P_v m_w(v-w)$ and the denominators stay bounded away from zero.
Since $DF(0)h=(\zeta-P_0)h'-P_0(\nu_r'/\nu_r)h$ has bounded coefficients with the decay above, integration along line segments gives the first-variation and difference bounds.

Finally, the normal velocity of $A Y_v$ and homogeneity of $\flowSpeed$ give
\[
 (1-\vartheta/\mu)v_s
 =\frac{\flowSpeed[Y_v]}{\nu_r b(v)}
   +(\mu-\vartheta)\frac{\langle Y_v,\nu_{Y_v}\rangle}{\nu_r b(v)}.
\]
Subtracting $(\mu-\vartheta)Dv$ proves the first identity in \eqref{eq:normal-neck-material-equation}.
At the centre, evenness and $v(0)=0$ give $C_U^\nu(v)(0)=1$ and $\matD v(0)=0$; evaluation there gives the formula for $\vartheta$.
\end{proof}

It is useful to isolate the conical product estimate used repeatedly below.
For $a\in\R$, set
\[
 [u]_{k,a}:=\sum_{j=0}^k
 \operatorname*{ess\,sup}_{\zeta\in\R}\langle\zeta\rangle^{j-a}
 |\partial_\zeta^ju(\zeta)|.
\]
Thus $\coneSize_k(u)=[u]_{k,1}$ by \eqref{eq:pincher-global-smallness}.

\begin{lemma}[Conical products and estimates for the surface diffusion remainders]
\label{lem:app-conical-tame-calculus}\label{lem:sdf-neck-nonlinear-estimates}
For $k\geq2$,
\begin{equation}\label{eq:app-conical-product}
 [uv]_{k,a+b}\leq C[u]_{k,a}[v]_{k,b}.
\end{equation}
If the ranges of the degree-zero coefficient variables $\mathbf z,\widehat {\mathbf z}$ lie in a fixed compact subset of the domain of a smooth function $G$, and their conical seminorms are bounded, then
\begin{equation}\label{eq:app-conical-composition}
 [G(\mathbf z)-G(\widehat{\mathbf z})]_{k,0}
 \leq C_G[\mathbf z-\widehat{\mathbf z}]_{k,0}.
\end{equation}
The coefficient variables used below are $(U+v)/\langle\zeta\rangle$, $U'+v'$ and their bounded conical derivatives, not the unbounded radius $U+v$ itself.

For the remainder estimates, let $U$ be a normalised positive conical profile for $\flowSpeed=-\Delta H$ and let $k\ge10$ be an integer.
The estimates hold for the radial remainder $\radRem$ of \eqref{eq:terminal-radial-remainder} and the normal remainder $\radRem^\nu$ of \eqref{eq:normal-neck-remainder}.
In the input and target norms of \eqref{eq:neck-remainder-norms}, the indicated operators extend from smooth paths to $C^2$ maps when the conical size $r_n(v)$ of \eqref{eq:neck-slab-conical-size} is sufficiently small.
For the unweighted interval estimate we also require $\|v\|_{L^\infty(I_n;H^k(J))}$ sufficiently small.
The constants are uniform in the slab, spatial interval, box centre and initial conical radius $R_0\ge1$.
For the weighted Sobolev estimate put
\[
 Q_n(v):=\|v\|_{L^\infty(I_n;H^k_2)}
              +\|v\|_{L^2(I_n;H^{k+2}_2)}.
\]
We first estimate the projected spatial remainder $\Pi\radRem(v)$ in \eqref{eq:defined-core-nonlinearity}, where $\Pi$ is the normalising projection \eqref{eq:projected-pincher-operator}.
On a sufficiently small conical coefficient ball,
\begin{align}
 \|\Pi\radRem(v)\|_{L^2(I_n;H^{k-2}_2)}
 &\leq C r_n(v)Q_n(v),
 \label{eq:app-hilbert-spacetime-tame}\\
 \|\Pi(\radRem(v)-\radRem(w))\|_{L^2(I_n;H^{k-2}_2)}
 &\leq C\{r_n(v)+r_n(w)\}Q_n(v-w)
 \nonumber\\
 &\quad+C r_n(v-w)\{Q_n(v)+Q_n(w)\}.
 \label{eq:app-hilbert-spacetime-difference}
\end{align}
These estimates use the Hilbert norm for the highest derivative and the conical norm for the remaining factors.

For the spatial remainder $\radRem$ of \eqref{eq:terminal-radial-remainder}, use the enlarged box $Q_\rho^*=B_\rho^*\times(0,1)$ of \eqref{eq:neck-end-input-norm}.
The local graph norm is the weighted sum on this single enlarged box,
\[
\begin{aligned}
 \|v\|_{\slabGraph^k(Q_\rho^*)}
 :=\sum_{j=0}^k\langle\rho\rangle^{j-1}
 \Bigl\{
 &\|\partial_y^j\widetilde v\|_{L^\infty(0,1;H^2(B_\rho^*))}\\
 &+\|\partial_y^j\widetilde v\|_{L^2(0,1;H^4(B_\rho^*))}\\
 &+\|\partial_y^j\partial_\theta\widetilde v\|_{L^2(Q_\rho^*)}
 \Bigr\},
\end{aligned}
\]
where $\widetilde v=\dilPull v(\cdot,n+\theta)$ as in \eqref{eq:dilation-straightening}.
In straightened coordinates we have
\begin{align}
 &\sum_{j=0}^k\langle\rho\rangle^{j-1}
 \|\partial_y^j\widetilde{\radRem(v)}\|_{L^2(Q_\rho)}
 \leq C\langle\rho\rangle^k\|\widetilde v\|_{L^\infty(0,1;H^k(B_\rho^*))}
             \|v\|_{\slabGraph^k(Q_\rho^*)},
 \label{eq:app-local-mixed-nonlinear}\\
 &\sum_{j=0}^k\langle\rho\rangle^{j-1}
 \|\partial_y^j\{\widetilde{\radRem(v)}
                    -\widetilde{\radRem(w)}\}\|_{L^2(Q_\rho)}
 \nonumber\\
 &\quad\leq C\langle\rho\rangle^k\{\|\widetilde v\|_{L^\infty(0,1;H^k(B_\rho^*))}
                  +\|\widetilde w\|_{L^\infty(0,1;H^k(B_\rho^*))}\}
                    \|v-w\|_{\slabGraph^k(Q_\rho^*)}
 \nonumber\\
 &\qquad+C\langle\rho\rangle^k\{\|v\|_{\slabGraph^k(Q_\rho^*)}
                 +\|w\|_{\slabGraph^k(Q_\rho^*)}\}
                    \|\widetilde v-\widetilde w\|_{L^\infty(0,1;H^k(B_\rho^*))}.
 \label{eq:app-local-mixed-nonlinear-difference}
\end{align}
To estimate the subtraction $\radRem(v)(0)DU$ we use the trace at the centre.
Consequently
\begin{align}
 \|\Pi\radRem(v)\|_{\slabSrc_R^k(I_n)}
 &\leq C\|v\|_{\slabGraph_R^k(I_n)}^2,
 \label{eq:app-box-spacetime-tame}\\
 \|\Pi(\radRem(v)-\radRem(w))\|_{\slabSrc_R^k(I_n)}
 &\leq C\{\|v\|_{\slabGraph_R^k(I_n)}
                  +\|w\|_{\slabGraph_R^k(I_n)}\}
                    \|v-w\|_{\slabGraph_R^k(I_n)}.
 \label{eq:app-box-spacetime-difference}
\end{align}
On an end box $|\rho|\geq4$ the appropriate refined estimate is
\begin{align}
 &\sum_{j=0}^k\langle\rho\rangle^{j-1}
 \|\partial_y^j\widetilde{\radRem(v)}\|_{L^2(Q_\rho)}
 \nonumber\\
 &\quad\leq C r_n(v)
 \left\{\sum_{j=0}^k\langle\rho\rangle^{j-1}
       \|\partial_y^{j+4}\widetilde v\|_{L^2(Q_\rho^*)}
       +\langle\rho\rangle^{-1}
          \|v\|_{\slabGraph^k(Q_\rho^*)}\right\},
 \label{eq:app-nonlinear-conical-decay}
\end{align}
and its polarised version holds.
In particular the principal term in this estimate has no additional inverse power of $\langle\rho\rangle$.
\end{lemma}

\begin{proof}
Leibniz' rule gives \eqref{eq:app-conical-product}, and the chain rule gives \eqref{eq:app-conical-composition} on bounded coefficient sets.
For the remaining assertions specialise to surface diffusion, so $\flowRad=\radOp$.
We calculate the normal-chart operator directly, so no regularity of a nonlinear coordinate inverse is being assumed.
For the meridian \eqref{eq:app-normal-neck-meridian}, put $\sigma_v=\sqrt{(R_v')^2+(Z_v')^2}$.
The mean curvature and its Laplacian are
\begin{equation}\label{eq:app-sdf-normal-chart-operator}
 H_v=\frac12\left(\frac{Z_v'}{R_v\sigma_v}
       +\frac{R_v'Z_v''-Z_v'R_v''}{\sigma_v^3}\right),
 \qquad
 \Delta_v H_v=\frac1{R_v\sigma_v}
       \partial_\zeta\left(\frac{R_v}{\sigma_v}\partial_\zeta H_v\right).
\end{equation}
Here $\flowRad^\nu(v)=-(\Delta_v H_v)/(\nu_r b(v))+\mu C_U^\nu(v)$.
The coefficient of $v''$ in $H_v$ is
\[
 \frac{R_v'\nu_r(\nu_U)_z-Z_v'\nu_r^2}{2\sigma_v^3}
 =\frac{\nu_r b(v)}{2\sigma_v^2},
\]
because $\nu_{Y_v}=(-Z_v'e_r+R_v'e_z)/\sigma_v$.
Thus the principal coefficient is exactly
\begin{equation}\label{eq:app-sdf-normal-chart-principal}
 \flowRad^\nu(v)=-\frac{1}{2\sigma_v^4}\partial_\zeta^4v
                   +\text{terms involving at most three derivatives of }v.
\end{equation}
On the coefficient ball, $R_v/\langle\zeta\rangle$, $\sigma_v$ and $b(v)$ are bounded away from zero.
The remaining coefficients in \eqref{eq:app-sdf-normal-chart-operator} are smooth functions of these degree-zero quantities and the scaled derivatives of $U$ and $v$.
For example $R_v^{-1}$ has degree $-1$, $\partial_\zeta^jv$ has degree $1-j$, and each differentiated coefficient gains the corresponding inverse power of $\langle\zeta\rangle$.
Consequently a lower-order product has total degree at most $-3$, while the principal coefficient difference is $O(r_n(v))$.
The geometric term \eqref{eq:app-normal-neck-dilation} has the same or stronger decay after its linear part is subtracted.

Now write \eqref{eq:terminal-operator} and \eqref{eq:app-sdf-normal-chart-operator} in quasilinear fourth-order form, and subtract their constant and linear terms as in \eqref{eq:terminal-radial-remainder} and \eqref{eq:normal-neck-remainder}.
Every remaining monomial contains at least two perturbation factors.
For the Hilbert source through order $k-2$, place one high derivative in $L^2H^{k+2}_2$ and the other factors in their conically weighted $L^\infty$ norms.
Their inverse powers of $\langle\zeta\rangle$ compensate the coefficient weights.
The one-dimensional product estimate gives \eqref{eq:app-hilbert-spacetime-tame}; the fixed functional is bounded by the local $H^k$ trace.
Subtracting two product expansions gives \eqref{eq:app-hilbert-spacetime-difference}.

We shall also use the unweighted estimates for the spatial remainder before projection on an interval $J$ of length at least one.
For $k\ge10$ the product rule puts the factor with up to $k+2$ derivatives in $L^2$ and the lower derivatives in $L^\infty$.
Taylor's formula at $U$ supplies a second factor of $v$.
For two inputs $v,w$ the same product estimate holds with one factor replaced by $v-w$.
Writing $L^p H^j$ for $L^p(I_n;H^j(J))$ in the following display, we obtain
\begin{equation}\label{eq:app-local-sobolev-quadratic}
 \begin{aligned}
 \|\radRem(v)\|_{L^2H^{k-2}}
 &\le C\|v\|_{L^\infty H^k}\|v\|_{L^2 H^{k+2}},\\
 \|\radRem(v)-\radRem(w)\|_{L^2H^{k-2}}
 &\le C\bigl(\|v\|_{L^\infty H^k}+\|w\|_{L^\infty H^k}\bigr)
             \|v-w\|_{L^2 H^{k+2}}\\
 &\quad+C\|v-w\|_{L^\infty H^k}
          \bigl(\|v\|_{L^2 H^{k+2}}+\|w\|_{L^2 H^{k+2}}\bigr).
 \end{aligned}
\end{equation}
Both estimates hold also for $\radRem^\nu$ by the normal graph formula \eqref{eq:app-sdf-normal-chart-operator} and the dilation identity \eqref{eq:app-normal-neck-dilation}.
The constant is uniform in $J$ on a small $L^\infty(I_n;H^k(J))$ ball.
Indeed, the differentiated coefficients are uniformly bounded there by positivity of $U$ and \eqref{eq:general-profile-tail}, and $H^1(J)$ embeds into $L^\infty(J)$ with a uniform constant for $|J|\ge1$.

For the local mixed estimates we first keep all spatial norms unweighted.
For smooth height variations $h,h_1,h_2$, write
\[
 L_\rho(h)=\|\widetilde h\|_{L^\infty(0,1;H^k(B_\rho^*))},
 \qquad
 K_\rho(h)=\|\widetilde h\|_{L^2(0,1;H^{k+4}(B_\rho^*))}.
\]
The coefficient ball bounds the normalised variables $v/\langle\zeta\rangle$ and $v'$ and their conical derivatives through orders $k$ and $k-1$, respectively; it does not require $L_\rho(v)$ to be small uniformly in $\rho$.
In both \eqref{eq:terminal-operator} and \eqref{eq:app-sdf-normal-chart-operator}, the fourth derivative occurs linearly, with coefficient depending only on $v,v'$ and the fixed profile.
The remaining terms are products of derivatives of order at most three, with smooth coefficients depending on the same variables.
Their coefficients and derivatives with respect to these variables are uniformly bounded on the coefficient ball: a derivative with respect to $v$ carries the inverse radius factor, while $R_v/\langle\zeta\rangle$, $\sigma_v$ and $b(v)$ stay bounded away from zero.
For the geometric term in the normal chart this assertion follows from the cancellation \eqref{eq:app-normal-neck-dilation}; no coefficient proportional to $\zeta$ remains.
The pullback \eqref{eq:dilation-straightening} introduces only factors $e^{c\mu\theta}$ with $0\le\theta\le1$, so the same bounds hold in straightened coordinates, uniformly in $n$.

After $j\le k$ spatial differentiations, two perturbation factors of order at least $k$ would require at least $2(k-4)>k$ differentiations.
Thus at most one factor requires an $L^2H^{k+4}$ bound; all the others have order at most $k-1$ and are bounded pointwise by the local $H^k$ norm.
Extra factors of the base field are bounded in their conical coefficient norms, including the inverse radius accompanying an undifferentiated base field.
Applying this product rule to the first variation of the remainder, and then to its second variation, gives
\begin{equation}\label{eq:app-local-mixed-variations}
 \begin{aligned}
 \|\widetilde{D\radRem(v)h}\|_{L^2(0,1;H^k(B_\rho))}
 &\le C\{L_\rho(v)K_\rho(h)+L_\rho(h)K_\rho(v)\},\\
 \|\widetilde{D^2\radRem(v)[h_1,h_2]}\|_{L^2(0,1;H^k(B_\rho))}
 &\le C\{L_\rho(h_1)K_\rho(h_2)+L_\rho(h_2)K_\rho(h_1)\\
 &\hspace{2.8em}+L_\rho(h_1)L_\rho(h_2)K_\rho(v)\}.
 \end{aligned}
\end{equation}
These bounds also hold for $\radRem^\nu$.
In the first line a low factor of $v$ is present because $D\radRem(0)=0$.
In the second line the last term is needed when both variations fall on a coefficient and the highest derivative remains on $v$; for example, varying the principal term in \eqref{eq:app-sdf-normal-chart-principal} twice produces $D^2[-(2\sigma_v^4)^{-1}][h_1,h_2]v^{(4)}$.

The output weights and the local graph norm satisfy
\begin{equation}\label{eq:app-local-full-output-weight}
 \begin{aligned}
 \sum_{j=0}^k\langle\rho\rangle^{j-1}
       \|\partial_y^j\widetilde f\|_{L^2(Q_\rho)}
 &\le C\langle\rho\rangle^{k-1}
       \|\widetilde f\|_{L^2(0,1;H^k(B_\rho))},\\
 K_\rho(h)&\le C\langle\rho\rangle
       \|h\|_{\slabGraph^k(Q_\rho^*)}.
 \end{aligned}
\end{equation}
The second inequality uses the smallest graph weight $\langle\rho\rangle^{-1}$ to control every derivative through order $k+4$.
Integrating the first line of \eqref{eq:app-local-mixed-variations} along the segment from zero to $v$, or from $w$ to $v$, and applying \eqref{eq:app-local-full-output-weight} proves \eqref{eq:app-local-mixed-nonlinear}--\eqref{eq:app-local-mixed-nonlinear-difference} with the factor $\langle\rho\rangle^k$.
The same conversion gives the local weighted bounds for the first two variations.
Putting the low factors in their conical coefficient norms instead gives the uniform box estimates \eqref{eq:app-box-spacetime-tame} and \eqref{eq:app-box-spacetime-difference}.
The rank-one subtraction is added only after these local estimates: $DU$ has finite conical source norm and $|\radRem(v)(0)|\leq C\|v\|_{H^k(-2,2)}^2$.

The first term in \eqref{eq:app-nonlinear-conical-decay} comes from leaving the difference of principal coefficients undifferentiated.
This difference is $O(r_n(v))$ and multiplies $\partial_y^{j+4}\widetilde v$.
If a derivative hits that coefficient, its conical inverse power, or the inverse power on the remaining low derivative factor, gives at least $\langle\rho\rangle^{-1}$ after the graph and source weights are matched.
The third-order and lower geometric terms have at least this gain.
For example a high coefficient derivative paired with a low fourth derivative is bounded using the graph derivative and the conical bound $|v^{(4)}|\leq Cr_n(v)\langle\rho\rangle^{-3}$.
This proves the refined estimate and its polarisation.
Here $\partial_y^{j+4}\widetilde v$ carries the weight $\langle\rho\rangle^{j-1}$ of the $j$-th graph summand in \eqref{eq:conical-parabolic-graph-norm}.
The corresponding Hilbert variation estimates are
\begin{equation}\label{eq:app-hilbert-second-variation}
 \begin{aligned}
 \|\Pi D\radRem(v)h\|_{L^2(I_n;H^{k-2}_2)}
 &\le C\{r_n(v)Q_n(h)+r_n(h)Q_n(v)\},\\
 \|\Pi D^2\radRem(v)[h_1,h_2]\|_{L^2(I_n;H^{k-2}_2)}
 &\le C\{r_n(h_1)Q_n(h_2)+r_n(h_2)Q_n(h_1)\\
 &\hspace{2.8em}+r_n(h_1)r_n(h_2)Q_n(v)\}.
 \end{aligned}
\end{equation}
With the end norm $E_\rho$ of \eqref{eq:neck-end-input-norm}, the same conical product calculation gives
\begin{equation}\label{eq:app-end-remainder-variations}
 \begin{aligned}
 \sum_{j=0}^k\langle\rho\rangle^{j-1}
       \|\partial_y^j\widetilde{D\radRem(v)h}\|_{L^2(Q_\rho)}
 &\le C\{r_n(v)E_\rho(h)+r_n(h)E_\rho(v)\},\\
 \sum_{j=0}^k\langle\rho\rangle^{j-1}
       \|\partial_y^j\widetilde{D^2\radRem(v)[h_1,h_2]}\|_{L^2(Q_\rho)}
 &\le C\{r_n(h_1)E_\rho(h_2)+r_n(h_2)E_\rho(h_1)\\
 &\hspace{2.8em}+r_n(h_1)r_n(h_2)E_\rho(v)\}.
 \end{aligned}
\end{equation}
The extra base-field term has the same origin as in \eqref{eq:app-local-mixed-variations}.
All these variation estimates hold for the normal remainder as well, since \eqref{eq:app-normal-neck-dilation} gives additional decay for its geometric contribution.
Taylor's integral remainder for the smooth coefficient functions on their compact normalised coefficient set gives continuity of the displayed derivatives in the indicated norms.
Indeed, subtracting two second variations either replaces their highest base-field factor by the difference of the base fields, or introduces that difference into a bounded coefficient; the same product estimates bound both terms.
The bounds are uniform in the box centre, so the same Taylor estimate holds after taking the suprema in the conical graph and source norms.
This proves the stated $C^2$ dependence.
\end{proof}

The term $\vartheta_{\rm rad}(v)Dv$ in \eqref{eq:defined-core-nonlinearity} requires a different estimate.
Indeed $Dv=v-\zeta v'$, and the graph norm does not control the extra factor $\zeta$ on its highest derivative.
In the global height equation \eqref{eq:global-height-residual} the scalar multiplies the material time derivative in the inner coordinates.
\Cref{cor:app-summable-spatial-remainder} gives the source estimate for this term.

We now propagate the conical graph bound from one slab to the next.
We use the whole-line slab norms $G_n,H_n$ from \eqref{eq:conical-slab-norms} and the weighted bound $Z_\eta$ from \eqref{eq:conical-hilbert-decay}.
For the quotient trace norm \eqref{eq:conical-parabolic-trace-norm} write
\[
 T_k(u)=\|u\|_{\slabTr_\infty^k}.
\]
The static norm
\begin{equation}\label{eq:app-static-trace-energy}
 T_k^{\rm st}(u):=\sum_{j=0}^k\sup_{\rho\in\R}
 \langle\rho\rangle^{j-1}\|\partial^ju\|_{H^2(B_\rho)}
\end{equation}
is equivalent to $T_k$ on the even subspace with zero value at the centre.
The trace estimate \eqref{eq:app-static-conical-trace-bound} gives one inequality.
For the converse we construct a path with trace $u$ and graph norm bounded by $CT_k^{\rm st}(u)$, using the model equation on the whole line.

For the model operator $B$ of \eqref{eq:far-model-operator}, the semigroup formula \eqref{eq:far-semigroup} has the convolution form
\begin{equation}\label{eq:app-model-semigroup-formula}
 (S_0(s)u)(\zeta)
 =e^{\mu s}(K_{a_\infty\tau_{\rm h}(s)}*u)(e^{-\mu s}\zeta),
 \qquad \widehat K_\sigma(\xi)=e^{-\sigma\xi^4}.
\end{equation}
The generator of the quotient semigroup \eqref{eq:model-waist-semigroup} acts on representatives vanishing at zero by
\begin{equation}\label{eq:app-model-waist-quotient}
 B^{\rm q}u
 =B u+a_\infty u^{(4)}(0),\qquad u(0)=0.
\end{equation}
Constants form an invariant eigenspace of $S_0$, so $S_{\rm q}$ is a semigroup on the quotient identified with functions vanishing at zero.
Passing to the quotient removes the solution $S_0(s)1=e^{\mu s}$, whose growth would preclude the uniform bound we seek.

\begin{lemma}[Uniform model trace and slab estimates]
\label{lem:app-model-conical-estimate}
For an even datum $u$ with $u(0)=0$,
\begin{equation}\label{eq:app-model-uniform-trace}
 \sup_{s\geq0}T_k^{\rm st}(S_{\rm q}(s)u)
 \leq C T_k^{\rm st}(u).
\end{equation}
If $z_s-B^{\rm q}z=g$, with $z(0)=u$ and $g(0,s)=0$, then
\begin{equation}\label{eq:app-model-whole-line-graph}
 \sup_n G_n(z)\leq C\{T_k(u)+\sum_n H_n(g)\}.
\end{equation}
The model solution on a single slab also gives a bounded linear trace lift with zero value at the centre.
\end{lemma}

\begin{proof}
The biharmonic kernel is Schwartz at positive time.
Scaling gives
\begin{align}
 \sup_{0<\sigma\leq a_\infty/(4\mu)}
 \int_\R\langle x\rangle^{k+8}|K_\sigma(x)|\,\dd x&<\infty,\nonumber\\
 \sup_{0<\sigma\leq a_\infty/(4\mu)}
 \int_{|x|\geq d}\langle x\rangle^{k+8}
                 |\partial_x^rK_\sigma(x)|\,\dd x
 &\leq C_{M,r}d^{-M},\quad d\geq1 .
 \label{eq:app-biharmonic-kernel-moments}
\end{align}
The second bound holds for every fixed $M,r$.
Consequently convolution by $K_\sigma$ is uniformly bounded in the weighted uniformly local $H^2$ norms of \eqref{eq:app-static-trace-energy}.
This includes $\sigma=0$ by the identity operator.

For $E\geq1$ put $D_Eh(\zeta)=E\{h(\zeta/E)-h(0)\}$.
For $j\geq1$ its derivative of order $j+r$, $r=0,1,2$, has factor $E^{1-j-r}$.
The conical weight ratio satisfies
\[
 \langle\zeta\rangle^{j-1}
 \langle\zeta/E\rangle^{1-j}\leq E^{j-1}.
\]
For $r=0,1$, the local $H^2$ norm supplies the required pointwise bound.
For $r=2$, changing variables on an output box costs at most $E^{1/2}$ and hence leaves the factor $E^{-3/2}$.
For $j=0$, use $|D_Eh(\zeta)|\leq|\zeta|\|h'\|_\infty$ and the $j=1$ trace bound for the next two derivatives.
Thus $T_k^{\rm st}(D_Eh)\leq C T_k^{\rm st}(h)$.
Applying this with $E=e^{\mu s}$ and $h=K_{a_\infty\tau_{\rm h}(s)}*u$ proves \eqref{eq:app-model-uniform-trace}.

On a straightened slab the unquotiented model is $\partial_\theta w+a_\infty e^{-4\mu\theta}\partial_y^4w=F$.
On $0\le\theta\le1$ its diffusion coefficient is bounded above and below by positive constants.
Pairing with $w$ and $\partial_y^4w$ on $\R$ and using Young's inequality gives
\begin{equation}\label{eq:app-biharmonic-slab-energy}
 \|w\|_{C([0,1];H^2)}+\|\partial_y^4w\|_{L^2(0,1;L^2)}
 \le C\bigl(\|w(0)\|_{H^2}+\|F\|_{L^2(0,1;L^2)}\bigr).
\end{equation}
To obtain the weighted uniformly local version, split each differentiated input into a part supported in an enlarged target box and the remaining unit boxes.
The first part is controlled by \eqref{eq:app-biharmonic-slab-energy}.
For separated boxes move derivatives onto the kernel and use \eqref{eq:app-biharmonic-kernel-moments}; the resulting bounds decay faster than any power of the box separation, including after time integration.
Summation over the input boxes absorbs the polynomial weight ratio.
Returning to the original similarity coordinates by \eqref{eq:app-dilation-pullback}, this proves
\begin{equation}\label{eq:app-model-single-slab}
 G_0(\dilPull^{-1}w)\leq C\{T_k^{\rm st}(w(0))+H_0(\dilPull^{-1}F)\}.
\end{equation}
For each output box we have estimated its spacetime $L^2$ norm by summing over all input boxes, including those just outside the enlargement.
The kernel bounds make the sum over separated boxes uniform in the output box, so we may now take the supremum as in \eqref{eq:conical-parabolic-graph-norm}.

The projection $w(\cdot,\theta)\mapsto w(\cdot,\theta)-w(0,\theta)$ onto representatives vanishing at zero is bounded on this graph space.
The $j=0,1$ core components control the subtracted constant and its time derivative, since $|\partial_\theta w(0,\theta)|\leq C\|\partial_\theta w(\cdot,\theta)\|_{H^1(-2,2)}$.
Subtraction of the value at the centre also commutes with dilation straightening.
This gives the same slab estimate for the quotient equation and constructs the asserted lift.
A source supported on $I_m$ has outgoing trace at most $CH_m(g)$; on later faces propagate it by \eqref{eq:app-model-uniform-trace} and use the single-slab estimate \eqref{eq:app-model-single-slab}.
The homogeneous datum is treated in the same way.
Summing these source pieces proves \eqref{eq:app-model-whole-line-graph}.
\end{proof}

We record the interior estimate used on the slowly expanding core.
On $[n-1,n+1]$ write
\[
 w(y,s)=e^{-\mu(s-n)}v(e^{\mu(s-n)}y,s),\qquad
 F(y,s)=e^{-\mu(s-n)}f(e^{\mu(s-n)}y,s).
\]
If $B_\rho\Subset B_\rho^*$ are fixed-width nested intervals, a time cutoff on $[n-1,n+1]$ gives, for $n\geq1$, the graph estimate for each $\partial_y^jw$ on $B_\rho\times I_n$ from its lower $L^2H^{j+2}$ norm on the enlarged two-slab box and the local $\partial_y^jF$ source.
To see the derivative count, commute $\partial_y^j$ through the straightened equation: every coefficient commutator has at most $j+3$ derivatives of $w$.
Nested spatial cutoffs and
\[
 \|w\|_{H^{j+3}}\leq
 \varepsilon\|w\|_{H^{j+4}}+C_\varepsilon\|w\|_{H^{j+2}}
\]
absorb these terms.
The time cutoff removes the incoming trace.
This proves the interior estimate separately at each differentiated order, without giving a lower source derivative the weight of a higher one.
The symbol bounds
\begin{equation}\label{eq:app-end-remainder-symbol-bounds}
 |\partial_y^rr_{4,\theta}|\leq C_r\langle y\rangle^{-4-r},
 \qquad
 |\partial_y^rr_{\ell,\theta}|
       \leq C_r\langle y\rangle^{\ell-4-r},\quad0\leq\ell\leq3,
\end{equation}
follow by applying \eqref{eq:app-straightened-operator} to the coefficient bounds \eqref{eq:fredholm-detail-a4-decay}--\eqref{eq:fredholm-detail-lower-decay}.
Here $r_{j,\theta}$ is the coefficient of $\partial_y^j$ in the difference between the straightened operator $\slabLin$ and $-a_\infty e^{-4\mu\theta}\partial_y^4$.
These bounds make the constants uniform in the box centre.
The normalising term is treated separately as the scalar $\ell(v)$ from \eqref{eq:projected-pincher-operator} times the fixed function $DU$.

\begin{proposition}[Whole-line conical Green estimate]
\label{prop:app-whole-line-conical-green}
If $v_s-\waistLin v=f$, then
\begin{equation}\label{eq:app-whole-line-conical-green}
 \sup_nG_n(v)\leq
 C\{T_k(v(0))+Z_\eta(v)+\sum_nH_n(f)\}.
\end{equation}
Here we assume the initial trace and the source have finite whole-line norms $T_k(v(0))$ and $\sum_nH_n(f)$, respectively.
The constant is independent of the initial compactification scale.
\end{proposition}

\begin{proof}
Take the slowly expanding radius $\widehat R(s)$ of \eqref{eq:slow-conical-radius}, with $R_*$ to be chosen sufficiently large.
Let $\chi_{\rm e}$ be even, zero on $[-1,1]$ and one outside $[-2,2]$, and set $z(\zeta,s)=\chi_{\rm e}(\zeta/\widehat R(s))v(\zeta,s)$.
Because $z$ vanishes near the centre, $z^{(4)}(0,s)=0$, and its model equation may be written either with $B$ or with $B^{\rm q}$:
\begin{align}
 z_s-B^{\rm q}z
 ={}&\chi_{\rm e}f
 +\chi_{\rm e}(\radLin+a_\infty\partial_\zeta^4)v
 -\chi_{\rm e}\ell(v)DU \nonumber\\
 &+[\partial_s-B,\chi_{\rm e}]v .
 \label{eq:app-moving-end-model-equation}
\end{align}
All cutoff arguments are $\zeta/\widehat R(s)$.
In particular the nonlocal normalisation source is present explicitly.

By \eqref{eq:app-end-remainder-symbol-bounds}, the principal coefficient error has source norm $C\widehat R_n^{-4}G_n(v)$ and the lower coefficient errors have source norm $C\widehat R_n^{-1}G_n(v)$.
Spatial cutoff commutators have order at most three and the latter bound as well.
The material cutoff term is $(\mu-a)(\zeta/\widehat R)\chi_{\rm e}'(\zeta/\widehat R)v$.
Converting a weighted local derivative on its transition to $H^k_2$ costs at most $C\widehat R_n^{k+1}$, including cutoff derivatives.
Also $|\ell(v)(s)|\leq Ce^{-\eta s}Z_\eta(v)$ and $DU$ has finite whole-line conical source norm.
Thus the right-hand side $g$ of \eqref{eq:app-moving-end-model-equation} satisfies
\begin{equation}\label{eq:app-slow-cutoff-source}
 H_n(g)\leq C H_n(f)+C\widehat R_n^{-1}G_n(v)
                   +C\widehat R_n^{k+1}e^{-\eta n}Z_\eta(v).
\end{equation}
It vanishes at the centre, as required in the model lemma.
The initial cutoff acts boundedly on $T_k$ uniformly for $R_*\geq1$.

On the remaining boxes $|\rho|\leq C_*\widehat R_n$, the interior estimate above and weighted Hilbert conversion give, for $n\geq1$,
\begin{equation}\label{eq:app-expanding-core-graph}
 G_n^{\rm near}(v)\leq C\{H_{n-1}(f)+H_n(f)
                   +\widehat R_n^{k+1}e^{-\eta(n-1)}Z_\eta(v)\}.
\end{equation}
The normalisation source contributes at most $Ce^{-\eta(n-1)}Z_\eta(v)$, since $DU$ has finite whole-line conical source norm.
Here $C_*$ includes the fixed enlargements and the two-slab straightening; on a box of radius comparable to $\widehat R_n$ the conversion factor is $\langle\rho\rangle^{j-1}\langle\rho\rangle^2 \leq C\widehat R_n^{k+1}$.
On the initial slab the near region has fixed size.
Using the incoming trace in the local parabolic estimate gives
\begin{equation}\label{eq:app-initial-core-graph}
 G_0^{\rm near}(v)\le C_{R_*}\{T_k(v(0))+H_0(f)+Z_\eta(v)\}.
\end{equation}
In both estimates the lower derivatives of $v$ on the enlarged boxes are bounded by $Z_\eta(v)$.

Apply \eqref{eq:app-model-whole-line-graph} to $z$ up to a finite number of slabs and put $\Gamma_N=\max_{n\leq N}G_n(v)$.
The far boxes are controlled by $z$, and \eqref{eq:app-expanding-core-graph} controls the rest.
Consequently
\[
 \Gamma_N\leq C_{R_*}
   \{T_k(v(0))+Z_\eta(v)+\sum_nH_n(f)\}
   +\frac{CR_*^{-1}}{1-e^{-a}}\Gamma_N .
\]
The constant in the last term is independent of $R_*$.
Choose $R_*$ so that this coefficient is below $1/2$ and let $N\to\infty$.
The estimate is initially obtained for smooth data in the conical class.
Approximation on finite time intervals, the weighted Hilbert uniqueness theorem and lower semicontinuity of each local graph norm extend it to all the stated data.
To justify the absorption, first fix $N$.
In the local energy estimates retain the lower-order norm of the solution, take the supremum over all boxes and apply Gronwall's inequality on $[0,N+1]$.
This gives $\Gamma_N<\infty$ for each $N$.
Absorbing the last term as above then gives the bound independent of $N$.
\end{proof}

\subsection{Summable fourth derivatives and localisation of the nonlinear forcing}

We use the quantities $\coneDiss_{4,n}(v)$ and $\coneDiss_4(v)$ of \eqref{eq:principal-spatial-dissipation}.
In each slab the supremum over boxes is taken after the local spacetime $L^2$ norm.

\begin{proposition}[Summable principal dissipation]
\label{prop:app-summable-principal-dissipation}
Suppose $v_s-\waistLin v=f$ and
\[
 T_k(v(0))+Z_\eta(v)+\sup_nG_n(v)+\sum_nH_n(f)<\infty.
\]
Then
\begin{equation}\label{eq:app-summable-principal-dissipation}
 \coneDiss_4(v)\leq C\{T_k(v(0))+Z_\eta(v)
                       +\sup_nG_n(v)+\sum_nH_n(f)\}.
\end{equation}
The constant is independent of the initial compactification scale.
Consequently, for the half-line Green operator the same data bound both the whole-line conical graph norm and $\coneDiss_4(v)$.
\end{proposition}

\begin{proof}
We first prove the assertion for the model $B$ of \eqref{eq:far-model-operator}.
We work first with $B$ on the full space, measuring the initial and terminal traces in $T_k^{\rm st}$ of \eqref{eq:app-static-trace-energy}.
On one straightened slab its diffusion coefficient is bounded above and below.
The fourth-order estimate \eqref{eq:app-biharmonic-slab-energy}, applied separately to $\partial_y^jv$, bounds the principal derivative and the terminal $H^2$ trace by the incoming $H^2$ trace and the source.
To obtain a uniformly local estimate, split each input into a cutoff supported in $B_\rho^*$ and its complement.
The first part is estimated by the whole-line $L^2$ energy estimate.
For the second, move derivatives onto the kernel: input and output are separated, so \eqref{eq:app-biharmonic-kernel-moments} gives a summable bound in the separation, also after time integration.
Sum over the exterior unit boxes and then take the supremum.
Polynomial conical weights are absorbed by the kernel moments.
Two derivatives of each differentiated initial trace suffice.
By \eqref{eq:far-semigroup}, $\tau_{\rm h}(\theta)=\int_0^\theta e^{-4\mu t}\,\dd t\ge e^{-4\mu}\theta$ on the unit slab, so the Fourier multiplier for the squared, time-integrated fourth derivative satisfies
\[
 \int_0^1|\xi|^8e^{-2a_\infty\tau_{\rm h}(\theta)|\xi|^4}\,\dd\theta
 \le \int_0^1|\xi|^8e^{-2a_\infty e^{-4\mu}\theta|\xi|^4}\,\dd\theta
 \le \frac{e^{4\mu}}{2a_\infty}|\xi|^4 .
\]

For a homogeneous model solution, the slab-$n$ straightening of \eqref{eq:app-model-semigroup-formula} is
\[
 \widetilde v_n(y,\theta)
 =e^{\mu n}(K_{a_\infty\tau_{\rm h}(n+\theta)}*v(0))(e^{-\mu n}y).
\]
For $n\geq1$ the kernel parameter lies in a fixed positive compact interval.
Transfer two derivatives to the kernel and retain $\partial^{j+2}v(0)$ as input.
The resulting kernels are uniformly Schwartz; summing weighted Cauchy--Schwarz over input boxes gives
\[
 |\partial^{j+4}(K_{a_\infty\tau_{\rm h}(n+\theta)}*v(0))(x)|
 \leq CT_k^{\rm st}(v(0))\langle x\rangle^{1-j}.
\]
With $E=e^{\mu n}$, the weighted output derivative is therefore bounded by
\[
 CT_k^{\rm st}(v(0))E^{-(j+3)}
       \langle\rho\rangle^{j-1}\langle\rho/E\rangle^{1-j}.
\]
For $j\geq1$ this is at most $CT_k^{\rm st}(v(0))E^{-4}$; for $j=0$ it is at most $CT_k^{\rm st}(v(0))E^{-3}$.
The pointwise smoothing estimate avoids a loss from the shrinking preimage of an output box.
Thus
\begin{equation}\label{eq:app-principal-slab-decay}
 \coneDiss_{4,n}(S_0(\cdot)v(0))\leq Ce^{-3\mu n}T_k^{\rm st}(v(0)),
 \qquad n\geq1.
\end{equation}
Decompose an inhomogeneous source $g$ into unit slabs.
A piece supported in $I_m$ has same-slab principal norm and outgoing trace at most $CH_m(g)$ by \eqref{eq:app-model-single-slab} after translating the slab to $[0,1]$.
On the following slab use that trace, and on subsequent slabs use \eqref{eq:app-principal-slab-decay}.
Minkowski's inequality and the summable kernel $e^{-3\mu r}$ yield
\begin{equation}\label{eq:app-model-principal-dissipation}
 \coneDiss_4(z)\leq C\{T_k^{\rm st}(z(0))+\sum_nH_n(g)\},
 \qquad z_s-B z=g.
\end{equation}
Approximation by finitely many smooth source pieces proves the same estimate for every summable source.

The same estimate holds for $B^{\rm q}$ because subtracting a spatial constant does not change any derivative in $\coneDiss_4$.
For its initial trace, which vanishes at the centre, $T_k^{\rm st}$ is equivalent to $T_k$.
For the true operator use the cutoff $z$ and its exact equation \eqref{eq:app-moving-end-model-equation}.
By \eqref{eq:app-slow-cutoff-source},
\[
 \sum_nH_n(g)\leq C\{\sum_nH_n(f)
              +\sup_nG_n(v)+Z_\eta(v)\}.
\]
Here $R_*$ has already been fixed; the series $\sum \widehat R_n^{-1}$ and $\sum \widehat R_n^{k+1}e^{-\eta n}$ converge.
Apply \eqref{eq:app-model-principal-dissipation} to $z$.
This bounds the contribution of every box $|\rho|\geq C_*\widehat R_n$, where the cutoff is identically one throughout the straightened box.

For the remaining boxes use \eqref{eq:app-expanding-core-graph}: its graph norm contains the principal derivative, so
\begin{equation}\label{eq:app-expanding-core-dissipation}
 \coneDiss_{4,n}^{\rm near}(v)
 \leq C\{H_{n-1}(f)+H_n(f)
              +\widehat R_n^{k+1}e^{-\eta(n-1)}Z_\eta(v)\},\qquad n\geq1.
\end{equation}
On $I_0$ we have $\coneDiss_{4,0}^{\rm near}(v)\leq G_0^{\rm near}(v)$, and \eqref{eq:app-initial-core-graph} bounds the latter in terms of the initial trace, the source and $Z_\eta(v)$.
Summing these near-region estimates proves \eqref{eq:app-summable-principal-dissipation}.

For the Green operator, combine this result with \eqref{eq:app-whole-line-conical-green} and \eqref{eq:app-continuous-halfline-estimate}.
Here $T_k(v(0))$ measures the full initial trace, including any compactly supported perturbations used to prescribe its controlled component.
\end{proof}

\begin{corollary}[Summable spatial and scalar-clock remainders]
\label{cor:app-summable-spatial-remainder}
Assume the nonlinear neck condition of \cref{def:neck-nonlinear-estimates} at order $k$.
Let $\eta>0$ and choose $0<a<\min\{\mu,\eta/[2(k+1)]\}$ in \eqref{eq:slow-conical-radius}.
Set
\[
 \|v\|_{\neckSize}
 :=Z_\eta(v)+\sup_nG_n(v)+\coneDiss_4(v).
\]
On a sufficiently small ball in this space,
\begin{equation}\label{eq:app-summable-spatial-remainder}
 \sum_nH_n(\Pi(\radRem(v)-\radRem(w)))
 \leq C(\|v\|_{\neckSize}+\|w\|_{\neckSize})
                         \|v-w\|_{\neckSize}.
\end{equation}
The same estimate holds with $\radRem^\nu$ in place of $\radRem$.
For both remainders the first two variations satisfy
\begin{equation}\label{eq:app-summable-spatial-variations}
 \begin{aligned}
 \sum_nH_n(\Pi D\radRem(v)h)
 &\le C\|v\|_{\neckSize}\|h\|_{\neckSize},\\
 \sum_nH_n(\Pi D^2\radRem(v)[h_1,h_2])
 &\le C(1+\|v\|_{\neckSize})
              \|h_1\|_{\neckSize}\|h_2\|_{\neckSize}.
 \end{aligned}
\end{equation}
If a spatially constant scalar $\vartheta(v)(s)$ obeys the core trace bounds
\[
 \sup_{I_n}|\vartheta(v)|\leq Ce^{-\eta n}\|v\|_{\neckSize},
 \qquad
 \sup_{I_n}|\vartheta(v)-\vartheta(w)|
       \leq Ce^{-\eta n}\|v-w\|_{\neckSize},
\]
then both $\vartheta(v)\matD v$ and $\vartheta(v)(C_U^\nu(v)-DU)$ satisfy the same summable quadratic difference estimate.
This is the clock term of the height-evolution residual, rather than $\vartheta_{\rm rad}(v)Dv$.
\end{corollary}

\begin{proof}
For either spatial remainder the hypotheses \eqref{eq:neck-first-variation}--\eqref{eq:neck-second-variation} apply with the norms of \eqref{eq:neck-remainder-norms}.
Integrating the first variation from $0$ to $v$, or from $w$ to $v$, gives the quadratic and difference bounds.
We give the common summation argument using $\radRem$.
Use the radius \eqref{eq:slow-conical-radius}.
On its near boxes, $\langle\rho\rangle\le C\widehat R_n$ and the density $\varpi=\langle\zeta\rangle^{-4}$ gives
\[
 \|\widetilde v\|_{L^\infty(0,1;H^k(B_\rho^*))}
 \le C\widehat R_n^2e^{-\eta n}Z_\eta(v).
\]
For $Q_\rho$ the first variation estimate uses unweighted Sobolev norms.
Applying \eqref{eq:app-local-full-output-weight} to the quadratic bound gives
\[
 H_n^{\rm near}(\radRem(v))
 \leq C\widehat R_n^{k+2}e^{-\eta n}Z_\eta(v)G_n(v).
\]
The factor $\widehat R_n^k$ retains the full output weight and removes the smallest input graph weight, as in \eqref{eq:app-local-full-output-weight}; the additional factor $\widehat R_n^2$ converts the low Hilbert norm to its unweighted local norm.
On the far boxes the last row of \eqref{eq:neck-remainder-norms} gives
\[
 H_n^{\rm far}(\radRem(v))
 \leq Cr_n(v)\{\coneDiss_{4,n}(v)+\widehat R_n^{-1}G_n(v)\}.
\]
Fixed box enlargement changes these norms by a uniform constant.
The near bound is summable because $(k+2)a<\eta$, which follows from $2(k+1)a<\eta$ in \eqref{eq:slow-conical-radius}.
The far bound is summable because $\sum_n\coneDiss_{4,n}(v)<\infty$, $\sum_n\widehat R_n^{-1}<\infty$ and $r_n(v)\le C\sup_mG_m(v)$.
The subtraction at the centre has the separate bound
\[
 H_n(\radRem(v)(0)DU)
 \leq Ce^{-2\eta n}Z_\eta(v)^2.
\]
This accounts for its nonlocal tail; it was not included in the local box estimate.
Subtraction of product expansions proves \eqref{eq:app-summable-spatial-remainder}.

For the first variation, \eqref{eq:neck-first-variation} and \eqref{eq:app-local-full-output-weight} give the same near bound with $Z_\eta(v)G_n(v)$ replaced by $Z_\eta(v)G_n(h)+Z_\eta(h)G_n(v)$.
For the second variation, \eqref{eq:neck-second-variation} and the same conversion of weights give
\[
 \begin{aligned}
 H_n^{\rm near}(D^2\radRem(v)[h_1,h_2])
 &\le C\widehat R_n^{k+2}e^{-\eta n}
       \{Z_\eta(h_1)G_n(h_2)+Z_\eta(h_2)G_n(h_1)\}\\
 &\quad+C\widehat R_n^{k+4}e^{-2\eta n}
       Z_\eta(h_1)Z_\eta(h_2)G_n(v).
 \end{aligned}
\]
The last term is summable too, since $(k+4)a<2\eta$ follows from $2(k+1)a<\eta$.
On the far boxes, inserting $L=r_n$ and $K=E_\rho$ from \eqref{eq:neck-remainder-norms} into the two variation estimates gives the corresponding sums of principal dissipation and $\widehat R_n^{-1}G_n$ terms.
The centre traces are bounded by the local $H^k$ norms, giving the factor $e^{-2\eta n}$ for two variations just as for \eqref{eq:app-summable-spatial-remainder}.
These estimates prove \eqref{eq:app-summable-spatial-variations}.
Continuity of the second derivative follows by subtracting its coefficient expansions: the conical coefficient norm controls coefficient differences uniformly on all boxes, and the two convergent near-region series and the far-region dissipation bound control the remaining high factors.

Finally the pullback identity \eqref{eq:app-material-pullback} and the definitions of the graph and source norms in \eqref{eq:conical-parabolic-graph-norm}--\eqref{eq:conical-parabolic-source-norm} imply $H_n(\matD v)\leq CG_n(v)$.
The scalar has no spatial derivatives, so
\[
 \sum_nH_n(\vartheta(v)\matD v)
 \leq C\|v\|_{\neckSize}\sup_nG_n(v)\sum_ne^{-\eta n}.
\]
Splitting the difference into its two factors proves the last assertion.
For $\vartheta(v)(C_U^\nu(v)-DU)$, use \eqref{eq:app-normal-neck-coefficient-bound} in the same summation.
\end{proof}

It remains to estimate the part of the nonlinear forcing in the controlled space.
We do this by pairing the forcing with the adjoint vectors that represent the controlled projection.

\begin{lemma}[Localisation of the projected forcing]
\label{lem:app-adjoint-localisation}
Under the hypotheses of \cref{lem:controlled-mode-localisation}, there are $C>0$ and $\sigma_N>\omega$ such that on every $I_n$ we have
\begin{equation}\label{eq:app-projected-localisation}
 \|\Pc\Pi\{\chi_{R(s)}\neckNL(v(s))\}\|_
 {L^2(I_n;H^k_2)}
 \leq CM^2e^{-\sigma_Nn},
\end{equation}
with the corresponding polarised difference estimate.
\end{lemma}

\begin{proof}
By \cref{lem:app-riesz-scale-compatibility}, the controlled projection is represented on the source and trace spaces by finitely many $H^0_2$ pairings with intrinsic adjoint Jordan vectors.
Their eigenvalues satisfy $\Re\lambda>-\omega>-\mu/2$, so \cref{lem:fredholm-detail-adjoint-chain-decay} gives smoothness on $\R$ and the differentiated exponential decay \eqref{eq:fredholm-detail-adjoint-chain-exponential}.
We use this decay to estimate the projected forcing.

Choose $0<a<\min\{\mu,(2\eta-\omega)/4,\eta/[2(k+1)]\}$ and split the pairing using $\chi_{e^{an}}$ on $I_n$.
On the inner part, converting each of the two weighted Sobolev factors in $\radRem(v)$ to an unweighted local factor costs at most $Ce^{2an}$.
In the interval norms \eqref{eq:neck-interval-norms}, this gives $L(v),L(w)\le5M_0e^{-(\eta-2a)n}<\varepsilon$ for $5M_0<\varepsilon$.
Integrating \eqref{eq:neck-first-variation} therefore puts one factor in $L^\infty H^k$ and the other in $L^2H^{k+2}$.
This gives $CM^2e^{-(2\eta-4a)n}$ in the $H^{k-2}$ source norm.
For the subtraction $\radRem(v)(0)DU$ it suffices to bound $v$ near the centre.
There the weight is bounded below, giving $CM^2e^{-2\eta n}$ without the factors $e^{2an}$.
It remains to estimate the projection of $\vartheta_{\rm rad}(v)Dv$ in \eqref{eq:defined-core-nonlinearity}.
Indeed, for an adjoint representative $\phi^*$ we apply the pairing identity \eqref{eq:controlled-dilation-pairing}.
The bracket is exponentially localised uniformly in $R$, including the cutoff derivative, by \eqref{eq:fredholm-detail-adjoint-chain-exponential}.
Thus the pairing with $Dv$ is bounded by $C\|v\|_{H^0_2}$, and multiplication by $\vartheta_{\rm rad}(v)$ again gives $CM^2e^{-2\eta n}$.
Subtracting the corresponding products for $v$ and $w$ gives the difference estimate.
On the outer part \eqref{eq:fredholm-detail-adjoint-chain-exponential} dominates every polynomial loss from the conical graph norm and moving cutoff.
Choose $\omega<\sigma_N<2\eta-4a$.
Finite-rank norm equivalence proves \eqref{eq:app-projected-localisation}.
Polarisation gives the difference estimate.
Since $\neckNL(v)(0)=0$ and $\chi_R(0)=1$, the source is already normalised and $\Pi\{\chi_R\neckNL(v)\}=\chi_R\neckNL(v)$; there is no additional rank-one contribution here.

The same reasoning applies to the projected spatial remainder $\Pi\radRem(v)$ alone.
For the height-evolution scalar-clock term $\vartheta(v)\matD v$, the core bound for $\vartheta$ and the discounted material-derivative Hilbert graph bound give the quadratic rate $e^{-2\eta n}$ after projection; the adjoint end decay treats the complementary region as above.
Thus this projected conclusion requires no source bound for $\vartheta_{\rm rad}(v)Dv$.
\end{proof}

The difference estimate gives \eqref{eq:controlled-projected-quadratic} of \cref{lem:controlled-mode-localisation}.

\section{Estimates for compactification}
\label{app:fixed-interface-details}

This appendix gives the detailed matching, volume and coordinate estimates used in \cref{sec:compact-new}.
We then construct the inverse for the global height equation and prove the nonlinear, selection and regularity estimates that complete the compactification.

\subsection{Reference surfaces, volume and coordinates}
\label{sec:appC-reference-estimates}

The collar calculations establish the matching estimate used to construct the reference family in \cref{sec:compact-new}.
We then estimate the defect and, when the additional volume hypotheses hold, correct the volume before treating normal coordinates.

\begin{lemma}[Differentiated collar matching]\label{lem:appC-collar-matching}
Under the hypotheses of \cref{lem:terminal-collar-match}, the radii $r_{\rm in}$ and $r_{\rm out}^{(1)}$ satisfy \eqref{eq:fixed-match-geometric}, both before and after applying $A\partial_A$.
\end{lemma}
\begin{proof}
By reflection it suffices to work on $2\delta\le z\le4\delta$.
We compare both radii with $\alpha z+c_1A^4z^{-3}$.
Write $E(\zeta)$ for the remainder in \eqref{eq:general-profile-tail}.
Substitution of $\zeta=z/A$ gives \eqref{eq:collar-inner-expansion}.
Here $|E^{(j)}(\zeta)|\le C_j\zeta^{-11-j}$ for large $\zeta$.
Since $A\partial_A[A E(z/A)]=A(E-\zeta E')(\zeta)$, the chain rule gives
\[
 \bigl|\partial_z^j(A\partial_A)^i[A E(z/A)]\bigr|
 \le C_j A^{1-j}(z/A)^{-11-j}=C_j A^{12}z^{-11-j},
 \qquad i=0,1 .
\]

Where $x=Z_\tau^{-1}(z)$ lies in the conical interval, \eqref{eq:collar-circle-position} gives $r_{\rm out}^{(1)}(z,\tau)=\alpha z-b\tau x^{-3}$.
To expand at fixed $z$, set $u=\tau/z^4$ and $x=z h(u)$.
The equation $Z_\tau(x)=z$, with $c=\alpha b/(1+\alpha^2)$ as in \eqref{eq:collar-circle-position}, becomes
\[
 h(u)+cu h(u)^{-3}=1,\qquad h(0)=1 .
\]
The derivative of the left-hand side with respect to $h$ at $(u,h)=(0,1)$ is one, so $h$ is smooth near zero.
In terms of $G(u)=\alpha-bu h(u)^{-3}$ we now have
\[
 r_{\rm out}^{(1)}(z,\tau)=zG(\tau/z^4),\qquad
 \partial_\tau r_{\rm out}^{(1)}(z,0)=z^{-3}G'(0)=-bz^{-3}.
\]

Taylor's formula $G(u)=G(0)+uG'(0)+u^2\int_0^1(1-t)G''(tu)\,\dd t$ gives
\[
 r_{\rm out}^{(1)}(z,\tau)=\alpha z-b\tau z^{-3}+\tau^2 z^{-7}B(u),
 \qquad B(u)=\int_0^1(1-t)G''(tu)\,\dd t .
\]
To bound the last term, observe that $u=\tau/z^4$ satisfies $z\partial_z u=-4u$ and $A\partial_A u=4u$.
Hence
\begin{align*}
 \partial_z[z^{-m}B(u)]&=z^{-m-1}\{-mB(u)-4uB'(u)\},\\
 A\partial_A[\tau^2 z^{-7}B(u)]&=4\tau^2 z^{-7}\{2B(u)+uB'(u)\}.
\end{align*}
All derivatives of $B$ are bounded on a fixed interval $0\le u\le u_0$ for $u_0$ sufficiently small.
Iterating the first identity gives
\[
 \bigl|\partial_z^j(A\partial_A)^i[\tau^2 z^{-7}B(u)]\bigr|
 \le C_j\tau^2 z^{-7-j},\qquad i=0,1 .
\]

At the outer edge of the collar, $x$ can lie just beyond the exact cone when $b<0$.
The smooth cap extension used in \cref{sec:compact-new} gives a smooth radial function on the whole closed collar for sufficiently small $\tau$.
Set $y=z/\delta$, $\sigma=\tau/\delta^4$ and $p_{\rm out}(y,\sigma)=\delta^{-1}r_{\rm out}^{(1)}(\delta y,\delta^4\sigma)$.
At $\sigma=0$, every point of $2\le y\le4$ lies on the terminal cone, so $p_{\rm out}(y,0)=\alpha y$ and $\partial_\sigma p_{\rm out}(y,0)=-by^{-3}$.
Taylor's formula at fixed $y$ therefore gives
\begin{equation}\label{eq:collar-outer-taylor}
 \begin{split}
 r_{\rm out}^{(1)}(z,\tau)
 &=\alpha z-b\tau z^{-3}+\tau^2\delta^{-7}B_\delta(y,\sigma),\\
 B_\delta(y,\sigma)
 &=\int_0^1(1-t)\partial_\sigma^2p_{\rm out}(y,t\sigma)\,\dd t.
 \end{split}
\end{equation}
The derivatives of $B_\delta$ are bounded on a fixed compact neighbourhood of $[2,4]\times\{0\}$; the bounds are uniform on the specified cap neighbourhood through the derivative orders used.
For $0\le i\le1$, the identities
\[
 \partial_z=\delta^{-1}\partial_y,\qquad
 A\partial_A[\tau^2\delta^{-7}B_\delta]
 =4\tau^2\delta^{-7}(2B_\delta+\sigma\partial_\sigma B_\delta)
\]
bound the $j$th spatial derivative of the remainder and its $A\partial_A$ derivative by $C_j\tau^2\delta^{-7-j}$ on the entire collar.

Finally, \eqref{eq:general-collar-coefficient} gives $-b\tau=c_1A^4$.
Subtracting \eqref{eq:collar-outer-taylor} from \eqref{eq:collar-inner-expansion} cancels the $z^{-3}$ terms.
The remaining terms and their $A\partial_A$ derivatives are bounded after $j$ spatial derivatives by $C_j A^8z^{-7-j}$, since $\tau=A^4/(4\mu)$ and $A/z$ is small.
Using $z\ge2\delta$ gives \eqref{eq:fixed-match-geometric}.
Evenness gives the same bounds on $-4\delta\le z\le-2\delta$.
\end{proof}

\begin{remark}[Matching to higher orders]\label{rem:appC-higher-order-matching}
For the rest of the construction we take only the first variation \eqref{eq:outer-terminal-normal-variation} of the cap.
This already gives the collar error \eqref{eq:fixed-match-geometric}, which is all that we need.
One can nevertheless continue the matching formally.
If $U$ has a full differentiated expansion in powers $\zeta^{1-4n}$ extending \eqref{eq:general-profile-tail}, substituting $A^4=4\mu\tau$ at fixed $z>0$ gives the coefficients $r_0(z)=\alpha z$ and $r_n(z)=(4\mu)^nc_n z^{1-4n}$ of $\tau^n$ in $AU(z/A)$ for $n\ge1$.
The neck solves $-\partial_\tau r=\flowRad[r]$ by \eqref{eq:general-flow-scaling} and \eqref{eq:general-profile-equation}.
Equating coefficients of $\tau^n$ therefore gives the same recursion that determines the formal outer expansion from the terminal cone,
\[
 r_{n+1}=-\frac1{n+1}[\tau^n]\flowRad\!\left[\sum_{j=0}^n\tau^jr_j\right],
 \qquad n\ge0 .
\]
Here $[\tau^n]$ extracts the coefficient of $\tau^n$.
\end{remark}

We next estimate the normal defect of the matched family and the volume error it produces.
The geometric cap condition is assumed here; its verification for the two cap types is independent of these estimates and is given in \cref{sec:dumbbell-pinchoff}.

\begin{lemma}[Reference defect and volume correction]\label{lem:appC-reference-defect-volume}
Under the hypotheses of \cref{prop:fixed-interface-matching}, set $\tau=A^4/(4\mu)$.
For a smooth cap and every integer $j\ge0$, the defect of the matched family $\matchX_A$ vanishes on $|z|\le2\delta$ and satisfies
\[
 \|\refDef_A\|_{C^j}\le C_{j,\delta}\tau,
 \qquad \|\partial_\tau\refDef_A\|_{C^0}\le C_\delta
\]
on the collars and fixed cap.
If the flow preserves volume and the cap satisfies \cref{def:compactification-hypotheses}(ii), the volume error also satisfies $|\partial_\tau^i\Delta V_{A(\tau)}|\le C_\delta\tau^{2-i}$ for $0\le i\le2$.
Under these additional hypotheses there is a unique small correction $d(A)$ restoring volume $V_0$, smooth for $A>0$ and satisfying \eqref{eq:outer-volume-amplitude}, and the corrected family has the same defect bounds.
The estimates are uniform for $j\le k$ on the $C^{k+8}$ cap neighbourhood in \cref{def:compactification-hypotheses}.
\end{lemma}
\begin{proof}
We first estimate the invariance defect $\refDef_A$ of \eqref{eq:matching-defect-definitions} for the matched family $\matchX_A$.
On $|z|\le2\delta$ it vanishes by \eqref{eq:general-flow-scaling} and \eqref{eq:general-profile-equation}, since $r_{\rm in}=AU(z/A)$ solves the radial equation with $\dd A/\dd t=\beta_0(A)$.
On the collar $2\delta\le z\le4\delta$, let $r$ be the interpolated radius.
Since $\xi$ vanishes on a neighbourhood of $4$, choose $\varepsilon_\xi>0$ such that it vanishes for $z/\delta\ge4-\varepsilon_\xi$.
For sufficiently small $\tau$, the inverse circle coordinates of the compact interval $2\le z/\delta\le4-\varepsilon_\xi$ remain on the exact cone.
We use the radial calculation on this interval and the fixed cap charts on the remaining collar, where the family is exactly \eqref{eq:outer-terminal-normal-variation}.
Set $y=z/\delta$, $\sigma=\tau/\delta^4$, and write $p(y,\sigma)=r(\delta y,A)/\delta$ and $p_{\rm in}(y,\sigma)=r_{\rm in}(\delta y,A)/\delta$.
We also need $\partial_\tau\refDef_A$ to control the second derivative of the volume correction.

In these variables \eqref{eq:collar-inner-expansion} and \eqref{eq:collar-outer-taylor} have the same constant and linear terms in $\sigma$.
By \eqref{eq:matched-collar-radius}, $p-p_{\rm in}$ is $1-\xi(y)$ times the difference between the rescaled outer and inner radii.
On $2\le y\le4-\varepsilon_\xi$ this gives
\[
 \|\partial_\sigma^i(p-p_{\rm in})\|_{C^{j+4}([2,4-\varepsilon_\xi])}
 \le C_j\sigma^{2-i},\qquad 0\le i\le2 .
\]
Indeed, the exact-cone calculation in the proof of \cref{lem:appC-collar-matching} writes the outer remainder divided by $\delta$ as $\sigma^2y^{-7}B(\sigma/y^4)$ on this interval, with all derivatives bounded independently of the cap variation.
For the inner remainder put $a=A/\delta=(4\mu\sigma)^{1/4}$.
At fixed $y$, $4\sigma\partial_\sigma=a\partial_a$ and
\[
 (a\partial_a)^i[aE(y/a)]
 =a\bigl[(1-\zeta\partial_\zeta)^iE\bigr](y/a).
\]
The differentiated bounds following \eqref{eq:collar-inner-expansion} make this $O(a^{12})=O(\sigma^3)$ in $C^{j+4}([2,4-\varepsilon_\xi])$ for $i\le2$.
Thus its first two $\sigma$ derivatives are $O(\sigma^2)$ and $O(\sigma)$, respectively.

Write $e_r$ for the radial unit vector and $\nu$ for the outward normal of the interpolated graph.
Since $\beta_0\partial_A=-\partial_\tau$ and $r_{\rm in}$ has zero defect, the scaling identity \eqref{eq:general-flow-scaling} gives
\[
 \refDef_A(\delta y)
 =\delta^{-3}\langle e_r,\nu\rangle
 \bigl\{-\partial_\sigma(p-p_{\rm in})
       -(\flowRad[p]-\flowRad[p_{\rm in}])\bigr\}.
\]
The rescaled radii stay bounded away from zero and have uniformly bounded $C^{j+4}$ norms.
Hence the smooth fourth-order radial expression \eqref{eq:general-radial-speed} gives $\|\flowRad[p]-\flowRad[p_{\rm in}]\|_{C^j}\le C_j\|p-p_{\rm in}\|_{C^{j+4}}=O(\sigma^2)$.
The velocity term is $O(\sigma)$ in $C^j$.
Both terms have bounded first $\sigma$ derivative there, as does $\langle e_r,\nu\rangle$, by the same bounds on the radii.
Returning to $z$ and using reflection yields the following bounds on this part of the collar:
\begin{equation}\label{eq:matched-collar-normal-defect}
 |\partial_z^j\refDef_A|\le C_j A^4\delta^{-7-j},\qquad
 \tau|\partial_\tau\refDef_A|\le C A^4\delta^{-7}.
\end{equation}
On the remaining collar and cap \eqref{eq:outer-terminal-normal-variation} gives $-\partial_\tau X_{\rm out}^{(1)}=\surfSpeed_0\nu_0$, so the defect vanishes at $\tau=0$.
Its first $\tau$ derivative is bounded in $C^j$ in the fixed cap charts, since the speed depends smoothly on spatial derivatives of the cap embedding through order four.
The bounds are uniform for $j\le k$ because $\surfSpeed_0\nu_0$ is bounded in $C^{k+4}$ on the given $C^{k+8}$ neighbourhood.
Thus $\refDef_A=O_\delta(\tau)$ on its whole support.
On the rest of the collar the fixed cap coordinates and the radial coordinate have uniformly bounded transition maps, so \eqref{eq:matched-collar-normal-defect} holds for $2\delta\le|z|\le4\delta$ as well.
For its $C^0$ time derivative the change of coordinates contributes only a bounded transport vector times one spatial derivative of the defect, which is already bounded by the same cap estimates.

This proves the defect bounds for the uncorrected reference family.
For the remainder of the proof assume the two additional volume hypotheses and compute the volume error.
Recall $A(\tau)=(4\mu\tau)^{1/4}$.
The first variation of volume and \eqref{eq:volume-preserving-speed} give
\begin{equation}\label{eq:matched-volume-derivative}
 -\frac{\dd}{\dd\tau}\Delta V_A
 =\beta_0(A)\frac{\dd}{\dd A}\Vol(\matchSurf_A)
 =\int_{\matchSurf_A}\refDef_A\,\dd\mu_A .
\end{equation}
Let $K\subset M$ be the union of the matching collars and the fixed outer cap, which contains the support of $\refDef_A$.
Write $\dd\mu_A=J_\tau\,\dd\lambda$ on $K$, with $\dd\lambda$ a fixed smooth area measure.
On the cap \eqref{eq:outer-terminal-normal-variation} bounds $J_\tau$ and $\partial_\tau J_\tau$.
On the collars the same bounds follow from the area density $r\sqrt{1+r_z^2}$ and the expansions \eqref{eq:collar-inner-expansion} and \eqref{eq:collar-outer-taylor} of the two radii in \eqref{eq:matched-collar-radius}.
Since $\refDef_A=O_\delta(\tau)$ and $\partial_\tau\refDef_A=O_\delta(1)$ on $K$, \eqref{eq:matched-volume-derivative} gives
\[
 \begin{aligned}
 -\partial_\tau\Delta V_A
   &=\int_K \refDef_A J_\tau\,\dd\lambda=O_\delta(\tau),\\
 -\partial_\tau^2\Delta V_A
   &=\int_K\bigl((\partial_\tau\refDef_A)J_\tau
                 +\refDef_A\partial_\tau J_\tau\bigr)\,\dd\lambda=O_\delta(1).
 \end{aligned}
\]
It remains to identify the limiting volume.
By \eqref{eq:general-profile-tail}, $U(\zeta)-\alpha|\zeta|$ is bounded, so $r_{\rm in}(z,A)=\alpha|z|+O(A)$ uniformly on $|z|\le2\delta$.
The difference of the squared radii bounds the neck contribution to the volume error, and the collars and cap converge smoothly.
Hence
\[
 |\Delta V_A|\le\pi\int_{-2\delta}^{2\delta}
       \bigl|r_{\rm in}(z,A)^2-\alpha^2z^2\bigr|\,\dd z+o(1)
       =O_\delta(A)+o(1)\longrightarrow0 .
\]
Integration from $\tau=0$ now gives $|\Delta V_A|\le C_\delta\int_0^\tau s\,\dd s=O_\delta(\tau^2)$.

For $A=A(\tau)$ and $|b|$ small, the change in enclosed volume is
\[
 W(\tau,b)=\Vol\bigl(\matchX_A+b\phi\nu_0\bigr)-\Vol(\matchX_A).
\]
The variation is supported on the smooth cap away from the collars, so $W$ is smooth up to $\tau=0$ and $W(\tau,0)=0$.
The first variation of volume and \cref{def:compactification-hypotheses}(ii) give
\[
 W_b(0,0)=\int_{C_0}\phi\,\dd\mu_{g_0}\ne0 .
\]
The implicit function theorem therefore gives a smooth function $b(\tau,v)$ near $(0,0)$ such that $W(\tau,b(\tau,v))=v$ and $b(\tau,0)=0$.
Put $v(\tau)=-\Delta V_{A(\tau)}$, so $d(A(\tau))=b(\tau,v(\tau))$ restores the volume $V_0$.
We have $v=O_\delta(\tau^2)$, $v'=O_\delta(\tau)$ and $v''=O_\delta(1)$.
Since $b(\tau,0)=0$ for all small $\tau$, both $b$ and $b_\tau$ vanish at $v=0$ and are $O(|v|)$ by smoothness.
Thus $d=O_\delta(\tau^2)$, and the chain rule at $(\tau,v(\tau))$ gives
\[
 \begin{aligned}
 \frac{\dd}{\dd\tau}d(A(\tau))
   &= b_\tau+b_v v'=O_\delta(\tau),\\
 \frac{\dd^2}{\dd\tau^2}d(A(\tau))
   &= b_{\tau\tau}+2b_{\tau v}v'+b_{vv}(v')^2+b_v v''=O_\delta(1),
 \end{aligned}
\]
since the derivatives of $b$ in the second line are bounded.
Now $A\partial_A=4\tau\partial_\tau$ and $(A\partial_A)^2=16(\tau^2\partial_\tau^2+\tau\partial_\tau)$, proving \eqref{eq:outer-volume-amplitude}.

It remains to estimate the defect after correcting the volume.
Use tildes and bars for the quantities in \eqref{eq:reference-normal-speeds} and \eqref{eq:matching-defect-definitions} evaluated on the matched and corrected embeddings, respectively, at the same point of $M$.
On $\supp\phi$, the forward velocity of the matched cap is $-\partial_\tau\matchX_A=\surfSpeed_0\nu_0$ by \eqref{eq:outer-terminal-normal-variation}.
At a fixed point $p$, the vector $\phi(p)\nu_0(p)$ in \eqref{eq:outer-volume-amplitude} is independent of $A$.
Differentiating the displacement $d(A)\phi(p)\nu_0(p)$ at fixed $p$ and subtracting the two defects gives
\[
 \overline{\refDef}_A-\widetilde{\refDef}_A
 =-(\partial_\tau d)\phi\langle\nu_0,\refNu_A\rangle
   +\langle\surfSpeed_0\nu_0,\refNu_A-\matchNu_A\rangle
   -(\overline{\surfSpeed}_A-\widetilde{\surfSpeed}_A),
\]
where $\overline{\surfSpeed}_A$ and $\widetilde{\surfSpeed}_A$ are the scalar speeds $\flowSpeed$ of the two surfaces.
The first term is $O_\delta(\tau)$ in $C^j$ since $\partial_\tau d=O_\delta(\tau)$.
On the fixed support of $\phi$, the embeddings differ by $d(A)\phi\nu_0=O_\delta(\tau^2)$ in $C^{j+4}$.
The unit normal depends smoothly on the first derivatives of the embedding and the speed on its derivatives through order four.
Thus the second and third terms are $O_\delta(\tau^2)$ in $C^j$.
After differentiating in $\tau$ at fixed cap coordinates, the first term is $O_\delta(1)$ because $\partial_\tau^2d=O_\delta(1)$ and $\partial_\tau\refNu_A$ is bounded.
The other two derivatives are $O_\delta(\tau)$ in $C^0$, since the embeddings differ by $O_\delta(\tau^2)$ in $C^4$ and their $\tau$ derivatives by $(\partial_\tau d)\phi\nu_0=O_\delta(\tau)$ in $C^4$.
This gives $\partial_\tau\overline{\refDef}_A=O_\delta(1)$ on $\supp\phi$.
Outside $\supp\phi$ the defects agree.
Since $\widetilde{\refDef}_A=O_\delta(\tau)$ and $\tau=A^4/(4\mu)$, we obtain $\overline{\refDef}_A=O_\delta(A^4)$, with zero defect on the neck.
Embeddedness follows from \cref{def:compactification-hypotheses}(i).
\end{proof}

We next record the uniform estimate on the doubled cap used in \cref{lem:fixed-cap-linear}.

\begin{lemma}[Uniform biharmonic estimate on the cap double]\label{lem:app-fixed-cap-linear}
Under the hypotheses of \cref{lem:fixed-cap-linear}, the biharmonic initial-value problem stated there has a unique solution satisfying \eqref{eq:fixed-cap-linear-bound} and \eqref{eq:outer-short-interval-interpolation}, uniformly in the interval and cap neighbourhood specified there.
\end{lemma}
\begin{proof}
Let $\bar g$ be the extended metric of the unperturbed cap.
In the equivalent $H^k$ inner product induced by $(1-\Delta_{\bar g})^{k/2}$, the operator $B=1+\tfrac12\Delta_{\bar g}^2$ is positive self-adjoint with domain $H^{k+4}$ and $B\ge1$.
Thus $B$ has maximal $L^2$ regularity on the half-line by de Simon's theorem \cite[\S3]{deSimon1964}.
The unweighted case of \cite[Proposition~3.1(b) and Theorem~3.2]{PrussSimonett2004} gives the solution of $\partial_s v+Bv=h$ with continuous trace in $(H^k,H^{k+4})_{1/2,2}=H^{k+2}$ and a bound in terms of $h$ and $v(0)$.

Take $v(0)=w_a$ and $h(s)=e^{-s}f(a+s)$ for $0\le s<b-a$, extended by zero thereafter.
Then $w(a+s)=e^s v(s)$ solves the problem for $\bar g$ on $I$.
Since $e^s\le e$ there, the bound is uniform in $I$.
The fixed cutoffs give $g_*\to\bar g$ in $C^{k+3}$ as the cap perturbation tends to zero in $C^{k+8}$.
In the fixed charts,
\[
 \|(\Delta_{g_*}^2-\Delta_{\bar g}^2)w\|_{L^2(I;H^k)}
 \le C\|g_*-\bar g\|_{C^{k+3}}\|w\|_{L^2(I;H^{k+4})}.
\]
Shrinking the cap neighbourhood, a Neumann series gives \eqref{eq:fixed-cap-linear-bound} for $g_*$ with at most twice the constant for $\bar g$.
To obtain the factor $|I|^{1/4}$ in \eqref{eq:outer-short-interval-interpolation}, square the spatial interpolation inequality $\|w\|_{H^{k+3}}\le C\|w\|_{H^{k+2}}^{1/2}\|w\|_{H^{k+4}}^{1/2}$ and integrate over $I$.
Cauchy--Schwarz gives $\int_I\|w\|_{H^{k+4}}\,\dd t\le |I|^{1/2}\|w\|_{L^2(I;H^{k+4})}$, as required.
\end{proof}

\begin{lemma}[Scale derivatives of the neck radius]\label{lem:app-neck-scale-derivatives}
Let $U$ satisfy \cref{def:normalised-conical-profile}, and put $D_A=A\partial_A$.
For $r_{\rm in}(z,A)=AU(z/A)$ and $\zeta=z/A$, differentiation at fixed $z$ gives
\begin{equation}\label{eq:neck-radius-scale-derivatives}
 \begin{aligned}
 D_A^i\partial_z^j r_{\rm in}
   &= A^{1-j}(1-j-\zeta\partial_\zeta)^i U^{(j)}(\zeta),\\
 |D_A^i\partial_z^j r_{\rm in}|&\le C_j(A+|z|)^{1-j},\qquad 0\le i\le2,\quad j\ge0 .
 \end{aligned}
\end{equation}
The constants are independent of $A>0$ and $z\in\R$.
\end{lemma}
\begin{proof}
The identity follows by the chain rule.
Smoothness of $U$ bounds the differentiated expressions for bounded $\zeta$, while \eqref{eq:general-profile-tail} gives the stated powers for large $|\zeta|$.
Multiplying these bounds by $A^{1-j}$ gives the estimate.
\end{proof}

\begin{proposition}[Volume estimates uniform in the central radius]\label{prop:app-uniform-volume-chart}
Suppose the flow preserves volume in the sense of \cref{def:volume-preserving-flow} and $C_0$ satisfies both parts of \cref{def:compactification-hypotheses}, with volume variation $\phi$ as in part~(ii).
For the corrected reference family of \cref{prop:fixed-interface-matching}, write
\[
 W_A(u)=\Vol(\refX_A+u\refNu_A)-V_0,
\]
where scale derivatives $D_A=A\partial_A$ hold $u$ fixed on $M$.
There are constants $A_0,r>0$ such that the following statements hold uniformly for $0<A\le A_0$ and $\|u\|_{\relNorm(A)}<r$.
The volume difference is given by
\begin{equation}\label{eq:volume-normal-jacobian}
 W_A(u)=\int_{\refSurf_A}\int_0^{u(p)}J_A(p,t)\,\dd t\,\dd\mu_A(p),
 \qquad J_A(p,t)=\det(I+t\,\dd\refNu_A(p)),
\end{equation}
where $\dd\refNu_A(p)$ acts on the tangent plane.
If $\dd\mu_A=\rho_A\,\dd\lambda$ on $M$ for a fixed smooth positive measure $\dd\lambda$, then
\begin{equation}\label{eq:volume-functional-derivative-bounds}
 \bigl\|D_A^i D_u^q W_A(u)\bigr\|
 \le C\int_M\ell_A\rho_A\,\dd\lambda\le C_\delta,
 \qquad 1\le i+q\le2 ,
\end{equation}
where height derivatives have the operator norms induced by \eqref{eq:relative-height-norm}.
There are also uniform constants $r_{\rm vol},\theta_*,C>0$ such that every symmetric height with $\|h\|_{\relNorm(A)}<r_{\rm vol}$ has a unique correction $\Theta_A(h)\in(-\theta_*,\theta_*)$ satisfying \eqref{eq:uniform-volume-chart}.
The corrected graph is embedded and agrees with the graph of $h$ off $\supp\phi$.
The correction is jointly $C^2$ in $(-\log A,h)$ and satisfies \eqref{eq:uniform-volume-chart-bound}, with $\Theta_A(0)=0$.
\end{proposition}
\begin{proof}
By \eqref{eq:outer-volume-amplitude}, $\refX_A$ differs from $\matchX_A$ by a cap variation of size $O_\delta(A^8)$ in $C^2$.
Hence the tube and curvature bounds of \cref{def:compactification-hypotheses}(i) hold for $\refX_A$.
Since $W_A(0)=0$, integration along the outward normal segments gives \eqref{eq:volume-normal-jacobian}.
In particular $J_A=1+b_1t+b_2t^2$, with $b_1=\tr(\dd\refNu_A)$ and $b_2=\det(\dd\refNu_A)$.
The curvature bound gives $|b_j|\le C\ell_A^{-j}$.
Thus for a fixed sufficiently small $r>0$ the normal tube contains $|t|\le r\ell_A$ and $J_A$ is uniformly bounded there.
For $\|h\|_{\relNorm(A)}\le r$,
\[
 |W_A(h)|\le C\|h\|_{\relNorm(A)}
     \int_{\refSurf_A}\ell_A\,\dd\mu_A
 \le C_V\|h\|_{\relNorm(A)}
\]
by the area estimate in \cref{rem:reference-area-defect} and $\ell_A\le\delta$, with $C_V$ independent of $A$.
Put $c_\phi=\int_{C_0}\phi\,\dd\mu_{g_0}\ne0$.
On $\supp\phi$ the caps converge smoothly and $\ell_A\asymp\delta$.
Hence $\int_{\refSurf_A}\phi\,\dd\mu_A\to c_\phi$, $\|\phi\|_{\relNorm(A)}$ is uniformly bounded, and the polynomial for $J_A$ gives
\[
 \left|\partial_\theta W_A(h+\theta\phi)-\int_{\refSurf_A}\phi\,\dd\mu_A\right|
 =\left|\int_{\refSurf_A}\phi\bigl(J_A(p,h+\theta\phi)-1\bigr)\,\dd\mu_A\right|
 \le C_\delta\bigl(\|h\|_{\relNorm(A)}+|\theta|\bigr).
\]
We may therefore choose $A_0,r_*,\theta_*>0$ so that whenever $0<A\le A_0$, $\|h\|_{\relNorm(A)}\le r_*$ and $|\theta|\le\theta_*$, the graph of $h+\theta\phi$ lies in the tube and $\partial_\theta W_A(h+\theta\phi)$ has the sign of $c_\phi$ and magnitude at least $m=|c_\phi|/2$.
Choose
\[
 r_{\rm vol}=\min\{r_*,\,m\theta_*/(2C_V)\}.
\]
Then $|W_A(h)|\le m\theta_*/2$.
Since $|\partial_\theta W_A|\ge m$ with constant sign, the values $W_A(h-\theta_*\phi)$ and $W_A(h+\theta_*\phi)$ lie on opposite sides of $W_A(h)$, each at distance at least $m\theta_*$.
Strict monotonicity gives a unique root $\Theta_A(h)$, and the mean value theorem gives $|\Theta_A(h)|\le |W_A(h)|/m$.
Injectivity of the normal tube gives embeddedness.

It remains to bound two derivatives uniformly as $A\downarrow0$.
On the neck we use \eqref{eq:neck-radius-scale-derivatives}, proved independently of the volume constraint in \cref{lem:app-neck-scale-derivatives}.
Write $r=r_{\rm in}$ and $S=S[r]$ as in \eqref{eq:profile-geometry}.
The two eigenvalues of $\dd\refNu_A$ are $-r''/S^3$ and $1/(rS)$ up to a common sign, and the area density is $rS$.
Since $r\asymp A+|z|\asymp\ell_A$, the product rule gives $|D_A^i b_j|\le C\ell_A^{-j}$ and $|D_A^i(rS)|\le C rS$ for $i\le2$.
On the collars and the fixed cap the same bounds follow from \eqref{eq:collar-inner-expansion}, \eqref{eq:collar-outer-taylor}, \eqref{eq:outer-terminal-normal-variation} and \eqref{eq:outer-volume-amplitude}.

Put $b_0=1$.
The preceding bounds on the curvatures and the area density give $|D_A^i(b_j\rho_A)|\le C\ell_A^{-j}\rho_A$ for $0\le i,j\le2$.
Integrating \eqref{eq:volume-normal-jacobian} in $t$ gives a cubic polynomial in the height $u$.
Holding $u$ and the height variations $v,w$ fixed on $M$, we obtain
\[
\begin{aligned}
 D_A^i W_A(u)
   &=\sum_{j=0}^2\int_M D_A^i(b_j\rho_A)\frac{u^{j+1}}{j+1}\,\dd\lambda,
       &&0\le i\le2,\\
 D_A^i D_uW_A(u)[v]
   &=\sum_{j=0}^2\int_M D_A^i(b_j\rho_A)u^jv\,\dd\lambda,
       &&0\le i\le1,\\
 D_u^2W_A(u)[v,w]
   &=\sum_{j=1}^2 j\int_M b_j\rho_A u^{j-1}vw\,\dd\lambda .
\end{aligned}
\]
For $\|h\|_{\relNorm(A)}\le r_*$ and $|\theta|\le\theta_*$, the height $u=h+\theta\phi$ satisfies $|u|\le C\ell_A$ by the uniform relative bound on $\phi$.
For $v,w$ of relative norm at most one, \eqref{eq:relative-height-norm} gives $|v|,|w|\le\ell_A$.
In each line the $j$th term contains $j+1$ factors from $u,v,w$, so its integrand is bounded by $C\ell_A^{-j}\ell_A^{j+1}\rho_A=C\ell_A\rho_A$.
Integration gives \eqref{eq:volume-functional-derivative-bounds}, using $\ell_A\le\delta$ and the area bound in \cref{rem:reference-area-defect}.

To differentiate the volume constraint, set $\sigma=-\log A$, $x=(\sigma,h)$, $F(x,\theta)=W_{e^{-\sigma}}(h+\theta\phi)$ and $\Theta(x)=\Theta_{e^{-\sigma}}(h)$.
The polynomial integral has coefficients smooth in $A>0$, so $F$ is jointly $C^2$.
Its first two derivatives are bounded by \eqref{eq:volume-functional-derivative-bounds}, since $\partial_\sigma=-D_A$ and differentiation in $\theta$ takes the height variation $v=\phi$, of uniformly bounded relative norm.
At $\theta=\Theta(x)$ we have $|F_\theta|\ge m$, so the implicit function theorem gives joint $C^2$ dependence of $\Theta$.
Let subscripts $X,Y$ denote derivatives in constant directions in the variable $x$.
Differentiating $F(x,\Theta(x))=0$ once and twice gives
\[
\begin{aligned}
 F_\theta\Theta_X&=-F_X,\\
 F_\theta\Theta_{XY}
   &=-F_{XY}-F_{X\theta}\Theta_Y-F_{Y\theta}\Theta_X
      -F_{\theta\theta}\Theta_X\Theta_Y .
\end{aligned}
\]
Take $X,Y$ to be the scale direction or height directions of relative norm one.
If $K$ bounds the first and second partial derivatives of $F$ in these directions and in $\theta$, the first identity gives $|\Theta_X|\le K/m$.
Substituting this in the second gives $|\Theta_{XY}|\le (K/m)(1+K/m)^2$.
Both constants are independent of $A$, proving \eqref{eq:uniform-volume-chart-bound}.
\end{proof}

\begin{lemma}[Uniform change between normal and radial coordinates]\label{lem:app-normal-radial-coordinates}
Under the hypotheses of \cref{lem:radial-normal-gauge}, there is a threshold $\varepsilon_{J'}>0$ with the stated dependence such that the coordinate changes in \eqref{eq:radial-normal-exact-chart} exist uniquely and satisfy the relative $C^1$ bounds stated there.
They preserve smoothness, have first variations \eqref{eq:radial-normal-linearisation}, and satisfy $x_A'>1/2$ on $J'$.
\end{lemma}
\begin{proof}
Write $r=\refRad_A$ and $(\refNu_z,\refNu_r)=\refNu_A$ in the meridian plane, suppressing $A$ on the two components, so that $\refNu_z=-r'\refNu_r$ and $\refNu_r^2=(1+r'^2)^{-1}$.
By \eqref{eq:neck-radius-scale-derivatives} on the inner neck and the matching estimates \eqref{eq:collar-inner-expansion} and \eqref{eq:collar-outer-taylor} on the collars,
\[
 |r'|+(A+|z|)\bigl(|r''|+|\partial_z\refNu_A|\bigr)\le C .
\]
Given $q$, the axial component of \eqref{eq:radial-normal-exact-chart} gives $x=z+u\refNu_z$.
Its radial component requires
\[
 K(z,u):=(r+q)(z+u\refNu_z)-r(z)-u\refNu_r=0 .
\]
Choose $\varepsilon_{J'}=c\min\{1,\dist(J',\partial J)/(A_0+4\delta)\}$, where $c>0$ is a sufficiently small constant depending only on the reference geometry.
For $z\in J'$ and $|u|\le2\varepsilon_{J'}(A+|z|)$ the shift satisfies
\[
 |x-z|\le |u|\le 2\varepsilon_{J'}(A_0+4\delta)\le \tfrac12\dist(J',\partial J) .
\]
Thus $x\in J$.
Along the segment from $z$ to $x$ the weight $A+|\cdot|$ is comparable to $A+|z|$, so the bound on $r''$ gives $|r'(x)-r'(z)|\le C|u|/(A+|z|)$.
At fixed $z$ we have $K(z,0)=q(z)$ and
\[
 K_u=(r'+q')(x)\refNu_z-\refNu_r
     =-1/\refNu_r+\bigl(r'(x)-r'(z)+q'(x)\bigr)\refNu_z .
\]
Since $|1/\refNu_r|\ge1$, choosing $c$ small gives $|K_u|\ge1/2$ on this interval of heights.
It follows by monotonicity that there is exactly one root there, between $-2|q(z)|$ and $2|q(z)|$.
At fixed $u$,
\[
 K_z=(r'+q')(x)(1+u\refNu_z')-r'(z)-u\refNu_r' .
\]
Hence $u'=-K_z/K_u$ is bounded by $C(\|q'\|_\infty+|u|/(A+|z|))$.
The implicit function theorem also gives smoothness for smooth data.

Conversely, a normal height $u$ gives
\[
 x_A(z)=z+u(z)\refNu_z(z),\qquad
 q(x_A(z))=r(z)-r(x_A(z))+u(z)\refNu_r(z).
\]
In both directions $|x_A'-1|\le C\|u\|_{A,J'}\le C'\varepsilon_{J'}$, using the bound just proved when $u=\toNormal(q)$.
We choose $c$ small enough that $x_A'>1/2$, so $x_A$ is invertible on its image.
Differentiation there gives
\[
 (q'\circ x_A)x_A'=r'(z)-r'(x_A)+(u\refNu_r)'-r'(x_A)(u\refNu_z)' .
\]
The bound on $r'(x_A)-r'(z)$ gives the asserted estimate for $q'$, and $|q(x_A)|\le\|r'\|_\infty|x_A-z|+|u|\le C|u|$ gives the height bound.
For a variation $\varepsilon q$, differentiating $K(z,u)=0$ at $\varepsilon=0$ gives $\dot u=q\refNu_r$ and $\dot x_A=\dot u\refNu_z=q\refNu_r\refNu_z$.
These are the two expressions in \eqref{eq:radial-normal-linearisation}.
\end{proof}

\subsection{The uniform global-height inverse}
\label{sec:appC-height-inverse}

To solve the linearised equation for a normal height $\globalHeight$ over the matched family, we first solve separate problems on the neck and on the cap.
We write $v_{\globalHeight}$ for the inner representative of $\globalHeight$ in \eqref{eq:global-height-representatives}.
It is important to distinguish $v_{\globalHeight}$ from the whole-line solution used to construct the inverse, since localising the equation for $\globalHeight$ produces a forcing term on the collar.

We fix a speed $\flowSpeed$ and an admissible neck $U$, $\delta$, the cap neighbourhood, $k\geq10$ and the weights in \eqref{eq:mixed-weight-choice}.
By the base clock \eqref{eq:outer-base-clock}, $\dd t_*=A^4\dd s$ and $|I_*|=A_0^4/(4\mu)$.
Constants below are uniform as $A_0/\delta\downarrow0$.

The inner norm $\neckPath$ in \eqref{eq:mixed-inner-space} uses whole-line conical boxes.
Its principal-derivative term \eqref{eq:principal-spatial-dissipation} satisfies
\begin{equation}
 \coneDiss_4(v)\asymp
       \|\partial_\zeta^4v\|_{\coneSrc_{\infty,1}^k}.
 \label{eq:appC-enhanced-inner-norm}
\end{equation}
Indeed, the straightening factors in \eqref{eq:dilation-straightening} are uniformly bounded on each unit time block.
The estimate for this term is \eqref{eq:app-summable-principal-dissipation}.
The physical component of the global norms \eqref{eq:mixed-product-spaces} uses the cap double $M_o$.
If $M_o$ has two components, the physical operator acts on each separately and the squared Sobolev norm is the sum of the two squared norms.
Thus a common bound in \eqref{eq:outer-cap-forward-estimate} on the two components gives the same estimate on their disjoint union.
The weight $\eta_o^4$ vanishes near each artificial seam, so $\eta_o^4 \globalHeight$ extends by zero there.
For a fixed physical annulus $C$, write $\capPath(I_*;C)$ for the space in \eqref{eq:mixed-outer-space} with its spatial Sobolev norms taken on $C$.
The cutoffs in \eqref{eq:global-parametrix-cutoffs} satisfy $\beta_j\alpha_j=\alpha_j$ and
\begin{equation}
 (1-\beta_i)(1-\beta_o)=0,\qquad
 \eta_o^4=1\ \hbox{on }\supp\alpha_o,\qquad
 \kappa_i\beta_i=\beta_i .
 \label{eq:appC-cutoff-identities}
\end{equation}
The transition annuli of $\beta_i$ and $\beta_o$ are separated from the corresponding source and initial-trace supports.
Write $C_{\rm tr}=\{\delta/8<|z|<\delta/4\}$ for the physical annuli containing the transition of $\eta_o$.
Choose fixed open annuli $C_{\beta_j}\Subset C_{\beta_j}'$, $j=i,o$, containing $\supp\partial_z\beta_j$ in $C_{\beta_j}$, with
\[
 C_{\beta_i}'\Subset\{5\delta/4<|z|<2\delta\},\qquad
 C_{\beta_o}'\Subset\{\delta/4<|z|<\delta\}.
\]
These are subsets of the physical neck chart, including its angular coordinate.

\begin{lemma}[Fourth-power physical localisation]
\label{lem:appC-fourth-power-commutator}
Let $P_A^0$ be a local fourth-order physical parabolic operator with time coefficient one.
Suppose that in the fixed charts on $C_{\rm tr}$ its spatial coefficients have $C^k$ norms bounded uniformly for $t_*\in I_*$ and all sufficiently small $A_0$.
For a smooth symmetric height $h:M\times I_*\to\R$ vanishing at the centre, let $v_h$ be its canonical inner representative from \eqref{eq:global-height-representatives}.
For every $\epsilon>0$,
\begin{align}
 \|[P_A^0,\eta_o^4]h\|_{H^k}
 \leq{}&\epsilon\|\eta_o^4 h\|_{H^{k+4}}
   +C_\epsilon\bigl(\|\eta_o^4 h\|_{H^k}
                   +\|h\|_{H^k(C_{\rm tr})}\bigr).
 \label{eq:appC-weighted-commutator}
\end{align}
If $P_A^0$ agrees on $K_{\delta/8}^{(1)}$ with an operator $P_o$ on the double satisfying \eqref{eq:outer-cap-forward-estimate} on $I_*$, and $v_h\in\neckPath$, then for sufficiently small $A_0$,
\begin{align}
 \|\eta_o^4 h\|_{\capPath}\leq C\bigl\{
 &\|\eta_o^4 P_A^0h\|_{L^2H^k}
 +\|\eta_o^4 h(t_0)\|_{H^{k+2}}
 +A_0^2\|v_h\|_{\neckPath}\bigr\}.
 \label{eq:appC-weighted-physical-estimate}
\end{align}
The same physical cutoff occurs in the path, source and trace norms.
\end{lemma}

\begin{proof}
In a fixed meridian chart, let $u$ be a smooth compactly supported function and put $a_j=\|\eta_o^j\partial_z^ju\|_{L^2}$.
Integration by parts and the bounded derivatives of $\eta_o$ give
\[
 a_j^2\leq a_{j-1}a_{j+1}+C_ja_{j-1}a_j,\qquad 1\leq j\leq3.
\]
Young's inequality and finite induction imply $a_j\leq\epsilon a_4+C_\epsilon a_0$ for $1\leq j\leq3$.
Moreover $|\partial_z^q(\eta_o^4)|\leq C_q\eta_o^{4-q}$ for $q\leq4$.
Differentiating the commutator through order $k$ puts at most $k$ derivatives on each spatial coefficient, so each resulting coefficient is uniformly bounded.
Terms with more than $k$ derivatives on $h$ are bounded by the weighted intermediate derivatives above, applied to derivatives of $h$ through order $k$.
Terms with at least four derivatives on the cutoff involve at most $k$ derivatives of $h$.
Differentiating $\eta_o^4 h$ four times and absorbing the same intermediate terms replaces the top weighted derivative by $\|\eta_o^4 h\|_{H^{k+4}}$ and the lower norms.
This proves \eqref{eq:appC-weighted-commutator}.
The time commutator is zero.

The height $\eta_o^4h$ extends by zero to the double, and its support lies where $P_o$ agrees with $P_A^0$.
Apply \eqref{eq:outer-cap-forward-estimate} to $P_o(\eta_o^4 h)=\eta_o^4 P_A^0h+[P_A^0,\eta_o^4]h$ and absorb the first term on the right of \eqref{eq:appC-weighted-commutator}.
Since $\|\eta_o^4 h\|_{L^2H^k}\le |I_*|^{1/2}\|\eta_o^4 h\|_{C(\overline I_*;H^k)}$, this term is absorbed for small $A_0$ by \eqref{eq:outer-base-clock}.
On $C_{\rm tr}$ we have $\kappa_i=1$, so \eqref{eq:global-height-representatives} and the conical part of $\|v_h\|_{\neckPath}$ give a uniform physical $H^k$ bound for $h$.
Integration over $I_*$ contributes $|I_*|^{1/2}=O(A_0^2)$.
This proves \eqref{eq:appC-weighted-physical-estimate}.
Smooth approximation extends it to the completed spaces.
\end{proof}

\begin{lemma}[Localisation and outer lifts]
\label{lem:appC-global-localisation}
Multiplication by a fixed smooth even physical cutoff whose first derivative is supported on fixed annuli away from the centre preserves $\neckPath$ uniformly.
If $\chi$ is a smooth cutoff supported on fixed physical annuli away from the centre and $w\in\capPath$, both rotationally and reflection-symmetric, then
\begin{equation}
 \|(J_i^A)^{-1}(\chi w)\|_{\neckPath}
 \leq C_\delta\|w\|_{\capPath}.
 \label{eq:appC-outer-inner-lift}
\end{equation}
For a rotationally and reflection-symmetric physical source $G$ supported on such annuli,
\begin{align}
 \|e^{\eta s}Q_i^AG\|_{L^2_s X_k}
 &\leq C_\delta A_0^{5/2}\|G\|_{L^2H^k},\nonumber\\
 \|Q_i^AG\|_{\coneSrc_{\infty,1}^k}
 &\leq C_\delta A_0^{3/2}\|G\|_{L^2H^k}.
 \label{eq:appC-fixed-annulus-source-conversion}
\end{align}
\end{lemma}

\begin{proof}
For a physical cutoff $\chi$, write $\chi^A(\zeta,s)=\chi(A(s)\zeta)$ as in the cutoff convention following \eqref{eq:global-parametrix-cutoffs}.
The material derivative satisfies
\[
 \matD(\chi^Av)=\chi^A\matD v.
\]
In the straightened coordinates \eqref{eq:dilation-straightening} on $I_n$, the multiplier is $\chi(A(n)y)$, independent of time.
On a box meeting the support of a cutoff derivative, $\langle\rho\rangle\asymp_\delta A(n)^{-1}$.
In the $j$th summand of the graph norm \eqref{eq:conical-parabolic-graph-norm}, let $q$ be the total number of spatial derivatives on the cutoff.
The remaining derivatives of the height fit the summand with index $i=\max\{0,j-q\}$, with at most the same number of extra derivatives as in the original summand.
The weights coincide for $q=0$, and for $q\ge1$ their ratio times the cutoff factor is $\langle\rho\rangle^{j-i}A(n)^q\le C_\delta$.
This proves boundedness on the slab graph and source spaces, and multiplication of a trace lift gives the same bound for the initial trace.
The product rule also gives boundedness on the weighted Hilbert spaces in \eqref{eq:mixed-inner-space}, whose spatial derivatives all have the same weight.

For the summable fourth derivatives in \eqref{eq:appC-enhanced-inner-norm}, differentiate the product through order $j+4$, $0\le j\le k$.
The term with all derivatives on $v$ is bounded by its fourth-derivative norm.
If $1\le q\le j+4$ derivatives fall on the cutoff, the remaining $j+4-q$ fit the graph summand $i=\max\{0,j+1-q\}\le k$ with at most three extra derivatives.
The resulting factor is
\[
 \langle\rho\rangle^{j-i}A(n)^q
 \le C_\delta A(n)^{q+i-j}\le C_\delta A(n).
\]
Thus the cutoff terms are summable since $\sum_n A(n)=A_0/(1-e^{-\mu})$.
The same count bounds the spatial commutator in the source $\neckEq(\chi^Av)$ of \eqref{eq:mixed-inner-space}, since each derivative on an operator coefficient gains an inverse power of $\langle\rho\rangle$ by \eqref{eq:app-end-remainder-symbol-bounds}.
The normalising term in $\neckEq(\chi^Av)-\chi^A\neckEq v$ is $(\chi(0)-\chi^A)\ell(v)DU$, supported away from the centre and bounded as in \eqref{eq:appC-normalising-tail-bound}.

For \eqref{eq:appC-fixed-annulus-source-conversion}, spatial scaling gives
\[
 \|\partial_\zeta^jQ_i^AG(s)\|_{H^0_2}^2
 \leq C_\delta A^9\|G(t_*(s))\|_{H^k}^2,\qquad j\leq k-2 .
\]
Since $\dd t_*=A^4\dd s$ and $2\eta<\mu$, integration gives the first bound.
The image of $I_n$ under the reference clock \eqref{eq:outer-base-clock} is $t_*(I_n)=[t_*(n),t_*(n+1)]$, and the conical source norm on this block satisfies
\[
 \|Q_i^AG\|_{\slabSrc_\infty^k(I_n)}
 \leq C_\delta A(n)^{3/2}\|G\|_{L^2(t_*(I_n);H^k)}.
\]
Cauchy--Schwarz and $\sum_{n\ge0}A(n)^3=A_0^3/(1-e^{-3\mu})$ give the second.

For the lift $v=(J_i^A)^{-1}(\chi w)$, \eqref{eq:fixed-interface-injections} gives
\[
 v(\zeta,s)=\frac{\chi(A\zeta)w(A\zeta,t_*(s))}{A\refNuRad(A\zeta)},\qquad A=A(s).
\]
Its Hilbert path components are bounded by $C_\delta A_0^{1/2}\|w\|_{\capPath}$, using the first bound in \eqref{eq:appC-fixed-annulus-source-conversion} for the material derivative.
In the conical graph norm, the zeroth and first extra spatial derivatives are controlled by the physical $H^{k+2}$ trace and Sobolev embedding.
The remaining derivatives are controlled by that trace or by the physical $L^2H^{k+4}$ norm, with nonnegative scaling powers.
The graph-domain source is the conversion of $P_A^0(\chi w)$ and is covered by \eqref{eq:appC-fixed-annulus-source-conversion}.
Finally the term in $\partial_\zeta^4v$ with all four derivatives on $\chi w$ is $A^3\partial_z^4(\chi w)/\refNuRad$, evaluated at $(z,t_*)=(A\zeta,t_*(s))$.
The second bound in \eqref{eq:appC-fixed-annulus-source-conversion} gives a summable norm of at most $C_\delta A_0^{3/2}\|w\|_{\capPath}$.
\end{proof}

We construct a right inverse using the separated cutoffs in \eqref{eq:global-parametrix-cutoffs}.
The linearised equation \eqref{eq:linear-global-operator} couples the physical operator $P_A^0$ to the normalising functional \eqref{eq:projected-pincher-operator}.
Its radial form on the exact profile is \eqref{eq:global-linear-conjugacy}.
On every fixed outer set, the reference variation \eqref{eq:outer-terminal-normal-variation} and the collar estimate \eqref{eq:fixed-match-geometric} give
\begin{equation}
 \|K_A\|_{H^k}\leq C_\delta,\qquad
 \|K_A\ell(v)\|_{L^2(I_*;H^k)}
 \leq C_\delta A_0^2\|v\|_{\neckPath}.
 \label{eq:appC-speed-cross-bound}
\end{equation}
The free trace is the second component of $L\globalHeight$ in \eqref{eq:fully-defined-section}.

\begin{proposition}[Global right parametrix]
\label{prop:appC-global-parametrix}
Let $U$ be an admissible neck for $\flowSpeed$, let the reference family be that of \cref{prop:fixed-interface-matching}, and suppose $C_0$ satisfies the parabolic cap condition of \cref{def:outer-cap-parabolic}.
There is a uniformly bounded map $\mathbf P:\mixY_{\rm eq}\times\freeTr \to\mixX$ such that
\begin{equation}
 L\mathbf P=I+\mathbf E,\qquad
 \|\mathbf E\|\leq C_\delta A_0^2 .
 \label{eq:appC-global-right-error}
\end{equation}
Its trace error is zero.
\end{proposition}

\begin{proof}
For target $(F,b)$ put
\[
 f_i=Q_i^A(\alpha_iF),\qquad
 a_i=(J_i^{A_0})^{-1}(\alpha_i b).
\]
Both are controlled by the global source and quotient trace norms.
Choose the compact controlled coordinates by
\begin{equation}
 B_{\rm c}c=-\Pc a_i-
 \int_0^\infty e^{-s\waistLin_{\rm c}}\Pc f_i(s)\,\dd s.
 \label{eq:appC-global-controlled-choice}
\end{equation}
Solve $\neckEq v=f_i$ with stable initial component $\Ps(a_i+\Psi c)$ and the controlled future condition \eqref{eq:linear-stable-future-conditions}.
Its full initial trace is exactly $a_i+\Psi c$.
The weighted estimate \eqref{eq:linear-convective-halfline-estimate} and the conical estimate \eqref{eq:conical-halfline-summable-estimate}, which includes the fourth-derivative term, give
\begin{equation}
 \|v\|_{\neckPath}+|c|
 \leq C\|(F,b)\|_{\mixY}.
 \label{eq:appC-raw-core-green-bound}
\end{equation}
Here we apply the conical half-line estimate directly to the \emph{full} initial trace $a_i+\Psi c$, so we need no separate bound for the Riesz projection on the conical trace space.
The strict gap $\eta>\gamma_{\rm c}$ controls the future integral, including the polynomial factors of Jordan blocks.

On the cap double, use the operator $P_o$ of \cref{def:outer-cap-parabolic}(ii) to solve
\begin{equation}
 P_ow=\alpha_oF-(1-\beta_i)K_A\ell(v),\qquad
 w(t_0)=\alpha_ob .
 \label{eq:appC-global-outer-reference}
\end{equation}
The sources vanish near the artificial seam and are extended by zero.
Their physical norms are controlled because $\eta_o^4=1$ on $\supp\alpha_o$ and on $\supp(1-\beta_i)$.
The outer estimate \eqref{eq:outer-cap-forward-estimate} and \eqref{eq:appC-speed-cross-bound} give $\|w\|_{\capPath}\leq C\|(F,b)\|_{\mixY}$.
Define
\begin{equation}
 \globalHeight=\mathbf P(F,b):=\beta_iJ_i^Av+\beta_ow .
 \label{eq:appC-global-parametrix}
\end{equation}

The inner norm of this height is bounded by \cref{lem:appC-global-localisation}.
For its physical norm we apply \cref{lem:appC-fourth-power-commutator} to $\beta_iJ_i^Av$.
On $C_{\rm tr}$ the reference radius is $AU(z/A)$ with $|z|\ge\delta/8$.
The differentiated expansion \eqref{eq:general-profile-tail} bounds its spatial derivatives and those of its reference velocity $\beta_0\partial_A[AU(z/A)]=-\mu A^{-3}DU(z/A)$ uniformly there, since $\partial_\zeta^jDU=O(|\zeta|^{-3-j})$.
The linearisation of the height operator \eqref{eq:global-height-coefficients} at zero therefore has the required $C^k$ coefficient bounds, and \cref{def:outer-cap-parabolic}(ii) gives the forward estimate.
The weighted source $\eta_o^4\beta_i\alpha_iF$ and the initial trace $\eta_o^4\alpha_i b$ are controlled by the target and trace norms.
On the transition of $\beta_i$, the raw inner field has zero initial trace and its only local source is $-K_A\ell(v)$.
Interior physical maximal regularity on nested annuli therefore gives
\begin{equation}\label{eq:appC-inner-cutoff-annulus}
 \|J_i^Av\|_{\capPath(I_*;C_{\beta_i})}
 \leq C_\delta A_0^2\|v\|_{\neckPath}.
\end{equation}
Indeed the speed source has that bound by \eqref{eq:appC-speed-cross-bound}.
For the lower $L^2(I_*;H^k)$ remainder, the factors $A^{1-j}$ in the physical derivatives of $A v(z/A,s)$ on the fixed annulus are comparable to the weights $\langle\zeta\rangle^{j-1}$ in \eqref{eq:conical-parabolic-graph-norm}.
We therefore obtain a uniform physical $H^k$ bound from \eqref{eq:fixed-interface-injections} and a factor $|I_*|^{1/2}=O(A_0^2)$ upon integration in time.
This bounds the commutator with $\beta_i$.
On the transition of $\eta_o^4$ one has $\beta_i=1$, so the lower term in \eqref{eq:appC-weighted-physical-estimate} is supplied directly by the inner conical norm.
This proves $\|\eta_o^4\beta_iJ_i^Av\|_{\capPath}\leq C\|(F,b)\|_{\mixY}$.
The outer contribution $\eta_o^4\beta_ow$ is bounded directly.
Consequently $\|\mathbf P(F,b)\|_{\mixX}\leq C\|(F,b)\|_{\mixY}$.

The initial trace is
\[
 \globalHeight(t_0)=\alpha_ib+\alpha_ob+J_i^{A_0}\Psi c
       =b+J_i^{A_0}\Psi c,
\]
so its free part is exactly $b$.

It remains to compute the equation error.
Near the centre the reconstructed height is $J_i^Av$, so its speed functional is exactly $\ell(v)$.
Equations \eqref{eq:global-linear-conjugacy} and \eqref{eq:appC-global-outer-reference} give
\begin{align*}
 L_{\rm eq}\globalHeight={}&F+
 [1-\beta_i-\beta_o(1-\beta_i)]K_A\ell(v)\\
 &+[P_A^0,\beta_i]J_i^Av+[P_A^0,\beta_o]w .
\end{align*}
The bracket is $(1-\beta_i)(1-\beta_o)=0$ by \eqref{eq:appC-cutoff-identities}.
Thus the exact equation error is
\begin{equation}
 E(F,b)=[P_A^0,\beta_i]J_i^Av+[P_A^0,\beta_o]w.
 \label{eq:appC-global-error-identity}
\end{equation}
The inner commutator has order at most three on $C_{\beta_i}$, so \eqref{eq:appC-inner-cutoff-annulus} bounds its $L^2H^k$ norm by $C_\delta A_0^2\|v\|_{\neckPath}$.
On a neighbourhood of the transition of $\beta_o$ both sources in \eqref{eq:appC-global-outer-reference} and its initial datum vanish.
Interior physical maximal regularity gives
\[
 \|w\|_{\capPath(I_*;C_{\beta_o})}
 \leq C_\delta\|w\|_{L^2(I_*;H^k(C_{\beta_o}'))}
 \leq C_\delta A_0^2\|w\|_{\capPath}.
\]
All commutators are spatial and of order at most three.
Hence
\[
 \|E(F,b)\|_{L^2H^k}\leq C_\delta A_0^2\|(F,b)\|_{\mixY} .
\]
Their supports are fixed physical annuli where $\eta_o^4=1$.
Applying \eqref{eq:appC-fixed-annulus-source-conversion} controls the inner target as well, and proves $\|E(F,b)\|_{\mixY_{\rm eq}}\leq C_\delta A_0^2\|(F,b)\|_{\mixY}$.
There is no trace error.
Thus $\mathbf E(F,b)=(E(F,b),0)$ in $\mixY_{\rm eq}\times\freeTr$.
This is \eqref{eq:appC-global-right-error}.
\end{proof}

The right parametrix proves surjectivity for small $A_0$.
To prove injectivity, we estimate the canonical inner representative of a height in the kernel.

\begin{lemma}[Coupled core localisation]
\label{lem:appC-global-core-localisation}
Let $v$ be the canonical inner representative of $\globalHeight$ in \eqref{eq:global-height-representatives}.
Then
\begin{equation}
 \neckEq v
 =\kappa_i^AQ_i^A(L_{\rm eq}\globalHeight)
   +Q_i^A[P_A^0,\kappa_i]\globalHeight
   +(1-\kappa_i^A)\ell(v)DU .
 \label{eq:appC-global-core-localisation}
\end{equation}
The last two terms are well-defined in $\neckSrc$, and
\begin{align}
 \|Q_i^A[P_A^0,\kappa_i]\globalHeight\|_{\neckSrc}
 &\leq C_\delta A_0^{3/2}\|\eta_o^4 \globalHeight\|_{\capPath},
 \label{eq:appC-global-collar-coupling}\\
 \|(1-\kappa_i^A)\ell(v)DU\|_{\neckSrc}
 &\leq C_\delta A_0^4\|v\|_{\neckPath}.
 \label{eq:appC-normalising-tail-bound}
\end{align}
\end{lemma}

\begin{proof}
The product rule and \eqref{eq:global-linear-conjugacy} give the identity.
The cutoff is one near the centre, so the speed functional is unchanged; comparing its full $DU$ term with its localised term produces the displayed tail.
The commutator is supported where $\eta_o^4=1$ and has order at most three, so \eqref{eq:appC-fixed-annulus-source-conversion} gives the first estimate.
On the tail, $|\zeta|\geq c_\delta/A(s)$ and the expansion \eqref{eq:general-profile-tail} gives $\partial_\zeta^j(DU)=O(|\zeta|^{-3-j})$.
The weighted Hilbert tail norm is $O_\delta(A(s)^{9/2})|\ell(v)(s)|$.
The $n$th conical source block is $O_\delta(A(n)^4)e^{-\eta n}\|v\|_{\neckPath}$.
Integration and block summation give the second estimate.
\end{proof}

\begin{lemma}[Trivial kernel]
\label{lem:appC-global-injectivity}
For an admissible neck, suppose the physical operator $P_A^0$ agrees on $K_{\delta/8}^{(1)}$ with an operator $P_o$ on the double satisfying the forward estimate \eqref{eq:outer-cap-forward-estimate} on $I_*$.
For $A_0/\delta$ sufficiently small, $\ker L=\{0\}$.
\end{lemma}

\begin{proof}
Let $L_{\rm eq}\globalHeight=0$ and $\freeProj \globalHeight(t_0)=0$.
Write $\globalHeight(t_0)=J_i^{A_0}\Psi c$, and put $I=\|v_{\globalHeight}\|_{\neckPath}$ and $O=\|\eta_o^4 \globalHeight\|_{\capPath}$.
The full initial inner trace is $\Psi c$.
At $s=0$ the controlled future formula \eqref{eq:linear-stable-future-conditions} for \eqref{eq:appC-global-core-localisation} bounds $|B_{\rm c}c|$, and hence $|c|$, by the source norm.
Applying \eqref{eq:linear-convective-halfline-estimate} and \eqref{eq:conical-halfline-summable-estimate} with this full trace, and bounding the two source terms by \eqref{eq:appC-global-collar-coupling} and \eqref{eq:appC-normalising-tail-bound}, gives
\begin{equation}
 I\leq C_\delta A_0^{3/2}O+C_\delta A_0^4I.
 \label{eq:appC-global-inner-kernel-estimate}
\end{equation}
The physical equation is $P_A^0\globalHeight=-K_A\ell(v_{\globalHeight})$.
Its weighted physical initial trace is zero, because the controlled lift is supported where $\eta_o^4=0$.
By \eqref{eq:appC-weighted-physical-estimate} and \eqref{eq:appC-speed-cross-bound},
\[
 O\leq C_\delta A_0^2I .
\]
Substitution into \eqref{eq:appC-global-inner-kernel-estimate}, followed by reduction of $A_0$, gives $I=O=0$.
Since the two cutoffs cover the physical surface, $\globalHeight=0$.
\end{proof}

\begin{corollary}[Uniform coupled global inverse]
\label{cor:appC-uniform-fibre-inverse}
Under the hypotheses of \cref{prop:appC-global-parametrix}, for $A_0/\delta$ sufficiently small,
\[
 L:\mixX\longrightarrow
 \mixY_{\rm eq}\times\freeTr
\]
is an isomorphism with uniformly bounded inverse.
The same bound controls $\|\ell(v_{\globalHeight})\|_{\speedSpace}$.
\end{corollary}

\begin{proof}
The right error bound \eqref{eq:appC-global-right-error} gives the right inverse $\mathbf P(I+\mathbf E)^{-1}$ by a Neumann series.
\Cref{lem:appC-global-injectivity} makes it the two-sided inverse.
The trace estimate \eqref{eq:app-waist-trace-bound} and the time weight in \eqref{eq:mixed-inner-space} control $\ell(v_{\globalHeight})$ in $\speedSpace$.
\end{proof}

We can now see how a variation on the cap enters the equation on the neck.
Even when the initial variation is supported on the cap, the commutator on the collar in \eqref{eq:appC-global-core-localisation} can force the inner representative $v_{\globalHeight}$.
The full localisation identity gives the linear part of the forcing in \eqref{eq:future-functional-identity}.
The separate whole-line problem for $v$ enters only in the construction of the parametrix.

\subsection{Nonlinear estimates, selection and geometric consequences}
\label{sec:appC-nonlinear-completion}

\begin{lemma}[Uniform elimination of the scale speed]
\label{lem:appC-waist-speed}
Evaluating the height equation \eqref{eq:global-height-residual} at the centre defines a $C^2$ map $\Theta_\beta(\globalHeight)$ into $\speedSpace$, determined by $v_{\globalHeight}$ and its first four spatial derivatives at $\zeta=0$, on a uniform small ball in $\mixX$.
It satisfies
\[
 \Theta_\beta(0)=0,\qquad
 D\Theta_\beta(0)\globalHeight=\ell(v_{\globalHeight}),\qquad
 \|\Theta_\beta(\globalHeight)\|_{\speedSpace}\leq C\|\globalHeight\|_{\mixX}.
\]
\end{lemma}

\begin{proof}
At the centre, $\globalHeight=0$ and its time derivative is zero.
Thus the scalar equation uses only the spatial jets at the centre.
In the radial residual normalisation its derivative in the speed variable is exactly one: $\partial_A\refRad_A(0)=1$.
The trace embedding for $k\geq10$ controls the required fourth-order jets by the inner Hilbert trace.
Their smooth coefficient maps are $C^2$ on the fixed small core neighbourhood.
The pointwise implicit function theorem is uniform in $s$; multiplying its difference estimates by $e^{\eta s}$ gives the stated bounds.
Differentiation at zero gives $D\Theta_\beta(0)\globalHeight=(\radLin v_{\globalHeight})(0)$.
\end{proof}

\begin{proposition}[Mixed-space nonlinear estimates]
\label{prop:appC-mixed-nemytskii}
Let $U$ be an admissible neck for $\flowSpeed$, use the reference family of \cref{prop:fixed-interface-matching}, and suppose the cap satisfies \cref{def:outer-cap-parabolic}(iii).
Consider the normal-height evolution residual \eqref{eq:global-height-residual}, whose physical time coefficient after eliminating the scale speed is the scalar $1-\Theta_\beta(\globalHeight)/\mu$.
The resulting global section is $C^2$ from its small $\mixX$ ball to $\mixY$, and its quadratic remainder $\mixRem$ satisfies \eqref{eq:section-quadratic-estimate}.
The constants are uniform for fixed $\delta$.
\end{proposition}

\begin{proof}
We apply the nonlinear neck condition of \cref{def:neck-nonlinear-estimates} to the normal remainder $\radRem^\nu$, since $v=v_{\globalHeight}$ in \eqref{eq:global-height-representatives} is a linear coordinate for normal height.
Where $\kappa_i=1$, the exact normal-coordinate equation \eqref{eq:normal-neck-material-equation} has nonlinear source
\begin{equation}\label{eq:appC-normal-neck-nonlinearity}
 \Pi\radRem^\nu(v)
 -\vartheta\{C_U^\nu(v)-DU\}
 +\frac{\vartheta}{\mu}\matD v,\qquad
 \vartheta=\ell(v)+\radRem^\nu(v)(0).
\end{equation}
In particular $\radRem^\nu$ is the remainder of the combined normal speed and reference dilation in \eqref{eq:normal-neck-remainder}.
The linearisation of that combined expression is $\radLin$, as shown in \cref{lem:app-normal-neck-chart}.
The reference dilation accounts for the change of coordinates.

For the weighted Hilbert source of $\Pi\radRem^\nu(v)$, use the whole-line norms \eqref{eq:neck-weighted-line-norms}.
The lower norm $r_n$ is bounded by the conical graph norm $G_n$ of \eqref{eq:conical-slab-norms}.
The trace embedding \eqref{eq:app-time-trace-map}, applied after the dilation \eqref{eq:dilation-straightening}, bounds the higher norm on $I_n$ by
\[
 K(h)\le C\bigl(\|h\|_{L^2(I_n;Y_k)}
              +\|\matD h\|_{L^2(I_n;X_k)}\bigr),
\]
using \eqref{eq:app-dilation-equivalence} and \eqref{eq:app-material-pullback}.
Thus $\{e^{\eta n}K(h)\}_{n\ge0}$ is square summable by \eqref{eq:mixed-inner-space}, while $\sup_n r_n(h)$ is finite.
Substitution in \eqref{eq:neck-first-variation} and \eqref{eq:neck-second-variation} gives the weighted Hilbert estimates for the first two variations.

For the conical source, \cref{cor:app-summable-spatial-remainder} sums the same two variation bounds using \eqref{eq:neck-ordinary-box-norms} on $|\rho|\le\widehat R_n$ and \eqref{eq:neck-end-box-norms} on $|\rho|>\widehat R_n$.
On the former boxes the factor $\widehat R_n^{k+2}e^{-\eta n}$ is summable by \eqref{eq:slow-conical-radius}.
On the latter, \eqref{eq:mixed-inner-space} controls the sum of the principal fourth derivatives and the lower terms gain $\widehat R_n^{-1}$.
The corollary includes the projection at the centre, so \eqref{eq:app-summable-spatial-remainder} and \eqref{eq:app-summable-spatial-variations} apply directly to $\Pi\radRem^\nu$.

For the geometric scale term, \eqref{eq:app-normal-neck-coefficient-bound} gives
\[
 H_n\bigl(\vartheta\{C_U^\nu(v)-DU\}\bigr)
 \le C e^{-\eta n}\|\vartheta\|_{\speedSpace}G_n(v).
\]
The material-clock term satisfies the same bound because the definition of $G_n(v)$ controls $\matD v$.
Summing these estimates gives the conical source bounds.
In the Hilbert source, multiplication by $\vartheta$ gives the product of its uniform bound and the weighted Hilbert graph norm of $v$; the geometric term is first order with bounded decaying coefficients.
For the variations of $C_U^\nu$, the inner graph norm controls $r_n(h)$ of \eqref{eq:neck-slab-conical-size}, and $Z_\eta(h)$ controls $e^{\eta n}\|h\|_{L^2(I_n;H^{k-1}_2)}$.
Thus the mixed bound \eqref{eq:app-normal-coefficient-second-variation} places one variation in the time supremum and the other in the weighted time integral.
Together with the graph-to-source estimates of \cref{lem:app-normal-neck-chart} and the speed bounds of \cref{lem:appC-waist-speed}, the product rule gives the first two variations of both clock terms.
On the transition of $\kappa_i$, all coordinates lie in the fixed collar and the cutoff derivatives are uniformly bounded in physical coordinates.
There the corresponding local products are covered by the physical estimate below and the conversion \eqref{eq:appC-fixed-annulus-source-conversion}.

For the physical part, apply the local coefficient bounds and \eqref{eq:cap-local-tame-bounds} from \cref{def:outer-cap-parabolic}(iii).
On the region $\eta_o^4=1$, the $C(I_*;H^{k+2})$ trace controls every lower factor in that estimate, while the $L^2(I_*;H^{k+4})$ norm controls its single higher factor.
On the transition of $\eta_o^4$, write the local differentiated expressions as products of jets with the single weight $\eta_o^4$.
If a factor has order above $k$, there is at most one such factor: $k\geq10$, while the original expression has order four.
All other jets are bounded in $L^\infty$ by the conical lower-jet norm on the transition of $\eta_o^4$, and by the physical trace norm where $\eta_o^4=1$.
Allocate the factor $\eta_o^4$ to that highest derivative.
For a smooth height or height variation $h$ and $1\leq j\leq4$, weighted interpolation as in \cref{lem:appC-fourth-power-commutator} gives
\begin{equation}\label{eq:appC-weighted-high-derivative}
 \|\eta_o^4\nabla^{k+j}h\|_{L^2}
 \leq C\bigl(\|\eta_o^4 h\|_{H^{k+4}}
             +\|h\|_{H^k(C_{\rm tr})}
             +\|\eta_o^4 h\|_{H^k}\bigr).
\end{equation}
Here $\eta_o^4\leq\eta_o^j$.
This controls each product in $L^2(I_*;H^k)$, including those in the first three height derivatives of \eqref{eq:cap-local-tame-bounds}.
In particular no assertion that the full spatial coefficient is a uniform $C^k$ multiplier is required.
It remains to estimate the time and scale terms in \eqref{eq:global-height-residual}.
After subtracting their linearisation, they are
\[
 -\frac{\Theta_\beta(\globalHeight)}{\mu}\partial_{t_*}\globalHeight
 +\{\Theta_\beta(\globalHeight)-\ell(v)\}K_A
 +A^{-3}\Theta_\beta(\globalHeight)\{V_A(\globalHeight)-V_A(0)\}.
\]
The first term has physical source norm at most
\[
 C\|\Theta_\beta(\globalHeight)\|_{L^\infty}
   \|\partial_{t_*}(\eta_o^4\globalHeight)\|_{L^2H^k},
\]
since the cutoff is independent of $t_*$.
For the second, \cref{lem:appC-waist-speed} gives a quadratic scalar remainder, and $K_A$ is bounded by \eqref{eq:appC-speed-cross-bound}.
For the last term, the normal graph formula gives
\[
 V_A(\globalHeight)-V_A(0)=
 -\left\langle(\partial_A\refX_A)^\top
       +\globalHeight\,\partial_A\refNu_A,\,
       (I+\globalHeight\,\dd\refNu_A)^{-1}\nabla\globalHeight
  \right\rangle .
\]
On $\supp\eta_o$, the quantities $A^{-3}\partial_A\refX_A$ and $A^{-3}\partial_A\refNu_A$ have uniformly bounded spatial derivatives in the required orders.
Indeed, $A^{-3}\partial_A=\mu^{-1}\partial_\tau$ on the cap family \eqref{eq:outer-terminal-normal-variation}, and \eqref{eq:collar-inner-expansion} and \eqref{eq:fixed-match-geometric} give the same bounds on the collars; the optional volume displacement satisfies these bounds by \eqref{eq:outer-volume-amplitude}.
Thus $A^{-3}\{V_A(\globalHeight)-V_A(0)\}$ is a smooth first-order expression vanishing at zero.
The weighted product estimate \eqref{eq:appC-weighted-high-derivative} bounds this expression in the physical source norm, and multiplication by $\Theta_\beta(\globalHeight)$ makes the term quadratic.
The same calculation after one or two height differentiations, together with \cref{lem:appC-waist-speed}, proves the required derivative bounds for all three terms.

The free trace is linear.
All residuals are global physical fields, so there is no additional nonlinear compatibility condition or change of target space.
The component estimates therefore give $\|D^2\mixMap(\globalHeight)[K_1,K_2]\|_{\mixY} \leq C\|K_1\|_{\mixX}\|K_2\|_{\mixX}$.
Integrating along a line segment proves \eqref{eq:section-quadratic-estimate}.
\end{proof}

The next estimate gives the dependence on the cap parameter used in the contraction argument.

\begin{lemma}[Dependence of the height equation on the cap]
\label{lem:appC-cap-parameter-dependence}
Let $U$ be an admissible neck for $\flowSpeed$, suppose the cap satisfies \cref{def:compactification-hypotheses}(i) and \cref{def:outer-cap-parabolic}, and let its smooth variations vanish near the neck and be measured in $C^{k+9}$ in the fixed charts.
On the common ball used for the contraction in the proof of \cref{thm:terminal-mixed-inverse}, the map $\heightFix$ is jointly $C^1$ in the height, the cap parameter and the free trace.
Its fixed point has the same parameter dependence.
\end{lemma}
\begin{proof}
To vary the cap, use the fixed charts of \cref{def:compactification-hypotheses}.
The cap variations vanish near the neck, so the inner coordinates, the weights on the overlap and the free-trace projection stay fixed.
On the outer cap $\surfSpeed_0\nu_0$ in \eqref{eq:outer-terminal-normal-variation} uses four spatial derivatives of $X_0$.
For the uncorrected reference family $d(A)=0$.
If the volume-preserving alternative in \cref{prop:fixed-interface-matching} is used, the scalar volume equation in \cref{lem:appC-reference-defect-volume} has a nonzero derivative in the correction amplitude, so $d(A)$ again depends $C^1$ on the cap.
On the fixed outer region $K_{\delta/16}$ the resulting maps in these charts
\[
 X_0\in C^{k+9}\ \longmapsto\ \refX_A\in C^{k+5}
       \ \longmapsto\ \refNu_A\in C^{k+4}
\]
are $C^1$ uniformly in $A$.
Passing from $\refX_A$ to its normal costs one spatial derivative, and the operator $\flowSpeed$ differentiates the full graph $\refX_A+\globalHeight\refNu_A$ four more times.
The term $A^{-3}\partial_A\refX_A$ and its counterpart for the normal in \eqref{eq:global-height-coefficients} obey the same bounds on the support of the cap variation, since $A^{-3}\partial_A=\mu^{-1}\partial_\tau$ and \eqref{eq:outer-terminal-normal-variation} is affine in $\tau$.
Thus the coefficients of the cap residual depend $C^1$ on this parameter in the norms used for \eqref{eq:section-quadratic-estimate}.
The uniform inverse bound \eqref{eq:fibre-product-linear-estimate} and $D(L^{-1})=-L^{-1}(DL)L^{-1}$ give the same dependence for $L^{-1}$.
The map $\heightFix$ is jointly $C^1$ in height, cap and free trace, with contraction constant at most $1/2$ on the common ball.
The parameter-dependent contraction theorem \cite[Theorem~1.174]{Chicone1999} gives the asserted dependence of the fixed point.
\end{proof}

To select the data at fixed volume we must estimate how a variation of the outer cap changes the controlled neck trace.

\begin{lemma}[Small outer influence on the controlled trace]
\label{lem:appC-off-diagonal-outer}
Let $\phi$ be a fixed smooth symmetric physical outer height supported outside the matching region.
Under the hypotheses of \cref{thm:terminal-mixed-inverse}, the nonlinear solution map satisfies
\[
 \|D_b\Phi_\varepsilon(0)p_\phi\|_{\Ec}
 \leq C_\delta(A_0^2+A_0^6)\|\phi\|_{H^{k+2}}=o(1).
\]
\end{lemma}

\begin{proof}
Apply the right parametrix to $(0,p_\phi)$, with zero equation source and free trace $p_\phi=(0,\phi)$ as in \cref{lem:fully-coupled-selection}.
Since $\phi$ is supported outside the matching region, the raw inner source and initial trace are both zero.
Thus $v=0$ and $c=0$ in \eqref{eq:appC-global-controlled-choice}, while the outer field is the homogeneous cap solution with initial datum $\phi$.
The reconstructed initial height is precisely $\phi$, so its controlled trace is zero.
The parameter norm is controlled by $\|\phi\|_{H^{k+2}}$: use this outer solution, multiplied by $\beta_o$, as a physical path lift and apply \cref{lem:appC-global-localisation}.
Since
\[
 L^{-1}=\mathbf P(I+\mathbf E)^{-1},
 \qquad \|\mathbf E\|\leq C_\delta A_0^2,
\]
the controlled trace of the base-linear response $L^{-1}(0,p_\phi)$ has norm at most $C_\delta A_0^2\|\phi\|_{H^{k+2}}$.

The nonlinear solution at zero free trace is generally nonzero.
Let $\globalHeight_*$ be the selected solution at zero free trace, and put $L_*=D\mixMap(\globalHeight_*)$.
The solution bound \eqref{eq:coupled-height-speed-bound} at $b=0$ gives $\|\globalHeight_*\|_{\mixX}\leq C_\delta A_0^6$.
By \cref{prop:appC-mixed-nemytskii}, $\|L_*-L\|\leq C_\delta A_0^6$.
The uniform inverse and a Neumann series give a uniform bound for $L_*^{-1}$, and
\[
 L_*^{-1}-L^{-1}
 =L_*^{-1}(L-L_*)L^{-1}.
\]
The derivative of the solution map in the direction $p_\phi$ is $L_*^{-1}(0,p_\phi)$.
Taking its controlled initial trace proves the additional $O_\delta(A_0^6)$ error and the stated estimate.
\end{proof}

We now solve the coupled selection equations on a neighbourhood independent of the initial central radius.

\begin{lemma}[Uniform solution of the selection equations]
\label{lem:appC-uniform-joint-selection}
Under the hypotheses of \cref{lem:fully-coupled-selection}, use the volume-corrected reference family of \cref{prop:fixed-interface-matching}.
For sufficiently small $A_0/\delta$, there are $\rho,r>0$, independent of $A_0$, such that every $h\in\volComp$ with $\|h\|_{\compTr}<\rho$ determines a unique solution $|(a,\lambda)|<r$ of \eqref{eq:joint-future-volume-system}.
The solution depends $C^1$ on $h$ and, for smooth variations measured in $C^{k+9}$ in the fixed charts, jointly $C^1$ on $h$ and the cap.
\end{lemma}
\begin{proof}
The equations \eqref{eq:joint-future-volume-system} use the solution map of \cref{thm:terminal-mixed-inverse}, before imposing either selection equation.
Their derivatives at the origin are computed in \eqref{eq:joint-selection-jacobian} using \cref{lem:appC-off-diagonal-outer} and the first variation of volume.
Put $\varepsilon=A_0/\delta$ and let $J$ denote the Jacobian in \eqref{eq:joint-selection-jacobian} and $J_0=\diag(B_{\rm c},c_\phi)$.
The limit $J_0$ is invertible.
For small $\varepsilon$, $\|J_0^{-1}(J-J_0)\|\le1/2$, so $\|J^{-1}\|\le2\|J_0^{-1}\|$.
The proposed height at $(h,a,\lambda)=(0,0,0)$ is zero and its volume is $V_0$ by \cref{prop:fixed-interface-matching}.
The evolution with zero free trace may still require a neck correction.
Indeed, at the origin the first equation has residual $-\Pc v_\varepsilon[0](\cdot,0)=O_\delta(A_0^6)$ by \eqref{eq:coupled-height-speed-bound}.

The volume estimate \eqref{eq:volume-functional-derivative-bounds} measures the initial normal height $u_0$ in the relative norm, so we need to bound this norm in terms of its trace $\tau$.
For any path with initial trace $\tau$, \eqref{eq:global-height-representatives} and Sobolev embedding bound the relative height at $s=0$ by the inner conical trace where $\kappa_i=1$ and by the $H^{k+2}$ cap trace where $\eta_o^4=1$.
Taking the infimum over lifts in \eqref{eq:compatible-trace-norm} gives $\|u_0\|_{\relNorm(A_0)}\le C_\delta\|\tau\|_{\compTr}$.

Write $x=(a,\lambda)$ and $\traceVol^\varepsilon=(\traceVol_{\rm c}^\varepsilon,\traceVol_V^\varepsilon)$.
The first two height derivatives in \eqref{eq:volume-functional-derivative-bounds} therefore bound the derivatives of the volume row.
For the first row, differentiating $\mixMap(\globalHeight)=(0,b)$ gives $D_b\globalHeight=[D\mixMap(\globalHeight)]^{-1}(0,\,\cdot\,)$.
The uniform inverse estimate \eqref{eq:fibre-product-linear-estimate} and the bound on $D^2\mixMap$ in \cref{lem:section-quadratic-remainder} control the variation of $D_b\globalHeight$ on the solution ball.
On a fixed small ball $|x|\le r$ we can thus arrange $\|I-J^{-1}D_x\traceVol^\varepsilon(h,x)\|\le1/2$.
Taking $A_0$ and $\|h\|_{\compTr}$ small gives $|J^{-1}\traceVol^\varepsilon(h,0)|\le r/2$.
The map $x\mapsto x-J^{-1}\traceVol^\varepsilon(h,x)$ is a contraction of this ball.
Its unique fixed point solves \eqref{eq:joint-future-volume-system}.
There the Jacobian is invertible, so the implicit function theorem gives $C^1$ dependence on $h$ and on the cap in the topology of \cref{thm:terminal-mixed-inverse}.
\end{proof}

We next estimate the meridian curvature and the separation of distinct meridian arcs needed for the normal tubes in \cref{sec:dumbbell-pinchoff}.
On the neck we use only the rounded profile and the conical collar.
For a T-type cap we also need a separation estimate for the annular return, which stays away from the axis.
For a D-type cap we use the same neck estimates and extend the tube over the two polar discs in \cref{sec:dumbbell-pinchoff}.
The tube estimates apply first to the matched meridian before volume correction.
A sufficiently small $C^2$ normal variation supported on the outer region stays in the same cap neighbourhood and leaves the neck fixed, so the estimates also apply to the volume-corrected meridian $\refMer_A$.
For a T-type cap, write $\refMer_A:\Sphere^1\to\R\times(0,\infty)$ for the entire closed meridian, $\refRad_A$ for its radial coordinate and $\refNu_A$ for its outward unit normal in the meridian plane.
On the graphical neck we use the axial coordinate $z$, recovering $\refMer_A(z)=(z,\refRad_A(z))$ from \cref{lem:radial-normal-gauge}.
We pull back $\ell_A$ to the parameter circle by $\refMer_A$.

The scale $\ell_A$ in \eqref{eq:local-geometric-scale} is comparable to $A+|z|$ on the neck.
Separation of the remaining arc from the apex makes $\ell_A$ comparable to $\delta$ there.
Its defining formula extends to an ambient function with a uniform Lipschitz constant.

\begin{lemma}[Scale-adapted geometry of the matched family]
\label{lem:appC-scaled-matched-geometry}
There are constants $c,C>0$, independent of sufficiently small $A$, such that on the graphical core
\begin{align}
 \refRad_A(z)&\geq c\ell_A(z),
 \label{eq:appC-radius-lower-bound}\\
 |\partial_z^q\refRad_A(z)|
 &\leq C_q\ell_A(z)^{1-q},
 \qquad q\geq1.
 \label{eq:appC-scaled-derivative-bound}
\end{align}
On the outer arc the corresponding arclength derivatives are uniformly bounded.
Every subarc of intrinsic length at most $c\ell_A(z)$ has normal injectivity radius at least $c\ell_A(z)$.
\end{lemma}

\begin{proof}
For $|z|\leq\delta/2$, $\refRad_A(z)=AU(z/A)$.
Positivity and the differentiated conical expansion \eqref{eq:general-profile-tail} imply
\[
 U(\zeta)\geq c(1+|\zeta|),
 \qquad
 |\partial_\zeta^qU(\zeta)|
 \leq C_q(1+|\zeta|)^{1-q}.
\]
Scaling gives \eqref{eq:appC-radius-lower-bound}-- \eqref{eq:appC-scaled-derivative-bound}.
The collar estimate \eqref{eq:fixed-match-geometric} gives the same bounds on the matching annulus.
On the outer arc they follow from the admissible cap bounds and smooth convergence of the approximate family.

The curvature of the meridian is at most $C/\ell_A$.
The inverse function theorem for the normal exponential map therefore gives a local normal tube of radius $c\ell_A$ on a subarc of the same size.
The constant is uniform in the three rescaled limits: the rounded positive profile, a straight conical ray, and a compact piece of the outer cap.
\end{proof}

We still have to exclude intersections between normal fibres based on different local arcs.
For this we need a lower bound on the Euclidean distance between points that are separated in arclength along the meridian.

\begin{lemma}[Nonadjacent separation]
\label{lem:appC-nonadjacent-separation}
There are $c_1,c_2>0$ such that, if $p,q\in\Sphere^1$ and the shorter meridian subarc joining $\refMer_A(p)$ and $\refMer_A(q)$ has length at least
\[
 c_1\min\{\ell_A(p),\ell_A(q)\},
\]
then
\begin{equation}
 |\refMer_A(p)-\refMer_A(q)|
 \geq c_2\min\{\ell_A(p),\ell_A(q)\}.
 \label{eq:appC-global-separation}
\end{equation}
The constants are uniform for caps in the fixed admissible neighbourhood.
The arclength threshold $c_1$ may be chosen smaller than the local-tube constant in \cref{lem:appC-scaled-matched-geometry}.
\end{lemma}

\begin{proof}
We first make the uniformity in the outer-cap parameter explicit.
Enlarge the reference outer arc by the two protected overlap subarcs and denote the result by $\gamma_o^*$.
This is a compact embedded $C^2$ arc and hence has positive reach, say
\[
 \operatorname{reach}(\gamma_o^*)=r_o>0.
\]
The admissible caps are $C^2$ normal graphs over $\gamma_o^*$ with norm at most a fixed sufficiently small multiple of $r_o$.
Write $\gamma_o$ for the corresponding outer arc of a varying cap.
The normal-coordinate map for $\gamma_o^*$ and the inverse function theorem then give
\begin{equation}
 \operatorname{reach}(\gamma_o)\geq r_o/4
 \label{eq:appC-uniform-outer-reach}
\end{equation}
for every admissible cap.
One direct way to see the global part is to write the perturbed normal tube in the reference tube.
Its base component is a $C^1$ perturbation of the identity of norm less than $1/4$, while its normal component is strictly monotone for $|\sigma|<r_o/4$; hence two such normal fibres cannot meet.

For a compact arc $\gamma$ of reach at least $r>0$, the metric distortion theorem of Boissonnat, Lieutier and Wintraecken \cite[Theorem~1]{BoissonnatLieutierWintraecken2019} bounds the intrinsic distance $d_\gamma$ by
\begin{equation}
 d_\gamma(p,q)
 \leq 2r\arcsin\!\left(\frac{|p-q|}{2r}\right)
 \quad\text{whenever } |p-q|<2r.
 \label{eq:appC-reach-chord-arc}
\end{equation}
Apply this with $r=r_o/4$ on the outer scale.
It gives a positive uniform Euclidean separation for every pair whose outer arclength separation is at least a fixed multiple of $\delta$.
The same reference tube includes the protected overlaps; consequently an outer-only point cannot approach a nonadjacent core point without violating the injectivity used in \eqref{eq:appC-uniform-outer-reach}.

Suppose now that \eqref{eq:appC-global-separation} fails.
Then there are central radii $A_j$ and parameter points $p_j,q_j\in\Sphere^1$ satisfying the arclength condition for which
\[
 \frac{|\refMer_{A_j}(p_j)-\refMer_{A_j}(q_j)|}
      {\min\{\ell_{A_j}(p_j),\ell_{A_j}(q_j)\}}\longrightarrow0 .
\]
Interchanging the points if necessary, we may assume $\ell_{A_j}(p_j)\le\ell_{A_j}(q_j)$.
We rescale about $\refMer_{A_j}(p_j)$ by $1/\ell_{A_j}(p_j)$.
When this point is on the neck, write $z_j$ for its axial coordinate.
After passing to a subsequence we consider the following cases.

If $|z_j|/A_j$ remains bounded, the rescaled meridians converge in $C^2_{\rm loc}$ to a compact subarc of the positive rounded profile.
That profile is a proper embedded graph and distinct nonadjacent compact subarcs have positive separation.

If $|z_j|/A_j\to\infty$ while $|z_j|/\delta\to0$, the rescaled meridians converge to a straight conical ray.
Points on the same ray have Euclidean separation equal to arclength separation.
Points on the reflected rays have axial separation comparable to $|z_j|$ unless both lie in the preceding rounded-core regime.
Thus \eqref{eq:appC-global-separation} holds in this case as well.

If the smaller scale is comparable to $\delta$, compactness gives convergence to two points on a compact portion of the terminal meridian away from the apex, including the protected neck subarcs.
Its positive reach and the uniform separation between nonadjacent compact subarcs give a contradiction.
Finally, if the two points lie in different regimes, either their physical $z$-coordinates differ by a fixed multiple of the smaller scale or one lies in the outer set and the other outside the adjacent overlap.
The admissible cap separation then supplies the same bound.
These alternatives exhaust the loop and prove the lemma.
\end{proof}

In the proof of \cref{prop:t-cap-verification} we combine this separation with the local tubes to obtain a single injective normal tube on the T-type surface.

\begin{corollary}[Embeddedness under the geometric cap condition]
\label{cor:appC-analytic-embeddedness}
Suppose the reference family satisfies the tube and curvature bounds of \cref{def:compactification-hypotheses}(i).
There is $\varepsilon_*>0$, independent of sufficiently small $A_0$, such that every normal height with the representatives \eqref{eq:global-height-representatives} satisfying \eqref{eq:embeddedness-smallness} gives an embedded surface diffeomorphic to the reference surface.
Its neck has positive radius and is a graph in the axial coordinate.
\end{corollary}

\begin{proof}
By hypothesis the reference surface has an injective normal tube.
We first bound the relative norm \eqref{eq:relative-height-norm} to show that the perturbed surface lies in this tube.
On the neck we must also show that the radius stays positive and that the axial coordinate increases along the displaced meridian.
Where $\kappa_i=1$, \eqref{eq:global-height-representatives} gives $\globalHeight(z,t_*)=A\refNuRad(z)v(z/A,s)$.
The height and slope terms of $\coneSize_{1,R}$ give $A|v|\le C\ell_A\coneSize_{1,R}$ and $|\partial_\zeta v|\le\coneSize_{1,R}$.
Since $|\refNuRad|+\ell_A|\partial_z\refNuRad|\le C$ by \eqref{eq:radial-normal-linearisation} and the scaled profile bounds, we obtain
\[
 |\globalHeight|+\ell_A|\nabla \globalHeight|
 \le C\coneSize_{1,R}(v;s)\ell_A(z).
\]
Where $\eta_o^4=1$ we have $\globalHeight=w$ and $\ell_A$ is bounded below independently of $A$.
These two regions cover the surface by \eqref{eq:global-parametrix-cutoffs}.
Thus \eqref{eq:embeddedness-smallness} makes the relative norm at most $C_\delta\varepsilon_*$.
For small $\varepsilon_*$ the graph lies in the injective tube of \cref{def:compactification-hypotheses}(i), where normal projection inverts the graph map.

On $|z|\le2\delta$, positivity and the conical expansion \eqref{eq:general-profile-tail} give $\refRad_A=AU(z/A)\ge c_r\ell_A$ for some $c_r>0$.
On the matching collars the same lower bound, with a possibly smaller $c_r$, follows from \eqref{eq:fixed-match-geometric} and the positive cone radius there.
Hence the corrected radius on the neck is at least $\refRad_A-|\globalHeight|\ge c_r\ell_A/2$.
Writing $\refNu_z=\langle\refNu_A,e_z\rangle$, the axial coordinate of the displaced point is $Z(z)=z+\globalHeight(z)\refNu_z(z)$.
The curvature bound and the bounded reference slope give $|\refNu_z'|\le C/\ell_A$.
Hence
\[
 |Z'-1|\le|\partial_z \globalHeight|+|\globalHeight|\,|\refNu_z'|
       \le C_\delta\varepsilon_*<1/2
\]
after reducing $\varepsilon_*$.
This expresses the positive radius as a function of $Z$ on the shifted neck interval.
\end{proof}

The remaining estimates identify the strong height solutions with classical flows and establish their smooth terminal limits.

\begin{proposition}[Classical regularity and smooth terminal limits]
\label{prop:appC-classical-terminal-regularity}
Let $U$ be an admissible neck for $\flowSpeed$, let $k\ge10$ and let $\globalHeight\in\mixX$ be a small solution of \eqref{eq:compact-invariance-equation} over the reference family of \cref{prop:fixed-interface-matching}, with smooth initial trace and $\vartheta=\Theta_\beta(\globalHeight)\in\speedSpace$, $\|\vartheta\|_{L^\infty}<\mu/2$.
Then the following hold.
\begin{enumerate}[label=\textup{(\roman*)}]
\item On every compact interval of positive central radii, the physical height equation is classical in fixed local normal charts.
\item Suppose also that the surfaces are embedded, the necks are radius graphs and the physical outer height equation is uniformly parabolic in fixed local graph charts.
Assume that $t(s)\uparrow T<\infty$ and that the radial perturbation $p(\zeta,s)=A(s)^{-1}r(A(s)\zeta,t(s))-U(\zeta)$ satisfies $\|p(\cdot,s)\|_{C^{k-2}(K)}\le C_K e^{-\eta s}$ on every compact interval $K\Subset\R$ for sufficiently large $s$.
Then $p\to0$ in $C^\infty_{\rm loc}(\R)$, and the surfaces converge smoothly in fixed physical charts away from the origin.
\end{enumerate}
\end{proposition}
\begin{proof}
\emph{Part (i).}
On each compact interval of positive central radii both changes of time are regular.
In a fixed local normal chart the height is continuous in space and time with four continuous spatial derivatives, by the inner and outer trace bounds and $k\ge10$.
More precisely those bounds and Sobolev embedding give continuity in $h^{4+\gamma}$ for $0<\gamma<1$ on a smaller chart.
The right-hand side of the physical height equation \eqref{eq:fixed-terminal-cap-height} is then continuous in $h^\gamma$.
Integrating the strong equation gives a classical time derivative.
For smooth traces, local uniqueness for the flow class of \cref{def:geometric-flow-class} identifies this with its classical solution.

\emph{Part (ii).}
Fix compact intervals $K\Subset K^\prime\Subset\R$.
The rescaled radius $U+p$ satisfies the radial equation \eqref{eq:general-normalised-evolution} in the parabolically rescaled time $\simTime$.
By \eqref{eq:dynamic-terminal-scaling} and \eqref{eq:actual-clock-ratio},
\[
 \frac{\dd\simTime}{\dd s}
 =A^{-4}\frac{\dd t}{\dd s}=\frac{1}{1-\vartheta/\mu}.
\]
For the self-similar solution generated by $U$, $\vartheta=0$, so the two times agree when both start at zero.
For the perturbed graph $\partial_{\simTime}=(1-\vartheta/\mu)\partial_s$ at fixed $\zeta$.
Subtracting the stationary equation \eqref{eq:general-profile-equation} now gives
\begin{equation}\label{eq:compact-radial-perturbation}
 (1-\vartheta/\mu)\partial_s p
 =\flowRad[U+p]-\flowRad[U]+\mu Dp-\vartheta D(U+p).
\end{equation}
On each fixed compact similarity interval $U+p$ stays positive and its slope stays bounded.
In time $\simTime$ the coefficient of $-\partial_\zeta^4p$ in \eqref{eq:compact-radial-perturbation} has a positive lower bound on each such interval, by the parabolicity in \cref{def:geometric-flow-class} and the smallness of $p$.
The spatial coefficients in the remaining terms depend smoothly on the profile and its lower derivatives.
To gain derivatives, fix $0<\gamma<1$ and start with the uniform $C^{k-3,\gamma}$ bound on a larger compact interval.
If $p$ is bounded in $C^{j+3,\gamma}$ there, differentiating the equation $j$ times in space leaves the same fourth-order principal term acting on $\partial_\zeta^jp$.
Every term on the right has a spatial $C^\gamma$ bound using derivatives through order $j+3$.
The scalar $\vartheta$ is constant in space.
The interior Schauder estimate of Dong and Zhang \cite[Theorem~2.2(I)]{DongZhang2015} therefore bounds $\partial_\zeta^jp$ in spatial $C^{4,\gamma}$ on smaller cylinders.
Since $|\vartheta|<\mu/2$, we have $2/3<\dd\simTime/\dd s<2$ on every $[n-1,n+1]$, so iteration bounds every spatial derivative on $K\times[n,n+1]$ uniformly in $n$.
For $m\ge1$, interpolation on a slightly larger compact interval gives
\[
 \|p(\cdot,s)\|_{C^m(K)}
 \le C_{m,K,K'}\|p(\cdot,s)\|_{C^0(K')}^{1/2}
                  \|p(\cdot,s)\|_{C^{2m}(K')}^{1/2}\longrightarrow0 .
\]

It remains to take the limit in physical coordinates away from the apex.
On a fixed neck annulus $0<d\le |z|\le d'$ where $\kappa_i=1$, the change of variables in \eqref{eq:physical-annulus-derivatives} gives $|\partial_z^j[Av(z/A,s)]|\le C_{d,d',j}\|v\|_{\coneGraph_\infty^k}$ for $0\le j\le k$.
The derivatives of $\refNuRad$ are bounded on this annulus, so the same conclusion holds for $\globalHeight$ by \eqref{eq:global-height-representatives}.
Where $\eta_o^4=1$ we have $\globalHeight=w$, so the $H^{k+2}$ cap trace in \eqref{eq:mixed-outer-space} gives the corresponding bounds.
On each smaller fixed graph chart away from the apex, uniform parabolicity and interior estimates then bound all spatial derivatives for $t_1\le t<T$, $t_1>t_0$.
If $g(\cdot,t)$ is the height in this fixed chart, \eqref{eq:fixed-terminal-cap-height} (with the chart's fixed reference) bounds $\partial_t g$ in $C^m$ using spatial derivatives through order $m+4$.
Thus $\|g(\cdot,t')-g(\cdot,t)\|_{C^m}\le C_m|t'-t|$ for $t_1\le t<t'<T$, which gives the smooth terminal limit.
\end{proof}

\section{The conical tail and profile certificate}
\label{app:profile-certificate}

We first give the two supporting arguments used in the construction of the conical tail in \cref{prop:finite-base-tail-chart}.
The first verifies the transformed equation and its nonlinear remainder estimates; the second establishes parameter dependence and the full differentiated expansion of the constructed profile.

\begin{lemma}[Normal form and remainder estimates]\label{lem:appD-tail-normal-form}
Fix a compact set $\paramSet\Subset(0,\infty)^2$ of pairs $(\alpha,\mu)$ and an integer $N\ge1$.
There exist $L_0>0$, $\delta_*>0$ and $K\ge1$ such that the following holds uniformly for $(\alpha,\mu)\in\paramSet$, $L\ge L_0$ and $t\ge\kappa L^{4/3}$.
In the coordinates $t,y$ of \eqref{eq:finite-base-cone-rescaling} and \eqref{eq:finite-base-tail-polynomial}, with $z=V^{-1}(y,y_t,y_{tt},y_{ttt})^T$, $z_0,z_s\in\R$ and $z_-=\overline{z_+}$, the profile equation \eqref{eq:main-profile-equation} is equivalent to \eqref{eq:finite-base-tail-normal-form}.
Its remainder satisfies \eqref{eq:finite-base-tail-remainder} and \eqref{eq:tail-remainder-difference} for $z,\widetilde z$ in this real subspace with $|z|,|\widetilde z|\le\delta_*$, and the reconstructed profile satisfies $U(\zeta)\ge(\min_{\paramSet}\alpha)\zeta/2$ there.
\end{lemma}

\begin{proof}
\emph{The linear part.}
Let $Q(\zeta)=\alpha\zeta$ and take $a_\alpha$ from \eqref{eq:cone-geometric-coefficients}.
The derivative $\coneLin=D\radOp[Q]$ of the radial operator \eqref{eq:terminal-operator} at the cone is
\begin{align}
 \coneLin v={}&
 -a_\alpha v^{(4)}
 -\frac{2a_\alpha}{\zeta}v'''
 +\frac{a_\alpha(\alpha^2-1)}{\alpha^2\zeta^2}v''
 +\frac{2a_\alpha}{\alpha^2\zeta^3}v'
 -\frac{3v}{2\alpha^2(1+\alpha^2)\zeta^4}.
 \label{eq:exact-cone-linearisation}
\end{align}
Under the rescaling \eqref{eq:finite-base-cone-rescaling}, the homogeneous equation $\coneLin v+\mu(v-\zeta v')=0$ becomes
\begin{equation}
 y^{(4)}-\frac2t y'''
 +\frac{A_2}{t^2}y''
 +\left(1-\frac{A_1}{t^3}\right)y'
 +\left(-\frac2t+\frac{A_0}{t^4}\right)y=0,
 \label{eq:exact-cone-normal-form}
\end{equation}
where
\[
 A_2=\frac{47}{8}+\frac{9}{16\alpha^2},\quad
 A_1=\frac{771}{64}+\frac{135}{64\alpha^2},\quad
 A_0=\frac{3403}{256}+\frac{873}{256\alpha^2}.
\]
The limiting characteristic polynomial is $\lambda^4+\lambda$.
Changing to its eigenbasis $V$ gives the matrix $B_1$ in \eqref{eq:finite-base-tail-remainder}.
Subtracting $P_N$ and using the recursion \eqref{eq:conical-tail-recursion} cancels the first $N$ forcing terms.
The scalar residual is $O(\zeta^{-4N-3})$; division by the coefficient of $y^{(4)}$, which is a constant multiple of $\zeta^{-1/3}$, gives $O(t^{-3N-2})$.
The remaining linear terms are $O(t^{-2}|z|)$.
Together with the quadratic estimate proved next, this gives \eqref{eq:finite-base-tail-remainder}.

\emph{The nonlinear remainder.}
Fix a small pointwise neighbourhood $|z|\le\delta_*$ and let $\alpha_{\min}=\min_{\paramSet}\alpha>0$.
Since $U=P_N+\zeta^{-5/3}y$ by \eqref{eq:finite-base-tail-polynomial} and $|y|\le 4|z|$, increasing $L_0$ ensures $U(\zeta)\ge\alpha_{\min}\zeta/2$ throughout this neighbourhood.
The denominators in the jet field $f_\mu$ of \eqref{eq:tail-jet-vector-field} are then bounded away from zero.

We need a difference estimate for $\tailRem$ in addition to its size bound \eqref{eq:finite-base-tail-remainder}.
To see the power of $t$ in the quadratic term, put $w=\zeta^{-1/3}$ and write
\begin{equation}\label{eq:tail-scaled-jet}
 (U,U',U'',U''')=(\alpha\zeta,\alpha,0,0)
       +(w^5b_0,w^4b_1,w^3b_2,w^2b_3).
\end{equation}
Since $U=P_N+w^5y(t)$ and $\partial_\zeta t=(4\kappa/3)w^{-1}$, the chain rule shows that each $b_j$ is affine in $z$, with uniformly bounded coefficients.
The coefficient of $y^{(4)}(t)$ in $U^{(4)}(\zeta)$ is $(4\kappa/3)^4w$.
The other terms from differentiating $P_N+w^5y(t)$ four times are affine in $z$.

It therefore suffices to estimate $D_z^2(\Phi_\mu/w)$, where $\Phi_\mu(\zeta,u)$ is the right-hand side of \eqref{eq:scalar-profile-field}.
In the two terms of \eqref{eq:scalar-profile-field} containing $U'''$, the jet powers in \eqref{eq:tail-scaled-jet} give
\[
 \frac{U''U'''}{w}=w^4 b_2b_3,\qquad
 \frac{U'''}{wU}=\frac{w^4b_3}{\alpha+w^8b_0}.
\]
The four curvature terms containing $(U'')^3$, $(U'')^2/U$, $U''/U^2$ and $1/U^3$ each carry a factor $w^8$ after division by $w$.
Their remaining coefficients and their first two $z$-derivatives are bounded since $w^3U\ge\alpha_{\min}/2$.
Finally, the drift in $\Phi_\mu/w$ is $2\mu(1+U'^2)^2(w^4b_0-b_1)$.
Since $w^4b_0-b_1$ is affine in $z$, at least one of the two $z$-derivatives falls on $(1+U'^2)^2$.
This gives a factor $w^4$ since $U'=\alpha+w^4b_1$.
These bounds give $\|D_z^2\tailRem(t,z)\|\le Cw^4\le C'/t$.

For a unit vector $e$ in the real coordinate space, Taylor's formula at $z=0$ gives
\[
 |D_z\tailRem(t,0)e|
 \le t\bigl(|\tailRem(t,t^{-1}e)|+|\tailRem(t,0)|\bigr)+Ct^{-2}
 \le Ct^{-2}
\]
by \eqref{eq:finite-base-tail-remainder}.
Integrating the derivative along the segment from $\widetilde z$ to $z$ gives \eqref{eq:tail-remainder-difference}.
The coefficients extend smoothly in $w=\zeta^{-1/3}$ and in $(\alpha,\mu)$ down to $w=0$.
The forcing, linear remainder and quadratic remainder have factors $w^{12N+8}$, $w^8$ and $w^4$, respectively, with smooth remaining coefficients.
Consequently their vanishing orders are unchanged by parameter differentiation.
\end{proof}

\begin{lemma}[Parameter dependence and differentiated expansion]\label{lem:appD-tail-regularity}
Fix $\paramSet\Subset(0,\infty)^2$ and an integer $N\ge1$, and use the transformed scalar profile equation of \cref{lem:appD-tail-normal-form}, with its constants chosen uniformly on a compact neighbourhood of $\paramSet$ in $(0,\infty)^2$.
Put $r=3N$ and $t_0=\kappa L^{4/3}$.
Let $0<\delta\le\delta_*$ and $L_0$ satisfy $t_0\ge2(r+2)$ and the bounds \eqref{eq:tail-volterra-contraction}--\eqref{eq:tail-volterra-centre-of-ball} with Lipschitz constant at most $1/2$ and $\|\tailMap0\|_{r,t_0}\le\delta/2$ on the ball of radius $\delta$, uniformly for $L\ge L_0$, $(\alpha,\mu)$ in this compact neighbourhood and $|a_{\rm s}|<\delta/4$.
Let $z$ be the fixed point of \eqref{eq:finite-base-centre-volterra}--\eqref{eq:finite-base-unstable-volterra} in that ball, and let $U_\infty=P_N+\zeta^{-5/3}(Vz)_0$ be the reconstructed profile.
Then:
\begin{enumerate}[label=(\roman*),ref=(\roman*),leftmargin=*,widest=ii]
\item\label{it:appD-tail-parameters}
For each fixed $L$, the map $(\alpha,\mu,a_{\rm s})\mapsto z(t_0\,\cdot)$ is $C^1$ in the weighted space on $[1,\infty)$ obtained from \eqref{eq:finite-base-tail-norm} by setting $t=t_0\sigma$.
The flux endpoint $(U_\infty,q_\infty,K_\infty,J_\infty)(L)$ is therefore $C^1$ in these parameters.
\item\label{it:appD-tail-expansion}
For all integers $n\ge1$ and $j\ge0$,
\[
 \partial_\zeta^j(U_\infty-P_n)=O(\zeta^{-4n-3-j})
 \qquad(\zeta\to\infty),
\]
where $P_n$ is the truncation \eqref{eq:finite-base-tail-polynomial} with $N=n$.
\end{enumerate}
\end{lemma}

\begin{proof}
\emph{Part~\ref{it:appD-tail-parameters}.}
To prove $C^1$ dependence on $p=(\alpha,\mu,a_{\rm s})$ at fixed $L$, put $\sigma=t/t_0$.
The weights in the norm \eqref{eq:finite-base-tail-norm} become $\sigma^{r+1}$ for the centre component and $\sigma^r$ for the other three.
We thus have a single Banach space on $[1,\infty)$ and a single closed ball of radius $\delta$.
Taking the estimates on a compact neighbourhood of $\paramSet$ in $(0,\infty)^2$ and restricting to $|a_{\rm s}|<\delta/4$, we obtain an open parameter set on which the Volterra map preserves this ball and has Lipschitz constant at most $1/2$ by \eqref{eq:tail-volterra-contraction}--\eqref{eq:tail-volterra-centre-of-ball}.

In these variables $\zeta=L\sigma^{3/4}$ is independent of $p$.
The coefficients of the rescaled jet field vary smoothly with $p$, and the recursion \eqref{eq:conical-tail-recursion} cancels the same powers for every $(\alpha,\mu)$.
A parameter derivative therefore preserves the decay orders in the remainder estimates \eqref{eq:tail-remainder-difference}.
On substituting $s=t_0\theta$, the centre kernel becomes $-t_0(\sigma/\theta)^2$.
A parameter derivative of an exponential kernel introduces at most a factor $C(1+t_0|\sigma-\theta|)$ in the exponential bound.
The weighted integrals in \eqref{eq:tail-exponential-kernel-bounds} still converge with this factor since $t_0\ge2(r+2)$.
These bounds justify continuous differentiation of the Volterra map in $p$ and $z$ in the common norm.
The uniform contraction theorem \cite[Theorem~1.174]{Chicone1999} now gives $C^1$ dependence of its fixed point on $p$.
Evaluation at $\sigma=1$ is continuous linear and the change from jet to flux variables \eqref{eq:pincher-flux-variables} is smooth for $U>0$, so the endpoint at $\zeta=L$ also depends $C^1$ on $p$.

\emph{Part~\ref{it:appD-tail-expansion}.}
We show first that, for each $M>N$, the same profile agrees near infinity with one constructed using the longer truncation $P_M$.
Start the integral equations at a later point $t=b$.
Replacing $P_N$ by $P_M$ translates the jet of $y$ by the known jet of $\zeta^{5/3}(P_N-P_M)$, which is $O(t^{-3N-1})$.
Choose the stable datum for $P_M$ at $b$ from the translated jet of $U_\infty$.
This datum tends to zero as $b\to\infty$.
Apply \cref{lem:appD-tail-normal-form} and the integral estimates \eqref{eq:tail-centre-kernel-bound}--\eqref{eq:tail-exponential-kernel-bounds} with $N$ replaced by $M$.
For $b$ sufficiently large, the contraction calculation \eqref{eq:tail-volterra-contraction}--\eqref{eq:tail-volterra-centre-of-ball} gives a unique fixed point and hence a profile $U_M$.
The base $b$ may depend on $M$.
By \eqref{eq:tail-fixed-point-bound} with $r=3M$, its remainder is small in $\|\cdot\|_{3M,b}$ and hence in the weaker norm $\|\cdot\|_{3N,b}$.
Translating back adds a jet of size $O(b^{-3N-1})$ in this weaker norm.
The translated solution has the same stable datum as $U_\infty$, and its decay gives the centre and growing-mode conditions at infinity.
Thus, for sufficiently large $b$, both profiles solve the same Volterra problem for $P_N$ in its contraction ball.
Uniqueness gives $U_M=U_\infty$ on this tail.

Write $z_M$ for the transformed remainder associated with $U_M-P_M$ by \eqref{eq:finite-base-tail-polynomial} and \eqref{eq:finite-base-tail-normal-form}.
The exponential integral estimates \eqref{eq:tail-exponential-kernel-bounds}, applied as in Step~4 of the proof of \cref{prop:finite-base-tail-chart}, give $z_M=O(t^{-3M-1})$.
The right-hand side of \eqref{eq:finite-base-tail-normal-form} is smooth in $w=(\kappa/t)^{1/4}$ and $z_M$ down to $w=0$, with the forcing vanishing to order $w^{12M+8}$ by \cref{lem:appD-tail-normal-form}.
Since $\partial_t w=-w/(4t)$, differentiating that equation inductively gives $\partial_t^\ell z_M=O(t^{-3M-1})$ for each fixed $\ell$.
Now $U_M-P_M=\zeta^{-5/3}y_M(t)$, where $y_M$ is the first component of $Vz_M$, and $\partial_\zeta t=(4\kappa/3)\zeta^{1/3}$.
Thus
\begin{equation}\label{eq:tail-long-truncation-derivative}
 \partial_\zeta^j(U_M-P_M)=O(\zeta^{-4M-3+j/3}).
\end{equation}
For given $n\ge1$ and $j\ge0$, choose $M>\max\{N,n\}$ with $M\ge n+\lceil j/3\rceil$.
Both the right-hand side of \eqref{eq:tail-long-truncation-derivative} and the $j$th derivative of the explicit sum $P_M-P_n$ are then $O(\zeta^{-4n-3-j})$.
Since $U_\infty=U_M$ on this tail, this proves the claim.
\end{proof}

The computer-assisted proof is contained in \cref{lem:validated-compactified-tail-bounds,lem:validated-centre-miranda-cube} and the matching argument in the \hyperref[proof:ground-state]{proof of \cref*{thm:certified-positive-conical-profile}}.
We give here the expressions evaluated by the program and show how its recorded bounds enter the two proofs.
The \hyperref[app:certificate-replay]{replay instructions} identify the archived program and the output of the run used here.

\phantomsection\label{app:certificate-tail}
For the conical tail, \path{symbolic_tail.py} differentiates the rational expression for $U^{(4)}$ in \eqref{eq:scalar-profile-field}.
For the integration from the centre we use the equivalent flux system \eqref{eq:pincher-flux-system-a}--\eqref{eq:pincher-flux-system-b}.
The recursion \eqref{eq:conical-tail-recursion}, together with \eqref{eq:forced-conical-coefficient} and \eqref{eq:forced-conical-second-coefficient}, gives the remaining coefficients of the truncation $P_4$ in \eqref{eq:finite-base-tail-polynomial} explicitly as
\begin{align}
 c_3&=-\frac{3(99985\alpha^4+10972\alpha^2+183)}
 {1024\mu^3\alpha^5(1+\alpha^2)^5},\notag\\
 c_4&=-\frac{6104617253\alpha^6+715737891\alpha^4
             +16291251\alpha^2+91661}
 {32768\mu^4\alpha^7(1+\alpha^2)^7}.
 \label{eq:appD-tail-coefficients}
\end{align}
Writing $P_4(\zeta)=\zeta p_4(\tau)$ with $\tau=\zeta^{-4}$, the routine \path{symbolic_tail.tail_residual_polynomials} expresses the right-hand side of \eqref{eq:scalar-profile-field} evaluated on $P_4$, less $P_4^{(4)}$, as $\zeta^{-3}N(\tau)/D_{\rm res}(\tau)$, where $D_{\rm res}(\tau)=p_4(\tau)^3(1+(P_4')^2)^2$.
The symbolic check and the interval evaluation of the forcing use the same polynomials $N$ and $D_{\rm res}$.

We derive the matching matrix $\mathsf A$ of \eqref{eq:validated-tail-coordinates} in the variables $x,w$ used there.
Let $\lambda_j$ be the diagonal entries of $\Lambda$ in \eqref{eq:validated-tail-system}, in the order $j=0,{\rm s},+,-$.
Let $P$ be the matrix with $j$th column $(1,\lambda_j,\lambda_j^2,\lambda_j^3)^{\mathsf T}$.
In this eigenbasis the coefficient of $x^{-1}$ is
\[
 \tailCoupling=
 \begin{pmatrix}
 3/4&15/4&15/4&15/4\\
 -1/4&-5/4&-5/4&-5/4\\
 -1/4&-5/4&-5/4&-5/4\\
 -1/4&-5/4&-5/4&-5/4
 \end{pmatrix}.
\]
Its diagonal part is $\tailDiag$ in \eqref{eq:validated-tail-system}.
Choose the matrix $T$ with $T_{ii}=0$ and $T_{ij}=-(\tailCoupling)_{ij}/(\lambda_i-\lambda_j)$ for $i\ne j$.
For the jet $\mathbf Y=(V,V_x,V_{xx},V_{xxx})^{\mathsf T}$, the change of variables $\mathbf Y=P(I+T/x)Z$ removes the off-diagonal terms of order $x^{-1}$.
In terms of the real vector $\mathbf z=(Z_0,Z_{\rm s},\Re Z_+,\Im Z_+)^{\mathsf T}$ we have $Z=C\mathbf z$, where
\[
 C=\begin{pmatrix}
 1&0&0&0\\0&1&0&0\\0&0&1&i\\0&0&1&-i
 \end{pmatrix}.
\]
Multiplying out $\mathsf A=P(I+T/x)C$ gives the matrix \eqref{eq:appD-real-transform} used for both profiles at the matching point.
The symbolic checker verifies its determinant and the inverse used in the interval evaluation.
The $x$-jet vector field $f_{\mathbf Y}$ follows from \eqref{eq:scalar-profile-field} by changing derivatives according to $\partial_\zeta=\tfrac{4}{3w}\partial_x$ and subtracting the jet of $P_4$.
Thus the real transformed field is exactly
\begin{equation}
 f_{\mathbf z}(x,\mathbf z)=\mathsf A^{-1}
       \{f_{\mathbf Y}(x,\mathsf A \mathbf z)-\mathsf A_x\mathbf z\}.
 \label{eq:appD-transformed-field}
\end{equation}
For a general background $\overline U$, applying the change of coordinates in \eqref{eq:appD-transformed-field} to the jet of $U-\overline U$ gives the corresponding field $f_Z[\overline U]$ in the complex coordinates $Z=C\mathbf z$.
At $\overline U=P_4$, the terms in \eqref{eq:validated-tail-system} are
\begin{equation}\label{eq:appD-field-remainders}
\begin{aligned}
 g(x)&=f_Z[P_4](x,0),\\
 \tailLinRem(x)&=D_Zf_Z[P_4](x,0)-(\Lambda+\tailDiag/x),\\
 \tailQuad(x,Z)&=f_Z[P_4](x,Z)-f_Z[P_4](x,0)-D_Zf_Z[P_4](x,0)Z.
\end{aligned}
\end{equation}
For the linearisation at the cone we use $\coneRem=D_Zf_Z[\alpha\zeta](x,0)-(\Lambda+\tailDiag/x)$.

In the real coordinates $\mathbf z$, the conjugate-pair norm used in \cref{lem:validated-compactified-tail-bounds} takes the form $|\mathbf z|_{\rm c}=|W^{1/2}\mathbf z|_2$ with $W=\diag(1,1,2,2)$, since $Z=C\mathbf z$.
When applying a weighted Frobenius norm to a derivative of $f_Z[\overline U]$, we use its real coordinate representation $f_{\mathbf z}[\overline U](x,\mathbf z)=C^{-1}f_Z[\overline U](x,C\mathbf z)$.
For matrices the program bounds the induced norm by $\|W^{1/2}MW^{-1/2}\|_F$.
For a bilinear map $H$ it uses the weighted Frobenius norm
\[
 \|H\|_{F,W}^2
 =\sum_{i,j,\ell}\frac{W_{ii}}{W_{jj}W_{\ell\ell}}
                 |H_{ij\ell}|^2,
\]
which bounds the corresponding induced bilinear norm.

In the recorded execution of \path{run_certificate.py}, the tail routine \path{certify_tail} in \path{tail.py} used $256$-bit Arb arithmetic on the $1024\times4\times4$ boxes partitioning the interval $0\leq w\leq737/2000$ and the $(\alpha,\mu)$ rectangle of \eqref{eq:validated-tail-parameter-box}.
Each box has exact rational endpoints.
Since $x_0^{-1/4}=20^{-1/3}<737/2000$, these boxes cover the whole tail $x\geq x_0$.
The receipt \path{certificate_receipt.json} records the exact cancellations under \path{symbolic_tail} and the interval extrema under \path{tail.extrema}.
With $\beta=57/4$ as in the forcing bound \eqref{eq:validated-tail-field-bounds}, the cancellation and forcing estimates are
\begin{equation}\label{eq:appD-run-forcing}
\begin{aligned}
 [\tau^j]N(\tau)&=0\quad(0\le j\le3),\qquad [\tau^4]N(\tau)\ne0,\\
 \sup x^{\beta+1}|g(x)|_{\rm c}&<3.101931095\times10^{14}.
\end{aligned}
\end{equation}
By the first line the residual starts at $\zeta^{-19}$.
After cancelling the factor $\tau^4$ symbolically, the routine evaluates the resulting expression for $w^{-61}g(x)$ at $\tau=w^{12}$ on each box.
The program records the resulting bound in the second line under \path{forcing}.

For the linear remainder consider the segment $U_\theta=\alpha\zeta+ \theta(P_4-\alpha\zeta)$, $0\le\theta\le1$.
Let $\delta \mathbf z$ be the jet of $P_4-\alpha\zeta$ in the real coordinates \eqref{eq:validated-tail-coordinates} and write $H_\theta=D_Z^2 f_Z[U_\theta](x,0)$ in these coordinates.
The tail run gives
\begin{equation}\label{eq:appD-run-linear}
\begin{aligned}
 \sup x^2\|\coneRem(x)\|&<27.406732,\\
 \sup w^{-9}|\delta \mathbf z|_{\rm c}&<1.92354,\\
 \sup_{0\le\theta\le1}w\|H_\theta\|_{F,W}&<8.633983,\\
 \sup x^2\|\tailLinRem(x)-\coneRem(x)\|&<16.596376,\\
 \sup x^2\|\tailLinRem(x)\|&<44.003107.
\end{aligned}
\end{equation}
The first three bounds are recorded under \path{cone_remainder}, \path{base_coordinate} and \path{segment_weighted_hessian}.
Since $x^2=w^{-8}$, the mean-value formula bounds the fourth quantity on each box by the product of the bounds for $w\|H_\theta\|_{F,W}$ and $w^{-9}|\delta \mathbf z|_{\rm c}$ there.
The program takes the maximum of these products (\path{segment_jacobian_correction}), then adds the bound for the cone remainder using the full interval endpoints.
The sum is recorded under \path{tail.derived.full_linear_constant}.

For the nonlinear remainder bound in \eqref{eq:validated-tail-field-bounds}, the complete Volterra tube is enclosed by $\mathbf z=w^{57}E$, with the component bounds
\[
 |E_j|\leq10^{14}+10^{-11}20^{19}.
\]
Write the scaled jet as $(h_0,q,h_2,h_3)=(w^3U,U',w^2U'',w^3U''')$, keeping the slope $q$ of \eqref{eq:pincher-flux-variables}.
The symbolic routine differentiates the scalar field \eqref{eq:scalar-profile-field} in these variables to form its ten weighted Hessian entries
\[
 w^{1+e_a+e_b}\partial_a\partial_b(w^4U^{(4)}),
 \qquad e=(3,-1,0,0),\qquad 1\le a\le b\le4.
\]
Substituting $h_2=w^{17}\widehat h_2$ and $h_3=w^{21}\widehat h_3$ makes these rational expressions regular at $w=0$.
The interval routine transforms the Hessian using \eqref{eq:appD-transformed-field} and encloses it separately on the tube and on the segment.
The records \path{weighted_hessian}, \path{positive_ratio} and \path{segment_positive_ratio} give
\begin{equation}\label{eq:appD-run-tube}
\begin{aligned}
 \sup_{\rm tube} w\|D_Z^2f_Z[P_4](x,Z)\|_{F,W}&<8.632301,\\
 \inf_{\rm tube} U/\zeta&>1.0369,\\
 \inf_{0\le\theta\le1} U_\theta/\zeta&>1.0369.
\end{aligned}
\end{equation}

To pass from these field bounds to the contraction estimates, we need the following bound for the diagonal solution operator.

\begin{lemma}[The tail Green operator]\label{lem:appD-tail-green-bound}
Let $\beta=57/4$, $x_0>0$ and $\kappa x_0>26$.
For the matrices $\Lambda$ and $\tailDiag$ in \eqref{eq:validated-tail-system}, let $\tailGreen H$ solve $W_x=(\Lambda+\tailDiag/x)W+H$ by integrating the stable component from $x_0$ with zero initial value and the other three components from infinity.
For every continuous $H:[x_0,\infty)\to\C^4$ with $H_0,H_{\rm s}\in\R$, $H_-=\overline{H_+}$ and finite weighted norm $\sup_{x\ge x_0}x^{\beta+1}|H(x)|_{\rm c}$, the estimate \eqref{eq:validated-volterra-operator-bound} holds, with
\[
 \Gamma_L=\left\{\frac1{(\beta+3/4)^2}
                    +\frac4{(\kappa x_0)^2}\right\}^{1/2}.
\]
\end{lemma}

\begin{proof}
The centre kernel for $y\ge x$ is $(x/y)^{3/4}$.
Its weighted integral is
\[
 x^\beta\int_x^\infty(x/y)^{3/4}y^{-\beta-1}\,\dd y
 =\frac1{\beta+3/4}=\frac1{15}.
\]
For $y\ge x$, the ratio of an unstable kernel's magnitude to the centre kernel's magnitude is
\[
 e^{-\kappa(y-x)/2}(y/x)^2\le1,
\]
because $\log(y/x)\le(y-x)/x$ and $\kappa x_0/2>13>2$.
The rotation of the conjugate pair is an isometry in $|\cdot|_{\rm c}$.
Thus the entire future centre--unstable block has weighted norm at most $1/(\beta+3/4)=1/15$.
The weighted past stable kernel satisfies
\[
 y^{-1}e^{-\kappa(x-y)}(x/y)^{13}
 \le x_0^{-1}e^{-(\kappa-13/x_0)(x-y)},\qquad x_0\le y\le x,
\]
so its integral is at most $1/(\kappa x_0-13)<2/(\kappa x_0)$.
The stable component is orthogonal to the centre and unstable components in $|\cdot|_{\rm c}$.
Adding the squares of their bounds gives \eqref{eq:validated-volterra-operator-bound}.
\end{proof}

With this operator estimate, the routine evaluates the Volterra bounds \eqref{eq:validated-volterra-ratios} over the parameter box with $|C_L|\le10^{-11}$.
Its \path{tail.derived} records \path{kappa_x0}, \path{Gamma_L}, \path{tail_lipschitz} and \path{tail_ball_ratio} give
\begin{equation}\label{eq:appD-run-contraction}
\begin{aligned}
 \inf\kappa x_0&>26.01517,\\
 \sup\Gamma_L&<0.102,\\
 L_{\rm tail}&<0.201,\\
 B_{\rm tail}&<0.742.
\end{aligned}
\end{equation}
The first inequality uses the rational lower endpoint stored in \path{kappa_x0}.

\phantomsection\label{app:certificate-centre}
We now turn to the integration from the centre and the signs on the six faces of the matching cube.
For the parameter cube \eqref{eq:validated-centre-box}, the entries of the rational matrix $B$ are
\begin{equation}\label{eq:validated-centre-matrix}
\begin{aligned}
 B_{11}&=-1.9874531069408703921803188384096016152\times10^{-8},\\
 B_{12}&=-3.89919252084496093216178406302010914955\times10^{-5},\\
 B_{13}&=2.692876593501873590740705252497522410696\times10^{-4},\\[0.4ex]
 B_{21}&=9.93565749626417479846462113618818936\times10^{-11},\\
 B_{22}&=2.812547048420073565711689937996722092401\times10^{-5},\\
 B_{23}&=-1.724902102072632493014507961965114922251\times10^{-5},\\[0.4ex]
 B_{31}&=-5.0000228133190469535136000176975175576826\times10^{-2},\\
 B_{32}&=-6.146379043385490350935388358263219194\times10^{-4},\\
 B_{33}&=1.2908993723726045408121737752290791486\times10^{-3}.
\end{aligned}
\end{equation}
In the recorded run, \path{centre_validated.py} retained the cube variables $\xi$ as degree-three Taylor-model symbols on all $400$ slabs of $[0,20]$.
On a slab with left endpoint $\zeta_*$ and length $h=1/20$, use the local integration coordinate $t=\zeta-\zeta_*$ and write the vector polynomial as $P(t,\xi)=\sum_{n=0}^{16}a_n(\xi)t^n$.
Its shifted coefficients are evaluated by the exact identity
\begin{equation}
 [s^j]P(\tau+s,\xi)
 =\sum_{n=j}^{16}\binom{n}{j}a_n(\xi)\tau^{n-j},
 \qquad 0\leq\tau\leq h.
 \label{eq:appD-shifted-polynomial}
\end{equation}
The shifted calculation retains the constant and three linear parameter terms and encloses the higher monomials and the existing remainder in a symmetric affine error.
Since $|\xi^\nu|\leq1$, this relaxation preserves inclusion.
The normalised seventeenth coefficient of the flux system \eqref{eq:pincher-flux-system-a}--\eqref{eq:pincher-flux-system-b} evaluated on the shifted series \eqref{eq:appD-shifted-polynomial}, multiplied by $h^{17}$, encloses the omitted time tail.
The defect bound $R_i$ in \eqref{eq:centre-slab-error} includes the bound $T_i$ for this time tail.
With the error radii and derivative bounds of that estimate, the program verified the following strict inclusion before accepting each slab.
\begin{equation}\label{eq:appD-run-slabs}
 E_i^0+h\left(R_i+\sum_{j=1}^4 A_{ij}E_j\right)<E_i,
 \qquad i=1,\ldots,4.
\end{equation}
The field \path{centre.integration.steps} in the receipt holds one summary per slab.
The incoming radii include the error from selecting the initial coefficients, and the endpoint radii carry the rounding error into the next slab.
The maxima over all slabs and components, recorded under \path{centre.integration} as \path{max_defect}, \path{max_time_tail} and \path{max_solution_error}, satisfy
\begin{equation}\label{eq:appD-run-slab-extrema}
\begin{aligned}
 \max R_i&<3.929609\times10^{-25},\\
 \max T_i&<2.497570\times10^{-25},\\
 \max E_i&<2.741703\times10^{-12}.
\end{aligned}
\end{equation}
The \hyperref[app:certificate-replay]{full replay} reconstructs the componentwise radii and repeats the test \eqref{eq:appD-run-slabs}.

The program enclosed $U_{\rm c}$ on every correction tube.
At $L=20$ it converted the flux variables to the jet and applied the inverse of \eqref{eq:appD-real-transform} to enclose $Z^{\rm c}$.
The receipt fields \path{centre.positivity.minimum_U_tube}, \path{centre.endpoint.centre_residual} and \path{centre.endpoint.stable_bound} give, respectively,
\begin{equation}\label{eq:appD-run-centre}
\begin{aligned}
 \inf_{\xi\in[-1,1]^3,\,0\le\zeta\le20} U_{\rm c}(\zeta;p(\xi))
   &>0.9996,\\
 \|\pi Z^{\rm c}(0)\|_\infty
   &<4.913763\times10^{-12},\\
 \sup_{\xi\in[-1,1]^3}|Z^{\rm c}_{\rm s}(\xi)|
   &<1.774939\times10^{-12}.
\end{aligned}
\end{equation}
At $\xi=0$ the bound includes the Taylor-model remainder as well as the constant coefficient.

On each face $\xi_i=\sigma$ the program enclosed $(\pi Z^{\rm c})_i$ by ranging the monomials and including the global remainder.
It then subtracted the tail interval of radius $10^{14}/20^{19}$ from \eqref{eq:validated-tail-base-displacement}.
For the mismatch \eqref{eq:validated-exact-mismatch}, the entries below are strict lower bounds for $\inf_{\xi_i=\sigma}\sigma M_i(\xi)$.
\begin{equation}\label{eq:appD-run-faces}
\begin{array}{@{}ccc@{}}
\toprule
 i & \sigma=-1 & \sigma=+1\\
\midrule
 1 & 1.41413\times10^{-11} & 7.18941\times10^{-12}\\
 2 & 1.12970\times10^{-11} & 1.03852\times10^{-11}\\
 3 & 1.57006\times10^{-11} & 6.01275\times10^{-12}\\
\bottomrule
\end{array}
\end{equation}
The six entries are rounded down from the rational lower endpoints of \path{signed_mismatch} in \path{centre.faces}.
The receipt numbers the components $0,1,2$.
These are the signs we use in the \hyperref[proof:ground-state]{proof of \cref*{thm:certified-positive-conical-profile}} to apply the Poincar\'e--Miranda theorem, after which uniqueness for the profile ODE identifies the two solutions.

\phantomsection\label{app:certificate-replay}
The accompanying ancillary archive \path{conical_pincher_certificate.zip} contains the source code and the $256$-bit Arb generators.
It also includes all $400$ centre-step records and the canonical receipt \path{certificate_receipt.json}.
The unchanged recorded calculation is also available in the \href{https://glenw83.github.io/conical_pincher_certificate.zip}{original online archive}.
Implementation and receipt-format details are given in \path{documentation/README.md} in the accompanying archive.
The SHA--256 digests below identify this archive, the unchanged receipt and its recorded source manifest.
\begin{center}
\small
\begin{tabular}{@{}l@{\quad}l@{}}
Archive & \texttt{de3d29b91e7cb0b5754d5de9336ebb9d26b2f0b05e31f706b3e714159a6e4e88}\\
Receipt & \texttt{d4d46f248323d23c213cb5eb14c749703120c3f7342f47a2b3e9aec46d222ce7}\\
Source manifest & \texttt{9167a9a19b189be05231b7963282c58069890f378766d6ca6d27dad1e12330e7}
\end{tabular}
\end{center}

The numerical enclosures use outward-rounded interval analysis \cite{Moore1966}, high-order Taylor models \cite{BerzMakino1998} and arbitrary-precision ball arithmetic \cite{Johansson2017}.
They provide the Volterra contraction bounds \eqref{eq:appD-run-contraction} used in \cref{lem:validated-compactified-tail-bounds} and the face signs \eqref{eq:appD-run-faces} used in the Poincar\'e--Miranda argument \cite{Miranda1940} in the proof of \cref{thm:certified-positive-conical-profile}.
For related validated ODE and connecting-orbit frameworks see \cite{LessardReinhardt2014,vanDenBergSheombarsing2021}.

The recorded environment is CPython $3.12.6$, python-flint $0.8.0$ and SymPy $1.13.3$, with $256$-bit Arb precision in both interval stages.
To check the archived receipt and its bounds from an unpacked copy of the certificate directory, run
\begin{verbatim}
python verify_receipt.py certificate_receipt.json
\end{verbatim}
This standard-library check verifies the stored bounds and provenance without repeating the interval calculations.
To recompute the tail coefficients from \eqref{eq:conical-tail-recursion}, the bounds on all $1024\times4\times4$ tail boxes and the $400$ centre steps of order $16$, run the following commands with CPython $3.12$ in a fresh copy of the directory.
\begin{verbatim}
python -m pip install -r requirements.txt
python run_certificate.py
python verify_receipt.py certificate_receipt.json
\end{verbatim}
Run Python with \path{PYTHONOPTIMIZE} unset and without \path{-O} or \path{-OO}, so that the assertions in the programs are active.
Each acceptance test compares an outward-rounded bound with an exact rational threshold.
If a strict inequality fails, the program rejects the certificate.

The code and full machine records are supplied as ancillary data, while all mathematical reductions needed to interpret them appear in \cref{sec:profile-new} and this appendix.

\end{document}